\documentclass[11pt,twoside]{article}
\usepackage[colorlinks=true, linkcolor=blue, citecolor = blue]{hyperref}
\newcommand{\ifims}[2]{#1}   
\newcommand{\ifAMS}[2]{#1}   
\newcommand{\ifau}[3]{#1}  

\newcommand{\ifbook}[2]{#1}   

\ifbook{

  }
  {

  }

\def\thetitle{Finite sample and asymptotic behavior of a quasi profile estimator for the single index model}
\def\thanksa
{The first author is partially supported by
Laboratory for Structural Methods of Data Analysis in Predictive Modeling, MIPT, 
RF government grant, ag. 11.G34.31.0073.
Financial support by the German Research Foundation (DFG) through the Collaborative Research Center 649 ``Economic Risk'' is gratefully acknowledged. The second author is gratefull for financial support by the German Research Foundation (DFG) through the Collaborative Research Center 1735 ``Structural Inference in Statistics: Adaption and Efficiency''}
\def\theruntitle{Finite sample single index estimation}

\def\theabstract{
We apply the results of \cite{AASP2013} and \cite{AASPalternating} to analyse a sieve profile quasi maximum likelihood 
estimator in the single index model with linear index function. The link function is approximated with \(C^3\)-Daubechies-wavelets with compact support. We derive results like Wilks phenomenon and Fisher Theorem in a finite sample setup. Further we show that an alternation maximization procedure converges to the global maximizer and assess the performance of a projection pursuit procedure in that context.
The approach is based on showing that the conditions of \cite{AASP2013} and \cite{AASPalternating} can be satisfied under a set of mild regularity and moment conditions on the index function, the regressors and the additive noise. This allows to construct nonasymptotic confidence sets and to derive asymptotic bounds for the estimator as corollaries.}

\def\kwdp{62F10}
\def\kwds{62J12,62F25,62H12}

\def\thekeywords{maximum likelihood, local quadratic breaketing, spread,
concentration, semiparametric, single index}

\def\authora{Andresen Andreas}
\def\runauthora{andresen a.}
\def\addressa{
    Weierstrass-Institute, \\
    Mohrenstr. 39, 10117 Berlin, Germany     
    }
\def\emaila{andresen@wias-berlin.de}
\def\affiliationa{Weierstrass-Institute and Humboldt University Berlin}
\def\thanksa{The author is supported by Research Units 1735 
"Structural Inference in Statistics: Adaptation and Efficiency"
}

\usepackage{amsmath,amssymb,amsthm}
\usepackage[numbers]{natbib}
\usepackage{comment}
\usepackage{color}
\usepackage{srcltx}
\usepackage[mathscr]{eucal}
\usepackage[math]{easyeqn}
\usepackage{etoolbox}

\ifims{
\textheight=23cm
\textwidth=14.8cm
\topmargin=0pt
\oddsidemargin=1.0cm
\evensidemargin=1.0cm
\linespread{1.3}
\renewenvironment{abstract}
    {\centerline{\textbf{Abstract}}\bigskip
      \begin{center}
       \begin{minipage}{11cm}
        \begin{small}
    }
    {   \end{small}
       \end{minipage}
      \end{center}
     \bigskip
    }

}{ 
}

\numberwithin{equation}{section}
\numberwithin{figure}{section}
\newcounter{example}[section]
\numberwithin{example}{section}
\newcounter{remark}[section]
\numberwithin{remark}{section}
\newtheorem{theorem}{Theorem}[section]
\newtheorem{proposition}[theorem]{Proposition}
\newtheorem{lemma}[theorem]{Lemma}
\newtheorem{corollary}[theorem]{Corollary}

\newtheorem{exmp}[example]{Example}
\newtheorem{rmrk}[remark]{Remark}
\newenvironment{example}{\begin{exmp}\rm}{\end{exmp}}
\newenvironment{remark}{\begin{rmrk}\rm}{\end{rmrk}}

\begin{document}
\thispagestyle{empty}
\ifims{
\title{\thetitle}
\ifau{ 
  \author{
    \authora
    \ifdef{\thanksa}{\thanks{\thanksa}}{}
    \\[5.pt]
    \addressa \\
    \texttt{ \emaila}
  }
}
{  
  \author{
    \authora
    \ifdef{\thanksa}{\thanks{\thanksa}}{}
    \\[5.pt]
    \addressa \\
    \texttt{ \emaila}
    \and
    \authorb
    \ifdef{\thanksb}{\thanks{\thanksb}}{}
    \\[5.pt]
    \addressb \\
    \texttt{ \emailb}
  }
}
{   
  \author{
    \authora
    \ifdef{\thanksa}{\thanks{\thanksa}}{}
    \\[5.pt]
    \addressa \\
    \texttt{ \emaila}
    \and
    \authorb
    \ifdef{\thanksb}{\thanks{\thanksb}}{}
    \\[5.pt]
    \addressb \\
    \texttt{ \emailb}
    \and
    \authorc
    \ifdef{\thanksc}{\thanks{\thanksc}}{}
    \\[5.pt]
    \addressc \\
    \texttt{ \emailc}
  }
}

\maketitle
\pagestyle{myheadings}
\markboth
 {\hfill \textsc{ \small \theruntitle} \hfill}
 {\hfill
 \textsc{ \small
 \ifau{\runauthora}
      {\runauthora and \runauthorb}
      {\runauthora, \runauthorb, and \runauthorc}
 }
 \hfill}
\begin{abstract}
\theabstract
\end{abstract}

\ifAMS
    {\par\noindent\emph{AMS 2000 Subject Classification:} Primary \kwdp. Secondary \kwds}
    {\par\noindent\emph{JEL codes}: \kwdp}

} 
{ 
\begin{frontmatter}
\title{\thetitle}


\runtitle{\theruntitle}

\ifau{ 
\begin{aug}
    \author{\authora\ead[label=e1]{\emaila}}
    \address{\addressa \\
     \printead{e1}}
\end{aug}

 \runauthor{\runauthora}
\affiliation{\affiliationa} }
{ 
\begin{aug}
    \author{\authora\ead[label=e1]{\emaila}\thanksref{t21}}
    \and
    \author{\authorb\ead[label=e2]{\emailb}\thanksref{t22}}
    
    \address{\addressa \\
     \printead{e1}}
    \address{\addressb \\
     \printead{e2}}
    \thankstext{t21}{\thanksa}
    \thankstext{t22}{\thanksb}
    \affiliation{\affiliationa, \affiliationb} 
    \runauthor{\runauthora and \runauthorb}
\end{aug}
} 
{ 
\begin{aug}
    \author{\authora\ead[label=e1]{\emaila}\thanksref{t21}}
    \and
    \author{\authorb\ead[label=e2]{\emailb}\thanksref{t22}}
    \and
    \author{\authorc\ead[label=e3]{\emailc}\thanksref{t23}}
    
    \address{\addressa \\
     \printead{e1}}
    \address{\addressb \\
     \printead{e2}}
    \address{\addressc \\
     \printead{e3}}
    \thankstext{t21}{\thanksa}
    \thankstext{t22}{\thanksb}
    \thankstext{t23}{\thanksc}
    \affiliation{\affiliationa, \affiliationb, \affiliationc} 
    \runauthor{\runauthora, \runauthorb, and \runauthorc}
\end{aug}}

\begin{abstract}
\theabstract
\end{abstract}



\end{frontmatter}
} 

\newcommand{\bb}[1]{\boldsymbol{#1}}


\renewcommand{\(}{$\,}
\renewcommand{\)}{\,$}

\def\btri{\vfill{\( \blacktriangleright \) }}
\def\btrir{\vfill{\( \blacktriangleright \) }}

\def\nquad{\hspace{-1cm}}
\def\eqdef{\stackrel{\operatorname{def}}{=}}
\def\tod{\stackrel{d}{\longrightarrow}}
\def\tow{\stackrel{w}{\longrightarrow}}
\def\toP{\stackrel{\P}{\longrightarrow}}

\newcommand{\cc}[1]{\mathscr{#1}}
\renewcommand{\bar}[1]{\overline{#1}}
\renewcommand{\hat}[1]{\widehat{#1}}
\renewcommand{\tilde}[1]{\widetilde{#1}}

\renewcommand{\Gamma}{\varGamma}
\renewcommand{\Pi}{\varPi}
\renewcommand{\Sigma}{\varSigma}
\renewcommand{\Delta}{\varDelta}
\renewcommand{\Lambda}{\varLambda}
\renewcommand{\Psi}{\varPsi}
\renewcommand{\Phi}{\varPhi}
\renewcommand{\Theta}{\varTheta}
\renewcommand{\Omega}{\varOmega}
\renewcommand{\Xi}{\varXi}
\renewcommand{\Upsilon}{\varUpsilon}
\def\nn{\nonumber \\}

\def\suml{\sum\limits}
\def\supl{\sup\limits}
\def\maxl{\max\limits}
\def\infl{\inf\limits}
\def\intl{\int\limits}
\def\liml{\lim\limits}
\def\Cov{\operatorname{Cov}}
\def\Var{\operatorname{Var}}
\def\arginf{\operatornamewithlimits{arginf}}
\def\argsup{\operatornamewithlimits{argsup}}
\def\argmax{\operatornamewithlimits{argmax}}
\def\argmin{\operatornamewithlimits{argmin}}
\def\val{\operatorname{val}}

\def\D{\boldsymbol{D}}
\def\dd{\operatorname{d}}
\def\tr{\operatorname{tr}}
\def\I{I\!\!I}
\def\R{I\!\!R}
\def\E{I\!\!E}
\def\P{I\!\!P}
\def\X{\mathfrak{X}}
\def\Const{\mathrm{Const.} \,}
\def\cdt{\boldsymbol{\cdot}}
\def\tm{\!\times\!}
\def\T{\top}
\def\diag{\operatorname{diag}}
\def\diam{\operatorname{diam}}
\def\rank{\operatorname{rank}}
\def\loc{\operatorname{loc}}

\def\av{\bb{a}}
\def\bv{\bb{b}}
\def\cv{\bb{c}}
\def\dv{\bb{d}}
\def\ev{\bb{e}}
\def\fv{\bb{f}}
\def\gv{\bb{g}}
\def\hv{\bb{h}}
\def\iv{\bb{i}}
\def\jv{\bb{j}}
\def\kv{\bb{k}}
\def\lv{\bb{l}}
\def\mv{\bb{m}}
\def\nv{\bb{n}}
\def\ov{\bb{o}}
\def\pv{\bb{p}}
\def\qv{\bb{q}}
\def\rv{\bb{r}}
\def\sv{\bb{s}}
\def\tv{\bb{t}}
\def\uv{\bb{u}}
\def\vv{\bb{v}}
\def\wv{\bb{w}}
\def\xv{\bb{x}}
\def\yv{\bb{y}}
\def\zv{\bb{z}}

\def\Cv{\bb{C}}
\def\Gv{\bb{G}}
\def\Mv{\bb{M}}
\def\Sv{\bb{S}}
\def\Uv{\bb{U}}
\def\Xv{\bb{X}}
\def\Yv{\bb{Y}}
\def\Zv{\bb{Z}}

\def\alphav{\bb{\alpha}}
\def\epsv{\bb{\varepsilon}}
\def\etav{\bb{\eta}}
\def\gammav{\bb{\gamma}}
\def\varepsilonv{\bb{\varepsilon}}
\def\phiv{\bb{\phi}}
\def\psiv{\bb{\psi}}
\def\tauv{\bb{\tau}}
\def\upsilonv{\bb{\upsilon}}
\def\xiv{\bb{\xi}}
\def\zetav{\bb{\zeta}}

\def\Psiv{\bb{\Psi}}
\def\CONST{\mathtt{C}}

\def\itemv{\vfill\item}
\newenvironment{myslide}[1]
    {\begin{frame}\frametitle{#1}\vfill}
    {\vfill\end{frame}}

\def\vsp{\vspace{0.05\textheight} \vfill}
\def\summarysign{\resizebox{0.08\textwidth}{0.08\textheight}{\includegraphics{summary}}\,}
\def\nix{}
\def\wpu{$\bullet$}

\def\gps{s}
\def\GK{\cc{G}}
\def\Excgr{\diamondsuit}

\def\rdl{\epsilon}
\def\rd{\bb{\rdl}}
\def\rddelta{\delta}
\def\rdomega{\omega}
\def\rddeltab{\rddelta^{*}}
\def\rhor{\omega}
\def\rhorb{\rhor^{*}}

\def\Span{\operatorname{span}}
\def\span{\operatorname{span}}
\def\Exc{{\square}}
\def\UUs{U_{\circ}}
\def\errbm{\errb^{*}}
\def\corrDF{\nu}
\def\BBr{\breve{\BB}}
\def\taua{\tau}
\def\AssId{\mathcal{I}}
\def\assId{\iota}
\def\AFD{\cc{A}}

\def\dimp{p}
\def\dimtotal{\dimp^{*}}

\def\cdimb{\mathfrak{c}_{1}}
\def\pnnd{\mathfrak{u}}
\def\Lmgf{\mathfrak{M}}
\def\Lmgfb{\Lmgf^{*}}
\def\lambdam{\gm_{1}}
\def\lambdaB{{\lambda}^{*}}
\def\lambdac{{\lambda'}}

\def\LL{\cc{L}}

\def\rdb{\rd}
\def\rdm{\underline{\rdb}}

\def\ND{\cc{N}}

\def\dimh{m}
\def\LCS{C}
\def\Bc{B_{0}}
\def\AF{A}
\def\CF{C}
\def\Ab{A_{\rdb}}
\def\Am{A_{\rdm}}
\def\DPrp{\DPr_{\dimh}}
\def\DPrb{\DPr_{\rdb}}
\def\DPrm{\DPr_{\rdm}}
\def\Cb{\cc{C}_{\rdb}}
\def\Cm{\cc{C}_{\rdm}}
\def\Ub{\cc{U}_{\rdb}}
\def\xivrb{\breve{\xiv}_{\rd}}
\def\VPrb{\breve{\VP}_{\rdb}}
\def\Larb{\breve{\La}_{\rdb}}
\def\Lar{\breve{\La}}
\def\Larm{\breve{\La}_{\rdm}}
\def\VH{Q}
\def\VHc{\VH_{0}}

\def\zetavr{\breve{\zetav}}
\def\etavr{\breve{\etav}}
\def\xivr{\breve{\xiv}}
\def\zetavrm{\zetavr_{\rdm}}

\def\fvh{\bb{\dimh}}
\def\N{\mathbb{N}}
\def\Z{\mathbb{Z}}

\def\iic{\IF}
\def\iif{\breve{\iic}}
\def\DP{\text{D}}
\def\HH{\text{H}}
\def\A{\text{A}}
\def\ifc{\breve{\iic}}

\def\La{\mathbb{L}}
\def\Lab{\La_{\rdb}}
\def\Lam{\La_{\rdm}}

\def\DP{D}
\def\DPc{\DP_{0}}
\def\DPb{\DP_{\rdb}}
\def\DPm{\DP_{\rdm}}

\def\LabGP{\La_{\rdb,\GP}}
\def\LamGP{\La_{\rdm,\GP}}

\def\DPbGP{\DP_{\rdb,\GP}}
\def\DPmGP{\DP_{\rdm,\GP}}
\def\riskbGP{\riskt_{\rdb,\GP}}

\def\gmi{\mathtt{b}}
\def\gmiid{\mathtt{g}_{1}}
\def\kullbi{\Bbbk}
\def\Thetasi{\Theta_{\loc}}
\def\rri{\mathtt{u}}
\def\rris{\rri_{0}}

\def\Ipc{\bb{\mathrm{f}}}
\def\IF{\Bbb{F}}
\def\IFc{\IF_{0}}
\def\IFb{\IF_{\rdb}}
\def\IFm{\IF_{\rdm}}

\def\DF{\cc{D}}
\def\DFc{\DF_{0}}
\def\DFb{\DF_{\rdb}}
\def\DFm{\breve{\DF}_{\rd}}
\def\DFm{\DF_{\rdm}}

\def\DPr{\breve{\DP}}
\def\VF{\cc{V}}
\def\VFc{\VF_{0}}

\def\VP{V}
\def\VPc{\VP_{0}}

\def\HHc{\HH_{0}}
\def\HHb{\HH_{\rd}}
\def\HHm{\HH_{\rdm}}

\def\xib{\xi^{*}}
\def\xivb{\xiv_{\rdb}}
\def\xivm{\xiv_{\rdm}}

\def\Pdom{\mu_{0}}
\def\PDOM{\bb{\mu}_{0}}

\def\thetav{\bb{\theta}}
\def\thetavs{\thetav^{*}}
\def\thetavc{\thetav'}
\def\thetavd{\thetav^{\circ}}
\def\thetavdc{\thetav^{\sharp}}
\def\dthetavs{\thetav,\thetavs}

\def\thetas{\theta^{*}}
\def\thetac{\theta'}
\def\thetad{\theta^{\circ}}
\def\thetab{\theta^{\dag}}
\def\thetavb{\thetav^{\dag}}

\def\vtheta{\vartheta}
\def\vthetav{\bb{\vtheta}}
\def\prior{\Pi}

\def\Gam{\Xi}
\def\Gam{\mathcal{S}}
\def\RG{R}
\def\Psu{\Upsilon}
\def\Phim{\breve{\Phi}}

\def\Proj{P}

\def\gammav{\bb{\gamma}}
\def\gammavs{\gammav^{*}}
\def\gammavd{\gammav^{\circ}}
\def\etavs{\etav^{*}}
\def\etavd{\etav^{\circ}}
\def\etavc{\etav'}

\def\taus{\tau_{0}}
\def\taup{\lceil \tau \rceil}

\def\sigmas{{\sigma^{*}}}
\def\Sigmas{\Sigma_{0}}

\def\upsilonc{\upsilon'}
\def\upsilond{\upsilon^{\circ}}
\def\upsilonp{{\upsilon}^{*}}
\def\upsilonm{{\upsilon}_{*}}
\def\upsilonvs{\upsilonv^{*}}
\def\upsilons{\upsilon^{*}}
\def\upsilonb{\bar{\upsilon}}
\def\upsilonvd{\upsilonv^{\circ}}

\def\ups{\bb{\upsilon}}
\def\upss{\ups^*}
\def\upsc{\ups^{\prime}}
\def\upsd{\ups^{\circ}}
\def\upsdc{\ups^{\sharp}}
\def\upsdu{\ups^{\flat}}

\def\Ups{\varUpsilon}
\def\Upsd{\Ups^{\circ}}
\def\Upss{\Ups_{\circ}}
\def\UpsP{\Ups^{c}}

\def\Thetas{\Theta_{0}}
\def\ThetasGP{\Theta_{0,\GP}}
\def\varthetav{\bb{\vartheta}}

\def\glink{g}

\def\fs{f}
\def\fb{\fv^{\dag}}

\def\wv{\bb{w}}
\def\varthetav{\bb{\vartheta}}
\def\Lr{\breve{L}}
\def\zetavr{\breve{\zetav}}
\def\etavr{\breve{\etav}}
\def\xivr{\breve{\xiv}}

\def\err{\diamondsuit}
\def\errbm{\bar{\err}_{\rdomega}}
\def\errm{\err_{\rdm}}
\def\errb{\err_{\rdb}}

\def\deficiency{\Delta}
\def\spread{\Delta}

\def\fis{\mathfrak{a}}

\def\Upss{\Ups_{\circ}}

\def\Thetas{\Theta_{0}}

\def\rr{\mathtt{r}}
\def\rups{\rr_{0}}

\def\Cs{E}
\def\Csd{\Cs^{\circ}}
\def\Ca{A}
\def\CS{\cc{E}}
\def\CA{\cc{A}}
\def\CAb{\CA_{\rd}}
\def\CAC{\CA_{\CoFu}}

\def\ex{\mathrm{e}}
\def\entrl{\mathbb{Q}}
\def\entrlb{\entrl}
\def\entr{\entrl}

\def\nunu{\nu_{0}}

\def\mub{\mu^{*}}
\def\mubc{\mu}
\def\mubcb{\mubc^{*}}
\def\Mubc{\mathbb{M}}
\def\Mubcb{\mathrm{M}}

\def\penr{\operatorname{pen}}
\def\pen{\mathfrak{t}}

\def\Thetathetav#1{\substack{\\[0.1pt] \upsilonv\in\Theta \\[1pt] \Proj \upsilonv = #1}}
\def\Span{\operatorname{span}}
\def\Exc{{\square}}
\def\UUs{U_{\circ}}
\def\errbm{\errb^{*}}
\def\corrDF{\rho}
\def\BBr{\breve{\BB}}
\def\taua{\tau}
\def\AssId{\mathcal{I}}
\def\AFD{\cc{A}}

\def\BanX{\cc{X}}
\def\BanY{\cc{Y}}
\def\BanZ{\cc{Z}}
\def\basX{\ev}
\def\apprX{\alpha}
\def\fvs{\fv^{*}}
\def\lkh{\ell}
\def\Bc{B_{0}}
\def\h{\frac{1}{2}}
\def\basis{\ev}
\def\Proj{\Pi_{0}}
\def\Projh{\Pi_{\dimh}}

\def\Ij{\mathcal{I}}

\def\Mn{M_{\nsize}}
\def\bA{\breve{A}}
\def\cA{\bA_{\dimh}}

\def\Sdr{\cc{S}}
\def\xxn{\xx_{\nsize}}

\def\CONST{\mathtt{C}}
\def\Ij{\mathcal{I}}
\def\etas{\eta^{*}}
\def\zetavs{\zetav^{*}}
\def\zetavc{\zetav'}

\def\omegav{\bb{\phi}}
\def\omegavs{\omegav^{*}}
\def\omegavc{\omegav'}

\def\dimn{\dimp_{\nsize}}
\def\betan{\beta_{\nsize}}

\def\BB{I\!\!B}
\def\Id{I\!\!\!I}

\def\wv{\bb{w}}
\def\varthetav{\bb{\vartheta}}
\def\Lr{\breve{L}}

\def\zetav{\bb{\zeta}}
\def\zetavr{\breve{\zetav}}
\def\etavr{\breve{\etav}}
\def\xivr{\breve{\xiv}}

\def\rdb{\rd}
\def\rdm{\underline{\rdb}}

\def\taub{\tau_{\rdb}}
\def\taum{\tau_{\rdm}}
\def\kappab{\kappa_{\rd}}
\def\deltab{\delta_{\rd}}

\def\taubGP{\tau_{\rdb,\GP}}
\def\taumGP{\tau_{\rdm,\GP}}
\def\kappabGP{\kappa_{\rd,\GP}}
\def\deltabGP{\delta_{\rd,\GP}}
\def\nubm{\nu_{\rd}}
\def\uub{u_{\rd}}
\def\uubGP{u_{\rd,\GP}}
\def\nubmGP{\nu_{\rd, G}}

\def\rG{\rd,\GP}

\def\LinSp{\mathrm{L}}
\def\Id{I\!\!\!I}
\def\Ind{\operatorname{1}\hspace{-4.3pt}\operatorname{I}}

\def\BG{\mathcal{R}}
\def\bg{r}
\def\fmup{\phi}
\def\rg{r}
\def\uc{u_{c}}
\def\muc{\mu_{c}}
\def\mud{\mu_{0}}
\def\xxd{\xx_{0}}
\def\yyd{\yy_{0}}
\def\gmd{\gm_{0}}

\def\ms{m^{*}}
\def\Inv{A}
\def\InvT{\Inv^{\T}}
\def\Invt{\tilde{\Inv}}

\def\ssize{N}
\def\nsize{{n}}

\def\rhor{\omega}

\def\LT{L}
\def\LGP{\LT_{\GP}}
\def\La{\mathbb{L}}
\def\Lab{\La_{\rdb}}
\def\Lam{\La_{\rdm}}

\def\DP{D}
\def\DPc{\DP_{0}}
\def\DPb{\DP_{\rdb}}
\def\DPm{\DP_{\rdm}}

\def\LabGP{\La_{\rdb,\GP}}
\def\LamGP{\La_{\rdm,\GP}}

\def\DPbGP{\DP_{\rdb,\GP}}
\def\DPmGP{\DP_{\rdm,\GP}}
\def\riskbGP{\riskt_{\rdb,\GP}}

\def\gmi{\mathtt{b}}
\def\gmiid{\mathtt{g}_{1}}
\def\kullbi{\Bbbk}
\def\Thetasi{\Theta_{\loc}}
\def\rri{\mathtt{u}}
\def\rris{\rri_{0}}

\def\Ipc{\bb{\mathrm{f}}}
\def\IF{\Bbb{F}}
\def\IFc{\IF_{0}}
\def\IFb{\IF_{\rdb}}
\def\IFm{\IF_{\rdm}}

\def\DF{\cc{D}}
\def\DFc{\DF_{0}}
\def\DFb{\DF_{\rdb}}
\def\DFm{\breve{\DF}_{\rd}}
\def\DFm{\DF_{\rdm}}

\def\DPr{\breve{\DP}}
\def\VF{\cc{V}}
\def\VFc{\VF_{0}}

\def\HHc{\HH_{0}}
\def\HHb{\HH_{\rd}}
\def\HHm{\HH_{\rdm}}

\def\xib{\xi^{*}}
\def\xivb{\xiv_{\rdb}}
\def\xivm{\xiv_{\rdm}}
\def\CAm{\underline{\CA}}
\def\CAb{\CA}

\def\penr{\operatorname{pen}}
\def\pen{\mathfrak{t}}
\def\PEN{\operatorname{PEN}}
\def\RSS{\operatorname{RSS}}
\def\med{\operatorname{med}}

\def\ex{\mathrm{e}}
\def\entrl{\mathbb{Q}}
\def\entrlb{\entrl}
\def\entr{\entrl}

\def\kullb{\cc{K}} 
\def\kullbc{\kullb^{c}}

\def\gm{\mathtt{g}}
\def\gmc{\gm_{c}}
\def\gmb{\gm}
\def\gmbm{\gmb_{1}}

\def\yy{\mathtt{y}}
\def\yyc{\yy_{c}}
\def\xx{\mathtt{x}}
\def\xxc{\xx_{c}}
\def\tc{t_{c}}

\def\alp{\alpha}
\def\alpn{\rho}
\def\gmu{\mathfrak{a}}

\def\losst{\varrho}
\def\loss{\wp}
\def\lossp{u}
\def\closs{g}

\def\riskt{\cc{R}}
\def\emprisk{\ell}
\def\bias{b}
\def\bern{q}

\def\TT{\nsize}

\def\Pone{P}
\def\Pf{\P_{f(\cdot)}}
\def\Ef{\E_{f(\cdot)}}
\def\Ps{\P_{\thetas}}
\def\Es{\E_{\thetas}}
\def\Pu{\P_{\upsilons}}
\def\Eu{\E_{\upsilons}}

\def\Pvs{\P_{\thetavs}}
\def\Evs{\E_{\thetavs}}

\def\UPd{w}
\def\nunup{\nu_{1}}
\def\rru{\rr_{1}}
\def\rups{\rr_{0}}
\def\rupsb{\rups^{*}}
\def\rrf{\rr^{\flat}}

\def\smooths{\mathbb{S}}
\def\smooth{\smooths_{1}}

\def\elli{\bar{\ell}}

\def\K{K}

\def\Psir{\breve{\Psi}}

\def\af{a}
\def\afs{\af^{*}}

\def\kapla{\varkappa}

\newcommand{\mlew}[1]{\tilde{\thetav}_{#1}}
\newcommand{\mlea}[1]{\hat{\thetav}_{#1}}
\newcommand{\mluw}[1]{\tilde{\theta}_{#1}}
\newcommand{\mlua}[1]{\hat{\theta}_{#1}}
\newcommand{\penm}[1]{\boldsymbol{m}_{#1}}

\def\Pdom{\mu_{0}}
\def\PDOM{\bb{\mu}_{0}}
\def\EDOM{\E_{0}}

\def\mk{m}
\def\Mk{\cc{M}}
\def\SV{\cc{S}}

\def\Cs{E}
\def\Csd{\Cs^{\circ}}
\def\Ca{A}
\def\CS{\cc{E}}
\def\CA{\cc{A}}
\def\CAb{\CA_{\rd}}
\def\CAC{\CA_{\CoFu}}

\def\Ccb{m_{\rdb}}
\def\Ccm{m_{\rdm}}
\def\CcbGP{m_{\rdb,\GP}}
\def\CcmGP{m_{\rdm,\GP}}

\def\etas{\eta^{*}}

\def\omegav{\bb{\phi}}
\def\omegavs{\omegav^{*}}
\def\omegavc{\omegav'}

\def\nunu{\nu_{0}}
\def\numu{\nu_{1}}
\def\nupi{\nu^{+}}
\def\nubu{\beta}

\def\nus{\nu}
\def\nusb{\nus}
\def\nusr{\nus^{\bracketing}}
\def\Nusb{\mathbb{N}}
\def\Nusr{\mathbb{N}^{\diamond}}

\def\dist{d}
\def\distd{\mathfrak{a}}

\def\hatk{\kappa}
\def\ko{k^{\circ}}

\def\qqq{\mathfrak{q}}
\def\ppp{{s}}
\def\Cqq{C(\qqq)}
\def\Cqqb{C^{\diamond}(\qqq)}
\def\Crho{C(\mrho)}
\def\Cqqm{\log(4)}
\def\Cqpr{(\qqq \rrp + \dimp / 2)}

\def\Cdima{\mathfrak{e}_{0}}
\def\Cdimb{\mathfrak{e}_{1}}
\def\cdima{\mathfrak{c}_{0}}
\def\cdimb{\mathfrak{c}_{1}}
\def\cdim{\mathfrak{c}}

\def\deltaD{\delta}
\def\alphai{\alpha_{1}}
\def\alphaii{\alpha_{2}}
\def\alphaiii{\alpha_{3}}
\def\alphaiv{\alpha_{4}}

\def\err{\diamondsuit}
\def\errbm{\bar{\err}_{\rdomega}}
\def\errm{\err_{\rdm}}
\def\errb{\err_{\rdb}}

\def\errbGP{\err_{\rdomega,\GP}}
\def\errmGP{\err_{\rdm,\GP}}
\def\errbmGP{\bar{\err}_{\rd,\GP}}

\def\errs{\err_{\rdomega}^{*}}
\def\deltas{\alpha}

\def\xivbGP{\xiv_{\rdb,\GP}}
\def\xivmGP{\xiv_{\rdm,\GP}}

\def\SP{S}
\def\GP{G}
\def\GPt{\GP_{0}}
\def\GPn{\GP_{1}}
\def\gp{g}
\def\gs{s}

\def\errbGP{\err_{\rdb,\GP}}
\def\errmGP{\err_{\rdm,\GP}}
\def\errpmGP{\err_{\GP}^{\pm}}

\def\LCS{\cc{C}}

\def\DPGP{\DP_{\GP}}
\def\thetavsGP{\thetavs_{\GP}}

\def\LL{\cc{L}}
\def\LLb{\LL^{*}}
\def\LLh{\cc{L}}

\def\YY{\cc{Y}}
\def\LP{L^{\circ}}

\def\modcnrd{\mathfrak{A}}

\def\pens{\pi}
\def\pnn{\mathfrak{g}}
\def\pnnd{\mathfrak{u}}
\def\pnndGP{\pnnd_{\GP}}

\def\confpr{\mathfrak{c}}
\def\confprb{\confpr^{*}}

\def\pn{\pens^{*}}
\def\penInt{\mathfrak{D}}
\def\penH{\mathbb{H}}
\def\pmu{\mathfrak{u}}
\def\Closs{\cc{R}}

\def\dimp{p}
\def\riskb{\riskt_{\rdb}}
\def\dimpp{\dimp+1}
\def\BB{I\!\!B}
\def\vA{\mathtt{v}}

\def\deficiency{\Delta}
\def\spread{\Delta}
\def\dimtotal{\dimp^{*}}

\def\thetav{\bb{\theta}}
\def\thetavs{\thetav^{*}}
\def\thetavc{\thetav'}
\def\thetavd{\thetav^{\circ}}
\def\thetavdc{\thetav^{\sharp}}
\def\dthetavs{\thetav,\thetavs}

\def\thetas{\theta^{*}}
\def\thetac{\theta'}
\def\thetad{\theta^{\circ}}
\def\thetab{\theta^{\dag}}
\def\thetavb{\thetav^{\dag}}

\def\vtheta{\vartheta}
\def\vthetav{\bb{\vtheta}}
\def\prior{\Pi}

\def\Gam{\Xi}
\def\Gam{\mathcal{S}}
\def\RG{R}
\def\Psu{\Upsilon}
\def\Phim{\breve{\Phi}}

\def\Proj{P}

\def\gammavs{\gammav^{*}}
\def\gammavd{\gammav^{\circ}}
\def\etavs{\etav^{*}}
\def\etavd{\etav^{\circ}}
\def\etavc{\etav'}

\def\taus{\tau_{0}}
\def\taup{\lceil \tau \rceil}

\def\sigmas{{\sigma^{*}}}
\def\Sigmas{\Sigma_{0}}

\def\upsilonc{\upsilon'}
\def\upsilond{\upsilon^{\circ}}
\def\upsilonp{{\upsilon}^{*}}
\def\upsilonm{{\upsilon}_{*}}
\def\upsilonvs{\upsilonv^{*}}
\def\upsilons{\upsilon^{*}}
\def\upsilonb{\bar{\upsilon}}
\def\upsilonvd{\upsilonv^{\circ}}

\def\ups{\bb{\upsilon}}
\def\upsc{\ups^{\prime}}
\def\upsd{\ups^{\circ}}
\def\upsdc{\ups^{\sharp}}
\def\upsdu{\ups^{\flat}}

\def\Ups{\varUpsilon}
\def\Upsd{\Ups^{\circ}}
\def\Upss{\Ups_{\circ}}
\def\UpsP{\Ups^{c}}

\def\Thetas{\Theta_{0}}
\def\ThetasGP{\Theta_{0,\GP}}
\def\varthetav{\bb{\vartheta}}

\def\glink{g}

\def\fvs{\fv}
\def\fs{f}
\def\fb{\fv^{\dag}}

\def\uc{\uv'}
\def\ud{\uv^{\circ}}
\def\uvs{\uv^{*}}
\def\us{u^{*}}
\def\vs{v^{*}}

\def\reps{\epsilon}
\def\eps{\epsilon}

\def\repsc{\reps_{0}}
\def\repsb{\reps^{*}}
\def\repsg{g}

\def\lu{\delta}
\def\lub{\bar{\lu}}

\def\Uu{U}
\def\UU{\cc{Y}}
\def\UUM{\cc{M}}
\def\UP{\cc{U}}
\def\up{\mathfrak{u}}

\def\VP{V}
\def\VPc{\VP_{0}}
\def\VPV{\cc{U}}
\def\VPVc{\cc{\VPV}_{0}}
\def\VPGP{\VP_{\GP}}
\def\VPSP{\VP_{\SP}}

\def\VV{H}
\def\GV{\cc{G}}
\def\GVS{S}

\def\VVb{\VV^{*}}
\def\VVc{\VV_{0}}
\def\vv{\bb{h}}
\def\vva{g}
\def\vp{\mathbf{v}}
\def\vpc{\vp_{0}}
\def\VVca{\VV}
\def\Vtt{H}

\def\DG{D}

\def\Vn{V_{0}}
\def\vn{v_{0}}

\def\norm{\mathfrak{c}}
\def\normc{\delta}
\def\norma{c}

\def\egridd{\cc{E}_{\delta}}
\def\penb{\varkappa}

\def\dotzeta{\dot{\zeta}}
\def\mes{\pi}
\def\mesl{\Lambda}
\def\cprr{F}

\def\lambdam{\gm_{1}}
\def\lambdaB{{\lambda}^{*}}
\def\lambdac{{\lambda'}}

\def\cla{{b}}
\def\fis{\mathfrak{a}}
\def\fiss{\fis_{1}}

\def\Vd{{V}}
\def\vd{\bar{v}}

\def\klim{k^{\circ}}
\def\midm{\mid \!}

\def\Ldrift{M}
\def\ldrift{m}
\def\mY{b}
\def\Lvar{D}
\def\lvar{\sigma}

\def\Mubcu{\Upsilon}
\def\Dthetav{\bb{u}}

\def\B{\cc{B}}
\def\BD{\B^{\circ}}
\def\BU{B}
\def\BI{\B^{*}}

\def\mub{\mu^{*}}
\def\mubc{\mu}
\def\mubcb{\mubc^{*}}
\def\Mubc{\mathbb{M}}
\def\Mubcb{\mathrm{M}}

\def\zzc{\zz_{c}}
\def\ww{w}
\def\wwc{\ww_{c}}

\def\norms{\circ} 
\def\rs{\rr_{\norms}}
\def\yys{\yy_{\norms}}
\def\xxs{\xx_{\norms}}
\def\zzs{\zz_{\norms}}
\def\uu{\mathtt{u}}
\def\uus{\uu_{\norms}}
\def\mus{\mu_{\norms}}
\def\gms{\gm_{\norms}}
\def\wws{\ww_{\circ}}

\def\srho{s}
\def\mrho{\varrho}

\def\Lmgf{\mathfrak{M}}
\def\Lmgfb{\Lmgf^{*}}

\def\lmgf{\mathfrak{m}}
\def\lmgfb{\lmgf^{*}}

\def\Expzeta{\mathfrak{N}}
\def\expzeta{\mathfrak{s}}

\def\rr{\mathtt{r}}
\def\rrb{\rr^{*}}
\def\rru{\rr_{\circ}}
\def\rrc{\rr'}
\def\rs{r_{*}}

\def\zz{\mathfrak{z}}
\def\zzb{\tilde{\zz}}
\def\tt{\mathfrak{t}}
\def\zb{z_{\rd}}
\def\zzQ{\zz_{2}}
\def\zzq{\zz_{1}}

\def\Cr{\mathfrak{c}}
\def\Crp{\mathfrak{C}}
\def\Crl{\mathfrak{r}}
\def\Crlp{\mathfrak{R}}
\def\Crlq{\cc{T}}
\def\Crlmu{\cc{M}}

\def\zetah{\zeta_{h}}
\def\GG{G}
\def\HH{H}
\def\pG{p}
\def\pH{q}
\def\hh{H^{*}}

\def\mubch{\mubc_{1}}
\def\rhoh{\rho_{1}}
\def\CoFuh{\CoFu_{1}}
\def\dimh{p_{1}}
\def\VPh{\VP_{1}}
\def\VPt{\VP_{0}}

\def\LLh{L_{1}}
\def\pnndh{\pnnd_{1}}

\def\LCS{C}
\def\Ac{A_{0}}
\def\Ab{A_{\rd}}
\def\DPrb{\DPr_{\rdb}}
\def\DPrm{\DPr_{\rdm}}
\def\zetavrb{\zetavr_{\rd}}
\def\Cb{\cc{C}_{\rdb}}
\def\Ub{\cc{U}_{\rdb}}
\def\zetavrb{\zetavr_{\rd}}
\def\xivrb{\breve{\xiv}_{\rd}}
\def\VPrb{\breve{\VP}_{\rdb}}
\def\Larb{\breve{\La}_{\rdb}}
\def\Larm{\breve{\La}_{\rdm}}

\def\deltav{\bb{\delta}}

\def\score{\nabla}
\def\scorer{\breve{\nabla}}

\def\LCS{C}
\def\Ac{A_{0}}
\def\Bc{B_{0}}
\def\AF{A}
\def\Ab{A_{\rdb}}
\def\Am{A_{\rdm}}
\def\DPrc{\DPr_{0}}
\def\DPrb{\DPr_{\rdb}}
\def\DPrm{\DPr_{\rdm}}
\def\Cb{\cc{C}_{\rdb}}
\def\Cm{\cc{C}_{\rdm}}
\def\Ub{\cc{U}_{\rdb}}
\def\deltav{\bb{\delta}}
\def\nuv{\bb{\nu}}
\def\nuvs{\nuv^{*}}
\def\nuvc{\nuv'}

\def\xivrb{\breve{\xiv}_{\rd}}
\def\VPrb{\breve{\VP}_{\rdb}}
\def\Larb{\breve{\La}_{\rdb}}
\def\Lar{\breve{\La}}
\def\Larm{\breve{\La}_{\rdm}}
\def\VH{Q}
\def\VHc{\VH_{0}}
\def\N{\mathbb{N}}

\def\Span{\operatorname{span}}
\def\Exc{{\square}}
\def\UUs{U_{\circ}}
\def\errbm{\errb^{*}}
\def\corrDF{\nu}
\def\BBr{\breve{\BB}}
\def\taua{\tau}
\def\AssId{\mathcal{I}}
\def\assId{\iota}
\def\AFD{\cc{A}}

\def\BanX{\cc{X}}
\def\basX{\ev}
\def\apprX{\alpha}
\def\fvs{\fv^{*}}
\def\lkh{\ell}
\def\Bc{B_{0}}
\def\dimn{\dimp_{\nsize}}
\def\betan{\beta_{\nsize}}

\def\gps{s}
\def\GK{\cc{G}}
\def\Excgr{\diamondsuit}

\def\dimh{m}
\def\LCS{C}
\def\Bc{B_{0}}
\def\AF{A}
\def\CF{C}
\def\Ab{A_{\rdb}}
\def\Am{A_{\rdm}}
\def\DPrp{\DPr_{\dimh}}
\def\DPrb{\DPr_{\rdb}}
\def\DPrm{\DPr_{\rdm}}
\def\Cb{\cc{C}_{\rdb}}
\def\Cm{\cc{C}_{\rdm}}
\def\Ub{\cc{U}_{\rdb}}
\def\xivrb{\breve{\xiv}_{\rd}}
\def\VPrb{\breve{\VP}_{\rdb}}
\def\Larb{\breve{\La}_{\rdb}}
\def\Lar{\breve{\La}}
\def\Larm{\breve{\La}_{\rdm}}
\def\VH{Q}
\def\VHc{\VH_{0}}

\def\fvh{\bb{\dimh}}
\def\N{\mathbb{N}}
\def\Z{\mathbb{Z}}

\def\iic{\IF}
\def\iif{\breve{\iic}}
\def\DP{\text{D}}
\def\HH{\text{H}}
\def\A{\text{A}}
\def\ifc{\breve{\iic}}

\def\Upsthetav#1{\substack{\\[0.1pt] \upsilonv\in\Ups \\[1pt] \Proj \upsilonv = #1}}
\def\Span{\operatorname{span}}
\def\Exc{{\square}}
\def\UUs{U_{\circ}}
\def\errbm{\errb^{*}}
\def\corrDF{\rho}
\def\BBr{\breve{\BB}}
\def\taua{\tau}
\def\AssId{\mathcal{I}}
\def\AFD{\cc{A}}

\def\BanX{\cc{X}}
\def\basX{\ev}
\def\apprX{\alpha}
\def\fvs{\fv^{*}}
\def\lkh{\ell}
\def\Bc{B_{0}}
\def\h{\frac{1}{2}}
\def\basis{\ev}
\def\Proj{\Pi_{0}}
\def\Projh{\Pi_{\dimh}}

\def\Ij{\mathcal{I}}

\def\NU{\mathbb{H}}

\def\Mn{M_{\nsize}}
\def\bA{\breve{A}}
\def\cA{\bA_{\dimh}}

\def\Sdr{\cc{S}}
\def\xxn{\xx_{\nsize}}

\def\CONST{\mathtt{C}}
\def\Ij{\mathcal{I}}
\def\etas{\eta^{*}}
\def\omegav{\bb{\phi}}
\def\omegavs{\omegav^{*}}
\def\omegavc{\omegav'}

\def\dimn{\dimp_{\nsize}}
\def\betan{\beta_{\nsize}}
\def\NU{\mathbb{H}}

\def\bA{\breve{A}}
\def\cA{\bA_{\dimh}}

\def\corrDF{\rho}


\def\xivGP{\xiv_{\GP}}
\def\dimA{\mathtt{p}}
\def\dimAGP{\dimA}
\def\dime{\dimA_{e}}
\def\dimG{\dimA_{\GP}}
\def\dimS{\dimA_{s}}
\def\nubm{\nu_{\rd}}
\def\uub{u_{\rd}}
\def\uubGP{u_{\rd,\GP}}

\def\priorden{\pi}
\def\xivGP{\xiv_{\GP}}
\def\dimAGP{\dimA}
\def\nubm{\nu_{\rd}}
\def\uub{u_{\rd}}
\def\uubGP{u_{\rd,\GP}}

\def\CR{\mathcal{C}}
\def\CRb{\CR_{\rdb}}
\def\vthetavb{\bar{\vthetav}}
\def\Covpost{\mathfrak{S}}

\def\Db{\DP_{+}}
\def\Dm{\DP_{-}}
\def\uvb{\uv_{+}}
\def\uvm{\uv_{-}}
\def\uud{\omega}
\def\taub{\delta}
\def\Lip{L}
\def\Xb{X_{+}}
\def\Xm{X_{-}}
\def\deltam{\delta_{-}}
\def\betauv{\delta}
\def\betab{\betauv_{1}}
\def\betaf{\betauv_{2}}
\def\upsv{\bb{\varkappa}}
\def\upsvb{\bar{\upsv}}
\def\rhob{\varrho}
\def\alpb{\alp_{1}}
\def\betap{\betauv_{3}}
\def\Ec{\E^{\circ}}
\def\ff{f}
\def\fpos{g}
\def\fneg{h}
\def\alpb{\alp_{+}}
\def\alpm{\alp_{-}}

\def\kappak{\kappa}
\def\kappas{\kappak^{*}}
\def\Kappak{\cc{K}}
\def\DPk{\DP_{\kappak}}
\def\VPk{\VP_{\kappak}}

\def\ts{s}
\def\tsv{\bb{\ts}}
\def\mm{\kappa}
\def\mmc{\mm'}
\def\mmd{\mm^{\circ}}
\def\mmo{\mm^{*}}
\def\mmmmo{\mm,\mmo}
\def\mmt{\tilde{\mm}}
\def\mma{\hat{\mm}}
\def\pp{z}

\def\LLL{L_{1}}
\def\LLr{L_{0}}
\def\muL{\mu_{1}}
\def\mur{\mu_{0}}

\def\LmgfL{\Lmgf_{1}}
\def\Lmgfr{\Lmgf_{0}}
\def\Lmgfm{\Lmgf_{1}}

\def\Kappa{\cc{K}}
\def\CoFu{\cc{C}}
\def\CoFuc{\CoFu_{0}}
\def\CoFub{\CoFu^{*}}
\def\CoFuL{\CoFu_{1}}
\def\CoFur{\CoFu_{0}}
\def\CAL{\CA_{1}}
\def\CAr{\CA_{0}}
\def\CAzz{\cc{A}}

\def\pnnL{\pnn_{1}}
\def\pnnr{\pnn_{0}}
\def\ttd{\delta}
\def\alphaL{\alpha_{1}}
\def\alphar{\alpha_{0}}
\def\alpharL{\alpha}
\def\rat{\mathfrak{t}}
\def\mquad{\nquad}
\def\zzL{\zz_{1}}
\def\zzr{\zz_{0}}

\def\mmset{\mathcal{I}}
\def\xex{u}
\def\dcm{q}
\def\dc{g}
\def\dcL{\dc_{1}}
\def\dcr{\dc_{0}}
\def\kk{k}

\def\cpen{\tau}

\def\dens{f}
\def\jj{j}
\def\JJ{\cc{J}}
\def\Zphi{Z}
\def\Zphiv{\bb{\Zphi}}

\def\nuu{\mathfrak{u}}
\def\nud{\mathfrak{u}_{0}}
\def\nun{c_{\nuu}}
\def\rhork{\kullb}
\def\GH{\mbox{GH}}
\def\HYP{\mbox{HYP}}
\def\NIG{\mbox{NIG}}
\def\IR{{\rm I\!R}}
\def\taggr{b}
\def\penm{\boldsymbol{m}}
\def\Crlp{\cc{R}}

\def\Mh{M}
\def\Mht{\Mh^{c}}

\def\Mhh{\Mh^{-}}
\def\Mhc{G}
\def\Lh{L_{1}}
\def\Uh{\cc{U}}
\def\wloc{w}
\def\Bias{B}
\def\bias{b}
\def\ExpzetaU{\Expzeta_{1}}
\def\vpci{\vp_{i,0}}
\def\IFci{\IF_{i,0}}

\def\erqb{\Circle_{\rdb}}
\def\erqm{\Circle_{\rdm}}
\def\errqm{\errm^{*}}
\def\errqb{\errb^{*}}
\def\Nsize{N}
\def\VVD{\VV_{1}}
\def\AA{A}
\def\Wloc{W}

\def\gmc{\gm_{c}}
\def\gmb{\gm}
\def\gmbm{\gmb_{1}}
\def\dimA{\mathtt{p}}
\def\muc{\mu_{c}}

\def\zz{\mathfrak{z}}
\def\zzb{\tilde{\zz}}
\def\zzc{\zz_{c}}
\def\tt{\mathfrak{t}}
\def\zb{z_{\rd}}

\def\yy{\mathtt{y}}
\def\yyc{\yy_{c}}
\def\xx{\mathtt{x}}
\def\xxc{\xx_{c}}
\def\tc{t_{c}}

\def\alp{\alpha}
\def\alpn{\rho}
\def\gmu{\mathfrak{a}}

\def\bA{\breve{A}}
\def\cA{\bA_{\dimh}}
\def\vA{\mathtt{v}}

\def\Uu{U}
\def\UU{\cc{Y}}
\def\UUM{\cc{M}}
\def\UP{\cc{U}}
\def\up{\mathfrak{u}}

\def\nsize{{n}}
\def\xxn{\xx_{\nsize}}
\def\dimn{\dimp_{\nsize}}
\def\betan{\beta_{\nsize}}

\def\eps{\epsilon}

\def\LCS{C}

\def\score{\nabla}
\def\scorer{\breve{\nabla}}

\def\DPrc{\DPr_{0}}
\def\nuv{\bb{\nu}}

\def\corrDF{\rho}

\newcommand{\mygraphics}[3]{\begin{center}
    \resizebox{#1\textwidth}{#2\textheight}{\includegraphics{#3}}
    \end{center}
}

\newcommand{\mybox}[3]{\begin{center}
    \resizebox{#1\textwidth}{#2\textheight}{#3}
    \end{center}
}

\newenvironment{eqnh}
{
    \setbeamercolor{postit}{fg=black,bg=hellgelb} 
    \begin{beamercolorbox}[center,wd=\textwidth]{postit} 
    \begin{eqnarray*}}
    {\end{eqnarray*}\end{beamercolorbox}
}

\def\cond{\, \big| \,}
\def\kappav{\bb{\kappa}}
\def\kappavs{\kappav^*}

\def\Yb{\mathbb{Y}}

\def\HF{\mathcal H}
\def\VPr{\breve\VP}

\def\Ybb{\mathbb Y}
\def\Xbb{\mathbb X}
\def\KL{\text{K}_0}
\def\RR{\text{R}_0}
\def\zzq{\zz_{1}}
\def\zzQ{\zz_{Q}}
\def\sign{\operatorname{sign}}
\def\vec{\operatorname{vec}}
\def\Xvv{\bb{X}}

\def\upsilonv{\boldsymbol{\upsilon}}
\def\upsilonvs{\boldsymbol{\upsilon}^{*}}
\def\upsilonvd{\boldsymbol{\upsilon}^\circ}
\def\upsilonvc{\upsilonv'}
\def\rupf{\rr_{1}}
\def\gmone{\gm_{1}}
\def\rhorb{\rhor_{1}}
\def\deltar{\delta}
\def\sign{\operatorname{sign}}

\def\varsigmav{\mathbf \varsigma}
\def\Xv{\mathbf X}
\def\Pn{\text{P}_n}
\def\kappavp{\etavs_{2}}
\def\interior{\operatorname{int}}

\def\conv{\operatorname{conv}}
\def\vareps{\varepsilon}
\def\varepsv{\varepsilonv}

\newpage
\tableofcontents

\section{Finding the most interesting directions of a data set}
Assume observations \((Y_i,\Xv_i)\in \R\times \R^{\dimp}\) with \(\dimp\in \N\) 
\begin{EQA}[c]\label{eq: introduction of model true}
Y_i=g(\Xv_i)+\varepsilon_{i}, \text{ }i=1,...,\nsize,
\end{EQA}
where \(g:\R^{\dimp}\to\R\) is some continuous function, \(\varepsilon_i\in\R\) are additive centered errors independent of the random regressors \((\Xv_i)\). Consider the task of estimating 
\begin{EQA}[c]
\E[\Yv| \Xv]=g(\Xv).
\end{EQA}
Statistical theory for nonparametric models shows that even for moderate \(\dimp\in\N\) the accuracy of estimating \(g(\Xv)\) increases very slow in the sample size \(n\in\N\) as the rates are lower bounded by \(n^{-\alpha/(2\alpha+\dimp)}\) - with \(\alpha>0\) quantifying the smoothness of \(g:\R^{\dimp}\to\R\) - as was for instance noted in \cite{stone1980}. \cite{Friendman1981} propose to use a projection pursuit approach to circumvent this problem in situations where
\begin{EQA}[c]\label{eq: g as sum of single index models}
g(\Xv)\approx \sum_{l=1}^{M} \fs_{(l)}(\Xv^{\T}\thetavs_{(l)}),
\end{EQA}
for a set of functions \(\fs_{(l)}:\R\to \R\), vectors \(\thetavs_{(l)}\in S_{1}^{\dimp,+}:=\{\thetav\in\R^{\dimp}:\, \|\thetav\|=1,\, \theta_1> 0\}\subset\R^{\dimp}\) and some \(M\in\N\). As each nonparametric estimation task is uni-variate, better performance can be expected in comparison to a full nonparametric regression as long as \(M,\dimp\in\N\) are not very large. But of course \eqref{eq: g as sum of single index models} is a structural assumption whose usefulness depends on the size of \(M\in\N\) and \(\dimp\in\N\). For small \(M\in\N\) and \(\dimp\in\N\) one can get important gains but the assumption \eqref{eq: g as sum of single index models} becomes rather restrictive. On the other hand, for large \(M\in\N\) and large \(\dimp\in\N\) the assumption \eqref{eq: g as sum of single index models} becomes true for any smooth function. This can be seen as follows. Assume that one observes \((Y_i,\Zv_i)\) for a given vector of regressors \(\Zv\in\R^{\dimp_1}\) and that the aim is to estimate \(g^{\circ}(\Zv)=\E[Y|\Zv]\). We can define for some \(D\in\N\) an extended vector of regressors \(\Xv\in \R^{\dimp_1+\sum_{d=2}^{D+1}\dimp_1^{d}-\dimp_1}\) via
\begin{EQA}[c]
\Xv\eqdef (Z_1,\ldots, Z_{\dimp_1},Z_1Z_2, Z_1Z_3, \ldots, Z_{\dimp_1-1}Z_{\dimp_1},Z_1Z_1Z_2,\ldots, Z_{\dimp_1-1}Z_{\dimp_1}^{D} ).
\end{EQA}
For large \(D\in\N\) this means that \eqref{eq: g as sum of single index models} demands that \(g^\circ(\Zv)=g(\Xv)\) can be well approximated by polynomials of maximal degree \(D+1\in\N\), which of course is the case for smooth functions. See \cite{Huber1985} and \cite{Jones1987} for a more sophisticated approach of showing that smooth functions \(g\) can be well approximated as in \eqref{eq: g as sum of single index models}. \cite{Friendman1981} suggest to estimate the pairs \((\fs_l,\thetavs_{l})\) iteratively. The first task is to estimate
\begin{EQA}[c]\label{eq: estimation task single index model bias}
\thetavs_{(1)}\eqdef \argmin_{\thetav\in S_{1}^{\dimp,+}}\E\left[\left(g(\Xv)-\E[g(\Xv)|\Xv^\T\thetav]\right)^2\right].
\end{EQA}
Given an estimator \(\tilde\thetav_{(1)}\in S_{1}^{\dimp,+}\) one can determine an estimator \(\hat\fs_{(1)}\) for \(\fs_{(1)}\) and generate a new sample via
\begin{EQA}[c]
{Y_i}_{(1)}\eqdef Y_i-\hat\fs_{(1)}(\Xv_i^\T\tilde\thetav_{(1)}).
\end{EQA}
Using this new data set \(({Y_i}_{(1)})_{i=1,\ldots,n}\) one can estimate \(\thetavs_{(2)}\) and \(\fs_{(2)}\) as in the first step and again generate a new data set \(({Y_i}_{(2)})_{i=1,\ldots,n}\). These steps are repeated \(M-1\in\N\) times if \(M\in\N\) was fixed or known in the beginning, otherwise until a certain level of variability in the data is explained by the obtained sum
\begin{EQA}[c]
\sum_{l=1}^{M}\hat\fs_{(l)}(\Xv_i^\T\tilde\thetav_{(l)}).
\end{EQA}

We will mainly focus on the task \eqref{eq: estimation task single index model bias}. It has  been observed in \cite{Hallsingleindex} that the estimation of \({\thetavs}_{(1)}\) - from now on denoted simply by \(\thetavs\) - can be attained with root-n rate even though the full model is nonparametric.

In the particular case that \(M=1\), i.e. that
\begin{EQA}[c]
\label{eq: introduction of model}
g(\Xv) =\fs(\Xv^{\T}\thetavs),
\end{EQA}
for some \(\fs:\R\to \R\) and \(\thetavs\in S_{1}^{\dimp,+}\subset\R^{\dimp}\), the estimation problem \eqref{eq: estimation task single index model bias} becomes the task to estimate the linear response vector in a semiparametric  single-index model (see \cite{Ichimura}). The single-index model supposes that the observations satisfy with two functions \(\fs:\R\rightarrow \R\) and \(h:\R^{\dimp}\rightarrow \R\) and with errors \((\varepsilon_{i})\in\R\)
\begin{EQA}[c]\label{eq: single index model in intro to single index chapter}
Y_i=\fs(h(\Xv_i))+\varepsilon_{i}, \text{ }i=1,...,\nsize.
\end{EQA}
Usually it is assumed that the index function \(h\) is known up to some parameter \(\thetav\in \R^{\dimp}\) such that one writes \(h(\thetav,\xv)\). In our setting \(h(\thetav,\xv)=\thetav^\T\xv\). 
\cite{Xia2006} compares the asymptotic distributions of two different prominent estimation procedures for \(\thetavs\). The first is the average derivative estimation introduced by \cite{Powell1989} and refined by \cite{HJS2001} and is based on the fact that if \eqref{eq: introduction of model} is correct
\begin{EQA}[c]
\E\left[\frac{d}{d\Xv}g(\Xv)\right]=\E\left[\fs'(\thetavs\Xv)\right]\thetavs,
\end{EQA}
which suggests to estimate \(\thetavs\) via an estimate of \(\E\left[\fs'(\thetavs\Xv)\right]\). The second one is the minimal conditional variance estimation by \cite{MAVE2002} which is inspired by \cite{HaerdleHallIchimura} and aims at directly solving \eqref{eq: estimation task single index model bias} via a local linear approximation of \(\E[y|\Xv^\T\thetav]\). Further results are the asymptotic efficiency of a semiparametric maximum-likelihood estimator shown by \cite{DelecroixHaerdleHristache} for particular examples and in \cite{HaerdleHallIchimura} the right choice of the bandwidth for the nonparametric estimation of the link function. 

In this work we want to use a different approach to carry out the first step \eqref{eq: estimation task single index model bias} that allows to apply the results of \cite{AASP2013} and \cite{AASPalternating}. For this purpose denote
\begin{EQA}[c]\label{eq: function fs in cond exp of g}
\E[g(\Xv)|\Xv^\T\thetavs]=\fs(\Xv^{\T}\thetavs).
\end{EQA}
Assume that \(\fs\in \Span\{(\basX_k)_{k\in\N}\}\) for the set of Daubechies wavelet basis functions \((\basX_k)_{k\in\N}\subset \BanX\) that we present in Section \ref{sec: choice of basis}. For some \( \dimh \ge 1 \) and \( {\etav}\in\R^{\dimh} \) denote 
\begin{EQA}[c]
    \fv_{\etav}
    \eqdef
    \sum_{k=0}^{\dimh} \eta_{k} \basX_{k},
\label{fsdimhse}
\end{EQA}    
with properly selected coefficients \( \etav = (\eta_{1},\ldots,\eta_{\dimh})^{\T} \in \R^{\dimh} \). Further assume that \(\P(\Xv_i\in B_{s_{\Xv}}(0))\approx 1\) for some \(s_{\Xv}>0\). Our aim is to analyse for \(\dimh\in\N\) the properties of the estimator 
\begin{EQA}[c]\label{eq: def estimator}
\tilde \thetav_{\dimh}\eqdef\Pi_{\thetav}\tilde\ups_{\dimh}\eqdef\Pi_{\thetav} \argmax_{(\thetav,\etav)\in\Upsilon_{\dimh}}\LL_{\dimh}(\thetav,\etav),
\end{EQA}
where
\begin{EQA}[c]
\label{eq: likelihood functional}
 \LL_{\dimh}(\thetav,\etav)=-\sum_{\{i:\,\|\Xv_i\|\le s_{\Xv}\}}\|Y_i-\fv_{\etav}(\Xv_i^\T\thetav)\|^2/2.
\end{EQA}
The set \(\Upsilon_{\dimh}\) satisfies \(\Upsilon_{\dimh}= S_{1}^{\dimp,+}\times B_{\rr^\circ}^{\dimh}\subset \R^{\dimp}\times\R^{\dimh} \) where \(B_{\rr^\circ}^{\dimh}\subset \R^{\dimh}\) denotes the centered ball of radius \(\rr^\circ>0\). Note that this is exactly the type of estimator presented in Section 2.7 of \cite{AASP2013}.
In \cite{Ichimura} a very similar estimator is analyzed based on a "leave one out" kernel estimation of \(\E[Y_i|\Xv_i^\T\thetav]\) instead of using \(\fv_{\etav}(\Xv_i^\T\thetav)\). Ichimura shows \(\sqrt{\nsize}\)-consistency and asymptotic normality of his proposed estimator. 

\begin{remark}
The radius \(\rr^\circ\) is needed to control the large deviations of the full maximizer \(\tilde\ups_{\dimh}\). We ensure that the estimator \(\tilde\ups_{\dimh}\) does not lie on the boundary in Lemma \ref{lem: a priori a priori accuracy for rr circ}.
\end{remark}

\begin{remark}
To avoid undesirable boundary effects (see Remark \ref{rem: lower bound of HH}) we do not use all available data: We only consider realizations \((Y_i,\Xv_i)\) for which \(\|\Xv_i\|\le s_{\Xv}\) but in Section \ref{sec: assumptions} we assume in condition \((\mathbf{Cond}_{\Xv})\) that there is positive probability that \(\Xv\in B_{s_{\Xv}+c_{B}}(0)\backslash B_{s_{\Xv}}(0)\). We assume that the proportion of ignored data is small such that we can neglect this in the following and pretend that we can use the full data set. 
\end{remark}

The estimator \(\tilde\thetav_{\dimh}\) in \eqref{eq: def estimator} is supposed to approach
\begin{EQA}[c]\label{eq: def target single index}
\thetavs\eqdef\Pi_{\thetav}(\thetavs,\etavs)\eqdef\Pi_{\thetav}\argmax_{(\thetav,\etav)\in \Ups}\E\LL_{\infty}(\thetav,\etav),
\end{EQA}
where \(\Upsilon= S_{1}^{\dimp,+}\times l^2 \) and for \((\thetav,\etav)\in \Upsilon\)
\begin{EQA}[c]
\LL_{\infty}(\thetav,\etav)\eqdef -\sum_{\{i:\,\|\Xv_i\|\le s_{\Xv}\}}\left\|Y_i-\sum_{k=1}^{\infty}\eta_k\basX_k(\Xv_i^\T\thetav)\right\|^2/2.
\end{EQA}
\begin{remark}
To understand the motivation of this functional note that for any \(\thetav\in S_{1}^{\dimp,+}\) the sequence
\begin{EQA}[c]
\etavs_{\thetav}\eqdef\Pi_{\etav}\argmax_{\substack{\ups\in\R^{\dimp}\times l^2\\ \Pi_{\thetav}\ups=\thetav}}\E\LL(\ups),
\end{EQA}
solves by first order criteria of maximality for any \(A\in\mathcal{F}(\Xv^\T\thetav)\) - where \(\mathcal{F}(\Xv^\T\thetav)\) denotes the sigma algebra associated to the law of \(\Xv^\T\thetav\) - the equation
\begin{EQA}[c]
\E\left[\left(g(\Xv)-\fv_{\etavs_{\thetav}}(\Xv^\T\thetav) \right)1_A\right]=0.
\end{EQA}
This means that with equivalence in \(L^2(\P^{\Xv})\)
\begin{EQA}[c]\label{eq: cond exp representation for etavs}
\fv_{\etavs_{\thetav}}(\Xv^\T\thetav)=\E[g(\Xv)|\Xv^\T\thetav],
\end{EQA}
such that the target \eqref{eq: def target single index} indeed coincides with the most informative direction in \eqref{eq: estimation task single index model bias}.
\end{remark}

\begin{remark} 
Note that there is a model bias and an approximation bias of the form
\begin{EQA}
"model\,bias"&=&\min_{\ups\in\Ups}\E\|g(\Xv)-\fv_{\etav}(\Xv^\T\thetav)\|^2, \label{eq: def of model bias single index}\\ "approximation\,bias"&=&\min_{\ups\in\Ups_\dimh}\E\|\fv_{\etav}(\Xv^\T\thetav)-\fv_{\etavs}(\Xv^\T\thetavs)\|^2, \label{eq: intro single index approx bias}
\end{EQA}
which both have to be accounted for. 
\end{remark}
As pointed out we will analyze the properties of the estimator \(\tilde \thetav_{\dimh}\) in \eqref{eq: def estimator} using the results of \cite{AASP2013} and \cite{AASPalternating}. It turns out that this is possible with a series of conditions on the additive noise \(\varepsilon_i\in\R\), the function \(g:\R^{\dimp}\to\R\) and on the random design \(\Xv\in \R^{\dimp}\). In particular the choice of the basis is independent of the model. Due to the support structure of compactly supported wavelets - see Section \ref{sec. choice of basis} - we still manage to control the sieve bias in \eqref{eq: intro single index approx bias}. Even though we assume what is necessary to apply the results of \cite{AASP2013} and \cite{AASPalternating}, the calculations necessary to check the necessary conditions still remain rather tedious and lengthy. We present most steps in full detail, which at some points leads to repetitions of very similar arguments. Also the regression setup leads to some peculiarities that we elaborate on in Section \ref{sec: implications of regression setup}. The treatment of these issues involves bounds for the spectral norm or random matrices from \cite{Tropp2012}. It is worthy to point out here that a fixed design setting would not resolve these issues either as one for instance would still have to deal with convergence issues of the operator
\begin{EQA}[c]
\sum_{i=1}^{n}\nabla \LL(\Xv_i,Y_i,\ups)\nabla \LL(\Xv_i,Y_i,\ups)^\T\in\R^{\dimtotal\times\dimtotal}.
\end{EQA}

There is another peculiarity to the results we present in this work. A naive approach to satisfy the important condition \({(\cc{L}{\rr})}\) from Section \ref{sec: conditions semi} would include a bound for 
\begin{EQA}[c]\label{eq: naive way to show Lr}
\sup_{\ups\in\Ups_{\dimh}}\left|\E[\LL(\ups,\upss)|(\Xv_i)_{i=1,\ldots,n}]- \E\LL(\ups,\upss)\right|.
\end{EQA}
But as \(\LL\) is quadratic and \(\Ups_{\dimh}\subset\R^{\dimtotal}\) can be quite large this becomes hard to achieve with nice bounds. We circumvent this problem using an idea of \cite{Mendelson}. Mendelson's crucial insight is that to obtain \(\E[\LL(\ups,\upss)|(\Xv_i)_{i=1,\ldots,n}]\ge \gmi\rr^2\) one only has to ensure that 
\begin{EQA}[c]
\inf_{\ups\in\Upss(\rr)^c}\P\left(\|Y_i-\fv_{\etavs}(\Xv_i^\T\thetavs)\|^2/2-\|Y_i-\fv_{\etav}(\Xv_i^\T\thetav)\|^2/2\ge \gmi\rr^2/n\right)>0.
\end{EQA}
We follow this route in the proof of Lemma \ref{lem: conditions example}. But we only apply this idea in the case that \(\CONST_{bias}=0\). In the general case we derive a bound for \eqref{eq: naive way to show Lr} to avoid too lengthy derivations. The price is an additional \(\log(n)\)-factor in the sufficient full dimension i.e. we need \({\dimtotal}^{3}\log(n)=o(\sqrt{n})\) instead of \({\dimtotal}^{3}=o(\sqrt{n})\) to get accurate results when applying Theorem \ref{theo: main theo finite dim regression}.


\section{Main results}
\label{sec: main results}

\subsection{Assumptions}
\label{sec: assumptions}
  
To apply the technique presented in \cite{AASP2013} and \cite{AASPalternating} we need a list of assumptions. We denote this list by \((\mathbf{\mathcal A})\). We start with conditions on the regressors \(\Xv\in\R^{\dimp}\):

\begin{description} 
 \item[\((\mathbf{Cond}_{\Xv})\)] The random variables \((\Xv_{i})_{i=1,\ldots,n}\subset \R^{\dimp}\) are i.i.d with distribution denoted by \(\P^{\Xv}\) and independent of \((\varepsilon_i)_{i=1,\ldots,n}\subset\R\). The measure \(\P^{\Xv}\) is absolutely continuous with respect to the Lebesgue measure. The Lebesgue density \(p_{\Xv}\) of \(\P^{\Xv}\) Lipschitz continuous on \( B_{s_{\Xv}}(0)\subset \R^{\dimp}\) with Lipschitz constant \(L_{p_{\Xv}}>0\). Furthermore we assume that for any pair \(\thetav,\thetavd\in S_{1}^{+,\dimp}\) with \(\thetav\perp\thetavd\) we have \(\Var\left(\Xv^\T\thetav\big|\Xv^\T\thetavs\right)>\sigma^2_{\Xv|\perp}\) for some constant \(\sigma^2_{\Xv|\perp}>0\) that does not depend on \(\Xv^\T\thetavs\in\R\). Furthermore assume that for all such pairs \(\left\|\frac{p_{\thetavd,\thetav}}{p_{\thetav}}\right\|_{\infty}<\infty\) with \(p_{\thetavd,\thetav}:\R^2\to \R_+\) denoting the density of \((\Xv^\T\thetavd,\Xv^\T\thetav)\in\R^{2}\). Also let on \( B_{s_{\Xv}+c_{B}}(0)\) the density satisfy \(p_{\Xv}>c_{p_{\Xv}}>0\) for some constants \(c_{p_{\Xv}},c_{B}>0\).   
\end{description}

\begin{remark}
\(\Var\left(\Xv^\T\thetavd\big|\Xv^\T\thetavs\right)=0\) would mean that \(\Xv^\T\thetavd=\break a(\Xv^\T\thetavs)\) for some measurable function \(a:\R\to\R\). But then we would have for any \((\alpha,\beta)\in \R^2\) with \(\alpha^2+\beta^2=1\) that
\begin{EQA}[c]
f(\Xv^\T(\alpha\thetavs+\beta\thetavd))=f(\alpha\Xv^\T\thetavs+\beta a(\Xv^\T\thetavs))\eqdef f^{\circ}_{\alpha,\beta}(\Xv^\T\thetavs),
\end{EQA}
such that the problem would no longer be identifiable. We bound \(p_{\Xv}>c_{p_{\Xv}}>0\) on \( B_{s_{\Xv}+c_{B}}(0)\) to ensure identifiability, also see Remark \ref{rem: lower bound of HH}. 
\end{remark}

\begin{remark}
We assume that the support of \(\P^{\Xv}\) contains \(0\) without loss of generality. If that was not the case one could modify the sample as follows. Let \(\xv_0\) be an inner point of the support of \(\P^{\Xv}\). Generate a new sample \((\Xv_i')_{i=1,\ldots,n}=(\Xv_i-\xv_0)_{i=1,\ldots,n}\) and assume \((\mathbf{Cond}_{\Xv})\) for this new sample instead.
\end{remark}

Of course we need some regularity of the link function \(\fs\in \{f:[-s_{\Xv},s_{\Xv}]\mapsto \R\} \) in \eqref{eq: function fs in cond exp of g}:
\begin{description} 
  
  \item[\( (\mathbf{Cond}_{\fs}) \)]
   For some \(\etavs\in B_{\rr^\circ}(0)\subset l^2\eqdef\{(u_k)_{k\in\N}:\, \sum_{k=1}^{\infty} u_k^2<\infty\}\)
   \begin{EQA}[c]
   \label{eq: expansion of indexfunciton}
    \fs=\E[g(\Xv)|\Xv^\T\thetavs=\cdot]=\fv_{\etavs}=\sum_{k=1}^\infty \eta^*_k\basX_{k},
   \end{EQA}
   where \(\|\fv_{\etavs}'\|_{\infty}=\CONST_{\|\fv_{\etavs}'\|_{\infty}}<\infty\) and \(\|\fv_{\etavs}''\|_{\infty}=\CONST_{\|\fv_{\etavs}''\|_{\infty}}<\infty\) and where with some \(\alpha>2\) and a constant \(C_{\|\etavs\|}>0\)
   \begin{EQA}[c]\label{eq: smoothness of fs}
    \sum_{k=0}^\infty k^{2\alpha}{\eta^*_{k}}^2 \le C_{\|\etavs\|}^2< \infty.
   \end{EQA} 
\end{description}

\begin{remark}
We can now specify the parameter set \(\Ups\subset \R^{\dimp}\times l^2\) namely
\begin{EQA}[c]\label{eq: def of set Ups}
\Ups\eqdef\left\{(\thetav,\etav)\in \R^{\dimp}\times l^2,\, \thetav\in S_{1}^{\dimp,+} \right\}.
\end{EQA}
\end{remark}

\begin{remark}
Simply using \eqref{eq: smoothness of fs} does not - easily - yield a bound for \(\|\fv_{\etavs}''\|_{\infty}\) since (see proof of Lemma \ref{lem: condition L_0 is satisfied})
\begin{EQA}[c]
|\fv_{\etavs}''(\Xv^\T\thetav)|\le \sqrt{34}\|\psi''\|_{\infty}\left(\sum_{k=0}^{\infty} k^{2\alpha}{\eta^*_k}^2 
\right)^{1/2}\left(\sum_{j=0}^{\infty} 2^{5j-2\alpha}
\right)^{1/2}=\infty.
\end{EQA}

\end{remark}

\begin{remark}\label{rem: how to ensure smoothness in biased case}
In the case that the data is not from the model \eqref{eq: introduction of model} but from the model in \eqref{eq: introduction of model true} the implications of this condition to the function \(g:\R^{\dimp}\to \R\) become somewhat unclear. One way of ensuring that it is satisfied is to assume that for every \(\thetav\in S^{\dimp,+}_{1}\) and any \(\xv\in B_{s_{\Xv}}(0)\cap \thetav^{\perp}\) the function
\begin{EQA}
\fs_{\thetav,\xv}: \R\to \R, &\quad  t \mapsto g(\xv+\thetav t),
\end{EQA}
satisfies \eqref{eq: expansion of indexfunciton} with some \(\etav(\thetav,\xv)\) and \(\alpha(\thetav,\xv)>2+\eps\), where \(\eps>0\) is independent of \(\xv\). More precisely set for any \(\thetav\in  S^{\dimp,+}_{1}\)
\begin{EQA}[c]
\fs_{\thetav}(t)\eqdef \E[Y_i|\Xv^\T \thetav= t]= \int_{B_{s_{\Xv}}(0)\cap \thetav^{\perp}} \fs_{\thetav,\xv}(t) p_{\Xv|\Xv^\T\thetav=t}(\xv)d\xv,
\end{EQA}
where \(p_{\Xv|\Xv^\T\thetav=t}(\xv)\) is the conditional density of \(\Xv|\Xv^\T\thetav=t\). Due to the smoothness assumption on \(\fs_{\thetav,\xv}(t)\) the function \(\fs_{\thetav}(t)\) satisfies \eqref{eq: expansion of indexfunciton} as well with some \(\etav(\thetav)\) and \(\alpha(\thetav)\ge \inf_{\xv\in B_{s_{\Xv}}(0)\cap \thetav^{\perp}}\{\alpha(\thetav,\xv)\}>2\). We proof this in Section \ref{sec: proofs}.
\end{remark}


To control the large deviations of \(\tilde\ups_{\dimh}\in\R^{\dimtotal}\) we use 
the following assumption:
\begin{description} 
 \item[\((\mathbf{Cond}_{\Xv\thetavs})\)]
  On some ball \(B_{h}(\xv_0)\subseteq B_{s_{\Xv}}(0)\) with \(h>0\) it holds true that \(|\fv_{\etavs}'(\Xv^\T\thetavs)|> c_{\fv_{\etavs}'}\) for some \(c_{\fv_{\etavs}'}>0\). 
\end{description}

\begin{remark}
 Note that a condition of this kind is necessary to ensure identifiability. Otherwise the function \(g:\R^{\dimp}\to \R\) would be \(\P^{\Xv}\)-almost surely constant. But for a constant function \(\thetavs\in\R^{\dimp}\) in \eqref{eq: estimation task single index model bias} is not defined.
\end{remark}

To be able to apply the finite sample device we need constraints on the moments of the additive noise:

\begin{description} 
 \item[\((\mathbf{Cond}_{\varepsilon})\)] 
  The errors \((\varepsilon_{i})\in \R\) are i.i.d. with \(\E[\varepsilon_{i}]=0\), \(\Cov(\varepsilon_{i})=\sigma^2\) and satisfy for all \(|\mu|\le\tilde g\) for some \(\tilde g>0\) and some \(\tilde \nu>0\)
  \begin{EQA}[c]
    \log\E[\exp\left\{ \mu \varepsilon_{1} \right\}]\le \tilde \nu^{2}\mu^{2}/2.
  \end{EQA}
\end{description}

\begin{remark}
Note that our assumptions in terms of moments and smoothness are quite common in this model. For instance \cite{HaerdleHallIchimura} assume that the density \(p_{\Xv}\) of the regressors \((\Xv_i)\) is twice continuously differentiable, that \(\fs\) has two bounded derivatives and that the errors \((\varepsilon_i)\) are centered with bounded polynomial moments of arbitrary degree. 
\end{remark}

Unfortunately these conditions do not facilitate an easy proof of our desired results in the case that \(\CONST_{bias}>0\). To control the large deviations of \(\tilde\ups_{\dimh}\) and for identifiability we impose some more "esoteric" conditions on the interplay of the function \(g:\R^{\dimp}\to \R\) and the measure \(\P^{\Xv}\). 
\begin{description} 
 \item[(model bias)]
 Assume that 
\begin{EQA}[c]
 \|\E[g(\Xv)|\Xv^\T\thetavs]-g(\Xv)\|=\|\fv_{\etavs}(\Xv^{\T}\thetavs)- g(\Xv)\|\le \CONST_{bias},
 \end{EQA} 
for some constant \(\CONST_{bias}\ge 0\).
 Furthermore we need if \(\CONST_{bias}>0\) that there exists an open ball \(B_{\rr_{\thetav}}(\thetavs)\subset\R^{\dimp}\) around \(\thetavs\) and a constant \(\gmi_{\theta}>0\) such that for \(\thetav\notin B(\thetavs)\)
\begin{EQA}
-\E\left[\left(g(\Xv)-\E[g(\Xv)|\Xv^\T\thetav]\right)^2\right]+\E\left[\left(g(\Xv)-\E[g(\Xv)|\Xv^\T\thetavs]\right)^2\right]
	&\le& -\gmi_{\theta},
\end{EQA}
and such that on \(B_{\rr_{\thetav}}(\thetavs)\subset\R^{\dimp-1}\) the second derivative exists and satisfies with some \(\CONST_{\theta}>0\)
\begin{EQA}
\nabla_{\thetav}^2\E\left[\left(g(\Xv)-\E[g(\Xv)|\Xv^\T\thetav]\right)^2\right]&\ge& \gmi_{\theta}>0.
\end{EQA}
\end{description} 

\begin{remark}
The conditions (model bias) are of course rather peculiar and not a very accurate characterization of the class of functions that allow the application of our approach. As this paper - even with these conditions - is still very technical we do not elaborate on this issue further. We only point out that this condition is a kind of quantification of how salient the direction \(\thetavs\in\R^{\dimp}\) in \eqref{eq: estimation task single index model bias} is.
\end{remark}


\subsection{Some important objects}
\label{sec: objects}
In this subsection we introduce some important objects that are relevant for our results.

For given \( \dimtotal=\dimp+\dimh \), set \( \Pi_{\dimtotal}\upsilonv =(\upsilon_1,\ldots,\upsilon_{\dimtotal})=(\thetav,\Pi_{\dimh}\etav)\in \R^{\dimtotal} \). We represent the full parameter \( \upsilonv\in\R^{\infty} \) in the form 
\begin{EQA}[c]
    \upsilonv
    =
    (\thetav,\fv)
    =
    (\Pi_{\dimtotal}\upsilonv,\kappav)
    =
    (\thetav,\Pi_{\dimh}\etav,\kappav)\in \R^{\dimp+\dimh}\times l^2.
\label{etavsdimh}
\end{EQA}    
where \( \kappav = (\eta_{\dimh + 1},\ldots)^{\T} \) stands for the remaining components of the expansion \eqref{eq: expansion of indexfunciton}. We repeat the definitions the sieve estimator \(\tilde \upsilonv_{\dimh} \), its possibly biased target \(\upsilonvs_{\dimh} \) and the full oracle \(\upsilonvs\in \Ups\subset l^2 \)
\begin{EQA}
\tilde \upsilonv_{\dimh}&=&\argmax_{\upsilonv\in\Ups_{\dimh}}\LL_{\dimh}(\upsilonv),\\
\upsilonvs_{\dimh}&=&(\thetavs_{\dimh},\etavs_{\dimh})=\argmax_{\upsilonv\in{\Ups_{\dimh}^*}}\E[\LL_{\dimh}(\upsilonv)],\label{eq: def of full biased target}\\
\upsilonvs&=&(\Pi_{\dimtotal}\upsilonvs,\kappavs)=\argmax_{\upsilonv\in\Ups\subset l^2}\E[\LL(\upsilonv)],
\end{EQA}
where \(\LL(\cdot)\) is the likelihood functional in \eqref{eq: likelihood functional} for \(\dimh=\infty\). 
We set
\begin{EQA}
\Ups_{\dimh}\eqdef\{(\thetav,\etav)\subset S^{\dimp,+}_{1}\times \R^{\dimh},\, \|\etav\|\le \rr^{\circ}\}, &\quad \Ups^*_{\dimh}\eqdef\{(\thetav,\etav)\subset S^{\dimp,+}_{1}\times \R^{\dimh}\},
\end{EQA}
with some \( \rr^{\circ}>\infty\) defined in Lemma \ref{lem: a priori a priori accuracy for rr circ}. 
\begin{remark}
We will see that \((\upsilonvs_{\dimh},0)\in l^2\) lies close to the true point \(\upsilonvs\in l^2\) but we will not proof that it is unique. We neither proof or use uniqueness of the profile ME either. In the following we will denote by \(\upsilonvs_{\dimh}\) the set of maximizers and we will always make statements about \(\tilde\thetav_{\dimh}\in\R^{\dimp}\), whereby we mean any element of the set of maximizers of the profiled likelihood functional. Non-uniqueness is not a problem, as the concentration on the local set \(\Upss\) is ensured via Theorem \ref{theo: large def with K}. 
\end{remark}

\begin{remark}
Note that we maximize over different sets. To control the large deviations and avoid boundary effects we have to ensure that with overwhelming probability \(\tilde\ups_{\dimh}\subset \interior\{\Ups_{\dimh}\}\subset \Ups^*_{\dimh}\). We do this with Lemma \ref{lem: a priori a priori accuracy for rr circ}, which tells us that we may set \(\rr^{\circ}\le \CONST \sqrt{\dimh}\) with some constant \(\CONST\in\R\). This lemma also ensures that the alternating sequence \( (\tilde\thetav_k,\tilde\etav_{k(-1)})_{k\in\N}\) from Section \ref{sec: initial guess} lies in \(S^{\dimp,+}_{1}\times B^{\dimh}_{\rr^{\circ}}(0)\).
\end{remark}

We define the \emph{information operator} \(\DF^2 \) similarly to the Fisher information matrix as the Hessian operator of the expected value of the likelihood functional:
\begin{EQA}
\label{DFc2se}
    \DF^{2}(\ups)
    & \eqdef &
    - \nabla^{2} \E \LL(\upsilonv)= - \nabla^{2} \E \LL(\thetav,\fv).
\end{EQA}
%
%
Consider 
the following block representations of of the \emph{information operator}:
\begin{EQA}[c]
    \DF^{2}
    =n\left(
      \begin{array}{ccc}
        \DP^2 & \AF_{\thetav\etav} \\
        \AF_{\thetav\etav}^\T & \HF^2 
      \end{array}
    \right)=
     \left(
      \begin{array}{ccc}
        \DF_{\dimh}^2 & \AF_{\upsilonv\kappa} \\
        \AF_{\upsilonv\kappa}^\T & \HF^2_{\kappa\kappa} 
      \end{array}
    \right)=n
    \left(
      \begin{array}{ccc}
        \DP^2 & A_{\dimh} & \AF_{\thetav\kappav} \\
        A_{\dimh}^\T  & \HH^2_{\dimh} & \AF_{\etav\kappav} \\
        \AF_{\kappav\thetav} & \AF_{\kappav\etav} & \HF^2_{\kappa\kappa}  
      \end{array}
    \right).
\label{DFc012se}
\end{EQA}  
where \( \AF_{\upsilonv\kappav} \) is a - possibly unbounded - operator from \(l^2\) to \(\R^{\dimp + \dimh}\). Define \(c_{\DF}\eqdef\lambda_{\min}(\DF_{\dimh}(\upss_{\dimh}))/\sqrt n\), where \(\lambda_{\min}(\DF_{\dimh})\in\R\) denotes the smallest eigenvalue of \(\DF_{\dimh}\in\R^{\dimtotal\times\dimtotal}\). In Lemma \ref{lem: D_0 dimh is boundedly invertable} 
we derive that \(c_{\DF}>0\). Furthermore we introduce the influence matrix and the score
\begin{EQA}
\DPr_{\dimh}^{-2}=\Pi_{\thetav}\DF_{\dimh}^{-2}\Pi_{\thetav}^\T,\,&\quad \xivr_{\dimh}=\nabla_{\thetav}\zeta(\upss_{\dimh})-A_{\dimh} \HH_{\dimh}^{-2}\nabla_{\etav}\zeta(\upss_{\dimh}), &\quad \zetav=\LL-\E_{\varepsilon}\LL,	
\end{EQA}
where \(\E_{\varepsilon}\) denotes the expectation operator of the law of \((\varepsilon_i)_{i=1,\ldots,n}\) given \((\Xv_i)_{i=1,\ldots,n}\).

\subsection{Properties of the Wavelet Sieve profile M-estimator}
\label{sec: results}
This section presents the application of the results of \cite{AASP2013} to the estimator \(\tilde\thetav_{\dimh}\) in \eqref{eq: def estimator}. Unfortunately a presentation of the results in full detail would involve constants that are characterized by formulas that would cover many pages. This is why in this work we restrict ourselves to the mere presentation of an upper bound for the critical dimension. This means that we do not specify the size of the appearing constants even though this would be crucial in a true finite sample approach. So whenever there appears a constant \(\CONST>0\) without further remarks it is a polynomial of \(\|\psi\|_{\infty},\|\psi'\|_{\infty},\|\psi''\|_{\infty},\CONST_{\|\fvs\|},s_{\Xv}, \frac{1}{c_{p_{\Xv}}}\), etc. where \(\psi:\R\to\R\) is introduced in Section \ref{sec. choice of basis}.. Also - in the proofs - the same symbol \(\CONST\) can stand for different values, that do not depend on \(\dimtotal, \dimh,n,\xx\). We use this convention to make the presentation less cumbersome and hope the reader appreciates this despite the loss of rigor.

Define  
\begin{EQA}
\breve \Excgr(\xx)&=&\CONST_{\diamond}\frac{{\dimtotal}^{5/2}+\CONST_{bias}{\dimtotal}^{7/2}+\xx }{\sqrt{n}},
\end{EQA}
where \(\CONST_{\Excgr}>0\) is a polynomial of \(\|\psi\|_{\infty},\|\psi'\|_{\infty},\|\psi''\|_{\infty},\CONST_{\|\fvs\|},s_{\Xv}\), etc..
We get the following result by applying Theorem \ref{theo: main theo finite dim regression}
\begin{proposition}
\label{prop: main results finite dim}
Assume \((\mathbf{\mathcal A})\). If \(\CONST_{bias}=0\) suppose that \( \dimh^{-(2\alpha+1)}n\to 0\) and that \({\dimtotal}^4/n\to 0\). If \(\CONST_{bias}>0\) suppose that \({\dimtotal}^{6}\log(n)/n\to 0\) and that \(\dimh^{-2(\alpha-1)}n\to 0\). If \(n\in\N\) is large enough,
it holds with probability greater than \(1-12\ex^{-\xx}-\exp\left\{-\dimh^{3}\xx\right\}-\exp\left\{- n c_{\bb{(Q)}}/4\right\}\)
\begin{EQA}
    \bigl|2 \Lr(\tilde{\thetav}_{\dimh},\thetavs_{\dimh}) - \| \xivr_{\dimh} \|^{2} \bigr|
    &\le& \CONST\left(\sigma \sqrt{\dimp+\xx}+\breve \Excgr(\xx)\right)\breve \Excgr(\xx), \\
        \bigl\|\DPr_{\dimh}(\upss_{\dimh}) \bigl( \tilde{\thetav}_{\dimh} - \thetavs_{\dimh} \bigr) 
        - \xivr_{\dimh} 
    \bigr\|
    &\le& 
    \breve \Excgr(\xx).
\label{eq: main results finite dim}
\end{EQA}
where \(c_{\bb{(Q)}},\CONST>0\).
\end{proposition}

\begin{remark}
The constant \(c_{\bb{(Q)}}>0\) is derived in the proof of Lemma \ref{lem: condition L_rr} and does not depend on \(\xx,\,n,\,\dimtotal\).
\end{remark}

\begin{remark}\label{rem: size of n}
The necessary size of \(n\in\N\) is determined by the speed with which \({\dimtotal}^{4}/ n\to 0\) and \(\dimh^{-2\alpha-1}n\to 0\) or \({\dimtotal}^{6}\log(n)/ n \allowbreak\to 0\) and \(\dimh^{-2(\alpha-1)}n\to 0\) respectively in the cases \(\CONST_{bias}=0\) or \(\CONST_{bias}>0\) respectively. In the proof of Proposition \ref{prop: main results finite dim} we impose conditions on \(n\in\N\) of the kind
\begin{EQA}
{\dimtotal}^{2}/\sqrt n\le \CONST_1^{-1},& & \dimh^{-2\alpha-1}n\le\CONST_2^{-1},
\end{EQA}
for certain constants \(\CONST_1,\CONST_2>0\) that are polynomials of \(\|\psi\|_{\infty}\) ,\(\|\psi'\|_{\infty}\), \break\(\|\psi''\|_{\infty}\), \(\CONST_{\|\fvs\|}\), \(L_{\nabla \Phi}\), \(s_{\Xv}\).
\end{remark}

So far we only addressed the behavior of the sieve profile ME with respect to the possibly biased target \(\thetavs_{\dimh}\in\R^{\dimp}\) and with a weighting matrix that depends on the dimension \(\dimh\in\N\) of the nuisance parameter \(\etav\in\R^{\dimh}\). The next result will specify the finite sample properties of \(\DPr \bigl( \tilde{\thetav} - \thetavs \bigr)\in\R^{\dimp} \) where
\begin{EQA}[c]
    \DPr^{-2}=\Pi_{\thetav}\DF^{-2}(\upss)\Pi_{\thetav}^\T\in\R^{\dimp\times\dimp}.
\label{DPrc2se}
\end{EQA}
We get the following result.

\begin{proposition}
\label{prop: semi sieve bias}
Assume \((\mathbf{\mathcal A})\). If \(\CONST_{bias}=0\) suppose that \( \dimh^{-(2\alpha+1)}n\to 0\) and that \({\dimtotal}^4/n\to 0\). If \(\CONST_{bias}>0\) suppose that \({\dimtotal}^{6}\log(n)/n\to 0\) and that \(\dimh^{-2(\alpha-1)}n\to 0\). If \(n\in\N\) is large enough
it holds with probability greater than \(1-12\ex^{-\xx}-\exp\left\{-\dimh^{3}\xx\right\}-\exp\left\{- n c_{\bb{(Q)}}/4\right\}\)
\begin{EQA}
&&\nquad\bigl\| 
        \DPr_{\dimh}(\upss_{\dimh}) \bigl( \tilde{\thetav}_{\dimh} - \thetavs \bigr) 
        - \xivr_{\dimh}(\upsilonvs_{\dimh}) 
    \bigr\|\\
    	&\le& \left(2+ \sqrt{\frac{1+\corrDF^2}{1-\corrDF^2}}2\right)\breve \Excgr(\xx)+\alpha(\dimh),\end{EQA}    
and
\begin{EQA}
  &&\nquad  \bigl| 2 \Lr(\tilde{\thetav}_{\dimh},\thetavs) - \| \xivr_{\dimh}(\upsilonvs_{\dimh}) \|^{2}\bigr| \le \CONST\left(\sigma \sqrt{\dimp+\xx}+\breve \Excgr(\xx)+\alpha(\dimh)\right)\left( \breve \Excgr(\xx)+\alpha(\dimh)\right),
\end{EQA}    
where 
\begin{EQA}[c] 
\alpha(\dimh)\le\CONST\sqrt n\left(\dimh^{-(\alpha+1/2)}+\CONST_{bias} \dimh^{-(\alpha-1)} \right),
\end{EQA}
and
\begin{EQA}
\rr_{\dimp}^*&\le&\CONST\breve \Excgr(\xx)+\alpha(\dimh).
\end{EQA}
Further if \(\CONST_{bias}=0\) and \({\dimtotal}^{5/2}/\sqrt n\to 0\) we find as \( \nsize \to \infty \)
\begin{EQA}
    \DPr \bigl( \tilde{\thetav}_{\dimh} - \thetavs \bigr)
     \tow 
    \ND(0,\sigma^2\Id_{\dimp}), &\quad
      2 \Lr(\tilde{\thetav}_{\dimh},\thetavs)
   \tow 
  \chi^2_{\dimp}.&\quad\label{eq: single index asymptotics}
\end{EQA}
\end{proposition}

\begin{remark}
The constraints \( \dimh^{-(2\alpha+1)}n\to 0\) and \({\dimtotal}^{5/2}/\sqrt n\to 0\) exclude the case \(\alpha\le 2\). But note that if \(0<\alpha-2=\eps\) and \(\dimh\ge n^{1/5-\delta}\) with \(\delta>2\eps/(25+5\eps)\) we get
\begin{EQA}[c]
\dimh^{-2\alpha-1}n^{1}\le n^{-(1+2\varepsilon_{\alpha}/5)+\delta(2\alpha+1)+1}=  n^{-2\eps_{\alpha}/5+\delta(5+2\eps)}\to 0,
\end{EQA}
such that \(n=o(\dimh^{2\alpha+1})\) and \(\dimtotal=o(n^{1/5})\). Also note that the choice \(\dimh=n^{1/(2\alpha+1)}\) is the optimal choice for \(\dimh\) - for known \(\thetavs\in\R^{\dimh}\) - in the given setting as a consequence of the bias variance decomposition in nonparametric series estimation; see \cite{newey1997convergence}. It leads to the optimal rate for the mean squared error in the estimation of \(\fs_{\etavs}\), i.e. \(n^{\alpha/(2\alpha+1)}\).
\end{remark}

\begin{remark}
Assume that the model \eqref{eq: single index model in intro to single index chapter} is correct. We will see in Section \ref{sec: calculating the elements} that then 
\begin{EQA}[c]\label{eq: def of efficient full information single index}
\sigma^2\DF^{2}(\upss)=\Cov(\nabla\LL(\upss)),
\end{EQA}
Such that with \eqref{eq: single index asymptotics}
\begin{EQA}
\sqrt{n}\bigl( \tilde{\thetav}_{\dimh} - \thetavs \bigr)\tow \ND(0,\sigma^2\DPr^{-2}), &\quad \sigma^2\DPr^{-2}=\Pi_{\thetav}^\T \Cov(\nabla\LL(\upss))\Pi_{\thetav}.
\end{EQA}
It can be shown that this is the lower bound for the variance of regular estimators of \(\thetavs\in\R^{\dimh}\) if \(\varepsilon\sim \ND(0,\sigma^2)\) and \(\Xv\) is uniformly distributed on \(\B_{s_{\Xv}}\subset \R^{\dimp}\).
\end{remark}

\begin{remark}
Note that we do not show any weak convergence statements for the case that \(\CONST_{bias}>0\). The approach of \cite{AAbias2014} is not applicable - at least not with the arguments we use for the case  \(\CONST_{bias}=0\) in Lemma \ref{lem: conditions theta eta}. Also note that to control the approximation bias when \(\CONST_{bias}>0\) the necessary smoothness of \(\E[g(\Xv)|\Xv^\T\thetavs=\cdot]=\fv_{\etavs}(\cdot): \R\to \R\) measured in \(\alpha>0\) in \eqref{eq: smoothness of fs} increases from \(\alpha> 2\) to \(\alpha>14/3\) to ensure that \(\alpha(\dimh)\to 0\).
\end{remark}

\subsection{A way to calculate the profile estimator}
\label{sec: initial guess}
In this section we briefly sketch how to actually calculate \(\tilde\upsilonv\in\R^{\dimtotal}\) in practice. For this note that the maximization problem 
\begin{EQA}[c]
\tilde\upsilonv=\argmax \sum_{i=1}^n(Y_i-\fv_{\etav}(\thetav^\T\Xv_i))^2/2,
\end{EQA}
is not convex and thus computationally involved. We propose to obtain the maximizer via the alternation maximization procedure as it is analyzed in \cite{AASPalternating}. To remind the reader this sequential algorithm works as follows: Start with some initial guess \(\tilde \upsilonv^{(0)}\in\Ups\). Then calculate for \(k\in\N\) iteratively
\begin{EQA}
\tilde\upsilonv^{(k,k+1)}&\eqdef&(\tilde \thetav^{(k)},\tilde \etav^{(k+1)})=\left(\tilde \thetav^{(k)},\argmax_{\etav}\LL_{\dimh}(\tilde \thetav^{(k)},\etav)\right),\\
\tilde\upsilonv^{(k,k)}&\eqdef&(\tilde \thetav^{(k)},\tilde \etav_k)=\left(\argmax_{\thetav}\LL_{\dimh}(\thetav, \tilde \etav^{(k)}),\tilde \etav^{(k)}\right).
\label{eq: alternating sequence single index}
\end{EQA}
In the following we write \(\tilde \upsilonv^{(k,k(+1))}\) in statements that are true for both \(\tilde \upsilonv^{(k,k+1)}\) and \(\tilde\upsilonv^{(k,k)}\). For the initial guess we propose a simple grid search. For this generate a uniform grid \(G_N\eqdef(\thetav_1,\ldots,\thetav_N)\subset S_1^+\) and define
\begin{EQA}[c]
\label{eq: def of initial guess}
\tilde \upsilonv^{(0)}\eqdef \argmax_{\substack{(\thetav,\etav)\in \Ups \\ \thetav\in G_N} }\LL_{\dimh}(\upsilonv).
\end{EQA}
Note that given the grid the above maximizer is easily obtained. Simply calculate
\begin{EQA}
\label{eq: calculation of eta initial guess}
\tilde \etav^{(0)}_l&\eqdef&\argmax \LL(\thetav_l,\etav)\\
	&=&\left( \frac{1}{n}\sum_{i=1}^{n}\basX\basX^\T(\Xv_i^\T\thetav_l)\right)^{-1} \frac{1}{n}\sum_{i=1}^{n}Y_i\basX^\T(\Xv_i^\T\thetav_l)\in \R^{\dimh},
\end{EQA}
where by abuse of notation \(\basX=(\basX_1,\ldots,\basX_{\dimh})\in \R^{\dimh}\). Observe that
\begin{EQA}[c]
\tilde \upsilonv^{(0)}= \argmax_{l=1,\ldots, N}\LL_{\dimh}(\thetav_l,\tilde \etav^{(0)}_l).
\end{EQA}
Define the fineness of the grid via \(\tau\eqdef \sup_{\thetav,\thetavd\in G_N}\|\thetav-\thetavd\|\). To asses the statistical properties of the alternating procedure we can derive the following result via an application of Theorem \ref{theo: alternating stat props regression}.

\begin{proposition}\label{prop: alternating single index}
If \(\CONST_{bias}=0\) set \(\tau =o({\dimtotal}^{-3/2})\) and \(\dimh^{4}=o(n)\) and assume that \(\dimh^{-(2\alpha+1)}n\to 0\). If \(\CONST_{bias}>0\) set \(\tau=o(\dimh^{-11/4})\) and \(\dimh^{6}\log(n)/n\to 0\) and assume that \(\dimh^{-2(\alpha-1)}n\to 0\). Furthermore let \(\xx\le 2\tilde\nu^2\tilde \gm^2(1+\CONST_{bias})n\). With the initial guess given by Equation \eqref{eq: def of initial guess} the alternating sequence satisfies with probability greater than \(1-12\ex^{-\xx}-\exp\left\{-\dimh^{3}\xx\right\}-\exp\left\{- n c_{\bb{(Q)}}/4\right\}\)
\begin{EQA}
	\bigl\| 
        \DPr_{\dimh}(\upss_{\dimh}) \bigl( \tilde\thetav^{(k)} - \thetavs \bigr) 
        - \xivr 
    \bigr\|
    &\le& 
    \breve\Excgr_{Q}(\rr_{k},\xx) ,
\label{eq: alternating fisher single index}
	\\
\label{eq: alternating wilks single index}
    \bigl| 2 \Lr(\tilde\thetav^{(k)},\thetavs) - \| \xivr \|^{2} \bigr|
    &\le&
     5\left(\|\xivr \|+\breve\Excgr_{Q}(\rr_{k},\xx)\right)\breve\Excgr(\rr_{k},\xx),
\end{EQA} 
where
\begin{EQA}[c]
\breve\Excgr_{Q}(\rr,\xx)\le\CONST_{\Excgr}\left(\frac{\left({\dimtotal}^{3/2}+ \xx + \CONST_{bias}{\dimtotal}^{5/2}\right)\rr^2}{\sqrt n} \right),
\end{EQA}
and where with some constant \(\CONST\)
\begin{EQA}
\rr_{k}&\le& (1-\sqrt\corrDF)^{-1}\CONST\left(\sqrt{\xx+\dimtotal}+2\corrDF^{k}\RR(\xx)\right),\\
\RR(\xx)&\le&  \CONST \sqrt{\dimtotal(1+\CONST_{bias}\log(n))+\xx+ (1+\CONST_{bias}\sqrt{\dimh})n\tau^2+\sqrt{n}\tau\sqrt{\xx}}.
\end{EQA}
\end{proposition}
\begin{remark}
The constraint \(\tau=o({\dimtotal}^{-3/2})\) implies that for the calculation of the initial guess the vector \(\tilde \etav^{(0)}_{(l)}\) in \eqref{eq: calculation of eta initial guess} and the functional \(\LL(\cdot)\) have to be evaluated \(N={\dimtotal}^{3(\dimp-1)/2}\) times. This means - since \(\dimh^{5}=o(n)\) is necessary for the right-hand sides in \eqref{eq: alternating wilks single index} and \eqref{eq: alternating fisher single index} to vanish- that we need an accuracy of the first guess of order \(o({n}^{-3/10})\), while the accuracy of the output of the alternating procedure is of order \(n^{-1/2}\). In the case that \(\CONST_{bias}>0\) we need an accuracy of the first guess of order \(o(n^{-9/26})\) because \(\tau=o(\dimh^{-9/4})\) and \(\dimh^{13/2}= o(n)\). Although this difference does not seem large the number of grid points necessary for  \(n^{-1/2}\)-accuracy of the grid search is by a factor \(n^{(\dimp-1)/5}\) or \(n^{2(\dimp-1)/13}\) larger than those for a sufficient initial guess. 
\end{remark}

Define the local neighborhood around the ME \(\tilde\ups\) (we suppress \(\cdot_{\dimh}\) here)
\begin{EQA}[c]
\tilde \Upss(\rr)\eqdef\{\ups\in\Ups:\,\|\DF(\ups-\tilde\ups)\|\le\rr\}.
\end{EQA}
If not the statistical properties but mere convergence of the sequence \(\tilde \ups^{(k,k(+1))}\to\tilde\ups\) is desired we can prove the following result using Theorem 2.4 of \citep{AASPalternating}.

\begin{proposition}\label{prop: convergence to ME}
Take the initial guess given by Equation \eqref{eq: def of initial guess}. Assume \((\mathbf{\mathcal A})\). If \(\CONST_{bias}=0\) set \(\tau =o({\dimh}^{-3/2})\) and \(\dimh^{4}=o(n)\) and assume that \(\dimh^{-(2\alpha+1)}n\to 0\). If \(\CONST_{bias}>0\) set \(\tau=o(\dimh^{-11/4})\) and \(\dimh^{6}\log(n)= o(n)\) and assume that \(\dimh^{-2(\alpha-1)}n\to 0\). Let \(\xx>0\) be chosen such that
\begin{EQA}[c]
\xx\le \frac{1}{2}\left(\tilde \nu^2 n\tilde \gm^2  - \log(\dimtotal)\right)\wedge \dimtotal.
\end{EQA}
Then
\begin{EQA}
\P\left(\bigcap_{k\in\N}\left\{\tilde \ups^{(k,k(+1))}\in \tilde \Upss(\rr_k^*)\right\}\right)&\ge& 1-10\ex^{-\xx}-\exp\left\{-\dimh^{3}\xx\right\}\\
	&&-\exp\left\{- n c_{\bb{(Q)}}/4\right\},
\end{EQA}
where with \(\kappa(\xx,\RR)=O({\dimtotal}^2/\sqrt{n}+\CONST_{bias}{\dimtotal}^{3}/\sqrt{n})\to 0\)
\begin{EQA}\label{eq: bound for rrk sequence convergence single index}
\rr_k^*&\le&\begin{cases} \corrDF^{k} 2\sqrt{2} \frac{1}{1-\kappa(\xx,\RR) k}(\RR+\rups), & \kappa(\xx,\RR) k\le 1,\\
   2\frac{1-\corrDF}{\kappa(\xx,\RR)}\tau(\xx)^{k/\log(k)}(\RR+\rups) , & \text{otherwise,}\end{cases}
\end{EQA}
with 
\begin{EQA}
(\RR+\rups)&\le&  \CONST \sqrt{\dimtotal(1+\CONST_{bias}\log(n))+\xx+ (1+\CONST_{bias}\sqrt{\dimh})n\tau^2+\sqrt{n}\tau\sqrt{\xx}},\\
\tau(\xx)&\eqdef&\left(\frac{\kappa(\xx,\RR)}{1-\corrDF}\right)^{L(k)}<1,\\
L(k)&\eqdef& \left\lfloor \frac{\log(1/\corrDF)-\frac{1}{k}\left(\log(2\sqrt{2})-\log(\kappa(\xx,\RR) k-1)\right)}{\left(1+\frac{1}{\log(k)}\log(1-\corrDF)\right)}\right\rfloor\in\N,
\end{EQA}
where \(\lfloor x\rfloor\in\N\) denotes the largest natural number smaller than \(x>0\).
\end{proposition}
\begin{remark}
Note that in the case \(\CONST_{bias}=0\) the constraint on the size of the dimension \(\dimtotal\in\N\) for accurate results is weaker in Proposition \ref{prop: convergence to ME} than in Proposition \ref{prop: alternating single index}, as there are no "right-hand sides" and thus \(\dimh^4=o(n)\) is sufficient.
\end{remark}

\subsection{Performance of Projection Pursuit Procedure}
In this section we want to briefly assess the performance of the Projection Pursuit procedure of \cite{Friendman1981}. We assume that the iteration \(k\in\N\) in the alternation maximization procedure is large enough so that we can pretend that one can directly access the maximizer \(\tilde\ups\). Also we assume that the number of iterations \(M\in\N\) is fixed. In the previous sections we already established that for observations of the kind 
\begin{EQA}[c]
Y_i=g(\Xv_i)+\varepsilon_{i}, \text{ }i=1,...,\nsize,
\end{EQA}
the estimator in \eqref{eq: def estimator} satisfies
\begin{EQA}\label{eq: ppp induction start}
&&\nquad \left|\E[Y|\Xv^\T {\thetavs}_{(1)}]-\fv_{\tilde\etav_{(1)}}(\Xv^\T \tilde\thetav_{(1)})\right|\\
&\le& \CONST \left({\rr^*}+\alpha(\dimh) +\Excgr(\xx)+ \|{\DF_{(1)}}^{-1}\nabla \LL_{(1)}(\upss)\| \right)/\sqrt{n},
\end{EQA}
with high probability.
But in each step a new data set is generated, i.e. given \(Y_i(l),\tilde\ups(l)\) we generate
\begin{EQA}[c]
{Y_i}_{(l+1)}\eqdef {Y_i}_{(l)}-\fv_{\tilde\etav_{(l)}}(\Xv_i^\T \tilde\thetav_{(l)})=g_{(l+1)}(\Xv_i)+\varepsilon_{i}+{\tau_i}_{(l)},
\end{EQA}
where
\begin{EQA}
g_{(l)}(\Xv_i)&\approx& \sum_{s=l}^{M}\fv_{{\etavs}_{(s)}}(\Xv_i^\T{\thetavs}_{(s)}), \\ {\tau_i}_{(l)}&=&\sum_{s=1}^{l}\fv_{{\etavs}_{(s)}}(\Xv_i^\T {\thetavs}_{(s)})-\fv_{\tilde\etav_{(s)}}(\Xv_i^\T \tilde\thetav_{(s)}).
\end{EQA}
The errors \({\tau_i}_{(l)}\) are not i.i.d. and not necessarily centered such that we can not directly apply the results from above for \(l>1\). But a slight modification serves a remedy. For this remember that the central tool for Theorems of the type of \ref{theo: main theo finite dim regression} is to bound with probability \(1-\ex^{-\xx}\)
\begin{EQA}[c]
\sup_{\ups\in\Upss(\rups)}\left\|\DF^{-1}\left( \nabla\LL(\ups)-\nabla\LL(\upss) \right)+\DF(\ups-\upss) \right\|\le \Excgr(\rups,\xx),
\end{EQA}
and to show that \(\P(\tilde\ups,\tilde\ups_{\thetavs}\in\Upss(\rups))\ge 1-\ex^{-\xx}\). So we decompose (we suppress \(\cdot_\dimh\) to ease notation)
\begin{EQA}
&&\nquad\LL_{(l)}(\ups,{Y_i}_{(l)})\\
	&=&-\sum_{i=1}^{n}\left(g_{(l)}(\Xv_i)+\varepsilon_i-\fv_{\etav} (\Xv_i^\T\thetav) \right)^2\\
	&&-\sum_{i=1}^{n}{{\tau_i}_{(l-1)}}^2 +2\sum_{i=1}^{n}{\tau_i}_{(l-1)}\left(\fv_{\etav} (\Xv_i^\T\thetav) -\fv_{{\etavs}_{(l)}} (\Xv_i^\T{\thetavs}_{(l)})\right)\\
	&\eqdef& {\LL_{\varepsilon}}_{(l)}(\ups, {Y_i}_{(l)})+{\LL_{(l)}}_{\tau}(\ups, {Y_i}_{(l)}),
\end{EQA}
and define
\begin{EQA}
{\upss_{\dimh}}_{(l)}&\eqdef&\argmax_{\ups\in\Ups_{\dimh}}\E{\LL_{\varepsilon}}_{(l)}(\ups),\\
{{\DF_{\dimh}}_{(l)}}^2&\eqdef &\nabla^2\E[ {\LL_{\varepsilon}}_{(l)}( {\upss_{\dimh}}_{(l)})],\\
{\zetav_{\varepsilon}}_{(l)}(\ups)&\eqdef & {\LL_{\varepsilon}}_{(l)}(\ups)-\E{\LL_{\varepsilon}}_{(l)}(\ups).
\end{EQA}
We assume that the condition (model bias) holds for every function \(g_{(l)}\). With Remark \ref{rem: how to ensure smoothness in biased case}, Lemma \ref{lem: conditions example} and Lemma \ref{lem: conditions theta eta} this means that the conditions of Section \ref{sec: conditions semi} and \ref{sec: synapsis sieve tools} are met for \(({\LL_{\varepsilon}}_{(l)},\Ups_{\dimh},{{\DF_{\dimh}}_{(l)}} )\) with high probability for every \(l=1,\ldots ,M\). 
It remains to show that for each \(l\in\N\) and \(\dimh\in\N\) large enough the contribution of \({\tau_i}_{(l)}\) remains insignificant. We do this in the following Proposition. 

\begin{proposition}\label{prop: convergence of PPP}
Assume that \(M =O(\dimtotal)\) and that the conditions \((\mathcal A)\) hold for every \(l=1\ldots,M\). Assume further \(\frac{{\dimtotal}^{3}\log(n)M+\xx}{\sqrt{n}}\to 0\) and assume that \(\dimh^{-2(\alpha-1)}n\to 0\). With probability greater than 
\begin{EQA}[c]
1-\ex^{-\xx}-M \left(12\ex^{-\xx} +\exp\left\{-\dimh^{3}\xx\right\}+\exp\left\{- n c_{\bb{(Q)}}/4\right\}\right),
\end{EQA}
we have
\begin{EQA}
&&\nquad \sup_{\xv\in B_{s_{\Xv}}(0)}\left|\sum_{l=1}^{M}\fv_{{\etavs}_{(l)}}(\xv^\T {\thetavs}_{(l)})-\fv_{\tilde\etav_{(l)}}(\xv^\T \tilde\thetav_{(l)})\right|\\
	&\le& \CONST M \sqrt{\dimh}\left(\frac{{\dimtotal}^{7/2}+\xx}{n}+\frac{\sqrt{\dimtotal+\xx}}{\sqrt{n}}\right).
\end{EQA}
\end{proposition}


\begin{remark}
Denoting the bias
\begin{EQA}[c]
b(M)\eqdef \left\|g-\sum_{l=1}^{M}\fv_{{\etavs}_{(l)}}(\cdot^\T {\thetavs}_{(l)})\right\|_{\infty},
\end{EQA}
Proposition \ref{prop: convergence of PPP} implies if \(\xx\le \dimtotal=o(n^{1/6})\) that
\begin{EQA}[c]
\sup_{\xv\in B_{s_{\Xv}}(0)}\left|g(\xv)-\sum_{l=1}^{M}\fv_{\tilde\etav_{(l)}}(\xv^\T \tilde\thetav_{(l)})\right|\le \CONST M o(n^{-1/3})+b(M).
\end{EQA}
Depending on the speed with which \(b(M)\) decays in \(M\) the resulting rate can be substantially faster than \(n^{-\alpha/(2\alpha+\dimp)}\).
\end{remark}


\section{Technical peculiarities}
Before we explain in more detail how the above statements can be derived based on the theory presented in \cite{AASP2013} and \cite{AASPalternating}, we address two technical issues that arise with the regression setup with random design and due to the peculiarities of the sieve approach.

\subsection{Choice of basis}
\label{sec. choice of basis}
To control the approximation bias of the sieve estimator \(\tilde \thetav_{\dimh}\in\R^{\dimp}\) with the approach from \cite{AAbias2014} we can not use any basis \((\basX_k)_{k\in\N}\) in \(L^2([-s_{\Xv},s_{\Xv}])\). We need to show in the proof of Lemma \ref{lem: conditions theta eta} that the following terms vanish as \(\dimh\to \infty\)
\begin{EQA}[c]
\label{eq: basis vanish condition}
\int_{\R}\basX_{\dimh+k}(x)\basX_{\dimh+l}(x) p_{\Xv^\T\thetavs}(x)dx;\, l,k\in\N,
\end{EQA}
where \(p_{\Xv^\T\thetavs}\) denotes the density of \(\Xv^\T\thetavs\in\R\). But it is not clear whether terms as in \eqref{eq: basis vanish condition} vanish for any basis of \(L^2([-s_{\Xv},s_{\Xv}])\).  Of course - following \cite{Shen1997} - we could assume that the basis is orthogonal in the inner product induced by the Hessian \(\nabla^2\E\LL(\upss)\). But for this one would need to know the true parameter \(\thetavs\in\R^{\dimp}\) and the density \(p_{\Xv}:\R^{\dimp}\to \R\) in advance. We want to avoid such assumptions and also the tedious calculations resulting from using an estimator of \(\thetavs\) plugged into an estimator of \(p_{\Xv^\T\cdot}\) for the construction of a suitable basis. As it turns out an orthonormal wavelet basis is suitable for our purpose. For high indexes \(k\in\N\) the support of each wavelet \(\basX_k\) is contained in a small interval on which the density \(p_{\Xv^\T\thetavs}\) can be well approximated by a constant. Due to orthogonality and shrinking supports of the basis the term in \eqref{eq: basis vanish condition} can be shown to diminish sufficiently fast for a Lipschitz continuous density \(p_{\Xv^\T\thetavs}\) (see Lemma \ref{lem: conditions theta eta}). The trouble is that our approach relies on smoothness of the basis elements. Consequently we need a smooth orthogonal wavelet basis on an interval. Thanks to \cite{Daubechies} such a basis \((\Psi_{n,m})_{n,m\in\Z}\) is available on \(L^2(\R)\). We can use this basis to construct an appropriate basis \((\basX_k)_{k\in\N}\) for \(L^2([-s_{\Xv},s_{\Xv}])\). This basis will have all the properties needed for the proof of Lemma \ref{lem: conditions theta eta} and thus will allow us to control the approximation bias in \eqref{eq: intro single index approx bias}.

To understand the choice of this basis \((\basX_k)_{k\in\N}\) we first have to briefly explain how the Daubechies wavelets are derived. To ease understanding we adopt the notation of \cite{Daubechies}. 
Starting with a scaling function \(\phi:\R\to \R\) where \(\|\phi\|_{L^2(\R)}=1\) one obtains a sequence of nested spaces, i.e. for \(j\in\N\)
\begin{EQA}
V_j=\Span\{2^{-j/2}\phi(2^{-j}\cdot-n); n\in \Z\}&\subset& L^2(\R), \\
 \ldots\subset  V_{1}\subset V_{0}\subset V_{-1}\subset\ldots &\subset& L^2(\R). 
\end{EQA}
If the scaling function \(\phi:\R\to \R\) satisfies certain properties one can show that \(\bar{\bigcup_{n\in\Z}V_{n}}= L^2(\R)\) and that \((2^{-j/2}\phi(2^{-j}\cdot-n))_{m\in\Z}\) is an orthonormal basis in \(V_j\subset L^2(\R)\) for every \(j\in\Z\) (see Theorem 6.3.6 of \cite{Daubechies}). Denote for each \(j\in\Z\) by \(W_j \subset L^2(\R)\) the orthogonal complement of \( V_{j+1}\subset L^2(\R)\) in \(V_{j}\subset L^2(\R)\). This gives
\begin{EQA}
\label{eq: space decomposition}
V_j=V_{j+1}\oplus W_{j+1}=\bigoplus_{\substack{k > j,\\ k\in\Z}} W_k, & \text{such that }& L^2(\R)= \bigoplus_{j\in\Z} W_j.
\end{EQA}
The idea of Daubechies wavelets is to find a function \(\psi\in W_1\) that satisfies with \(\psi_{j,n}\eqdef 2^{-j/2}\psi(2^{-j}\cdot+n)\)
\begin{EQA}
W_j=\text{span}(\Psi_{j,n}; n\in \Z), & &\langle \psi_{j,n}, \psi_{j,n'}\rangle_{L^2} =\delta_{n,n'},\, n,n'\in \Z.
\end{EQA}
This is indeed possible. For this denote 
\begin{EQA}
h_n=\langle \phi, \phi(2\cdot+m)  \rangle, m\in\Z, &\text{i.e.} & \phi=\sqrt 2\sum_{m\in\Z}h_m  \phi(2\cdot-m),
\end{EQA}
and define
\begin{EQA}[c]
\psi=\sqrt 2\sum_{m\in\Z}(-1)^{n-1} h_{-n-1}  \phi(2\cdot-n).
\end{EQA}
Theorem 6.3.6 and Chapter 6.4 of \cite{Daubechies} and the table 3.1 of \cite{longmem} show that there exists a scaling function \(\phi_9:\R\to \R\) for which the associated family \(\psi_{j,n}\eqdef 2^{-j/2}\psi(2^{-j}\cdot+n)\) satisfies
\begin{EQA}
\label{eq: properites of Daubechies wavelets}
(\psi_{j,n})_{j,n\in \Z}\text{ ONB of }L^2(\R),\, & \text{supp}(\psi)\subseteq [0,17],\, & \psi\in C^2(\R).
\end{EQA}

Thus we obtain a well-suited basis for \(L^2(\R)\) but we only need one for \linebreak \(L^2([-s_{\Xv},s_{\Xv}])\). We could simply embed 
\begin{EQA}
L^2([-s_{\Xv},s_{\Xv}])\to L^2(\R), &&f(\cdot)\mapsto f(\cdot)1_{[-s_{\Xv},s_{\Xv}]},
\end{EQA}
and use that basis but this would mean that we have to include basis functions \(\psi_{j,n}\in L^2(\R)\) for positive \(j\in\N\) as well. We want to avoid this as it is not necessary for our purpose. Instead we do the following: First adapt the scale and support of the basis and the corresponding shift operation to the interval via redefining
\begin{EQA}
\phi_{9,s_{\Xv}}(t)=s_{\Xv}^{-1/2}\phi_9(s_{\Xv}^{-1}t+1), && \psi_{s_{\Xv}}(t)=s_{\Xv}^{-1/2}\psi(s_{\Xv}^{-1}t+1).
\end{EQA}
The associated wavelet basis \(\psi_{j,n}\eqdef 2^{-j/2}\psi_{s_{\Xv}}(2^{-j}\cdot+ns_{\Xv})\) still satisfies all properties in \eqref{eq: properites of Daubechies wavelets} where the support is adapted to read \([-s_{\Xv},16s_{\Xv}]\). Next note that \eqref{eq: space decomposition} and the definition of the subspaces implies
\begin{EQA}[c]
L^2(\R)= V_{0}\oplus\bigoplus_{j\in\N} W_{-j},
\end{EQA}
where the definition is adapted to read \(V_j=\bar{\text{span}}\{2^{-j/2}\phi_{7,s_{\Xv}}(2^{-j}\cdot-n s_{\Xv}); n\in \Z\}\subset L^2(\R)\). As we only have to approximate functions that are nonzero on \([-s_{\Xv},s_{\Xv}]\) this suggest the following basis: for \(k=(2^{j_k}-1)17+r_k \in\N\) where \(j_k\in\N_{0}\) and \(r_k\in\{0,\ldots, 2^{j_k}17-1\}\) we set
\begin{EQA}[c]
\basX_{k}\eqdef\begin{cases}
			\phi_{9,s_{\Xv}}(t+(k-1)s_{\Xv}) & \text{ if }k\le 17,\\
			\psi_{-j_k,r_k} &\text{ if } k>17.
			\end{cases}
\end{EQA}
So in words we include all elements of a basis for \(L^2(\R)\) which have a support with nonempty intersection with \([-s_{\Xv},s_{\Xv}]\). Restricting the preimage of the elements of the closed span of these functions to \([-s_{\Xv},s_{\Xv}]\) we end up with a basis for \(L^2([-s_{\Xv},s_{\Xv}])\), that is contained in 
\(C^3(\R)\) and satisfies for any \(l,k\in\N\) with \(k=(2^{j_k}-1)17+r_k\in\N\)
\begin{EQA}
\langle \basX_l,\basX_k\rangle_{L^2(\R)}=\delta_{l,k}, & &|\text{supp}(\psi_k)|\le 2^{-j_k}17 s_{\Xv}.
\end{EQA}
Furthermore this basis has another useful property that will come in handy in the proof of Lemma \ref{lem: conditions theta eta}: For any \(k\in\N\) with \(k=(2^{j_k}-1)17+r_k\in\N\) it holds
\begin{EQA}
&&\nquad\Big|\Big\{l=(2^{j_l}-1)17+r_l\Big|\,r_l\in\{0,\ldots, 2^{j_l}17-1 \},\, \text{supp}(\basX_k)\cap \text{supp}(\basX_l)\neq \emptyset\Big\}\Big|\\
	&=& \left\ulcorner2^{(j_l-j_k)}17\right\urcorner. \label{eq: number of intersections}
\end{EQA}
In words this means that the number of nonempty intersections of the supports of \(\basX_k\) and \(\basX_l\) can be controlled well. For nearly all basis functions \(\basX_l\) with \(l\ge k\) we have
\begin{EQA}[c]
\int_{\R}\basX_{k}(x)\basX_{l}(x) p_{\Xv^\T\thetavs}(x)dx=0.
\end{EQA}
This will allow to satisfy the conditions \( \bb{(}\kappav\bb{)} \) and \( \bb{(}\upsilonv\kappav{)} \) from \cite{AAbias2014} in Lemma \ref{lem: conditions theta eta}.

\subsection{Implications of Regression setup}
\label{sec: implications of regression setup}
Due to the regression set up there are some particularities to the analysis that we have to point out here. The definition of \(\upss_{\dimh}\in\Ups\) reads
\begin{EQA}[c]
\upss_{\dimh}\eqdef \argmax_{\ups\in\Ups_{\dimh}}\E\LL_{\dimh}(\ups),
\end{EQA}
where \(\E\) denotes the expectation operator with respect to the joint measure of \((\Xv,\varepsilon)\in\R^{\dimp}\times\R\), similarly \(\DF^{2}(\ups)\) is also based on the full expectation \(\E\). But in Lemma \ref{lem: conditions example} we show the conditions \( {(\CS \DF_{0}) }\), \( {(\CS \rr) }\) and \( {(\CS \DF_{1}) }\) for the random variables
\begin{EQA}[c]
 \nabla(1-\E_{\varepsilon})\LL_{\dimh}(\ups)\in\R^{\dimp+\dimh},
 \end{EQA} 
i.e. we use only the expectation with respect to the noise \((\varepsilon_i)\). This leads to rather weak conditions on the errors \((\varepsilon_i)\). Especially the conditions \( {(\CS \rr) }\) and \( {(\CS \DF_{1}) }\) would otherwise become quite restrictive. But on the other hand this means that a list of additional steps become necessary to apply the theory of \cite{AASP2013} and \cite{AASPalternating}. As becomes evident from the proof of Theorem 2.2 of \cite{AASP2013}, we have to bound the term
\begin{EQA}[c]
\sup_{\ups\in\Upss(\rr)}\|\DF_{\dimh}^{-1}\nabla(\E-\E_{\varepsilon})[\LL_{\dimh}(\upss_{\dimh})-\LL_{\dimh}(\ups)]\|,
\end{EQA}
and ad the obtained bound to the error term \(\breve\Excgr(\rups,\xx)\). Also the probability of the desired bound has to be subtracted from the probability under which the event in Theorem 2.2 of \cite{AASP2013} is valid. See Section \ref{sec: regression version of semi paper} for more details. 
The following lemma serves this bound.

\begin{lemma}\label{lem: additional error from different expectation operator}
With some constant \(\CONST>0\) 
\begin{EQA}
&&\nquad\P\Bigg(\bigcap_{\rr\le \RR}\bigg\{\sup_{\ups\in\Ups}\|\DF^{-1}_{\dimh}\nabla(\E-\E_{\varepsilon})[\LL_{\dimh}(\upss_{\dimh})-\LL_{\dimh}(\ups)]\|\\
	&&\ge \CONST \rr\sqrt{\xx+\dimtotal\log(\dimtotal)}  /\sqrt{n}\bigg\}\Bigg)\le \ex^{-\xx}.
\end{EQA}
\end{lemma}

\begin{remark}
We will see that the error term 
\begin{EQA}[c]
\CONST \rr\sqrt{\xx+\dimtotal\log(\dimtotal)}  /\sqrt{n},
\end{EQA}
is of smaller order than the bounds that we will derive for \(\Excgr(\rr)\) in the subsequent analysis. Consequently we neglect it in the following and let a constant \(\CONST_{\Excgr}>0\) account for its contribution in the formulation of Proposition \ref{prop: main results finite dim}.
\end{remark}
Further in the derivation of the conditions \( {(\CS \DF_{0}) }\), \( {(\CS \rr) }\) and \( {(\CS \DF_{1}) }\) we obtain bounds for \(\nu_{1},\nu_0,\nu_\rr\) that involve terms of the kind
\begin{EQA}
\left\|\E\left[\bb{S}_n\right]-\bb{S}_n \right\|, &\quad \bb{S}_n= \frac{1}{n}\sum_{i=1}^{n} \bb{M}(\Xv_i), &\quad \bb{M}(\Xv_i)\in\R^{\dimtotal\times \dimtotal}.
\end{EQA}
This leads to concentration bounds for sums of i.i.d. matrices which can be handled with the results of \cite{Tropp2012}. We do this in Section \ref{sec: matrix deviation bounds based on tropp}. Again the set on which Theorem 2.2 of \cite{AASP2013} occurs has to intersected with the set on which the matrix deviation bounds are valid. 
Another implication is that when proving condition \({(\cc{L}{\rr})}\) we have to consider \(\E_{\eps}\LL(\ups,\upss_{\dimh})\) instead of \(\E\LL(\ups,\upss_{\dimh})\), which makes the proof quite involved and again makes the restriction to a set of high probability necessary. This is why in Proposition \ref{prop: main results finite dim} the probability of the desired results can only be bounded from bellow by \(1-14\ex^{-\xx}-\CONST\ex^{-n c-\dimtotal \xx}\) instead of \(1-5\ex^{-\xx}\) as in Proposition 2.4 of \cite{AASP2013}.

\section{Synopsis of the finite sample theory for M-Estimators}
In this section we briefly summarize the results of \cite{AASP2013} and \cite{AASPalternating} and thereby adapt them to the regression setting of the given model. 
\subsection{Conditions}
\label{sec: conditions semi}
This section collects the conditions that underlie the results of \cite{AASP2013} and \cite{AASPalternating}. They are taken from \cite{AASP2013} but adapted to our setting. This means in particular that the expectation operator in the moment conditions is \(\E_{\varepsilon}\) and not the full one. \cite{AASP2013} assume that the function
\( \LL(\upsilonv) \colon \allowbreak\R^{\dimtotal}\allowbreak\to \allowbreak\R \) is sufficiently smooth in \( \upsilonv \in \R^{\dimtotal}\), 
\( \nabla \LL(\upsilonv)  \in \R^{\dimtotal}\) stands for the gradient and 
\( \nabla^{2} \E \LL(\upsilonv)  \in\R^{\dimtotal\times\dimtotal}\) for the Hessian of the expectation 
\( \E \LL: \R^{\dimtotal}\to \R  \) at \( \upsilonv \in\R^{\dimtotal}\). By smooth enough we mean that all appearing derivatives exist and that we can interchange \(\nabla \E\LL(\upsilonv) =\E\nabla\LL( \upsilonv) \) on \(\Upss(\rups)\), where \(\rups>0\) is defined in Equation \eqref{eq: def of rups} and \(\Upss(\rr)\) in equation \eqref{eq: def of Ups}.\\

Define the matrices with \(\upss_{\dimh}\in\R^{\dimtotal}\) from \eqref{eq: def of full biased target}
\begin{EQA}\label{DFcseg0}
	\DF_{\dimh}^{2}
    =
    - \nabla^{2} \E \LL_{\dimh}(\upsilonvs_{\dimh})=
    \left( 
      \begin{array}{cc}
        \DP^{2} & \A \\
        \A^{\T} & \HH^{2} \\
      \end{array}  
    \right)\in\R^{\dimtotal\times\dimtotal}, & \VF_{\dimh}^2=\Cov(\nabla \zetav_{\dimh}(\upsilonvs_{\dimh}))\in\R^{\dimtotal\times\dimtotal},
\end{EQA}
Using the matrix \(\DF_{\dimh}^{2} \) we define the local set \(\Upss(\rr) \subset \Ups_{\dimh} \subseteq \R^{\dimtotal}\) with some
\(\rr\ge 0\):
\begin{EQA}[c]
\label{eq: def of Ups}
	\Upss(\rr)
	\eqdef 
	\bigl\{ \upsilonv=(\thetav,\etav)\in\Ups_{\dimh} \colon \|\DF_{\dimh}(\upsilonv-\upsilonvs_{\dimh})\|\le \rr \bigr\}.
\end{EQA}
We introduce \( \tilde{\upsilonv}_{\thetavs_{\dimh}}\in\Ups \), which maximizes \( \LL(\upsilonv) \) subject to
\(\Proj \upsilonv = \thetavs_{\dimh}\):
\begin{EQA}[c]
\label{tuthLLuus}
    \tilde{\upsilonv}_{\thetavs_{\dimh}} \eqdef (\thetavs_{\dimh},\tilde \etav_{\thetavs_{\dimh}})
    \eqdef 
    \argmax_{\Thetathetav{\thetavs_{\dimh}}} 
    \LL(\upsilonv),
\end{EQA}
and define the radius \(\rups>0\)
\begin{EQA}[c]\label{eq: def of rups}
\rups(\xx)\eqdef \inf_{\rr>0}\left\{\P(\tilde{\upsilonv},\tilde{\upsilonv}_{\thetavs_{\dimh}}\in\Upss(\rr))\ge 1-\ex^{-\xx}\right\},
\end{EQA}
which we set to infinity if \(\tilde{\upsilonv}=\{\,\}\) or \(\tilde{\upsilonv}_{\thetavs}=\{\,\}\). Under the conditions 
\({(\cc{L}{\rr})} \) and \({(\CS\rr)} \) Theorem~2.3  of \cite{AASP2013} states that \(\rups = \rups(\xx) \approx \CONST\sqrt{\xx+\dimtotal}>0\). Further introduce the projected gradient and the covariance of the projected score
\begin{EQA}
    \scorer_{\thetav}
    =
    \score_{\thetav} - \A \HH^{-2} \score_{\etav}, &\quad \VPr^{2}
    =
    \Cov(\scorer_{\thetav}\zeta(\upsilonvd)).
\end{EQA}

\subsubsection*{A sufficient list of conditions}
The following three conditions ensure that \(\DF_{\dimh}^2\) is not degenerated and further quantify the smoothness properties on \(\Upss(\rr)\) of the expected log-likelihood 
\( \E \LL(\upsilonv) \) and of the stochastic component \( \zeta_{\dimh}(\upsilonv) = \LL_{\dimh}(\upsilonv) - \E_{\varepsilon} \LL_{\dimh}(\upsilonv) \). 

First we state an \emph{identifiability condition}.

\begin{description}
  \item[\( (\bb{\AssId}) \)] 
It holds for some \( \corrDF < 1 \)
\begin{EQA}[c]
    \| \HH^{-1} \A^{\T} \DP^{-1} \|
    \leq 
    \corrDF .
\label{regularity2}
\end{EQA}
\end{description}

\begin{description}
    \item[\(\bb{(\breve{\LL}_{0})} \)]
    For each \( \rr \le 4\rups\), 
    there is a constant \( \rddelta(\rr) \) such that
    it holds on the set \( \Upss(\rr) \):
    \begin{EQA}
\label{LmgfquadELGP}
    \|\DP^{-1}\DP^{2}(\upsilonv)\DP^{-1}-I_{\dimp}\|&\le& \breve\rddelta(\rr),\\
    \|\DP^{-1}(\A(\upsilonv)-\A)\HH^{-1}\|&\le& \breve\rddelta(\rr)\\
\left\|    \DP^{-1}\A\HH^{-1}\left(I_{\dimh}-\HH^{-1}\HH^2(\upsilonv)\HH^{-1}\right)\right\|
    &\le& \breve\rddelta(\rr),
\end{EQA}
where 
\begin{EQA}
\DF(\ups)^2\eqdef -\nabla^2\E\LL(\ups), &\quad \DF(\ups)=\left( 
      \begin{array}{cc}
        \DP^{2}(\ups) & \A(\ups) \\
        \A^{\T}(\ups) & \HH^{2}(\ups) \\
      \end{array}  
    \right).
\end{EQA}
\end{description}

\begin{description}
  \item[\( \bb{(\breve\CS \DF_{1})} \)]
    \(\zeta(\upsilonv)\to \zeta(\upsilonv')\) as \(\upsilonv\to \upsilonv'\). Further for all \( 0 < \rr < 4\rups \), 
    there exists a constant \( \omega \le 1/2 \) such that for all \( |\mubc| \le \breve\gm \) and \(\upsilonv,\upsilonv'\in \Upss(\rr)\)
\begin{EQA}[c]
    \sup_{\upsilonv,\upsilonv'\in\Upss(\rr)}\sup_{\|\gammav\|\le 1} \log \E_{\varepsilon} \exp\left\{ 
         \frac{\mubc} {\breve\omega}\frac{\gammav^\T \DPr^{-1} \bigl\{ 
		\scorer_{\thetav}\zetav(\upsilonv) - \scorer_{\thetav}\zetav(\upsilonv')\bigr\}}{\|\DF_{\dimh}(\upsilonv-\upsilonv')\|}\right\}
    \le 
    \frac{\breve\nu_{1}^{2} \mubc^{2}}{2}.
\end{EQA}
\end{description}

\begin{description}
  \item[\( \bb{(\breve\CS \DF_{0})} \)]
    There exist a matrix \(\VPr^2\in\R^{\dimp\times \dimp}\), constants \( \nunu>0 \) and \(\breve\gm > 0 \) such that for all 
    \( |\mubc| \le \breve\gm \)
\begin{EQA}[c]
    \sup_{\gammav \in \R^{\dimp}} \log\E_{\varepsilon} \exp\left\{ 
        \mubc \frac{\langle \breve\nabla_{\thetav} \zeta(\upsilonvd),\gammav \rangle}
                   {\| \VPr \gammav \|}
    \right\}
    \le 
    \frac{\breve\nu_{0} ^{2} \mubc^{2}}{2}.
\end{EQA}
\end{description}

\begin{remark}
Please see \cite{AASP2013} for a discussion and explanation of the above conditions.
\end{remark}

\subsubsection*{Stronger conditions for the full model}
In many situations the following, stronger conditions, are easier to verify and allow to derive more accurate results:

\begin{description}
    \item[\( \bb{(\LL_{0})} \)]
    For each \( \rr \le \rups \), 
    there is a constant \( \rddelta(\rr) \) such that
    it holds on the set \( \Upss(\rr) \):
    \begin{EQA}[c]
\label{LmgfquadELGP}
    \bigl\|
       \DF_{\dimh}^{-1} \bigl\{ \nabla^{2}\E\LL(\upsilonv) \bigr\} \DF_{\dimh}^{-1} - \Id_{\dimtotal} 
    \bigr\|
    \le
    \rddelta(\rr).
\end{EQA}

\end{description}
\begin{description}
  \item[\( \bb{(\CS \DF_{1})} \)]
    There exists a constant \( \rhor \le 1/2 \), such that for all \( |\mubc| \le \gm \) 
    and all \( 0 < \rr < \rups \)
\begin{EQA}[c]
    \sup_{\upsilonv,\upsilonvc\in\Upss(\rr)}
    \sup_{\|\gammav\|=1} 
    \log \E_{\varepsilon} \exp\left\{ 
         \frac{\mubc \, \gammav^{\T} \DFc^{-1} 
         		\bigl\{ \nabla\zetav(\upsilonv)-\nabla\zetav(\upsilonvc) \bigr\}}
         	  {\rhor \, \|\DF_{\dimh} (\upsilonv-\upsilonvc)\|}\right\}
    \le 
    \frac{\nu_1^{2} \mubc^{2}}{2}.
\end{EQA}
\end{description}

\begin{description}
  \item[\( \bb{(\CS \DF_{0})} \)]
    There exist a matrix \( \VFc^{2}\in \R^{\dimtotal\times \dimtotal} \), constants \( \nunu>0 \) and \( \gm > 0 \) such that for all 
    \( |\mubc| \le \gm \)
\begin{EQA}[c]
    \sup_{\gammav \in \R^{\dimtotal}} \log\E_{\varepsilon} \exp\left\{ 
        \mubc \frac{\langle \nabla \zeta(\upsilonvd),\gammav \rangle}
                   {\| \VF \gammav \|}
    \right\}
    \le 
    \frac{\nu_{0} ^{2} \mubc^{2}}{2}.
\end{EQA}
\end{description}

The following lemma shows, that these conditions imply the weaker ones from above:

\begin{lemma}[Lemma 2.1 of \citep{AASP2013}]
\label{lem: strong cond imply breve cond} 
Assume \( ({\AssId}) \). Then \( {(\CS \DF_{1})} \) implies \( {(\breve\CS \DF_{1})} \), \( {(\LL_{0})} \) implies \( {(\breve\LL_{0})} \), and \( {(\CS \DF_{0})} \) implies \( {(\breve\CS \DF_{0})} \) with
\begin{EQA}
\breve \gm=\frac{\sqrt{1-\corrDF^2}}{(1+\corrDF)\sqrt{1+\corrDF^2}}\gm, & \breve \nu_i=\frac{(1+\corrDF)\sqrt{1+\corrDF^2}}{\sqrt{1-\corrDF^2}}\nu_i, & \breve \delta(\rr)=\delta(\rr),\,\text{ and }\breve \omega=\omega.
\end{EQA}
\end{lemma}

\subsubsection*{Conditions to ensure concentration of the ME}\label{sec: conditions for large dev}
Finally we present two conditions that allow a specific approach to determine the radius \(\rups(\xx)>0\) from \eqref{eq: def of rups}. These conditions have to be satisfied on the whole set \(\Ups\subseteq \R^{\dimtotal}\). 

\begin{description}
  \item[\( \bb{(\cc{L}{\rr})} \)] 
     For any \( \rr > \rups\) there exists a value \( \gmi(\rr) > 0 \), 
     such that
\begin{EQA}[c]
    \frac{-\E \LL(\upsilonv,\upsilonvd)}{\|\DF_{\dimh}(\upsilonv-\upsilonvd)\|^{2}}
    \ge 
    \gmi(\rr),
    \qquad
    \upsilonv \in \Upss(\rr).
\end{EQA}

  \item[\( \bb{(\CS\rr)} \)] 
    For any \( \rr \ge \rups \) there exists a constant \( \gm(\rr) > 0 \) such that 
\begin{EQA}[c]
    \sup_{\upsilonv \in \Upss(\rr)} \, 
    \sup_{\mubc \le \gm(\rr)} \, 
    \sup_{\gammav \in \R^{\dimtotal}}
    \log\E_{\varepsilon} \exp\left\{ 
        \mubc \frac{\langle \nabla \zeta(\upsilonv),\gammav \rangle}
        {\|\DF_{\dimh}\gammav\|}
    \right\}
    \le \frac{\nu_\rr^{2} \mubc^{2}}{2}.
\end{EQA}
\end{description}

\subsection{General results for profile M-estimators}\label{sec: regression version of semi paper}
\cite{AASP2013} define for some \(\xx,\rr>0\) the semiparametric spread 
\begin{EQA}[c]
  \breve\Excgr(\rr,\xx)
    \eqdef
    4\left(  \frac{4}{(1-\corrDF^2)^2}\breve\rddelta(4\rr) + \, 6 \nu_{1}  \breve\omega \zzq(\xx,2\dimtotal+2\dimp) \right)\rr.
\label{eq: def of breve diamond rr}
\end{EQA}  
\begin{remark}\label{rem: bounf for entropy term}
The constant \(\zzq(\xx,\cdot)\) is of order of \(\sqrt{\xx+\cdot}\). For a precise definition see Appendix C of \cite{AASP2013}. 
\end{remark}

\cite{AASP2013} present the following three results, that we adapted for the regression setup. They can be proved in exactly the same way:

\begin{theorem}[\cite{AASP2013}, Theorem~2.3]
\label{theo: large def with K}
Suppose that on some set \(\bb{\mathcal N}(\xx)\subset\Omega\) the condition\( (\CS\rr) \) and \( (\cc{L}\rr) \) with \( \gmi(\rr) \equiv \gmi \) is met. Further define the following random set
\begin{EQA}[c]
 \Ups(K)\eqdef \{\ups\in\Ups: \LL(\ups,\upss)\ge -K\}.
\end{EQA}
If for a fixed \( \rups \) and any \( \rr \ge \rups \), the following conditions are fulfilled:
\begin{EQA}
    1 + \sqrt{\xx + 2\dimtotal} 
    & \le &
    3 \nu_{\rr}^{2} \gm(\rr)/\gmi ,\label{cgmi1rrc}
    \\
    6 \nu_{\rr} \sqrt{\xx + 2\dimtotal+\frac{\gmi}{9\nu_{\rr}^2}K}
    & \le &
    \rr\gmi ,
\label{cgmi2rrc}
\end{EQA}
then 
\begin{EQA}
	\P(\Ups(K) \subseteq\Upss(\rups))
	& \ge &
	1-\ex^{-\xx}-\P(\bb{\mathcal N}(\xx)^c).
\label{PLDttuttut}
\end{EQA}
\end{theorem}

\begin{theorem}[Theorem 2.2 of \cite{AASP2013}]
\label{theo: main theo finite dim regression}
Assume \({(\breve\LL_{0})}\) and \({(\AssId)}\). Further assume that on some set \(\bb{\mathcal N}(\xx)\subset\Omega\) the condition
\({(\breve\CS \DF_{1})} \) is met. Further assume that on \(\bb{\mathcal N}(\xx)\subset\Omega\) the sets of maximizers \(\tilde\ups,\,\tilde{\upsilonv}_{\thetavs}\) are not empty and that it contains with some \(\tau(\cdot)\in\R\) the set
\begin{EQA}[c]
\left\{ \sup_{\ups\in\Upss(\rups)}\|\nabla(\E-\E_{\varepsilon})[\LL({\upss_{\dimh}})-\LL(\ups)]\|\le \tau(\rups)\right\}\cap\{\tilde\ups,\,\tilde{\upsilonv}_{\thetavs}\in\Upss(\rups)\}.
\end{EQA}
Then it holds on a set of probability greater \( 1-\ex^{-\xx}-\P(\bb{\mathcal N}(\xx)^c) \) 
\begin{EQA}
	\bigl\| 
        \DPr \bigl( \tilde{\thetav} - \thetavs \bigr) 
        - \xivr 
    \bigr\|
    &\le& 
    \breve\Excgr(\rups,\xx) + \tau(\rups),
\label{eq: Fisher in main theo regression}
	\\
    \bigl| 2 \Lr(\tilde{\thetav},\thetavs) - \| \xivr \|^{2} \bigr|
    &\le&
     5\left(\|\xivr\|+\breve\Excgr(\rups,\xx)+ \tau(\rups)\right)\left(\breve\Excgr(\rups,\xx)+ \tau(\rups)\right),
\label{eq: Wilks in main theo}
\end{EQA}
where the spread \(\breve \Excgr(\rups,\xx) \) is defined in \eqref{eq: def of breve diamond rr} and where \(\rups>0\) is defined in \eqref{eq: def of rups}.
\end{theorem}

\begin{proposition}[Proposition 2.4 of \cite{AASP2013}]
\label{prop: large dev refinement regression} Assume the conditions of Theorem \ref{theo: main theo finite dim regression} and additionally assume \( {(\LL_{0})} \) and that \( {(\CS \DF_{1})} \) and \({(\CS \DF_{0})} \) are met on \(\bb{\mathcal N}(\xx) \). 
Then the results of Theorem \ref{theo: main theo finite dim regression} hold with \(\rr_1\le \rups\) instead of \(\rups\) and with probability greater \(1-4\ex^{-\xx}-\P(\bb{\mathcal N}(\xx)^c)\) where
\begin{EQA}[c]\label{eq: one step recursion rr bound}
\rr_1  \le \zz(\xx,\BB)+\Excgr_{Q}(\RR,\xx)\wedge \rups(\xx).
\end{EQA}
Further if there is some \(\eps>0\) such that \(\delta(\rr)/\rr\vee 6\nu_1\omega\le \eps\) for all \(\rr\le \rups\) and with \(6\eps \rups(\xx)< c\) and \(6\eps \rups(\xx)<1\) then \(\rups\) can be replaces with \(\rr_{0}^*\) which is bounded by
\begin{EQA}\label{eq: rrstar after large dev recursion in semi}
\rr_{0}^*&\le &  \zz(\xx,\BB)+\eps\zzQ(\xx,4\dimtotal)^2+ \eps^2\frac{18}{1-c}\zz_{\eps}(\xx).
\end{EQA}
\end{proposition}
\begin{remark}\label{rem: bounf for quad forms}
The constant \(\zz(\xx,\BB)\) is of order of \(\sqrt{\xx+\dimtotal}\). For a precise definition see Appendix A of \cite{AASP2013}. 
\end{remark}

\begin{remark}
This is a slightly refined version of Proposition 2.4 of \cite{AASP2013}, that can be derived using arguments that are similar to those underlying Theorem 2.4 of \cite{AASPalternating}.
\end{remark}

\subsection{A way to bound the sieve bias}\label{sec: synapsis sieve tools}
Theorem \ref{theo: main theo finite dim regression} involves two kinds of bias once it is applied to the sieve estimator \(\tilde\thetav_{\dimh}\): one that concerns the difference \(\thetavs_{\dimh} - \thetavs\in\R^{\dimp}\) and the other the difference between \(\DPr_{\dimh}(\upss_{\dimh})\in\R^{\dimp\times\dimp}\) and \(\DPr(\upsilonvs)\in\R^{\dimp\times\dimp}\) where
\begin{EQA}[c]
\DPr^{2}(\upsilonvs)\eqdef \left(\Pi_{\thetav}\nabla^{2}\E[\LL(\upsilonvs)]^{-1}\Pi_{\thetav}^{\T}\right)^{-1}\in\R^{\dimp\times\dimp},
\end{EQA}
i.e. the derivatives of \(\E[\LL]\) are taken with respect to all coordinates of \(\upsilonv\in l^{2}\) and the Hessian is calculated in the "true point" \(\upsilonvs\in l^{2}\). \cite{AASP2013} combined with \cite{AAbias2014} present the following conditions to control these biases: 

\begin{description}
  \item[\( \bb{(}\kappav\bb{)} \)] 
  The vector \( \kappavs\eqdef (Id_{l^2}-\Pi_{\dimtotal})\upsilonvs\in l^2\) satisfies \(\|\HF_{\kappav\kappav}\kappavs\|^2\le \CONST_{\kappavs}\dimh\) for some \(\CONST_{\kappavs}>0\) and with \( \alpha(\dimh)\to 0 \)
\begin{EQA}
    \|\DF_{\dimh}^{-1} \AF_{\kappav\ups_\dimh}^\T \kappavs\|
    &\le& 
    \hat\alpha(\dimh). 
\label{apprXdimhledimh}
\end{EQA} 
Further for any \(\lambda\in[0,1]\) with some \(\tau(\dimh)\to 0\)
\begin{EQA}
\label{eq: condition on smoothness of DF}
    \|\DF_{\dimh}^{-1} \left(\nabla_{\upsilonv_{\dimh}\kappav}\E\LL(\upsilonvs,\lambda\kappavs) -\AF_{\kappav\ups_\dimh}^\T\right) \kappavs\|
    &\le& 
    \tau(\dimh),\\
\left|{\kappavs}^\T(\HF_{\kappav\kappav}-\nabla_{\kappav\kappav}\E\LL\left((\Pi_{\dimtotal}\upsilonvs,\lambda\kappavs)\right)\kappavs\right| &\le& \CONST_{\kappavs}\dimh.
\label{eq: condition on smoothness of DF-zeta}
\label{eq: new condition for bias}
\end{EQA}     
\end{description}

\begin{description}
    \item[\( \bb{(}\upsilonv\kappav{)} \)]
Assume that with some \(\beta(\dimh)\to 0\)
\begin{EQA}[c]
\label{eq: convergence of breve matrix}
\|\HF_{\kappav\kappav}^{-1}\AF_{\kappav\ups_\dimh}^\T\DF_{\dimh}^{-1}\|\le \beta(\dimh).
\end{EQA}
\end{description}

\begin{description}
  \item[\(\bb{(\cc{L}{\rr}_{\infty})}\)] For any \( \rr > \rups\) there exists a value \( \gmi(\rr) > 0 \), 
     such that
\begin{EQA}[c]
    \frac{-\E \LL(\upsilonv,\upsilonvs)}{\|\DF(\upsilonv-\upsilonvs)\|^{2}}
    \ge 
    \gmi(\rr).
\end{EQA}
\end{description}

\begin{description}
\item[\(\bb{(bias'')}\)] As \(\dimh\to \infty\) with \(\|\cdot\|\) denoting the spectral norm
\begin{EQA}[c]
 \| \DPr_{\dimh}^{-1}(\upsilonvs_{\dimh})\VPr^{2}_{\dimh,\DF}(\upsilonvs_{\dimh})\DPr_{\dimh}^{-1}(\upsilonvs_{\dimh}) - \breve d^{-1}\breve v^{2} \breve d^{-1}\|\to 0.
 \end{EQA}
\end{description}

For some \(\rr>0\) define the set
\begin{EQA}[c]
{\Ups}_{0,\dimh}(\rr)\eqdef\{\ups \in \R^{\dimtotal},\, \|\DF_{\dimh}(\upss_{\dimh})(\ups-\upss_{\dimh})\|\}.
\end{EQA}

\begin{theorem}[Corollary 2.8, 2.10 and Theorem 2.9 of \cite{AASP2013}; Theorem 2.1 of \cite{AAbias2014}]
\label{theo: bias cond}
Let the condition \( {(\cc{L}{\rr}_{\infty})}\) with \(\gmi(\rr)\equiv \gmi>0\), \( (\kappav) \) and condition \( ({\AssId}) \) from Section \ref{sec: conditions semi} be satisfied for both \(\DF_{\dimh}(\upss)\) and \(\DF_{\dimh}(\upss_{\dimh})\) and for \( \E\LL: l^2\to \R \). Set \( {\rr^*}^{2}= 4\CONST_{\kappavs}^2\dimh/\gmi\). Assume that on some set \(\bb{\mathcal N}(\xx)\subset\Omega\) and some \(\dimh_0\in\N\) and all \(\dimh\ge \dimh_0\) the conditions \({ (\breve\CS \DFc)} \), 
\(  {(\breve\CS \DF_{1})} \) and \( { (\breve\LL_{0})} \) 
from Section~\ref{sec: conditions semi} are satisfied for all \(\dimh\ge\dimh_0\) for some \(\dimh_0\in\N\) and with \(\DFc^{2}=\allowbreak\nabla_{\dimp+\dimh}^{2}\E\LL_{\dimh}(\upsilonvs_{\dimh})\allowbreak\in\R^{\dimtotal\times\dimtotal}\), \(\VFc^{2}=\Cov[\nabla_{\dimp+\dimh}\LL_{\dimh}(\upsilonvs_{\dimh})]\in\R^{\dimtotal\times\dimtotal}\) and \(\upsilonvd=\upsilonvs_{\dimh}\in\R^{\dimtotal}\) and for any \(\rr\le \rr^*\vee \rups^{\circ}\). Further assume that on \(\bb{\mathcal N}(\xx)\subset\Omega\) the sets of maximizers \(\tilde\ups,\,\tilde{\upsilonv}_{\thetavs}\) are not empty and that it contains with some \(\tau(\cdot)\in\R\) the set
\begin{EQA}[c]
\left\{ \sup_{\ups\in\Upss(\rups)}\|\nabla(\E-\E_{\varepsilon})[\LL({\upss_{\dimh}})-\LL(\ups)]\|\le \tau(\rups^{\circ})\right\}\cap\{\tilde{\upsilonv}_{\dimh},\tilde{\upsilonv}_{\thetavs_{\dimh},\dimh},\tilde{\upsilonv}_{\thetavs,\dimh}\in{\Ups}_{0,\dimh}(\rups^{\circ})\}.
\end{EQA}
Then it holds for any \(\dimh\ge\dimh_0\) with probability greater \(1-\ex^{-\xxn}-\P(\bb{\mathcal N}(\xx)^c)\)
\begin{EQA}
\bigl\| 
        \DPr_{\dimh}(\upsilonvs_{\dimh}) \bigl( \tilde{\thetav}_{\dimh} - \thetavs \bigr) 
        - \xivr_{\dimh}(\upsilonvs_{\dimh}) 
    \bigr\|
    &\le& 
    \,\breve\Excgr(\rups^{\circ},\xx)+ \alpha(\dimh),\\
	\bigl| 
		2 \Lr(\tilde{\thetav}_{\dimh},\thetavs) - \| \xivr_{\dimh}(\upsilonvs_{\dimh}) \|^{2}
	\bigr|
	&\le& 
	5\left(\| \xivr_{\dimh}(\upsilonvs_{\dimh})\|+\breve\Excgr(\rr_{0}^{\circ},\xx)+\alpha(\dimh)\right)\\
		&&\left(\breve\Excgr(\rr_{0}^{\circ},\xx)+\alpha(\dimh)\right)
\end{EQA}    
where
 \begin{EQA}   
   \alpha(\dimh)&=&  \sqrt{\frac{1+\corrDF^2}{1-\corrDF^2}}\Bigg(\alpha(\dimh) +\tau(\dimh)+2\breve\delta(2\rr^*)\rr^*\Bigg).
 \end{EQA} 
If further the condition \( (\upsilonv\kappav) \) and \({(bias'')}\) are fulfilled and if
\begin{EQA}
 \breve\delta(\rr^*)\to 0, &\quad \breve\delta_n(\rr)\to 0, &\quad \rups(\xx)<\infty, \\
 \P(\bb{\mathcal N}(\xx)^c)\to 0,\text{ as }\xx\to \infty,
 \end{EQA} 
there is a sequence \(\dimh_n\to \infty\) such that as \( \nsize \to \infty \)
\begin{EQA}[ccl]
    \nsize \breve d \bigl( \tilde{\thetav}_{\dimh} - \thetavs \bigr)
    - \xivr
     &\toP& 
    0,\\
    \nsize \breve d \bigl( \tilde{\thetav}_{\dimh} - \thetavs \bigr)
     &\tow& 
    \ND(0,\breve d^{-1}\breve v^{2} \breve d^{-1}),\\
     2 \Lr(\tilde{\thetav}_{\dimh},\thetavs)
    &\tow& 
    \mathcal L(\|\xivr_{\infty}\|), \text{ }\xivr_{\infty}\sim \ND(0,\breve d^{-1}\breve v^{2} \breve d^{-1}).
\end{EQA}
\end{theorem}
\begin{remark}
With remark 2.26 of \cite{AASP2013} the radius \(\rups^{\circ}\) 
\begin{EQA}[c]
\P\left(\left\{\tilde{\upsilonv}_{\dimh},\tilde{\upsilonv}_{\thetavs_{\dimh},\dimh},\tilde{\upsilonv}_{\thetavs,\dimh}\in{\Ups}_{0,\dimh}(\rups^{\circ})\right\}\right)>1-\ex^{-\xx},
\end{EQA}
is close to \(\rups\) which satisfies
\begin{EQA}[c]
\P\left(\left\{\tilde{\upsilonv}_{\dimh},\tilde{\upsilonv}_{\thetavs_{\dimh},\dimh}\in{\Ups}_{0,\dimh}(\rups)\right\}\right)>1-\ex^{-\xx}.
\end{EQA}
The later can be determined using the arguments we present in Section \ref{sec: single index large dev}, using Theorem \ref{theo: large def with K}. 
\end{remark}

\subsection{Convergence results for the alternating procedure}\label{sec: synapsis alternating}
To derive convergence statements for the alternating procedure sketched in Section \ref{sec: initial guess}  \cite{AASPalternating} present the following list of conditions on the initial guess \eqref{eq: def of initial guess}.

\begin{description}
\item[\((\mathbf A_1)\)] With probability greater \(1-\beta_{(\mathbf A)}(\xx)\) the initial guess satisfies \(\LL(\tilde \upsilonv_0,\upsilonvs)\ge -\KL(\xx)\) for some \(\KL(\xx)\ge 0\).
\item[\((\mathbf A_2)\)] The conditions \({(\breve\CS \DF_{1})} \), \( {(\breve{\cc{L}}_{0})} \), \({(\CS \DF_{1})} \) and \( {(\cc{L}_{0})} \) from Section \ref{sec: conditions semi} hold for all \(\rr\le \RR(\xx,\KL)\) where 
\begin{EQA}[c]
\label{eq: def of radius RR}
  \RR(\xx,\KL)\eqdef \zz(\xx)\vee\frac{6 \nunu}{\gmi(1- \corrDF)} \sqrt{\xx + 2.4\dimtotal+\frac{\gmi^2}{9 \nunu^2}\KL(\xx)}.
\end{EQA} 
\item[\((\mathbf A_3)\)] There is some \(\eps>0\) such that \(\delta(\rr)/\rr\vee 12\nu_1\omega\le \eps\) for all \(\rr\le \RR\). Further \(\KL(\xx)\in\R\) and \(\eps>0\) are small enough to ensure 
\begin{EQA}
\label{eq: cond on eps with zz}
 c(\eps,\zz(\xx))&\eqdef&\eps 7\CONST(\corrDF)\frac{1}{1-\corrDF}\left(\zz(\xx)+\eps \zz(\xx)^2\right)<1,\\
\label{eq: cond on eps with RR}
 c(\eps,\RR) &\eqdef&\eps 7\CONST(\corrDF)\frac{1}{1-\corrDF}\RR<1,
\end{EQA} 
with
\begin{EQA}[c]\label{eq: def of CONST corrDF}
\CONST(\corrDF)\eqdef 2\sqrt{2}(1+\sqrt{\corrDF})(1-\sqrt\corrDF)^{-1}.
\end{EQA}

\item [\((\mathbf B_1)\)] Assume for all \(\rr\ge\frac{6\nunu}{\gmi}\sqrt{\xx + 4\dimtotal} \)
\begin{EQA}
    1 + \sqrt{\xx + 4\dimtotal} 
    & \le &
    \frac{3 \nu_{\rr}^{2}}{\gmi} \gm(\rr).
\label{eq: assumption A4}
\end{EQA}
\end{description}

\begin{theorem}[Theorem 2.2 of \cite{AASPalternating}]
 \label{theo: alternating stat props regression}
Assume that the conditions \( {(\cc{L}_{0})} \) and \( {(\breve{\cc{L}}_{0})} \) are met. Assume that on some set \(\bb{\mathcal N}(\xx)\subset\Omega\) the conditions \( {(\CS \DF_{0})} \),\( {(\CS \DF_{1})} \), \( {(\cc{L}_{\rr})} \), \( {(\breve\CS \DF_{1})} \) and \( {(\CS\rr)} \) of Section \ref{sec: conditions semi} are met with a constant \(\gmi(\rr)\equiv\gmi\) and where \(\VFc^2=\Cov\big(\score \LL(\upsilonvs)\big)\), \(\DFc^{2}  = - \nabla^{2} \E \LL(\upsilonvs)\) and where \(\upsilonvd=\upsilonvs\). Further assume that on \(\bb{\mathcal N}(\xx)\subset\Omega\) the sets \((\tilde\ups^{(k,k(+1))})\) are not empty and that it contains the set
\begin{EQA}[c]
\bigcap_{\rr\le \RR}\left\{ \sup_{\ups\in\Upss(\rr)}\|\nabla(\E-\E_{\varepsilon})[\LL({\upss_{\dimh}})-\LL(\ups)]\|\le \tau(\rr)\right\}\cap\{(\tilde\ups^{(k,k(+1))}) \subset {\Ups}_{0,\dimh}(\RR)\}.
\end{EQA}
Further assume \(( B_1)\) and that the initial guess satisfies \(( A_1)\) and \(( A_2)\). Then it holds with probability greater \(1-8\ex^{-\xx}-\beta_{(\mathbf A)}-\P(\bb{\mathcal N}(\xx)^c)\) for all \(k\in\N\) 
\begin{EQA}
	\bigl\| 
        \DPr \bigl( \tilde{\thetav}^{(k)} - \thetavs \bigr) 
        - \xivr 
    \bigr\|
    &\le& 
    \breve\Excgr_{Q}(\rr_{k},\xx) ,
\label{eq: alternating fisher}
	\\
\label{eq: alternating wilks}
    \bigl| 2 \Lr(\tilde{\thetav}^{(k)},\thetavs) - \| \xivr \|^{2} \bigr|
    &\le&
     5\left(\|\xivr \|+\breve\Excgr_{Q}(\rr_{k},\xx)\right)\breve\Excgr_{Q}(\rr_{k},\xx),
\end{EQA} 
where
\begin{EQA}[c]
\rr_k \le 2\sqrt{2}(1-\sqrt\corrDF)^{-1}\left\{\left(\zz(\xx)+\Excgr_{Q}(\RR,\xx)\right)+(1+\sqrt \corrDF)\corrDF^{k}\RR(\xx)\right\}.
\end{EQA}
If further condition \(( A_3)\) with \(\delta(\rr)+\tau(\rr)\vee\nu_{1,\dimh}\omega\rr \le \eps\rr\) then \eqref{eq: alternating fisher} and \eqref{eq: alternating wilks} are met with
\begin{EQA}
\rr_k&\le &  \left(\CONST(\corrDF)+ \frac{4\CONST(\corrDF)^3c(\eps,\zz(\xx))}{1-c(\eps,\zz(\xx))}\right)\left(\zz(\xx)+\eps \zz(\xx)^2\right)\\
	&&+\corrDF^k\left(\CONST(\corrDF)+ \corrDF \frac{4\CONST(\corrDF)^3c(\eps,\RR)}{1-c(\eps,\RR)} \right)\RR.
\end{EQA}
\end{theorem}

\cite{AASPalternating} also present a result that shows under which conditions the sequence of estimators  \((\tilde\ups^{(k,k(+1))})\) actually converges to the maximizer \(\tilde\ups\). For this result consider the following condition.

\begin{description}
  \item[\( \bb{(\CS \DF_{2})} \)]
    There exists a constant \( \rhor \le 1/2 \), such that for all \( |\mubc| \le \gm \) 
    and all \( 0 < \rr < \rups \)
\begin{EQA}
&&\nquad    \sup_{\upsilonv,\upsilonvc\in\Upss(\rr)}
    \sup_{\|\gammav_1\|=1}  \sup_{\|\gammav_2\|=1} 
    \log \E \exp\left\{ 
         \frac{\mubc \, \gammav_1^{\T} \DF^{-1} 
         		\bigl\{ \nabla^2\zetav(\upsilonv)-\nabla^2\zetav(\upsilonvc) \bigr\}\gammav_2}
         	  {\rhor_2 \, \|\DF (\upsilonv-\upsilonvc)\|}\right\}\\
    &\le& 
    \frac{\nu_2^{2} \mubc^{2}}{2}.
\end{EQA}
\end{description}
Define \(\zz(\xx,\nabla^2\LL(\upss))\) via
\begin{EQA}[c]
\P\left\{\|\DF^{-1}\nabla^2\LL(\upss)\|\ge \zz\left(\xx,\nabla^2\LL(\upss)\right) \right\}\le \ex^{-\xx},
\end{EQA}
and \(\kappa(\xx,\RR)\) as
\begin{EQA}
\kappa(\xx,\RR)&\eqdef& \frac{2\sqrt{2}(1+\sqrt \corrDF)}{\sqrt{1-\corrDF}} \bigg[ \delta(\RR)+9\omega_2\nu_2\|\DF^{-1}\|\zzq(\xx,6\dimtotal)\RR\\
	&&\phantom{ \frac{2\sqrt{2}(1+\sqrt \corrDF)}{\sqrt{1-\corrDF}} \bigg[}+  \|\DF^{-1}\|\zz\left(\xx,\nabla^2\LL(\upss)\right)\bigg].
\end{EQA}

%
%


\begin{theorem}[Theorem 2.4 of \cite{AASPalternating}]
\label{theo: convergence to MLE regression}
Assume that the condition \( {(\cc{L}_{0})} \) is met. Assume that on some set \(\bb{\mathcal N}(\xx)\subset\Omega\) the conditions \( {(\CS \DF_{0})} \),\( {(\CS \DF_{1})} \), \( {(\cc{L}_{\rr})} \) and \( {(\CS\rr)} \) of Section \ref{sec: conditions semi} are met with a constant \(\gmi(\rr)\equiv\gmi\) and where \(\VFc^2=\Cov\big(\score \LL(\upsilonvs)\big)\), \(\DFc^{2}  = - \nabla^{2} \E \LL(\upsilonvs)\) and where \(\upsilonvd=\upsilonvs\). Furthermore, assume that on \(\bb{\mathcal N}(\xx)\subset\Omega\) the sets \((\tilde\ups^{(k,k(+1))})\) are not empty and that it contains the set
\begin{EQA}[c]
\bigcap_{\rr\le \RR}\left\{ \sup_{\ups\in\Upss(\rr)}\|\nabla(\E-\E_{\varepsilon})[\LL({\upss_{\dimh}})-\LL(\ups)]\|\le \tau(\rr)\right\}\cap\{(\tilde\ups^{(k,k(+1))}) \subset {\Ups}_{0,\dimh}(\RR)\}.
\end{EQA}
Suppose \(( B_1)\) and that the initial guess satisfies \(( A_1)\) and \(( A_2)\). Assume that \(\kappa(\xx,\RR)<(1-\corrDF)\). Then
\begin{EQA}[c]
\P\left(\bigcap_{k\in\N}\left\{\tilde \ups^{(k,k(+1))}\in \tilde \Upss(\rr_k^*)\right\}\right)\ge 1-3\ex^{-\xx}-\beta_{(\mathbf A)},
\end{EQA}
where
\begin{EQA}\label{eq: bound for rrk sequence convergence}
\rr_k^*&\le&\begin{cases} \corrDF^{k} 2\sqrt{2} \frac{1}{1-\kappa(\xx,\RR) k}(\RR+\rups), & \kappa(\xx,\RR) k\le 1,\\
   2\frac{1-\corrDF}{\kappa(\xx,\RR)}\tau(\xx)^{k/\log(k)}(\RR+\rups) , & \text{otherwise,}\end{cases}
\end{EQA}
with \((\RR+\rups)\eqdef \RR+\rups\) and
\begin{EQA}
\tau(\xx)&\eqdef&\left(\frac{\kappa(\xx,\RR)}{1-\corrDF}\right)^{L(k)}<1\\
L(k)&\eqdef& \left\lfloor \frac{\log(1/\corrDF)-\frac{1}{k}\left(\log(2\sqrt{2})-\log(\kappa(\xx,\RR) k-1)\right)}{\left(1+\frac{1}{\log(k)}\log(1-\corrDF)\right)}\right\rfloor\in\N,
\end{EQA}
where \(\lfloor x\rfloor\in\N_0\) denotes the largest natural number smaller than \(x>0\).
\end{theorem}

\section{Application of the finite sample theory}
\label{sec: details}
We will now apply the results presented in the previous section to our problem. First we will show that the conditions \( {(\CS \DF_{0})} \), \( {(\CS \DF_{1})} \), \( {(\LL_{0})}\), \( {(\AssId)}\),
of Section \ref{sec: conditions semi} can be satisfied under the assumptions \((\mathbf{\mathcal A})\). These imply -  by Lemma  \ref{lem: strong cond imply breve cond} - \( {(\breve\LL_{0})} \), \( {(\breve\CS \DF_{1})} \) and \( {(\breve\CS \DF_{0})} \) from Section \ref{sec: conditions semi}, necessary for Theorem 2.2 of \cite{AASP2013}. Further we will show that the conditions \( {(\CS\rr)}\) and \( {(\cc{L}\rr)}\) from \ref{sec: conditions semi} are met. This will allow to determine \(\rups>0\) and ensure that the sets of maximizers \(\tilde\ups_{\dimh},\,\tilde{\upsilonv}_{\dimh}\text{}_{\thetavs}\) are not empty. The subsequent analysis will then serve to determine the necessary size of \(n\in\N\) that allows to obtain good bounds for \(\breve\Excgr(\rups,\xx)\in\R\).
 Concerning the alternation procedure we will show that the initial guess from \eqref{eq: def of initial guess} and the values of \(\delta(\rr),\omega\) from  \( {(\breve\LL_{0})} \), \( {(\breve\CS \DF_{1})} \) allow to apply the Theorems \ref{theo: alternating stat props regression} and \ref{theo: convergence to MLE regression}. 

\subsection{Conditions satisfied}
In this section we show that the conditions of section \ref{sec: conditions semi} are satisfied. First we derive an a priori bound for the distance between the target \(\upsilonvs_{\dimh}\in\R^{\dimp}\times\R^{\dimh}\) and the true parameter \(\upsilonvs\in\R^{\dimp}\times l^2\)
\begin{lemma} 
\label{lem: in example a priori distance of target to oracle}
Assume \((\mathbf{\mathcal A})\) then there is a constant \(\CONST>0\) that depends only on\break \(\|p_{\Xv^\T\thetavs}\|_\infty,\allowbreak \CONST_{\|\fvs\|}, s_{\Xv}, L_{p_{\Xv}}\) such that with
\begin{EQA}
\label{eq: def of rr star}
\rr^*&=&\CONST\sqrt n\dimh^{-(1+2\alpha)/2} \sqrt \dimh.
\end{EQA}
we get \(\|\DF_{\dimh}(\upsilonvs_{\dimh}-\upsilonvs)\|\le \rr^*\).
\end{lemma}



The next step is to determine a radius \(\rr^{\circ}\) that ensures that \(\tilde\ups\in  S^{\dimp,+}_1\times B_{\rr^{\circ}}(0)\) with large probability.
\begin{lemma}
\label{lem: a priori a priori accuracy for rr circ}
Define
\begin{EQA}
\tilde\etav^{(\infty)}_{\dimh,\thetav}&\eqdef& \argmax_{\etav\in \R^{\dimh}}\LL_{\dimh}(\thetav,\etav),
\end{EQA}
then with some constant \(\CONST\in\R\)
\begin{EQA}[c]
\P\left(\sup_{\thetav\in S^{\dimp,+}_1}\left\|\tilde\etav^{(\infty)}_{\dimh,\thetav}\right\|\ge \CONST \sqrt{\dimtotal\log(\dimtotal)+\xx}\right)\le  \ex^{-\xx}.
\end{EQA}
\end{lemma}

\begin{remark}
This Lemma also ensures that the alternating sequence \linebreak \( (\tilde\thetav_k,\tilde\etav_{k(-1)})\) introduced in Section \ref{sec: initial guess} lies in \(S^{\dimp,+}_{1}\times B^{\dimh}_{\rr^{\circ}}(0)\), with 
\begin{EQA}[c]\label{eq: def of rr circ}
\rr^{\circ}\le \CONST\sqrt{\dimtotal\log(\dimtotal)+\xx}.
\end{EQA}
Note that - using that by Lemma \ref{lem: D_0 dimh upss is boundedly invertable} we have \(\DF_{\dimh}\ge c_{\DF}>0\) - this also means that
\begin{EQA}[c]
\Ups_{\dimh}\subseteq \Upss(\sqrt{n}\rr^{\circ}/c_{\DF})\eqdef \left \{\ups\in\Ups:\, \|\DF_{\dimh}(\ups-\upss_{\dimh})\|\le \sqrt{n}\rr^{\circ}/c_{\DF}\right \}.
\end{EQA}
\end{remark}


Now we show that the general conditions of section \ref{sec: conditions semi} are met under the assumptions \((\mathbf{\mathcal A})\). For this we point out again that due to the random design regression approach we define the random component of \(\LL\) via \(\LL-\E_{\varepsilon}\LL\) where \(\E_{\varepsilon}\) denotes the expectation operator of the law of \((\varepsilon_i)_{i=1,\ldots,n}\) given \((\Xv_i)_{i=1,\ldots,n}\). This facilitates the proof of the conditions \( {(\CS \DF_{0}) }\), \( {(\CS \DF_{1}) }\) and \({(\CS\rr)}\) but leads to additional randomness, in the sense that the claim of the following lemma is only true with a certain high probability.

\begin{lemma}
\label{lem: conditions example}
Assume the conditions \((\mathbf{\mathcal A})\). Then with \(\upsilonvd=\upsilonvs_{\dimh}\in\R^{\dimtotal}\) and \begin{EQA}
\VFc^2=\Cov\big(\score \LL_{\dimh}(\upsilonvs_{\dimh})\big), & &
\DFc^{2}
    =
    - \nabla^{2} \E \LL_{\dimh}(\upsilonvs_{\dimh}),
\end{EQA}
and \(\xx\le\dimh\) we get the conditions of section \ref{sec: conditions semi} on the set 
\begin{EQA}[c]
\left\{\sup_{\thetav\in S^{\dimp,+}_1}\left\|\tilde\etav^{(\infty)}_{\dimh,\thetav}\right\|\le \CONST \sqrt{\dimtotal\log(\dimtotal)+\xx} \right\},
\end{EQA}
 with:
\begin{description}
\item \( \bb{(\CS \DF_{0})} \) with probability greater than \(1-\ex^{-\xx}\) and with 
\begin{EQA}
 \gm=\sqrt{\frac{n}{\CONST\dimh}} \tilde g,&\quad
 \nu_{\dimh}^{2}=2\tilde \nu^{2},
\end{EQA}

\item \(\bb{(\CS\rr)}\) with probability greater than \(1-\ex^{-\xx}\) and with 
\begin{EQA}
 \gm(\rr)&=&\sqrt n c_{\DF}\tilde g\CONST\left(\sqrt{\dimh}+\dimh^{3/2}\rr/\sqrt{n}\right)^{-1},\\
 \nu_{\rr,\dimh}^{2}&=&\tilde \nu^{2}\Big( 1+\CONST\left(\dimh^{3/2}+ \rr\dimh^2/\sqrt{n}\right)\rr/\sqrt{n}\Big)\\
 	&&+\CONST\left(\dimh+\dimh^3\rr^2/n\right)\Big(\xx+\log(2\dimh)\Big)^{1/2}/\sqrt{n}\Big).\label{eq: bounds for Er single index}
\end{EQA}

\item \( \bb{(\CS \DF_{1}) }\) on \(\Upss(\rr)\) for all \(\rr>0\) with \(\rr\dimh^2/\sqrt{n}\le 1\) with probability greater than \(1-\ex^{-\xx}\) and with 
\begin{EQA}
\gm\ge   \frac{\sqrt n}{\rr \dimh^{3/2}C_{\bb{(\CS \DF_{1}) }}},	 &\quad
 \omega \eqdef \frac{2}{\sqrt n c_{\DF}}, &\quad
 \nu_{1,\dimh}^2=\tilde \nu^{2}\CONST_{\bb{(\CS \DF_{1}) }} \dimh^2,
\end{EQA}
where \(C_{\bb{(\CS \DF_{1}) }}\) is some constant that only depends on \(\|\psi\|, \|\psi'\|, \|\psi''\|,\break L_{p_{\Xv}}, s_{\Xv}, c_{\DF}\), etc..

\item \(\bb{(\cc{L}_{0})} \) is satisfied for all \(\rr>0\) with \(\rr\dimh^{3/2}/\sqrt{n}\le 1\) and where 
\begin{EQA}[c]
\delta(\rr)=\frac{\CONST_{\bb{(\cc{L}_{0})}} \left\{\dimh^{3/2}+\CONST_{bias}\dimh^{5/2}\right\}\rr}{c_{\DF}\sqrt n}.
\end{EQA}
The constant \(\CONST_{\bb{(\cc{L}_{0})}}>0\) is polynomial of \(\|\psi\|_{\infty}\) ,\(\|\psi'\|_{\infty}\), \(\|\psi''\|_{\infty}\), \(\CONST_{\|\fvs\|}\), \(L_{\nabla \Phi}\), \(s_{\Xv}\), \(c_{\DF}^{-1}\) and \(\|p_{\Xv^\T\thetavs}\|_{\infty}\) and is independent of \(\xx,\,n,\,\dimtotal\). 

\item \(\bb{(\cc{L}{\rr})}\) if \(\CONST_{bias}=0\) and for \(n\in\N\) large enough with \(\gmi= c_{\bb{(\cc{L}{\rr})}}>0\) as soon as 
\begin{EQA}[c]\label{eq: lower bound rups in Lr}
\rr^2\ge(3(2+\CONST) {\rr^*}^2+\CONST_{\sum})/(c \gmi)\vee \dimh
\end{EQA}
for certain constants \(c_{\bb{(\cc{L}{\rr})}},\,c,\,\CONST,\,\CONST_{\sum} >0\) and with probability greater than \(1-\exp\left\{-\dimh^{3}\xx\right\}-\exp\left\{- n c_{\bb{(Q)}}/4\right\}\). 
In the case that \(\CONST_{bias}\neq 0\) we get for 
\begin{EQA}[c]\label{eq: bower bound rups in Lr model biased}
\rr^2\ge \sqrt{\xx+\CONST\dimtotal[\log(\dimtotal )+\log(n)]}/\gmi_{\E}\vee 2{\rr^*}^2,
\end{EQA}
that with some \(\gmi_{bias}>0\) independent of \(n,\dimh,\xx,\rr\) and with probability greater than \(1-\ex^{-\xx}\)
\begin{EQA}
-\E_{\eps}\LL_{\dimh}(\ups,\upss_{\dimh})&\ge& \gmi_{bias}\rr^2.
\end{EQA}
\end{description}
\end{lemma}

\begin{remark}\label{rem: conds for alternating satisfied to}
The condition \(\rr\dimh^2/\sqrt{n}\le 1\) needed for \( \bb{(\CS \DF_{1}) }\) can be relaxed to read \(\rr\dimh^{3/2}/\sqrt{n}\le 1\) if one increases \(\nu_{1,\dimh}^2=\tilde \nu^{2}\CONST_{\bb{(\CS \DF_{1}) }} \dimh^3\). This does not change the bounds for \(\Excgr(\rr,\xx)\), as \(\delta(\rr)\) then still is of the same order as \(\omega\nu_{1,\dimh}\). With this correction the conditions apply for all \(\rr\le \RR\), where \(\RR\) is the deviation bound for the elements of the alternation procedure started in \(\tilde\ups_{0}\) in \eqref{eq: def of initial guess}, as we explain in Remark~\ref{rem: RR small enough for conds domain}.
\end{remark}
For the regularity condition \(({\AssId})\) we use the following Lemma.

\begin{lemma}
\label{lem: cond identifiability}
Under the assumptions of the last lemma the identifiability 
condition \(({\AssId})\) is satisfied with
\begin{EQA}[c]
\corrDF^2\le 1-\frac{ c_{\DF}}{ \CONST\dimp}.
\end{EQA} 
\end{lemma}

\begin{proof}
This follows from \(\DF\ge c_{\DF}Id\) with Lemma B.5 of \cite{AASP2013} where
\begin{EQA}[c]
\corrDF^2\le 1-\frac{n c_{\DF}}{\lambda_{\max}\DP\wedge\lambda_{\max}\HF }\le  1-\frac{ c_{\DF}}{  \CONST\dimp}.
\end{EQA}
where we used Lemma \ref{lem: conditions theta eta} to bound \(\lambda_{\max}\DP\le \CONST\dimp\) in the last step.
\end{proof}

Finally we apply the following Lemma \ref{lem: strong cond imply breve cond} to obtain the conditions \( {(\breve\LL_{0})} \), \( {(\breve\CS \DF_{1})} \) and \( {(\breve\CS \DF_{0})} \).

\begin{remark}
We do not show the conditions \( {(\breve\LL_{0})} \), \( {(\breve\CS \DF_{1})} \) and \( {(\breve\CS \DF_{0})} \) directly. To benefit from the weaker conditions we would need entry-wise bounds for the operator \(\AF\HH^{-2}\) for better bounds in the proof of condition \( {(\breve\LL_{0})} \). As this work is very long and technical without this sophistication we postpone this improvement to future work.
\end{remark}


\subsection{Large deviations}
\label{sec: single index large dev}
Next we determine the necessary size of the radius \(\rr_{0}(\xx)\) defined by
\begin{EQA}
\rups(\xx)&\eqdef& \inf\{\rr>0:\text{ } \P\{\tilde\upsilonv_{\dimh},\tilde\upsilonv_{\thetavs_{\dimh},\dimh}\in\Upss(\rr)\}\le \ex^{-\xx}\}, \\
\tilde\upsilonv_{\thetavs_{\dimh},\dimh}&\eqdef&\argmax_{\substack{\ups\in\Ups_{\dimh}\\ \Pi_{\thetav}\ups=\thetavs_{\dimh}}}\LL_{\dimh}(\ups),\\
\Upss(\rr)&\eqdef& \{\ups\in\R^{\dimtotal}:\, \|\DF_{\dimh}(\ups-\upsilonvs_{\dimh})\|\le \rr\}.
\end{EQA}
We want to use Theorem \ref{theo: large def with K}. For this we have with Lemma \ref{lem: a priori a priori accuracy for rr circ} combined with Lemma \ref{lem: condition ED_0 and Er} that condition \(\bb{(\CS\rr)}\) is met with probability \(1-2\ex^{-\xx}\) and with (setting \(\rr=\CONST\sqrt{n}\sqrt{\dimtotal\log(\dimtotal)} \) in \eqref{eq: bounds for Er single index})
\begin{EQA}
 \gm(\rr)=\sqrt n c_{\DF}\tilde g\CONST\left(\sqrt{\dimh}+\dimh^{2}\log(\dimtotal)\right)^{-1}, &\quad
 \nu_{\dimh}^{2}\le\tilde \nu^{2}\CONST\dimh^{3}\log(\dimtotal)^2.
\end{EQA}
Furthermore due to \(\rr^*\le \CONST\sqrt{\dimtotal}\) and for moderate \(\xx>0\) we find if
\begin{EQA}[c]
\rr^2\ge  \begin{cases} \CONST\dimtotal,  & \text{ if } \CONST_{bias}=0,\\
	\CONST\dimtotal\log(n)& \text{ if } \CONST_{bias}>0.\end{cases}
\end{EQA}
that with some \(\gmi>0\)
\begin{EQA}[c]
\P\left(-\E_{\eps}\LL_{\dimh}(\ups,\upss_{\dimh})\ge \gmi \rr^2\right)\ge 1-\ex^{-\xx}-\exp\left\{-\dimh^{3}\xx\right\}-\exp\left\{- n c_{\bb{(Q)}}/4\right\}.
\end{EQA}
Note that the second condition \eqref{cgmi1rrc} of Theorem \ref{theo: large def with K} is satisfied in our setting for \(n\in\N\) large enough as we assume that \({\dimtotal}^5(1+\CONST_{bias}\log(n))/n\to 0\). Finally we only have to ensure that \(\rups>0\) is large enough to satisfy \eqref{eq: lower bound rups in Lr}, then Theorem \ref{theo: large def with K} yields the following corollary.
\begin{corollary}
 \label{lem: size of r first iteration}
Consider the set
\begin{EQA}[c]
\mathcal{A}\eqdef\left\{ \bb{(\CS\rr)}\text{ and }\bb{(\cc{L}_{\rr})}\text{ are met}\right\}\cap \left\{\sup_{\thetav\in S^{\dimp,+}_1}\left\|\tilde\etav^{(\infty)}_{\dimh,\thetav}\right\|\le \CONST \sqrt{\dimtotal\log(\dimtotal)+\xx} \right\},
\end{EQA}
Then it holds that
\begin{EQA}[c]
\P\left(\mathcal{A}\cap \left\{\sup_{\ups\in \Ups_{\dimh}\setminus \Upss(\rups^{\circ})}\LL(\ups,\upss_{\dimh})< 0\right\} \right)\ge 1-\ex^{-\xx}-\P(\mathcal{A}^c),
\end{EQA}
where
\begin{EQA}[c]\label{eq: def of rups circ}
\rups^{\circ}\eqdef \begin{cases} \CONST \dimh^{3/2} \sqrt{\xx+\dimtotal} & \text{ if } \CONST_{bias}=0,\\
 \CONST\left(\sqrt{\dimtotal\log(n)}\vee \dimh^{3/2} \sqrt{\xx+\dimtotal} \right)& \text{ if } \CONST_{bias}>0.\end{cases}
\end{EQA}
\end{corollary}
Repeating the same steps from above gives that on the set
\begin{EQA}[c]
\left\{ \bb{(\CS\rr)}\text{ and }\bb{(\cc{L}_{\rr})}\text{ are met}\right\}\cap \left\{\sup_{\thetav\in S^{\dimp,+}_1}\left\|\tilde\etav^{(\infty)}_{\dimh,\thetav}\right\|\le \CONST \sqrt{\dimtotal\log(\dimtotal)+\xx} \right\}\\
	\cap\left\{\sup_{\ups\in \Ups_{\dimh}\setminus \Upss(\rups^{\circ})}\LL(\ups,\upss_{\dimh})< 0\right\}.
\end{EQA}
condition \(\bb{(\CS\rr)}\) is actually met on \(\Upss(\rups^{\circ})\) with 
\begin{EQA}
 \gm(\rr)=\sqrt n c_{\DF}\tilde g\CONST\dimh^{-1}, &\quad
 \nu_{\dimh}^{2}\le\CONST\tilde \nu^{2}\dimh,
\end{EQA}
if \({\dimtotal}^5(1+\CONST_{bias}\log(n))/n\to 0\). This gives
\begin{corollary}
 \label{lem: size of r second iteration}
Consider the set
\begin{EQA}
\mathcal{B}&\eqdef&\left\{ \bb{(\CS\rr)}\text{ and }\bb{(\cc{L}_{\rr})}\text{ are met}\right\}\cap \left\{\sup_{\thetav\in S^{\dimp,+}_1}\left\|\tilde\etav^{(\infty)}_{\dimh,\thetav}\right\|\le \CONST \sqrt{\dimtotal\log(\dimtotal)+\xx} \right\}\\
	&&\cap\left\{\sup_{\ups\in \Ups_{\dimh}\setminus \Upss(\rups^{\circ})}\LL(\ups,\upss_{\dimh})< 0\right\}.
\end{EQA}
Then it holds that
\begin{EQA}[c]
\P\left(\mathcal{B}\cap \left\{\sup_{\ups\in \Ups_{\dimh}\setminus \Upss(\rups)}\LL(\ups,\upss_{\dimh})< 0\right\} \right)\ge 1-2\ex^{-\xx}-\P(\mathcal{A}^c),
\end{EQA}
where
\begin{EQA}[c]\label{eq: size of r second iteration}
\rups\le \begin{cases} \CONST \sqrt{\xx+\dimtotal} & \text{ if } \CONST_{bias}=0,\\
 \CONST\sqrt{\xx+\dimtotal\log(n)} & \text{ if } \CONST_{bias}>0.\end{cases}
\end{EQA}
\end{corollary}


\subsection{Proof of finite sample Wilks and Fisher expansion}
Combining Lemma \ref{lem: conditions example} and Corollary \ref{lem: size of r second iteration} we obtain the following bound if \(\CONST_{bias}=0\) and \({\dimtotal}^4/n\to 0\) and if \(n\in\N\) is large enough:
\begin{EQA}
\breve \Excgr(\rups,\xx)&\le&\CONST_{\diamond}\frac{{\dimtotal}^{5/2}+\xx }{\sqrt{n}},
\end{EQA}
where \(\CONST_{\Excgr}>0\) is a polynomial of \(\|\psi\|_{\infty},\|\psi'\|_{\infty},\|\psi''\|_{\infty},\CONST_{\|\fvs\|},L_{\nabla \Phi},s_{\Xv}\).

With these results the case \(\CONST_{bias}=0\) in Proposition \ref{prop: main results finite dim} is merely a corollary of Theorem \ref{theo: main theo finite dim regression} and of Lemma \ref{lem: strong cond imply breve cond}. More precisely define the set
\begin{EQA}
&&\nquad\bb{\mathcal N}(\xx)\\
	&\eqdef& \left\{\sup_{\thetav\in S^{\dimp,+}_1}\left\|\tilde\etav^{(\infty)}_{\dimh,\thetav}\right\|\le \CONST \sqrt{\dimtotal\log(\dimtotal)+\xx} \right\}\\
	&&\cap\left\{\sup_{\ups\in \Ups_{\dimh}\setminus \Upss(\rups^{\circ})}\LL(\ups,\upss_{\dimh})< 0\right\}\\
	&&\cap\{\tilde\ups_{\dimh},\tilde\ups_{{\thetavs_{\dimh},\dimh}}\in\{\|\DF_{\dimh}(\ups-{\upss_{\dimh}})\|\le \rups\}\}\\
	&&\cap\left\{ \sup_{\ups\in\Upss(\rups)}\|\nabla(\E-\E_{\varepsilon})[\LL({\upss_{\dimh}})-\LL(\ups)]\|\le \CONST(\xx+\dimtotal)^2 \rups/\sqrt{n}\right\}\\
	&&\cap\left\{\text{The conditions of Section \ref{sec: conditions semi} are met for \((\LL,\Ups_{\dimh},{\DF} )\)}\right\}.
\end{EQA}
It is of Probability greater \(1-7\ex^{-\xx}-\exp\left\{-\dimh^{3}\xx\right\}-\exp\left\{- n c_{\bb{(Q)}}/4\right\}\).
Finally with the results of Appendix A of \cite{AASP2013} on the deviation behavior of quadratic forms we can bound with some constant related to the finite value \(\tr(\DPr^{-1}\VPr^2\DPr^{-1})\)
\begin{EQA}
 \P(\|\DPr^{-1}\scorer\|)\le \zz(\xx, \breve\BB))\ge 1-2\ex^{-\xx}, && \zz(\xx, \breve\BB)\le \sigma\CONST\sqrt{\dimtotal+\xx}.
\end{EQA}
So we get the claim with Theorem \ref{theo: main theo finite dim regression} via adapting the size of \(\CONST_{\Excgr} >0\). 

For the case that \(\CONST_{bias}>0\) we want to apply Proposition \ref{prop: large dev refinement regression}. For this define
\begin{EQA}[c]
\eps\eqdef 6\nu_1\omega \vee \delta(\rr)/\rr\le \CONST_{\diamond}\dimh^{5/2}/\sqrt{n}.
\end{EQA}
Then  \(\rups>0\) in \eqref{eq: size of r second iteration} satisfies by assumption
\begin{EQA}[c]
6\eps\rups\to 0.
\end{EQA}
since \(\dimh^3\log(n)/\sqrt{n}\to 0\).
Consequently Proposition \ref{prop: large dev refinement regression} applies with \(\bb{\mathcal N}(\xx)\) from above, which yields the claim of Proposition \ref{prop: main results finite dim} with an error term 
\begin{EQA}[c]
 \CONST_{\diamond}(1+\CONST_{bias})\frac{\xx+{\dimtotal}^{5/2}{\rups^*}^2 }{\sqrt{n}},
\end{EQA}
where 
\begin{EQA}
\rr_{0}^*&\le &  \zz(\xx,\BB)+\eps\zzQ(\xx,4\dimtotal)^2+ \eps^2\frac{18}{1-c}\zz_{\eps}(\xx)\le \CONST\sqrt{\dimtotal+\xx}.
\end{EQA}

\subsection{Bounding the sieve bias}
We prove this claim via showing that the conditions of and Theorem \ref{theo: bias cond} are met, which can be adapted to the regression set up in the same way as we did with Theorem \ref{theo: main theo finite dim regression} and Proposition \ref{prop: large dev refinement regression}. This concerns especially condition \({(bias)}\) from Section 2.7 of \cite{AASP2013}. For this we use the conditions \({(\cc{L}{\rr}_{\infty})} \) and \({(\kappav)}\) from \cite{AAbias2014} and then we can use Theorem \ref{theo: bias cond}. But exactly this is done in Lemma \ref{lem: conditions theta eta}. So we simply have to plug in our estimates. 

Finally we determine an admissible rate for \(\dimh(n)\in\N\) which ensures that the error terms vanish. We exemplify this for the case \(\CONST_{bias}=0\). We can show that
\begin{EQA}[c]
 \breve\Excgr(\rups^\circ,\xx )\le  \CONST(\dimtotal+\xx)^{5/2}/\sqrt n.
\end{EQA}
If \({\dimtotal}^{5/2}/\sqrt n\to 0\), we can get that \(2(\|\DPr^{-1}\scorer \|+\rr_{\dimp}^*(\xx_n))\Excgr(\rr_2,\xx_n )\toP 0\) by choosing a sequence \(\xx_n>0\), that increases slow enough. If \(\sqrt n\dimh^{-\alpha-1/2}\to 0\) we get the desired result. Clearly such a sequence exists and in this case \(\P(\Omega(\xx_n))\to 1\).

For the the weak convergence statements we also focus on the case \(\CONST_{bias}=0\) and use Theorem \ref{theo: bias cond}. As \(\delta(\rr),\omega\to 0\) and \(\rups(\xx)<\infty\) we further only have to prove condition \(\bb{(bias')}\) which means that we have to bound 
\begin{EQA}[c]
\|I_{\dimtotal}-\DPr_{\dimh}^{-1}(\upss)\DPr(\upss)\DPr_{\dimh}^{-1}(\upss)\|\text{ and }\|I_{\dimtotal}-\DPr_{\dimh}^{-1}(\upss_{\dimh})\DPr_{\dimh}(\upss)\DPr_{\dimh}^{-1}(\upss_{\dimh})\|.
\end{EQA}
With \({(\upsilonv\kappav)} \) - as proven in Lemma \ref{lem: conditions theta eta} - we can apply Lemma A.4 of \cite{AAbias2014} to find
\begin{EQA}   
\|I-\DPr_{\dimh}^{-1}\DPr\DPr_{\dimh}^{-1}\|&\le&\sqrt{\frac{1+\corrDF^2+\dimh^{-1}}{1-\corrDF^2}\frac{\CONST_1^2\dimh^{-1}}{c_{\DF}^2-\CONST_1^2\dimh^{-1}}}\to  0,
\end{EQA}
and with Lemma A.5 of \cite{AAbias2014}
\begin{EQA}
&&\nquad\| I-\DPrp(\upsilonvs_{\dimh})^{-1}\DPrp(\upsilonvs)^2\DPrp(\upsilonvs_{\dimh})^{-1}\|\\
	&\le&  \frac{\sqrt{\corrDF}\left(2+ \sqrt{1-\breve\delta(\rr^*)}\right)+1+\breve\delta(\rr^*)}{(1-\sqrt{\corrDF})^2}\breve\delta(\rr^*)\to 0.
\end{EQA}
Furthermore we need to satisfy \(\bb{(}\bb{bias}''\bb{)}\), which in our setting becomes
\begin{description}
\item[\(\bb{(}\bb{bias}''\bb{)}\)]
The i.i.d. random variables \(Y_i(\dimh)\in\R^{\dimp}\) satisfy \(\Cov(Y_i(\dimh))\to 0\) where
\begin{EQA}
Y_i(\dimh)&\eqdef& (\frac{1}{\sqrt n}\DPr_{\dimh})^{-1}\left\{\score_{\thetav} \left(\lkh_{i}(\upsilonvs_{\dimh}) -\lkh_{i}(\upsilonvs)\right)\right.\\
	&&\phantom{\frac{1}{\sqrt n}\DPr_{\dimh})^{-1}}\left.- \A_{\dimh}\HH_{\dimh}^{-2}\score_{(\eta_1,\ldots,\eta_\dimh)} \left(\lkh_{i}(\upsilonvs_{\dimh}) -\lkh_{i}(\upsilonvs)\right)\right\}.
\end{EQA}
\end{description}
which is done with Lemma \ref{lem: cond bias prime is satisfied}. This completes the proof after plugging in the bounds.

\subsection{Proof of convergence of the alternating procedure}
\label{sec: proof initial guess}
Here we want to explain in more detail how the Propositions \ref{prop: alternating single index} and \ref{prop: convergence to ME} can be derived.

We want to use Theorem \ref{theo: alternating stat props regression}. For this it remains to check the conditions \((\mathbf A_1)\), \((\mathbf A_2)\) and \((\mathbf A_2)\) from Section \ref{sec: synapsis alternating} for the initial guess defined in \eqref{eq: def of initial guess}. 

\begin{remark}
Condition \((\mathbf B_1)\) is met in our case as we pointed out in Section \ref{sec: single index large dev}.
\end{remark}

We can prove the following lemma: 

\begin{lemma}
\label{lem: conditions of initial guess met}
It holds for \(\xx\le \CONST \tilde\nu^2\tilde \gm^2n\) that
\begin{EQA}[c]
\P\left( \LL_{\dimh}(\tilde \upsilonv^{(0)},\upsilonvs_{\dimh})\le  -\CONST\left\{ (1+\CONST_{bias}\sqrt{\dimh})n\tau^2 +(1+  \CONST_{bias})\sqrt{\xx}\tau\sqrt{n}\right\} \right)\le 2\ex^{-\xx}.
\end{EQA}
If \(\CONST_{bias}=0\) set \(\tau =o({\dimtotal}^{-3/2})\) and \(\dimh^{4}=o(n)\). If \(\CONST_{bias}>0\) set \(\tau=o(\dimh^{-9/4})\) and \(\dimh^6= o(n)\). Then the initial radius \(\RR>0\) in \eqref{eq: def of radius RR}satisfies \(\eps\RR \to 0\) such that the conditions \((\mathbf A_1)\),\((\mathbf A_2)\) and \((\mathbf A_3)\) are satisfied for \(n\in \N\) large enough (as in Lemma \ref{lem: conditions example}).
\end{lemma}
Together with Theorem \ref{theo: alternating stat props regression} this implies Proposition \ref{prop: alternating single index} as we can bound
\begin{EQA}[c]
    \breve\Excgr_{Q}(\rr,\xx)
    \le \CONST_{\diamond}\frac{\xx+{\dimtotal}^{3/2}\rr^2+\CONST_{bias}{\dimtotal}^{2}\rr^2 }{\sqrt{n}}.
\end{EQA}
\begin{remark}\label{rem: RR small enough for conds domain} \(\eps\RR \to 0\) implies \(\RR\dimh^{3/2}/\sqrt{n}\to 0\). As pointed out in Remark \ref{rem: conds for alternating satisfied to} this means that the conditions from Section \ref{sec: conditions semi} can be satisfied on \(\Upss(\RR)\).
\end{remark}

For Proposition \ref{prop: convergence to ME} we apply Theorem \ref{theo: convergence to MLE regression}. It remains to show condition \( {(\CS \DF_{2})} \) and to bound \(\zz(\xx,\nabla^2\LL(\upss))\) which is defined via
\begin{EQA}[c]
\P\left\{\|\DF^{-1}\nabla^2\LL(\upss)\|\ge \zz\left(\xx,\nabla^2\LL(\upss)\right) \right\}\le \ex^{-\xx}.
\end{EQA}
We derive a bound for \(\zz(\xx,\nabla^2\LL(\upss))\) in Lemma \ref{lem: bound for norm of hessian} which is based on Corollary 3.7 of \cite{Tropp2012}, as is proposed in Remark 2.17 of \cite{AASPalternating}. The claim of Proposition \ref{prop: convergence to ME} is shown with the following Lemma.
\begin{lemma}\label{lem: conditions for convergence are satisfied}
Assume \((\mathbf{\mathcal A})\). Assume further that \({\dimtotal}^4/n\to 0\) and \(\tau =o({\dimtotal}^{-3/2})\) if \(\CONST_{bias}=0\) and \({\dimtotal}^6/n\to 0\) and \(\tau =o({\dimtotal}^{-9/4})\) if \(\CONST_{bias}>0\). Let \(\xx>0\) be chosen such that
\begin{EQA}[c]
\xx\le \frac{1}{2}\left(\tilde \nu^2 n\tilde \gm^2  - \log(\dimtotal)\right).
\end{EQA}
then the conditions \( {(\CS \DF_{2})} \), \( {(\cc{L}_{0})} \), \( {(\cc{L}_{\rr})} \) and \( {(\CS\rr)} \) are met and \(\kappa(\xx,\RR)\to 0\) with \(n\to\infty\).
\end{lemma}
\begin{remark}
The bound for \(\xx\) comes from Lemma \ref{lem: bound for norm of hessian} but also from the definition of \(\zzq(\xx,\cdot)\) and ensures that \(\zzq(\xx,3\dimtotal)=O(\sqrt{\xx+\dimtotal})\).
\end{remark}

\subsection{Proof of Proposition \ref{prop: convergence of PPP}}
Define the set
\begin{EQA}[c]
\bb{\mathcal M}_{M}(\xx)\eqdef \left\{\sup_{\thetav\in S^{\dimp,+}_1}\left\|\tilde\etav^{(\infty)}_{\dimh,\thetav}\right\|\le \CONST(\xx) \sqrt{\dimh} \right\}\cap\bigcap_{l=1}^{M}\bb{\mathcal N}_{l}(\xx),
\end{EQA}		
where
\begin{EQA}	
&&\nquad\bb{\mathcal N}_{l}(\xx)\eqdef \left\{\sup_{\ups\in \Ups_{\dimh}\setminus \Upss(\rups^{\circ})}\LL(\ups,\upss_{\dimh})< 0\right\}\\
	&&\cap
		\{\tilde\ups_{\dimh}\text{}_{(l)},\tilde\ups_{\dimh}\text{}_{{\thetavs}_{(l)}(l)},\tilde\ups_{\dimh}\text{}_{{\etavs}_{(l)}(l)}\in\{\|\DF_{(l)}(\ups-{\upss_{\dimh}}_{(l)})\|\le \rups\}\}\\
	&&\cap\bigcap_{\rr\le\rups}\bigg\{\sup_{\ups\in\Upss(\rr)}\left\|{\DF_{(l)}}^{-1}\left(\nabla{\zeta_{\varepsilon}}_{(l)}(\ups)-\nabla{\zeta_{\varepsilon}}_{(l)}({\upss_{\dimh}}_{(l)}) \right)\right\|-2\rr^2\\
		&&\phantom{ \bigg\{\sup_{\ups\in\Upss(\rr)}\left\|{\DF_{(l)}}^{-1}\left(\nabla{\zeta_{\varepsilon}}_{(l)}(\ups)-\nabla{\zeta_{\varepsilon}}_{(l)}({\upss_{\dimh}}_{(l)}) \right)\right\|}\le \CONST\omega \nu_1(\xx+\dimtotal) \bigg\}\\
	&&\cap\left\{ \|{\DF_{(l)}}^{-1}\nabla {\LL_{\varepsilon}}_{(l)}({\upsilonvs_{\dimh}}_{(l)})\|\le \CONST\sqrt{\xx+\dimtotal}\right\}\\
	&&\cap\bigg\{ \sup_{\ups\in{\Upss}_{(l)}(\rr^{\infty})}\|\nabla(\E-\E_{\varepsilon})[{\LL_{\varepsilon}}_{(l)}({\upss_{\dimh}}_{(l)})-{\LL_{\varepsilon}}_{(l)}(\ups)]\|\\
		&&\phantom{\bigg\{ \sup_{\ups\in{\Upss}_{(l)}(\rr^{\infty})}\|\nabla(\E-\E_{\varepsilon})[{\LL_{\varepsilon}}_{(l)}({\upss_{\dimh}}_{(l)}}\le \CONST(\xx+\dimtotal)^2 \rr^{\infty}/\sqrt{n}\bigg\}\\
	&&\cap\left\{\text{The conditions of Section \ref{sec: conditions semi} are met for \(({\LL_{\varepsilon}}_{(l)},\Ups_{\dimh},{\DF_{(l)}} )\)}\right\},
\end{EQA}
where \(\rups=\CONST(\dimtotal+\xx)M\), \(\rups^{\circ}=\CONST[{\dimtotal}^{3/2}\sqrt{\dimtotal+\xx}\vee (\dimtotal+\xx)M]\) and where
\begin{EQA}[c]
\rr^{\infty}(\xx)= \CONST \sqrt{\dimtotal+\xx}.
\end{EQA}

\begin{remark}
For \(M=1\) this is the set on which Proposition \ref{prop: main results finite dim} applies.
\end{remark}

\begin{lemma}
We have on the set \(\bb{\mathcal M}_{M}(\xx)\) if \({\dimtotal}^{5}/n< l\)
\begin{EQA}[c]\label{eq: induction assumption in ppp}
{\tau_i}_{(l)}\le \CONST l \sqrt{\dimh}\left(\frac{{\dimtotal}^{7/2}+\xx}{n}+\frac{\sqrt{\dimtotal+\xx}}{\sqrt{n}}\right).
\end{EQA}
\end{lemma}

\begin{proof}
We obtain with Proposition \ref{prop: large dev refinement regression} that if
\begin{EQA}
\left(\delta(\rr)/\rr+6\nu_1\omega\right)\rups<1, &\quad \text{ and }\left(\delta(\rr)/\rr+6\nu_1\omega\right)\CONST\sqrt{\xx+\dimtotal}<1,
\end{EQA}
that then
\begin{EQA}[c]
\bb{\mathcal M}_{M}(\xx)\subset \{\tilde\ups_{\dimh}\text{}_{(l)},\tilde\ups_{\dimh}\text{}_{\thetavs_{(l)}(l)}\subset \Upss(\rr^{\infty})\},
\end{EQA}
where
\begin{EQA}[c]
\rr^{\infty}(\xx)\le \CONST \sqrt{\dimtotal+\xx}.
\end{EQA}
But by assumption
\begin{EQA}
\left(\delta(\rr)/\rr+6\nu_1\omega\right)\CONST\sqrt{\xx+\dimtotal}&\le&\CONST \frac{{\dimtotal}^{5/2}+\xx}{\sqrt{n}}\to 0, \\
\left(\delta(\rr)/\rr+6\nu_1\omega\right)\rups(\xx)&\le &\CONST \frac{{\dimtotal}^{3}\log(n)M+\xx}{\sqrt{n}}\to 0.
\end{EQA}
Consequently we can restrict our selves to the set \({\Upss}(\rr^{\infty})\). We show the claim via induction. For this note that with \eqref{eq: ppp induction start} we already showed the claim for \(l=1\). Assume that the claim is already shown for \(0<l-1<M\). 
Remember that
\begin{EQA}[c]
\varsigmav_{i,\dimh}(\upsilonv)\eqdef \Big(\fv'_{\etav}(\Xv_{i}^{\T}\thetav) \nabla\Phi(\thetav)^{\T} \Xv_{i},  \basX(\Xv_{i}^{\T}\thetav)\Big)\in\R^{\dimp+\dimh}.
\end{EQA}
We find with the same arguments as in the proof of Proposition \ref{prop: large dev refinement regression} and using Lemma \ref{lem: bounds for scores and so on} that on the set \(\bb{\mathcal M}_{M}\) (we suppress \(\cdot _{(l)}\))
\begin{EQA}
&&\nquad\sup_{\ups\in{\Upss}(\rr^{\infty})}\left\|{\DF}^{-1}\left( \nabla\LL(\ups)-\nabla\LL({\upss_{\dimh}}) \right)+\DF(\ups-{\upss_{\dimh}}) \right\|\\
	&\le& \sup_{\ups\in{\Upss}(\rr^{\infty})}\left\|{\DF^{-1}}\left( \nabla\LL_{\varepsilon}(\ups)-\nabla\LL_{\varepsilon}({\upss_{\dimh}}) \right)+\DF(\ups-{\upss_{\dimh}}) \right\|\\
	&&	+ \sup_{\ups\in{\Upss}(\rr^{\infty})}\left\|{\DF^{-1}}\left( \nabla\LL_{\tau}(\ups)-\nabla\LL_{\tau}({\upss_{\dimh}}) \right) \right\|\\
	&\le& 	\Excgr_{Q}(\rr^{\infty},\xx)+ \sup_{\ups\in{\Upss}(\rr^{\infty})}\frac{2}{c_{\DF}\sqrt{n}}\sum_{i=1}^{n}\tau_i(l-1)\left\|\varsigmav_{i,\dimh}(\upsilonv)-\varsigmav_{i,\dimh}(\upsilonvs_{\dimh}) \right\|\\	
	&\le& 	\Excgr_{Q}(\rr^{\infty},\xx)+ \frac{\CONST \dimh^{3/2}\rr^{\infty}}{c^2_{\DF}}\max_i|\tau_i(l-1)|,
\end{EQA}
Denote
\begin{EQA}[c]\label{eq: def of bb B for tau}
\bb{B}_{(l-1)}\eqdef\max_i|\tau_i(l-1)|.
\end{EQA}
Then we find
\begin{EQA}
&&\nquad\left\|\DF_{(l)}({\tilde\ups_{\dimh}}\text{}_{(l)}-{{\upss_{\dimh}}_{(l)}}) \right\|\\
	&\le& \left\|{\DF_{(l)}}^{-1}\nabla{\LL_{\varepsilon}}_{(l)}({\upss_{\dimh}}_{(l)}) \right\|+\CONST\frac{{\dimtotal}^{7/2}+\xx}{\sqrt{n}}+\CONST{\dimtotal}^{2}\bb{B}_{(l-1)}\\
	&\le&\CONST\left(\sqrt{\dimtotal+\xx}+\CONST\frac{{\dimtotal}^{7/2}+\xx}{\sqrt{n}}+{\dimtotal}^{2}\bb{B}_{(l-1)}\right).
\end{EQA}

It remains to address the bias \(\|\DF_{(l)}({\upss_{\dimh}}_{(l)}-{\upss}_{(l)})\|\).

Using that the assumptions \((\mathbf{\mathcal A})\) hold for all \((g_{(l)})_{l=1,\ldots,M}\) we can bound as in Lemma \ref{lem: cond Lr infty}
\begin{EQA}
\E{\LL_{\varepsilon}}_{(l)}(\ups,{\upss}_{(l)})&\le& -\gmi \rr^2,
\end{EQA}
where \(\rr=\|\DF_{(l)}(\ups-{\upss}_{(l)})\|\).
With  Lemma A.2 of \cite{AAbias2014} this gives
\begin{EQA}
\|\DF_{(l)}({\upss_{\dimh}}_{(l)}-{\upss}_{(l)})\|^2
	&\le&{\rr^{*}}^2,
\end{EQA}
where we point out that \({\rr^*}\le \CONST \sqrt n\dimh^{-\alpha} \) in \eqref{eq: def of rr star} is a uniform upper bound for all \(l\le M\). 
We derived that on the set \(\bb{\mathcal M}_{M}\) using that \(\rr^*\le \CONST \sqrt{\dimtotal+\xx}\)
\begin{EQA}
&&\nquad\left\|\DF_{(l)}({\tilde\ups_{\dimh}}\text{}_{(l)}-{{\upss}_{(l)}}) \right\|\\
	&\le&\CONST\left(\sqrt{\dimtotal+\xx}+\CONST\frac{{\dimtotal}^{7/2}+\xx}{\sqrt{n}}+{\dimtotal}^{2}\bb{B}_{(l-1)}\right)\\
	&\eqdef& \CONST \bb{T}_{(l-1)}.\label{eq: ppa final bound for l estimator}
\end{EQA}
Finally we bound
\begin{EQA}
\left|\fv_{{\etavs}_{(l)}}(\Xv_i^\T {\thetavs}_{(l)})-\fv_{\tilde\etav_{(l)}}(\Xv_i^\T \tilde\thetav_{(l)})\right|
&\le& \left|\fv_{{\etavs}_{(l)}-\tilde\etav_{(l)}}(\Xv^\T \tilde\thetav)\right|\\
	&&+\left|\fv_{{\etavs}_{(l)}}(\Xv^\T {\thetavs}_{(l)})-\fv_{{\etavs}_{(l)}}(\Xv^\T \tilde\thetav_{(l)})\right|.
\end{EQA}
We estimate separately using \eqref{eq: ppa final bound for l estimator}
\begin{EQA}
\left|\fv_{{\etavs}_{(l)}-\tilde\etav_{(l)}}(\Xv^\T \tilde\thetav_{(l)})\right|&\le&  \|\|\HF_{\dimh}^{-1}\basX\|_{\R^{\dimh}}\|_{\infty} \CONST\bb{T}_{(l-1)}\\
	&\le& \CONST\sqrt{\dimh}\bb{T}_{(l-1)}/\sqrt{n}.
\end{EQA}
Furthermore we find with \eqref{eq: ppa final bound for l estimator}
\begin{EQA}[c]
\left|\fv_{{\etavs}_{(l)}}(\Xv^\T {\thetavs}_{(l)})-\fv_{{\etavs}_{(l)}}(\Xv^\T \tilde\thetav)_{(l)}\right|\le  \CONST s_{\Xv} \|\fv_{{\etavs}_{(l)}}'\|\bb{T}_{(l-1)}/\sqrt{n}.
\end{EQA}
Consequently 
\begin{EQA}
\left|{\tau_i}_{(l)} \right| &=& \left|\sum_{s=1}^{l}\fv_{{\etavs}_{(s)}}(\Xv_i^\T {\thetavs}_{(s)})-\fv_{\tilde\etav_{(s)}}(\Xv_i^\T \tilde\thetav_{(s)}) \right|\\
	&\le& \CONST l \sqrt{\dimh}\left(\frac{{\dimtotal}^{7/2}+\xx}{n}+\frac{\sqrt{\dimtotal+\xx}}{\sqrt{n}}\right)+\CONST \sum_{s=1}^{l}\frac{{\dimtotal}^{5/2}}{\sqrt{n}}\bb{B}_{(s-1)}.	
\end{EQA}
Denote
\begin{EQA}
a\eqdef\CONST  \sqrt{\dimh}\left(\frac{{\dimtotal}^{7/2}+\xx}{n}+\frac{\sqrt{\dimtotal+\xx}}{\sqrt{n}}\right),&\quad
b\eqdef\frac{{\dimtotal}^{5/2}}{\sqrt{n}}.
\end{EQA}
Furthermore define
\begin{EQA}
{S_k}_{(l)}\eqdef \sum_{s=1}^{l}{S_{k-1}}_{(s-1)}, &\quad {S_0}_{(l)}=l.
\end{EQA}
Then we can write
\begin{EQA}
\left|{\tau_i}_{(l)} \right| &\le& a \sum_{k=0}^{l-1}b^{k}{S_k}_{(l)},
\end{EQA}
which gives with the crude bound \({S_k}_{(l)}\le l \sum_{s=0}^{k}l^{s}= l\frac{l^{k+1}-1}{l-1}\le 2 l^{k+1}\) that
\begin{EQA}
\left|{\tau_i}_{(l)} \right| &\le& 2 l a \sum_{k=0}^{l-1}b^{k}l^k\le \CONST l a,
\end{EQA}
if \(b< l \le M\). This gives the claim.
\end{proof}

To complete this section we show that the set \(\bb{\mathcal M}_{M}\) is of large probability as long as \(M\in\N\) is not too big.
\begin{lemma}
We have
\begin{EQA}[c]
\P(\bb{\mathcal M}_{M})\ge 1-\ex^{-\xx}-M \left(12\ex^{-\xx} +\exp\left\{-\dimh^{3}\xx\right\}+\exp\left\{- n c_{\bb{(Q)}}/4\right\}\right)
\end{EQA}
\end{lemma}

\begin{proof}
With Lemma \ref{lem: a priori a priori accuracy for rr circ} we find
\begin{EQA}[c]
 \P\left(\sup_{\thetav\in S^{\dimp,+}_1}\left\|\tilde\etav^{(\infty)}_{\dimh,\thetav}\right\|\ge \CONST(\xx) \sqrt{\dimh}\right)\le  \ex^{-\xx}.
 \end{EQA} 
Due to the assumptions we find with Lemma \ref{lem: conditions example} that
\begin{EQA}
&&\nquad\P\left(\text{The conditions of Section \ref{sec: conditions semi} are met for \(({\LL_{\varepsilon}}_{(l)},\Ups_{\dimh},{\DF_{(l)}} )\)} \right)\\
	&\ge& 1-4\ex^{-\xx}-\exp\left\{-\dimh^{3}\xx\right\}-\exp\left\{- n c_{\bb{(Q)}}/4\right\}.
\end{EQA}
On that set we find as in the proof of Proposition \ref{prop: large dev refinement regression}  for \(\CONST>0\) large enough
\begin{EQA}
&&\nquad\P\Bigg(\bigcap_{\rr\le\rups}\bigg\{\sup_{\ups\in\Upss(\rr)}\left\|{\DF_{(l)}}^{-1}\left(\nabla{\zeta_{\varepsilon}}_{(l)}(\ups)-\nabla{\zeta_{\varepsilon}}_{(l)}({\upss_{\dimh}}_{(l)}) \right)\right\|-2\rr^2\\
&\le& \CONST\omega \nu_1(\xx+\dimtotal) \bigg\}\Bigg)\ge 1-\ex^{-\xx}.
\end{EQA}
and
\begin{EQA}[c]
\P\left(\left\|{\DF_{(l)}}^{-1} \nabla{\zeta_{\varepsilon}}_{(l)}({\upss_{\dimh}}_{(l)})\right\|\ge \CONST\sqrt{\xx+\dimtotal}  \right)\ge 1-2\ex^{-\xx}.
\end{EQA}
Furthermore by Lemma \ref{lem: additional error from different expectation operator}
We have that 
\begin{EQA}
&&\nquad\P\left(\sup_{\ups\in{\Upss}_{(l)}(\rr)}\|\nabla(\E-\E_{\varepsilon})[{\LL_{\varepsilon}}_{(l)}({\upss_{\dimh}}_{(l)})-{\LL_{\varepsilon}}_{(l)}(\ups)]\ge \CONST(\xx+\dimtotal)^2\rr/\sqrt{n}\right)\\
	& \le& 2\ex^{-\xx}.
\end{EQA}
For the large deviation bound we proceed as follows. Note that
\begin{EQA}
\LL_{(l)}(\ups, {\upss_{\dimh}}_{(l)},{Y_i}_{(l)})&=&{\LL_{\varepsilon}}_{(l)}(\ups, {\upss_{\dimh}}_{(l)},{Y_i}_{(l)})\\
	&&+ 2\sum_{i=1}^{n}\tau_i(l-1)\left(\fv_{\etav} (\Xv_i^\T\thetav) -\fv_{{\etavs_{\dimh}}_{(l)}} (\Xv_i^\T{\thetavs_{\dimh}}_{(l)})\right).
\end{EQA}
Using \eqref{eq: def of bb B for tau} we can bound
\begin{EQA}
&&\nquad\sum_{i=1}^{n}\tau_i(l-1)\left(\fv_{\etav} (\Xv_i^\T\thetav) -\fv_{{\etavs_{\dimh}}_{(l)}} (\Xv_i^\T{\thetavs_{\dimh}}_{(l)})\right)\\
&\le& \CONST\bb{B}_{(l-1)}\sqrt{n}\sqrt{\dimh}\rr.
\end{EQA}
As the conditions \((\mathbf{\mathcal A})\) are satisfied for all \(l=1,\ldots,M\) we can establish as in Lemma \ref{lem: conditions example} for \(\rr^2\ge \CONST_{\gmi}\dimtotal\log(n)\)
\begin{EQA}
&&\nquad-\E_{\varepsilon} \sum_{i=1}^{n}\left(g_{(l)}(\Xv_i)+\varepsilon_i-\fv_{\etav} (\Xv_i^\T\thetav) \right)^2\\
	&&-\left(g_{(l)}(\Xv_i)+\varepsilon_i-\fv_{{\etavs_{\dimh}}_{(l)}} (\Xv_i^\T{\thetavs_{\dimh}}_{(l)}) \right)^2\le -\gmi_{(l)} \rr^2.
\end{EQA}
Together this implies for \(\rr\ge \CONST_{\gmi}\dimtotal\)
\begin{EQA}[c]
\E_{\varepsilon}\LL_{(l)}(\ups, {\upss_{\dimh}}_{(l)},{Y_i}_{(l)})\le -\gmi_{(l)} \rr^2 +\CONST\bb{B}_{(l-1)}\sqrt{n}\sqrt{\dimh}\rr .
\end{EQA}
This gives for \(\rr\ge \CONST\bb{B}_{(l-1)}\sqrt{n}\sqrt{\dimh}\) and \(\CONST>0\) large enough
\begin{EQA}[c]
\E_{\varepsilon}\LL(\ups, {\upss_{\dimh}}_{(l)},{Y_i}_{(l)})\le -\gmi_{(l)} \rr^2 /2.
\end{EQA}
Plugging in \eqref{eq: induction assumption in ppp} the lower bound becomes
\begin{EQA}[c]
{\rups}_{(l)}\ge \CONST \sqrt{\dimtotal+\xx}\left(1+  l \sqrt{\dimh}\frac{{\dimtotal}^{7/2}+\xx}{\sqrt{n}}\right)= \CONST ' M (\dimtotal+\xx).
\end{EQA}
For the remaining part we proceed as in section \ref{sec: single index large dev}. This gives the claim.
\end{proof}

\addtocontents{toc}{\protect\setcounter{tocdepth}{1}}
\section{Proofs}\label{sec: proofs}
In the following all the technical steps necessary to prove the Lemmas of section \ref{sec: details} are presented. But first we cite an important result that will be used in our arguments, namely the bounded difference inequality:
\begin{theorem}[Bounded differences inequality]
\label{theo: bounded differences inequality}
Let a function \(f:\mathcal X^{n}\to \R\) satisfy for any \(\Xv_1,\ldots,\Xv_{n},\Xv_{i}'\in \mathcal X\)
\begin{EQA}[c]
|f(\Xv_1,\ldots, \Xv_{i-1},\Xv_{i},\Xv_{i+1},\ldots,\Xv_{n})-f(\Xv_1,\ldots, \Xv_{i-1},\Xv_{i}',\Xv_{i+1},\ldots,\Xv_{n})|\le c_i.
\end{EQA}
Then for any vector of independent random variables \(\Xv\in \mathcal X^{n}\)
\begin{EQA}
\label{eq: implication bounded diff}
\P\left(f(\Xv)-\E f(\Xv)\ge t \right)&\le& \ex^{-\frac{2t^2}{\sum_{i=1}^{n}c_i^2}},\\
 \P\left(f(\Xv)-\E f(\Xv)\le -t \right)&\le& \ex^{-\frac{2t^2}{\sum_{i=1}^{n}c_i^2}}.
\end{EQA}
\end{theorem}

Furthermore we will use the basic chaining device as it was introduced by \cite{dudleyintegral} (see Section 2 of \cite{talagrandchaining1996} for a more concise description). As we use the idea several times, we summarize the central step in the following Lemma

\begin{lemma}\label{lem: basic chaining}
Let \(\{\UU(\ups)-\UU(\upss),\, \ups\in\Ups\}\) be a family of random variables index by a set \(\Ups\) that is contained in a normed space \((\BanX,\|\cdot\|)\). Define \(\Ups_0=\{\upss\}\) and with some \(\rr>0\) the sequence \(\rr_k=2^{-k}\rr\) and the sequence of sets \(\Ups_k\) each with minimal cardinality such that
\begin{EQA}
\Ups \subset \bigcup_{\ups\in\Ups_k} B_{\rr_k}(\ups), &\quad B_{\rr}(\ups)\eqdef \{\upsd\in \Ups,\,\|\upsd-\ups\|\le \rr\}.
\end{EQA}
Then for any \(\zz>0\) 
\begin{EQA}
&&\nquad\P\left( \sup_{\ups\in\Ups}|\UU(\ups)-\UU(\upss)|\ge \zz\right) \\
	&\le& \sum_{k=1}^{\infty}|\Ups_k|\sup_{\upsd\in\Ups_{k}}\P\left(\inf_{\ups\in\Ups_{k-1}} |\UU(\ups)-\UU(\upsd)|\ge 2^{-(k-1)/2}(1-1/\sqrt{2})\zz\right).
\end{EQA}
\end{lemma}
\begin{proof}
We simply use the definition and estimate
\begin{EQA}
&&\nquad\P\left(\sup_{\ups\in\Ups}|\UU(\ups)-\UU(\upss)|\ge \zz\right)\\
	&\le& \P\left(\sum_{k=1}^{\infty}\sup_{\ups_k\in\Ups_{k}}\inf_{\ups_{k-1}\Ups_{k-1}}|\UU(\ups_k)-\UU(\ups_{k-1})|\ge \zz\right)\\
	&\le&\sum_{k=1}^{\infty}\P\left(\sup_{\ups_k\in \Ups_{k}}\inf_{\ups_{k-1}\in\Ups_{k-1}} |\UU(\ups_k)-\UU(\ups_{k-1})|\ge 2^{-(k-1)/2} (1-1/\sqrt{2})\zz\right)\\
	&\le&\sum_{k=1}^{\infty}|\Ups_k|\\
		&&\sup_{\ups_k\in \Ups_{k}}\P\left(\inf_{\ups_{k-1}\in\Ups_{k-1}} |\UU(\ups_k)-\UU(\ups_{k-1})|\ge 2^{-(k-1)/2} (1-1/\sqrt{2})\zz\right),
\end{EQA}
where we used that \(\sum_{k=1}^{\infty}2^{-(k-1)/2}\le 1/(1-1/\sqrt{2})\).
\end{proof}

\subsection{Proof of Remark \ref{rem: how to ensure smoothness in biased case}}
\begin{proof}
This can be seen as follows. First with Fubini's Theorem we find
\begin{EQA}
\etav_k(\thetav)&\eqdef& \int_{[-s_{\Xv},s_{\Xv}]}\fs_{\thetav}(t)\basX_{k}(t)dt\\
	&=& \int_{[-s_{\Xv},s_{\Xv}]}\int_{B_{s_{\Xv}}(0)\cap \thetav^{\perp}} \fs_{\thetav,\xv}(t)\basX_{k}(t) p_{\Xv|\Xv^\T\thetav=t}(\xv)d\xv dt,\\
	&=& \int_{B_{s_{\Xv}}(0)\cap \thetav^{\perp}} \left(\int_{[-s_{\Xv},s_{\Xv}]}\fs_{\thetav,\xv}(t)\basX_{k}(t)dt\right) p_{\Xv|\Xv^\T\thetav=t}(\xv)d\xv ,\\
	&=& \int_{B_{s_{\Xv}}(0)\cap \thetav^{\perp}} \etav_k(\thetav,\xv)  p_{\Xv|\Xv^\T\thetav=t}(\xv)d\xv.
\end{EQA}
Note that the application of Fubini's theorem is justified since by assumtion \(|\fs_{\thetav,\xv}(t)\basX_{k}(t) p_{\Xv|\Xv^\T\thetav=t}(\xv)|<\infty\).
Furthermore with Jensen's inequality and exchanging the order integration and summation as the \(\limsup\) is finite we find
\begin{EQA}
\sum_{k=0}^\infty k^{2\alpha(\thetav)}\etav^2_k(\thetav)^2&=&\sum_{k=0}^\infty  k^{2\alpha(\thetav)}\left(\int_{B_{s_{\Xv}}(0)\cap \thetav^{\perp}}  \etav_k(\thetav,\xv)  p_{\Xv|\Xv^\T\thetav=t}(\xv)d\xv\right)^2\\
	&\le &\sum_{k=0}^\infty  \int_{B_{s_{\Xv}}(0)\cap \thetav^{\perp}}  k^{2\alpha(\thetav)}\etav_k(\thetav,\xv)^2  p_{\Xv|\Xv^\T\thetav=t}(\xv)d\xv\\
	&\le &\int_{B_{s_{\Xv}}(0)\cap \thetav^{\perp}}  \left(\sum_{k=0}^\infty k^{2\alpha(\thetav,\xv)}\etav_k(\thetav,\xv)^2 \right) p_{\Xv|\Xv^\T\thetav=t}(\xv)d\xv\\
	&<& \infty,
\end{EQA}
where we used in the second to last step that \(\alpha(\thetav)\le \alpha(\thetav,\xv) \).
\end{proof}

\subsection{Calculating the elements}
\label{sec: calculating the elements}
First we calculate the relevant objects in this setting. For this we have to emphasize one subtlety about this analysis. As the parameter \(\thetav\in\R^\dimp\) lies in \(S_{1}^{\dimp,+}\subset \R^{\dimp}\) a more appropriate parameter set is \(W_{S}\eqdef[0,\pi]\times[-\pi/2,\pi/2]\times[-\pi/2,\pi/2]\times...\times[-\pi/2,\pi/2]\subset\R^{\dimp-1}\). This gives, parametrising the half sphere \(S_{1}^{\dimp,+}\subset \R^{\dimp}\) via the standard spherical coordinates
\begin{EQA}[c]
 \Phi:[0,\pi]\times[-\pi/2,\pi/2]\times[-\pi/2,\pi/2]\times...\times[-\pi/2,\pi/2]\subset\R^{\dimp-1}\to S^{\dimp,+}_{1},
\end{EQA}
that our actual likelihood functional is defined on \(W_{S}\times \R^{\dimh} \) as
\begin{EQA}[c]
\label{eq: likelihood functional true parameter}
 \LL_{\dimh}(\thetav,\etav)=\sum_{i=1}^\nsize\|Y_i-\fv_{\etav}(\Xv_i^\T\Phi(\thetav))\|^2/2,
\end{EQA}
where with abuse of notation we denote the preimage of an element of the sphere by the same symbol. Fix any element of the set of maximizers \(\upsilonvs_{\dimh}\) for some \(\dimh\in\N\).

First we calculate
\begin{EQA}
 \zeta(\upsilonv,\upsilonvs)&:=&\LL_{\dimh}(\upsilonv,\upsilonvs)-\E_{\varepsilon}\LL_{\dimh}(\upsilonv,\upsilonvs)\\
 	&=&-\sum_{i=1}^{n}\varepsilon_{i}\Big( g(\Xv_{i})- \fv_{\etav}(\Xv_{i}^{\T}\Phi(\thetav))\Big).
\end{EQA}
This gives that with \(\nabla_{\dimtotal}=(\nabla_{\theta_1},\ldots,\nabla_{\theta_{\dimp-1}},\nabla_{\eta_1},\ldots, \nabla_{\eta_{\dimh}})\) and \(\bb{\varepsilon}=(\varepsilon_1,\ldots,\eps) \)
\begin{EQA}
 \nabla_{\dimtotal} \zeta(\upsilonv)&=&\sum_{i=1}^{n}\Big(\fv'_{\etav}(\Xv_{i}^{\T}\thetav) \nabla\Phi(\thetav)^{\T} \Xv_{i},  \basX(\Xv_{i}^{\T}\thetav)\Big)\varepsilon_{i}\\
 &\eqdef&\sum_{i=1}^{n} \varsigmav_{i,\dimh}(\upsilonv)\varepsilon_{i}\\
 &\eqdef&W_{\dimh}(\upsilonv)\bb{\varepsilon}.
\end{EQA}
where with \(\basX=(\basX_1,\ldots,\basX_{\dimh})\)
\begin{EQA}[c]
 W_{\dimh}(\upsilonv)=\left(\begin{array}{ccc}
        \fv'_{\etav}(\Xv_{1}^{\T}\thetav)\nabla\Phi(\thetav)^{\T} \Xv_{1}&...&\fv'_{\etav}(\Xv_n^{\T}\thetav)\nabla\Phi(\thetav)^{\T} \Xv_n \\
          \basX(\Xv_{1}^{\T}\thetav)& ... &  \basX(\Xv_n^{\T}\thetav)\\
      \end{array} \right).
\end{EQA}
As we use this notation in the following, we repeat the definition
\begin{EQA}[c]\label{eq: def of varsigmav}
\varsigmav_{i,\dimh}(\upsilonv)\eqdef\Big(\fv'_{\etav}(\Xv_{i}^{\T}\thetav) \nabla\Phi(\thetav)^{\T} \Xv_{i},  \basX(\Xv_{i}^{\T}\thetav)\Big)\in\R^{\dimtotal}.
\end{EQA}

By assumption the \(\varepsilon_{i}\)  are i.i.d. with covariance \(\sigma^{2}>0\) and the design points \((\Xv_i)\) are i.i.d. as well. We set
\begin{EQA}
\label{eq: calculation of VF}
 \VF^2_{\dimh}&\eqdef& \sigma^{2}\E W_{\dimh}(\upsilonvs)W_{\dimh}(\upsilonvs)^{\T}\\
 	&=&n \sigma^{2}\left( 
      \begin{array}{cc}
        d_{\thetav}^{2}(\upsilonvs) &a_{\dimh}(\upsilonvs) \\
        a_{\dimh}^{\T}(\upsilonvs) & h_{\dimh}^{2}(\upsilonvs)\\
      \end{array}  
    \right)\eqdef n \sigma^{2} d_{\dimh}^2\in\R^{(\dimp-1+\dimh) \times (\dimp-1+\dimh)}.
\end{EQA}
where with \(\E[\cdot]\) denoting the expectation under the measure \(\P^{\Xv_1}\)
\begin{EQA}
 d_{\thetav}^{2}(\upsilonv) &=&\E\Big[\fv'_{\etav}(\Xv_{1}^{\T}\thetav)^{2}\nabla\Phi(\thetav)^{\T} \Xv_{1}{\Xv_{1}}^{\T}\nabla\Phi(\thetav)\Big],\\
 h_{\dimh}^{2}(\upsilonv)&=& \E\Big[\basX\basX^{\T}(\Xv_{1}^{\T}\thetav)\Big],\\
 a_{\dimh}(\upsilonv)&=& \E\Big[\fv'_{\etav}(\Xv_{1}^{\T}\thetav) \nabla\Phi(\thetav)^{\T} \Xv_{1} \basX^{\T}(\Xv_{1}^{\T}\thetav)\Big].
\end{EQA}
Furthermore we get because of the quadratic functional and sufficient smoothness of the basis \((\basX_i)\) for any \(\upsilonv\in \R^{\dimtotal-1}\)
\begin{EQA}
\DF_{\dimh}^{2}(\upsilonv)&\eqdef&-\nabla_{\dimtotal}^{2}\E[\LL_{\dimh}(\upsilonv)]=n d_{\dimh}^2(\upsilonv)+n r_{\dimh}^2(\upsilonv) ,\\
	 d_{\dimh}^2 &=&	 \left( 
      \begin{array}{cc}
        d_{\thetav}^{2}(\upsilonv) &a_{\dimh}(\upsilonv) \\
        a_{\dimh}^{\T}(\upsilonv) & h_{\dimh}^{2}(\upsilonv)\\
      \end{array}  
    \right), \label{eq: def of matrix d}\\
  r_{\dimh}^2(\upsilonv)&=&\E\Bigg[\left[\fv_{\etav}(\Xv^{\T}\thetav)- g(\Xv)\right]\left( 
      \begin{array}{cc}
        v_{\thetav}^{2}(\upsilonv) & b_{\dimh}(\upsilonv) \\
        b_{\dimh}^{\T}(\upsilonv) & 0\\
      \end{array}  
    \right)\Bigg],\label{eq: def of matrix r in DF}\\
      v_{\thetav}^{2}(\upsilonv)&=&     2\fv''_{\etav}(\Xv^{\T}\thetav)\nabla\Phi_{\thetav}^{\T} \Xv\Xv^{\T}\nabla\Phi_{\thetav} +|\fv'_{\etav}(\Xv^{\T}\thetav)|^2\Xv^{\T}\nabla^2\Phi_{\thetav}^{\T}[\Xv,\cdot,\cdot],\\
      b_{\dimh}(\upsilonv)&=&\nabla\Phi_{\thetav}\Xv^{\T}\basX'^\T(\Xv^{\T}\thetav).
\end{EQA}

For the analysis of the sieve bias we also define the corresponding full operator \(\DF^{2} \in L(l^2,\{(x_k)_{k\in\N}, x\in\R\})\)
\begin{EQA}
\nquad\DF^{2}(\upsilonv) &=& nd^2(\upsilonv)+ \E\Bigg[\left[\fv_{\etav}(\Xv^{\T}\thetav)- g(\Xv)\right]\left( 
      \begin{array}{cc}
        v_{\thetav}^{2}(\upsilonv) & b_{\infty}^{\T}(\upsilonv) \\
        b_{\infty}^{\T}(\upsilonv) & 0\\
      \end{array}  
    \right)\Bigg],
\end{EQA}
where with the obvious adaptations
\begin{EQA}[c]
d^2(\upsilonv)=\left( 
      \begin{array}{cc}
        d_{\thetav}^{2}(\upsilonv) &a_{\infty}(\upsilonv) \\
        a_{\infty}^{\T}(\upsilonv) & h_{\infty}^{2}(\upsilonv)\\
      \end{array}  
    \right).\label{eq: representation of matrix d}
\end{EQA}
Furthermore we calculate - with \(\varsigmav_{i,\dimh}\) from \eqref{eq: def of varsigmav} -
\begin{EQA}[c]
\nabla^2\zeta(\ups)=\sum_{i=1}^{n}\nabla\varsigmav_{i,\dimh}(\ups),
\end{EQA}
where
\begin{EQA}
\Pi_{\thetav}\nabla_{\thetav}\varsigmav_{i,\dimh}(\ups)&=&  \fv''_{\etav}(\Xv_{i}^{\T}\thetav) \nabla\Phi(\thetav)^\T\Xv_{i}\Xv_{i}^\T\nabla\Phi(\thetav)\\
	&&+\fv'_{\etav}(\Xv_{i}^{\T}\thetav)\Xv_{i}\nabla^2\Phi(\thetav^\T\Xv_i)[\Xv_{i},\cdot,\cdot]  ,\\
\Pi_{\etav}\nabla_{\thetav}\varsigmav_{i,\dimh}(\ups)&=& \basX'(\thetav^\T\Xv_i)\Xv_{i}^{\T}\nabla\Phi(\thetav) ,\\
\nabla_{\etav}\varsigmav_{i,\dimh}(\ups)&=&0.
\end{EQA}
%
%
%
%
%
%
%
%
\subsection{Preliminary calculations}

\begin{lemma}\label{lem: bounds for basis scalar products}
We have
\begin{EQA}
&&\nquad|\E[\basX_k\basX_l(\Xv^\T\thetav)]|\\
	&\le&  17s_{\Xv}^{\dimp+1}L_{p_{\Xv}}\|\psi\|_{\infty}2^{-j_l-1}2^{j_k/2-j_l/2} 1_{\{I_l\cap I_k\neq \emptyset\}}(k,l),\text{ for }l\ge k \label{eq: bound for scalar product basis k l}\\
&&\nquad|\E[(\Xv^\T\thetav)\basX_k'\basX_l(\Xv^\T\thetavd)]|\\
	&\le&17\frac{\sqrt{p+2}}{2} \pi\|\psi'\|_\infty s^2_{\Xv}\|p_{\Xv^\T\thetav}\|_{\infty}2^{3j_k/2}2^{-(j_l\vee j_k)/2}1_{\{I_l\cap I_k\neq \emptyset\}}(k,l),\label{eq: bound for scalar product basis derivative k basis l and X theta}\\
&&\nquad|\E[\basX_l'\basX'_k(\Xv^\T\thetav)]|\\
	&\le& 17 s_{\Xv}\|\psi'\|_{\infty}\|p_{\Xv}\|_{\infty} 2^{3(j_l+j_k)/2-(j_l\vee j_k)}1_{\{I_k\cap I_l\neq \emptyset\}}(k,l), \label{eq: bound for scaler product basis derivatives l k}\\
&&\nquad\E\left[(\basX_k(\Xv^\T\thetav)-\basX_k(\Xv^\T\thetav'))(\basX_l(\Xv^\T\thetav)-\basX_l(\Xv^\T\thetav'))\right]\\
	&\le& \CONST\|\thetav-\thetav'\|^2 2^{j_k}2^{j_l}\|\psi'\|_{\infty}^2 s_{\Xv}^4 17^21_{\{I_k\cap I_l\neq \emptyset\}},\label{eq: bound for expectation of basis difference}\\
	&&\nquad\E\left[(\basX'_k(\Xv^\T\thetav)-\basX'_k(\Xv^\T\thetav'))(\basX'_l(\Xv^\T\thetav)-\basX'_l(\Xv^\T\thetav'))\right]\\
	&\le& \CONST\|\thetav-\thetav'\|^2 2^{2j_k}2^{2j_l}\|\psi''\|_{\infty}^2 s_{\Xv}^4 17^2 1_{\{I_k\cap I_l\neq \emptyset\}}	\label{eq: bound for expectation of basis difference first derivative}\\
	&&\nquad\E\left[\left(\basX_l(\Xv^\T\thetav)-\basX_l(\Xv^\T\thetavs_{\dimh})\right) \basX_k(\Xv^\T\thetav)\right]\\
	&\le&  \CONST \|\thetav-\thetav'\| 2^{j_l/2}2^{(j_k\wedge j_l)/2}\label{eq: bound expectation basis l times difference of basis vectors}
\end{EQA}
\end{lemma}

\begin{proof}
Observe that if the density of \(p_{\Xv}:\R^{\dimp}\mapsto \R\) is Lipshitz continuous with Lipshitz constant \(L_{p_{\Xv}}\) and its support contained in a ball of radius \(s_{\Xv}>0\) then the density \(p_{\Xv^\T\thetavs}:\R\mapsto \R\) of \(\Xv^\T\thetavs\in\R\) is Lipshitz continuous with Lipshitz constant \(L_{p_{\Xv^\T\thetavs}}\le s_{\Xv}^{\dimp}L_{p_{\Xv}}\). Furthermore for \(k,l\in\N\)
\begin{EQA}
\E[\basX_k\basX_l(\Xv^\T\thetav)]&=&\int_{[-s_{\Xv},s_{\Xv}]}\basX_k(x)\basX_l(x)p_{\Xv^\T\thetavs}(x)dx.
\end{EQA}
Denote by \(I_{k}\subset \R\) the support of \(\basX_k(x)\). We write 
\begin{EQA}
&&\nquad\E[\basX_k\basX_l(\Xv^\T\thetav)]=\int_{I_l}\basX_k(x)\basX_l(x) p_{\Xv^\T\thetavs}(x)dx\\
	&=&\int_{I_l}\basX_k(x)\basX_l(x) p_{\Xv^\T\thetavs}(x_0)dx 1_{\{I_l\cap I_k\neq \emptyset\}}(k,l)\\
		&&+\int_{I_l}\basX_k(x)\basX_l(x)\Big(p_{\Xv^\T\thetavs}(x)-p_{\Xv^\T\thetavs}(x_0)\Big)dx 1_{\{I_l\cap I_k\neq \emptyset\}}(k,l),
\end{EQA}
where \(x_0\in I_l\) is the center of the support of \(\basX_l(x)\), which is of length \(2^{-j_l}17 s_{\Xv}\) for \(l=(2^{j_l}-1)17+r_l\in\N\). Because of orthogonality the first summand on the right-hand side is equal to zero. For the second summand we use the Lipshitz continuity and Cauchy-Schwarz to estimate
\begin{EQA}
&&\nquad|\int_{I_l}\basX_k(x)\basX_l(x) \Big(p_{\Xv^\T\thetavs}(x)-p_{\Xv^\T\thetavs}(x_0)\Big)dx| 1_{\{I_l\cap I_k\neq \emptyset\}}(k,l)\\
	&\le& s_{\Xv}^{\dimp}L_{p_{\Xv}}2^{-j_l-1}\int_{I_l}|\basX_k(x)||\basX_l(x)|dx 1_{\{I_l\cap I_k\neq \emptyset\}}(k,l)\\
	&\le& s_{\Xv}^{\dimp}L_{p_{\Xv}}2^{-j_l-1}\left(\int_{I_l}\basX_l(x)^2dx\int_{I_l}\basX_k(x)^2dx\right)^{1/2} 1_{\{I_l\cap I_k\neq \emptyset\}}(k,l)\\
	&\le& s_{\Xv}^{\dimp}L_{p_{\Xv}}2^{-j_l-1}\left(\int_{I_l}\basX_k(x)^2dx\right)^{1/2} 1_{\{I_l\cap I_k\neq \emptyset\}}(k,l)\\
	&\le& 17s_{\Xv}^{\dimp+1}L_{p_{\Xv}}\|\psi\|_{\infty}2^{-j_l-1}2^{j_k/2-j_l/2} 1_{\{I_l\cap I_k\neq \emptyset\}}(k,l),
\end{EQA}
where we used that the \((\basX_k)\) form an orthonormal basis, that \(\|\basX_k\|_{\infty} \break \le 2^{j_k/2}\|\psi\|_{\infty}\) and that \(I_l\) is of length \(2^{-j_l}17 s_{\Xv}\). This gives \eqref{eq: bound for scalar product basis k l}. Using that for any \(\thetav\in W_{S}\) it holds true that \(\|\nabla\Phi(\thetavs)\thetav\|\le \frac{\sqrt{p+2}}{2} \pi\) we estimate similarly to before
\begin{EQA}
&&\nquad|\E[(\Xv^\T\thetav)\basX_k'\basX_l(\Xv^\T\thetavd)]|\\
&\le& \frac{\sqrt{p+2}}{2} \pi s^2_{\Xv}\E[|\basX_k'\basX_l(\Xv^\T\thetavd)|]\\
&\le& \frac{\sqrt{p+2}}{2} \pi s^2_{\Xv} \int_{I_l}\basX_k'(x)\basX_l(x) p_{\Xv^\T\thetavs}(x)dx \\
&\le& \frac{\sqrt{p+2}}{2} \pi s^2_{\Xv}\|p_{\Xv^\T\thetav}\|_{\infty}  \left(\int_{I_l}\basX_k'(x)^2dx\right)^{1/2} \left(\int_{I_l}\basX_l(x)^2 dx\right)^{1/2}  \\
&\le &17\frac{\sqrt{p+2}}{2} \pi\|\psi'\|_\infty s^2_{\Xv}\|p_{\Xv^\T\thetav}\|_{\infty}2^{3j_k/2}2^{-(j_l\vee j_k)/2}1_{\{I_l\cap I_k\neq \emptyset\}}(k,l).
	\end{EQA}
The bound \eqref{eq: bound for scaler product basis derivatives l k} follows with exactly the same calculations.
To show \eqref{eq: bound for expectation of basis difference} we calculate with \(M_k\eqdef\{(x,y)\in\R^2,\, x\in I_k\}\cup \{(x,y)\in\R^2,\, x+y\in I_k\}\) and with \(p_{\thetav,(\thetavd-\thetav)}:\R^2\to \R_+\) denoting the density of \((\Xv^\T\thetav,\Xv^\T(\thetavd-\thetav))\in\R^{2}\)
\begin{EQA}
&&\nquad\E\left[(\basX_k(\Xv^\T\thetav)-\basX_k(\Xv^\T\thetavd))(\basX_l(\Xv^\T\thetav)-\basX_l(\Xv^\T\thetavd))\right]\\
	&=&1_{\{I_k\cap I_l\neq \emptyset\}}\int_{M_k}(\basX_k(x)-\basX_k(x+y))(\basX_l(x)-\basX_l(x+y)) \\
	&&\phantom{1_{\{I_k\cap I_l\neq \emptyset\}}\int_{M_k}(\basX_k(x) }p_{\thetav,(\thetavd-\thetav)}(x,y)d(x,y)\\
	&\le& 1_{\{I_k\cap I_l\neq \emptyset\}}\left(\int_{M_k}(\basX_k(x)-\basX_k(x+y))^2 p_{\thetav,(\thetavd-\thetav)}(x,y)d(x,y)	\right)^{1/2}\\
		&&\left(\int_{M_l}(\basX_l(x)-\basX_l(x+y))^2 p_{\thetav,(\thetavd-\thetav)}(x,y)d(x,y)	\right)^{1/2}.
\end{EQA}	
We estimate separately
\begin{EQA}
&&\nquad\int_{M_k}(\basX_k(x)-\basX_k(x+y))^2 p_{\thetav,(\thetavd-\thetav)}(x,y)d(x,y)\\
&\le& 2^{3j_k}\|\psi''\|_{\infty}
^2 \int_{M_k} y^2p_{\thetav,(\thetavd-\thetav)}(x,y)d(x,y),
\end{EQA}
Note that \(p_{\thetav,(\thetavd-\thetav)}(x,y)>0\) only for \(|y|\le \|\thetav-\thetavd\| (s_{\Xv}+h)\), where we suppress \(h\) in the following such that
\begin{EQA}
&&\nquad\int_{M_k}(\basX_k(x)-\basX_k(x+y))^2 p_{\thetav,(\thetavd-\thetav)}(x,y)d(x,y)\\
&\le& \|\thetav-\thetavd\|^2 2^{3j_k}\|\psi''\|_{\infty}^2s_{\Xv}^2\\
	&&\left(\int_{\R} \int_{I_k-x}  p_{\thetav,(\thetavd-\thetav)}(x,y)dydx+\int_{I_k} \int_{\R}  p_{\thetav,(\thetavd-\thetav)}(x,y)dydx\right)\\
 &\le &\|\thetav-\thetavd\|^2 2^{3j_k}\|\psi''\|_{\infty}^2 s_{\Xv}^2\\
 	&&\left(\int_{\R} \P\left\{(\thetavd-\thetav)^\T\Xv \in I_k-x| \thetav^\T\Xv=x\right\} p_{\thetav}(x)dx +\int_{I_k}  p_{\thetav}(x)dx\right).
\end{EQA}
represent \(\thetavd=\alpha \thetav +\beta \thetav'\) where \(\thetav'\perp \thetav\) with \(\|\thetavd\|=1\). Then we find with condition \((\mathbf{Cond}_{\Xv})\)
\begin{EQA}
&&\nquad\P\left\{(\thetavd-\thetav)^\T\Xv \in I_k-x| \thetav^\T\Xv=x\right\}\\
	&=&\P\left\{{\thetav'}^\T\Xv \in \frac{1}{\beta}(I_k-(1-\alpha)x)| \thetav^\T\Xv=x\right\}\\
	&\le& \left\|\frac{p_{\thetav',\thetav}}{p_{\thetav}}\right\|_{\infty}\lambda\left\{\frac{1}{\beta}(I_k-(1-\alpha)x)\right\}\le \CONST 2^{-j_k}/\|\thetav-\thetavd\|.
\end{EQA}
With the bound \(p_{\thetav}(x)\le \CONST_{p_{\Xv}}\) we find (since \(\|\thetav-\thetavd\|<\sqrt{2}\))
\begin{EQA}\label{eq: bound for L2 norm of basis derivative difference}
&&\nquad\int_{M_k}(\basX_k(x)-\basX_k(x+y))^2 p_{\thetav,(\thetavd-\thetav)}(x,y)d(x,y)\\
	&\le& \CONST\|\thetav-\thetavd\| 2^{2j_k}\|\psi''\|_{\infty}^2 s_{\Xv}^4 17^2,
\end{EQA}
which yields \eqref{eq: bound for expectation of basis difference}. With the same calculations we can show \eqref{eq: bound for expectation of basis difference first derivative}.
with \(M_{l,k}\eqdef\{(x,y)\in I_k\times\R,\, x\in I_l\cap I_k\}\cup  \{(x,y)\in I_k\times\R,\, x+y\in I_l\}\)
\begin{EQA}
&&\nquad\E\left[\left(\basX_l(\Xv^\T\thetav)-\basX_l(\Xv^\T\thetavs_{\dimh})\right) \basX_k(\Xv^\T\thetav)\right]\\
	&\le&\left(\int_{M_{l,k}}\left(\basX_l(x)-\basX_l(x+y)\right)^2p_{\thetav,(\thetavs_{\dimh}-\thetav)}(x,y)d(x,y)\right)^{1/2}\\
	&&\left(\int_{M_{l,k}}\basX_k^2(x)p_{\thetav,(\thetavs_{\dimh}-\thetav)}(x,y)d(x,y)\right)^{1/2}.
\end{EQA}
We have by \eqref{eq: bound for expectation of basis difference}
\begin{EQA}
&&\nquad\int_{M_{l,k}}\left(\basX_l(x)-\basX_l(x+y)\right)^2p_{\thetav,(\thetavs_{\dimh}-\thetav)}(x,y)d(x,y)\\
	&\le& 2^{2j_l}\|\thetav-\thetavs_{\dimh}\|^2\|\psi'\|^2 s_{\Xv}^4 17^2\CONST_{p_{\Xv}}.
\end{EQA}
As above we can bound
\begin{EQA}
&&\nquad\int_{M_{l,k}} \basX_k^2(x)p_{\thetav,(\thetavs_{\dimh}-\thetav)}(x,y)d(x,y)\\
	&=&\int_{\R}\basX_k^2(x) \int_{I_l-x}  p_{\thetav,(\thetav'-\thetav)}(x,y)d(x,y)\\
		&&+\int_{I_l\cap I_k}\basX_k^2(x) \int_{\R}  p_{\thetav,(\thetav'-\thetav)}(x,y)d(x,y)\\
	&\le& \int_{\R}\basX_k^2(x) \P\left\{(\thetav'-\thetav)^\T\Xv\in (I_l-x)\right| \thetav^\T\Xv= x\}  p_{\thetav}(x)d(x)\\
		&&+	\int_{I_l\cap I_k}\basX_k^2(x)p_{\thetav}(x)d(x)\\
	&\le& \frac{\CONST}{\|\thetav-\thetavd\|}2^{-j_l}\CONST_{p_{\Xv}}+2^{-j_l}2^{(j_k\wedge j_l)}	\|\psi\|^2_{\infty}.
\end{EQA}
\end{proof}

\begin{lemma}\label{lem: bounds for objects} 
For any \((\thetav,\etav)\in\R^{\dimp+\dimh}\)
\begin{EQA}
 \| \basX(\xv)\|&\le &{\CONST}\|\psi\|_{\infty}\sqrt{\dimh}, \label{eq: bound for basis vector}\\
 \left|\fv_{\etav}(\xv)\right|&\le& {\CONST}\|\psi\|_{\infty}\sqrt{\dimh}\|\etav\|, \label{eq: bound for f eta general}\\
 \| \basX'(\xv)\|&\le& \sqrt{17}\|\psi'\|\dimh^{3/2},\label{eq: bound for derivative basis vector}\\
\end{EQA}
\end{lemma}

\begin{proof}
Clearly \( \left|\fv_{\etav}(\xv)\right|\le \|\etav\|\| \basX(\Xv_{i}^{\T}\thetavs_{\dimh})\|\). Because of the wavelet structure and the choice \(\dimh=2^{j_m}17-1\) we have for each \(j=0,\ldots,j_{\dimh}-1\) that
\begin{EQA}\label{def: index set M j}
&&\nquad|M(j)|\\
	&\eqdef&\Big|\Big\{k\in\{(2^{j}-1)17,\ldots,(2^{j+1}-1)17-1\}:|\basX_k(\xv)|\neq 0 \Big\}\Big|\le 17.
\end{EQA}
This implies
\begin{EQA}
\label{eq: bound for e-basis}
\| \basX(\xv)\|&=& \left(\sum_{k=0}^{\dimh-1}|\basX_k(\xv)|^2\right)^{1/2}=\left(\sum_{j=0}^{j_{\dimh}-1}\sum_{k\in M(j)}|\basX_{k}(\xv)|^2\right)^{1/2} \\
	&\le&\sqrt{17}\|\psi\|_{\infty}\left(\sum_{j=0}^{j_{\dimh}-1}2^{j}\right)^{1/2}=\sqrt{17}\|\psi\|_{\infty}2^{j_{\dimh}/2}\le\sqrt{17}\|\psi\|_{\infty}\sqrt{\dimh}.
\end{EQA}
The proof of \eqref{eq: bound for derivative basis vector} works analogously.
\end{proof}

\subsection{Lower bound for the information operator}
\begin{lemma}
 \label{lem: D_0 dimh upss is boundedly invertable}
Under \((\mathbf{Cond}_{\Xv,\basX})\), \((\mathbf{Cond}_{\Xv\thetavs})\) and (model bias) we find for all \(\dimh\in\N\cup \{\infty\}\) that \(\DF_{\dimh}(\upss)\ge c_{\DF}*\) with some constant \(c_{\DF}*>0\).
\end{lemma}
\begin{remark}
The constant \(c_{\DF}*>0\) is specified - to some extend - in the proof. 
\end{remark}

%

\begin{proof}
We represent for any \(\gammav\in\R^{\dimtotal}\) with \(\|\gammav\|=1\)
\begin{EQA}
&&\nquad\gammav^\T\DF_{\dimh}\gammav\\
	&=&n \lim_{h\to 0}\frac{1}{h^2}\bigg(\E\left[\left(g(\Xv)-\sum_{k=1}^{\dimh}(\etas_k+h\gamma_{\dimp+k}) \basX_k(\Xv^\T(\thetavs+h\Pi_{\thetav}\gammav))\right)^2\right]\\
	&&-\E\left[(g(\Xv)-\E[g(\Xv)|\Xv^\T\thetavs])^2\right]\bigg).
\end{EQA}
Using the properties of conditional expectation we can write
\begin{EQA}
&&\nquad\E\left[\left(g(\Xv)-\sum_{k=1}^{\dimh}(\etas_k+h\gamma_{\dimp+k}) \basX_k(\Xv^\T(\thetavs+h\Pi_{\thetav}\gammav))\right)^2\right]\\
	&=&\E\bigg[\bigg(\E[g(\Xv)|\Xv^\T(\thetavs+\Pi_{\thetav}\gammav)]\\
		&&-\sum_{k=1}^{\dimh}(\etas_k+h\gamma_{\dimp+k}) \basX_k(\Xv^\T(\thetavs+h\Pi_{\thetav}\gammav))\bigg)^2\bigg]\\
		&&+\E\left[(g(\Xv)-\E[g(\Xv)|\Xv^\T(\thetavs+h\Pi_{\thetav}\gammav)])^2\right]
\end{EQA}
Using assumption (model bias) we find
\begin{EQA}
&&\nquad\gammav^\T\DF_{\dimh}\gammav\ge n \gmi_{\theta}\|\Pi_{\thetav}\gammav\|^2\\
	&&+n \lim_{h\to 0}\frac{1}{h^2}\E\bigg(\E[g(\Xv)|\Xv^\T(\thetavs+h\Pi_{\thetav}\gammav)]\\
		&&\phantom{n \lim_{h\to 0}\frac{1}{h^2}}-\sum_{k=1}^{\dimh}(\etas_k+h\gamma_{\dimp+k}) \basX_k(\Xv^\T(\thetavs+h\Pi_{\thetav}\gammav))\bigg)^2.
\end{EQA}
In case that \(\|\Pi_{\thetav}\gammav\|^2\ge \tau^2>0\) with some \(\tau>0\) this implies \(\DF_{\dimh}\ge \gmi_{\theta}\tau^2\). Assume \(\|\Pi_{\thetav}\gammav\|^2\le \tau^2\). Using the smoothness of the density \(p_{\Xv}\) and of \(g\) we find with some constant
\begin{EQA}[c]
\Big|\E[g(\Xv)|\Xv^\T(\thetavs+h\Pi_{\thetav}\gammav)]-\E[g(\Xv)|\Xv^\T\thetavs]\Big|\le  \CONST\|\Pi_{\thetav}\gammav\|\le n \CONST h\tau.
\end{EQA}
Furthermore we show in Lemma \ref{lem: size of Q} that with some \(\gmi^*>0\) and \(Q>0\)
\begin{EQA}[c]
\inf_{\ups\in\Ups_{\dimh}}\P\left(\left|\E[g(\Xv)|\Xv^\T\thetavs]-\sum_{k=1}^{\dimh}(\eta_k) \basX_k(\Xv^\T\thetav) \right|
	\ge \gmi^* \|\ups-\upss\| \right)\ge Q>0.
\end{EQA}
\begin{remark}
A close look at the proof of Lemma \ref{lem: size of Q} reveals that the claim can be shown with \(\|\ups-\upss\|\) instead of \(\|\DF(\ups-\upss)\|\) on the right-hand side with the same arguments.
\end{remark}
Consequently
\begin{EQA}
&&\nquad\E\left[\bigg(\E[g(\Xv)|\Xv^\T(\thetavs+h\Pi_{\thetav}\gammav)]-\sum_{k=1}^{\dimh}(\etas_k+h\gamma_{\dimp+k}) \basX_k(\Xv^\T(\thetavs+h\Pi_{\thetav}\gammav))\bigg)^2\right]\\
	&\ge& Q 	h^2(\gmi^*-\CONST\tau)^2.	
\end{EQA}
Setting \(\tau\le \gmi^*/(2\CONST)\) gives the claim.
\end{proof}

\subsection{Regularity}

\begin{lemma}
\label{lem: conditions theta eta}
Assume that the density \(p_\Xv:\R^{\dimp}\to \R\) is Lipshitz continuous and that the \(\Xv\in\R\) are bounded by some constant \(s_{\Xv}>0\). Then using our orthogonal and sufficiently smooth wavelet basis we get for any \(\lambda\in[0,1]\)
\begin{EQA}
\|\HF_{\dimh} ^{1/2}\kappavs\|^2&<&\left(17\|p_{\Xv^\T\thetavs}\|_\infty\CONST_{\|\fvs\|}+17^2 \sqrt{36}s_{\Xv}^{\dimp+1}L_{p_{\Xv}}\|\psi\|_{\infty}\CONST_{\|\fvs\|}^2\right) n\dimh^{-2\alpha},\\
\alpha(\dimh)&\eqdef&\|\DF_{\dimh}^{-1}\AF_{\upsilonv\kappav}\kappavs\|\le  \CONST_1\sqrt n\left( \dimh^{-(\alpha+1/2)}+\CONST_{bias} \dimh^{-(\alpha-1)} \right),\\
\tau(\dimh)&\eqdef &\|\DF_{\dimh}^{-1} \nabla_{\upsilonv\kappav}\E[\LL\big((\Pi_{\dimtotal}\upsilonvs,\lambda\kappavs)-\AF_{\upsilonv\kappav}\big)]  \kappavs\|
    \le \CONST_1\dimh^{-2\alpha+1/2}\sqrt{n},\\
  0&=&\left|{\kappavs}^\T(\HF_{\dimh} -\nabla_{\kappav\kappav}\E\LL(\Pi_{\dimtotal}\upsilonvs,\lambda\kappavs))\kappavs\right|,
\end{EQA}
if \(\CONST_{bias}=0\) one can bound with some \(\CONST>0\)
\begin{EQA}
\beta(\dimh)&\eqdef&\|\DF_{\dimh}^{-1}\AF_{\upsilonv\kappav}\HF_{\dimh} ^{-1}\|\le \CONST \dimh^{-1/2}.
\end{EQA}
Furthermore we find that
\begin{EQA}
\|\DP^2\|&\le& 
		n\frac{p+2}{4}\CONST_{\|\fs\|}\|\psi'\|^2_{\infty}s_{\Xv}^2 \pi^2.
\end{EQA}
\end{lemma}

\begin{proof}
 We have that
\begin{EQA}[c]
 \|\DF_{\dimh}^{-1}\AF_{\upsilonv\kappav}\kappavs\|\le \|\DF_{\dimh}^{-1}\|\|\AF_{\upsilonv\kappav}\kappavs\| .
\end{EQA}
Due to Lemma \ref{lem: D_0 dimh upss is boundedly invertable}
\begin{EQA}[c]
\|\DF_{\dimh}^{-1}\|\le \frac{1}{ c_{\DF}\sqrt n}.
\end{EQA}
And we have by definition that for any \(\upsilonv=(\thetav,\etav)\in W_{S}\times\R^{\dimh}\)
\begin{EQA}[c]
\frac{1}{n}|\upsilonv^\T\AF_{\upsilonv\kappav}\kappavs|\le\frac{1}{n}|\thetav\AF_{\thetav\kappav}\kappavs|+\frac{1}{n}|\etav\AF_{\etav\kappav}\kappavs|.
\end{EQA}
We first analyze the second summand
\begin{EQA}
\frac{1}{n}\etav\AF_{\etav\kappav}\kappavs&=&\sum_{l=\dimh+1}^\infty\eta^*_{l}\sum_{k=1}^{\dimh}
\eta_k\E[\basX_k\basX_l(\Xv^\T\thetavs)].
\end{EQA}
We use \eqref{eq: bound for scalar product basis k l} from Lemma \ref{lem: bounds for basis scalar products} to find
\begin{EQA}
&&\nquad|\frac{1}{n}\etav\AF_{\etav\kappav}\kappavs|\\
	&\le& 17s_{\Xv}^{\dimp+1}L_{p_{\Xv}}\|\psi\|_{\infty}\sum_{l=\dimh+1}^\infty\sum_{k=1}^{\dimh}
|\eta^*_{l}||\eta_k|2^{-j_l-1}2^{j_k/2-j_l/2} 1_{\{I_l\cap I_k\neq \emptyset\}}(k,l).
\end{EQA}
Note that for each \(j_k=0,\ldots,j_{\dimh}\) there exists at most 17 \(r_k(l)\in \{0,\ldots,2^{j_k}17-1\}\) with \(I_l\cap I_k\neq \emptyset\). Remember that \(\dimh=2^{j_m}17-1\) and note that \(2^{j_{\dimh}}\le \dimh\). This implies using the Cauchy-Schwarz inequality and that \(\|\etav\|=1\)
\begin{EQA}
&&\nquad|\frac{1}{n}\etav\AF_{\etav\kappav}\kappavs|\\
	&\le&17s_{\Xv}^{\dimp+1}L_{p_{\Xv}}\|\psi^2\|_{\infty}\sum_{l=\dimh+1}^\infty\sum_{k=1}^{\dimh}|\eta^*_{l}||\eta_k|2^{-j_l-1}2^{j_k/2-j_l/2}1_{\{I_l\cap I_k\neq \emptyset\}}(k,l)\\
	&\le&17s_{\Xv}^{\dimp+1}L_{p_{\Xv}}\|\psi^2\|_{\infty}\sum_{l=\dimh+1}^\infty|\eta^*_{l}|2^{-3j_l/2}\left(\sum_{k=1}^\dimh  2^{j_k}1_{\{I_l\cap I_k\neq \emptyset\}}(k,l) \right)^{1/2}\\
	&\le&17\sqrt{17}s_{\Xv}^{\dimp+1}L_{p_{\Xv}}\|\psi^2\|_{\infty}\sum_{l=\dimh+1}^\infty|\eta^*_{l}|2^{-3j_l/2}\left(\sum_{j_k=0}^{j_{\dimh}-1}2^{j_k}\right)^{1/2}\\
	&\le&17^{3/2}s_{\Xv}^{\dimp+1}L_{p_{\Xv}}\|\psi^2\|_{\infty}\sqrt{\dimh}\left(\sum_{l=\dimh+1}^\infty|\eta^*_{l}|^2\right)^{1/2}\left(\sum_{l=\dimh}^\infty2^{-3j_l}\right)^{1/2}.
\end{EQA}
By assumption \(\mathbf{Cond}_{\upsilonvs}\)
\begin{EQA}[c]
\left(\sum_{l=\dimh+1}^\infty|\eta^*_{l}|^2\right)^{1/2}\le \dimh^{-\alpha}\left(\sum_{l=\dimh+1}^\infty l^{2\alpha}|\eta^*_{l}|^2\right)^{1/2}\le \dimh^{-\alpha}\CONST_{\|\fvs\|}.
\end{EQA}
Since \(\dimh=2^{j_m}17-1\) and \(l=(2^{j_l}-1)17+r_l\) with \(r_l\in\{0,\ldots, 2^{j_l}17-1\}\)
\begin{EQA}
\left(\sum_{l=\dimh+1}^\infty 2^{-3j_l}\right)^{1/2}&=&\left(\sum_{j_l=j_\dimh}^\infty C(\dimh)2^{j_l} 2^{-3j_l}\right)^{1/2}\\
	&=&C(\dimh)^{1/2}2^{-j_{\dimh}}2\le \sqrt{2} C(\dimh)^{3/2}m^{-1},
\end{EQA}
with
\begin{EQA}[c]
C(\dimh)=\frac{2^{j_m}17-1}{2^{j_{\dimh}}}\le 17.
\end{EQA}
Consequently
\begin{EQA}
|\frac{1}{n}\etav\AF_{\etav\kappav}\kappavs|&\le& \sqrt{2}17^{3}\CONST_{\|\fvs\|}s_{\Xv}^{\dimp+1}L_{p_{\Xv}}\|\psi^2\|_{\infty} \dimh^{-\alpha-1/2}.
\end{EQA}
For the second summand we remind the reader that
\begin{EQA}
\AF_{\thetav\kappav}&=&n a_{\thetav\kappav},\\
a_{\thetav\kappav}&=&\E[\fv_{\etavs}'(\Xv^\T\thetavs)\nabla\Phi_{\thetavs}^\T\Xv(\basX_{\dimh+1}(\Xv^\T\thetavs),\ldots)],
\end{EQA}
Similarly to the first summand we get by the dominated convergence theorem
\begin{EQA}
 \thetav a_{\thetav\kappav}\kappavs
 &=&\sum_{k=1}^{\infty}\sum_{l=\dimh+1}^{\infty}\eta^*_k\eta^*_l
 \E[(\Xv^\T\nabla\Phi(\thetavs)\thetav)\basX_k'\basX_l(\Xv^\T\thetavs)]
 1_{\{I_l\cap I_k\neq \emptyset\}}(k,l).
\end{EQA}
To justify the exchange of summation and expectation note that for each \(l\in\N\)
\begin{EQA}
&&\nquad\E[|(\Xv^\T\nabla\Phi(\thetavs)\thetav)\basX_l\fs'_{\etavs}(\Xv^\T\thetavs)|]\\
	&\le& \|\nabla\Phi(\thetavs)\thetav\|s_{\Xv}2^{j_l/2}\E[|\fs'_{\etavs}(\Xv^\T\thetavs)|]\\
	&\le&\|\nabla\Phi(\thetavs)\thetav\|s_{\Xv}2^{j_l/2}\E\left[\left| \sum_{k=1}^{\infty} \eta^*_k \basX'_k(\Xv^\T\thetavs)\right|\right]\\
	&\le&\|\nabla\Phi(\thetavs)\thetav\|s_{\Xv}2^{j_l/2} \left(\sum_{k=1}^{\infty} l^{2\alpha}{\eta^*_k}^2\right)^{1/2} \left(\sum_{k=1}^{\infty}l^{-2\alpha}2^{3j_k}\|\psi'\|^2\right)^{1/2} \\
	&\le&\|\nabla\Phi(\thetavs)\thetav\|s_{\Xv}\CONST_{\|\fvs\|}\|\psi'\|_{\infty} 2^{j_l/2} \left(\frac{17}{2}\sum_{j=0}^{\infty}l^{-2\alpha}2^{4j}\right)^{1/2} <\infty.
\end{EQA}
The exchange of the order of summation is justified by the subsequent bounds and again the dominated convergence theorem.
We again use Lemma \ref{lem: bounds for basis scalar products} to find with \eqref{eq: bound for scalar product basis derivative k basis l and X theta} and with similar arguments to those from above
\begin{EQA}
|\thetav a_{\thetav\kappav}\kappavs|&\le&17\frac{\sqrt{p+2}}{2} \pi\|\psi'\|_\infty s^2_{\Xv}\|p_{\Xv}\|_{\infty}\\
		&&\sum_{k=1}^{\infty}\eta^*_k 2^{3j_k/2}\sum_{l=\dimh+1}^{\infty}\eta^*_l2^{-(j_l\vee j_k)/2}1_{\{I_l\cap I_k\neq \emptyset\}}(k,l)\\
	&\le& 17\frac{\sqrt{p+2}}{2} \pi\|\psi'\|_\infty s^2_{\Xv}\|p_{\Xv}\|_{\infty}\sum_{k=1}^{\infty}\eta^*_k k^{3/2}\left(\sum_{l=\dimh+1}^{\infty}l^{2\alpha}{\eta^*_l}^2\right)^{1/2}\\
		&&\left(\sum_{j_l=j_{\dimh}+1}^{\infty}\sum_{r_l=0}^{2^{j_l}17-1}2^{-2\alpha j_l}2^{-(j_l\vee j_k)}1_{\{I_l\cap I_k\neq \emptyset\}}(k,l)\right)^{1/2}.
\end{EQA}
We have due to \eqref{eq: number of intersections} that
\begin{EQA}
&&\nquad\sum_{j_l=j_{\dimh}+1}^{\infty}\sum_{r_l=0}^{2^{j_l}17-1}2^{-2\alpha j_l}2^{-(j_l\vee j_k)}1_{\{I_l\cap I_k\neq \emptyset\}}(k,l)\\
	&=&\sum_{j_l=j_{\dimh}+1}^{\infty}2^{-2\alpha j_l}2^{-(j_l\vee j_k)}\sum_{r_l=0}^{2^{j_l}17-1}1_{\{I_l\cap I_k\neq \emptyset\}}(k,l)\\
	&=&\sum_{j_l=j_{\dimh}+1}^{\infty}2^{-2\alpha j_l}2^{-(j_l\vee j_k)}\\
		&&\Big|\Big\{l=(2^{j_l}-1)17+r_l\Big|\,r_l\in\{0,\ldots, 2^{j_l}17-1 \},\, I_l\cap I_k\neq \emptyset\Big\}\Big|\\
	&=&\sum_{j_l=j_{\dimh}+1}^{\infty}2^{-(2\alpha +1)j_l}2^{-(j_k-j_l)_+}\left\ulcorner2^{(j_l-j_k)}17\right\urcorner \\
	&\le& 2^{-(2\alpha+1)j_{\dimh}}18 \le 17\dimh^{-(2\alpha+1)}18.
\end{EQA}
Which gives
\begin{EQA}
|\thetav a_{\thetav\kappav}\kappavs|	&\le& 17^{3/2}\sqrt{18}\frac{\sqrt{p+2}}{2} \pi\|\psi'\|_\infty s^2_{\Xv}\|p_{\Xv}\|_{\infty}\CONST_{\|\fvs\|}\dimh^{-\alpha-1/2}\\
	&&\left(\sum_{k=1}^{\infty}{\eta^*_k}^2 k^{2\alpha}\right)^{1/2}
		\left(\sum_{k=1}^{\infty} k^{-(2\alpha-3)}\right)^{1/2}\\
	&\le&17^{3/2}\sqrt{18}\frac{\sqrt{p+2}}{2} \pi\|\psi'\|_\infty s^2_{\Xv}\|p_{\Xv}\|_{\infty} \CONST_{\|\fvs\|}^2\\
		&&\sqrt{(2\alpha-3)/(2\alpha-4)}\dimh^{-(\alpha+1/2)},
\end{EQA}
since \(\alpha>2\) such that \(\sum_{k=1}^{\infty} k^{-(2\alpha-3)}< (2\alpha-3)/(2\alpha-4)\).

Furthermore with \(\thetavd=\nabla\Phi_{\thetavs}\thetav\in{\thetavs}^\perp\)
\begin{EQA}
\left|\thetav b_{\thetav\kappav}\kappavs \right|&=& \left|\E\left[(\fv_{\etavs}(\Xv^\T\thetavs)-g(\Xv))\Xv^\T\thetavd\sum_{k=\dimh}^{\infty}\etas_k\basX_k'(\Xv^\T\thetavs)\right]\right|\\
	&\le& \CONST_{bias} \frac{\sqrt{p+2}}{2}\pi s_{\Xv}\E[\left|\fv_{\kappavs}'(\Xv^\T \thetavs)\right|].
\end{EQA}
\begin{remark}
If \(\Xv^\T\thetavs\) was independent to \(\Xv^\T\thetavd\) for any \(\thetavd\in{\thetavs}^\perp\), we would have \(\thetav b_{\thetav\kappav}\kappavs=0\) by the definition of \(\fv_{\etavs}(\Xv^\T\thetavs)\eqdef\E[g(\Xv)|\Xv^\T\thetavs]\). 
\end{remark}
We bound
\begin{EQA}
\E[\left|\fv_{\kappavs}'(\Xv^\T \thetavs)\right|]&\le& \sqrt{17} \|\psi'\|_{\infty}\left(\sum_{k=\dimh}^{\infty}{\etas_k}^2k^{2\alpha}  \right)\left(\sum_{j=j_{\dimh}+1}^{\infty}2^{-(2\alpha-3)j} \right)\\
	&\le& \CONST(\dimh)\CONST_{\|\etavs\|} \sqrt{17} \|\psi'\|_{\infty}\dimh^{-(\alpha-3/2)}<\infty.
\end{EQA}
We can exchange summation and expectation to find
\begin{EQA}[c]
\E[\left|\fv_{\kappavs}'(\Xv^\T \thetavs)\right|]=\sum_{k=\dimh}^{\infty}\etas_k\E[\left|\basX_k'(\Xv^\T\thetavs)\right|].
\end{EQA}
We estimate
\begin{EQA}
\E[\left|\basX_k'(\Xv^\T\thetavs)\right|]&=&\int_{\R}\left|\basX_k'(x)p_{\Xv^\T\thetavs}(x)\right|dx\\
	&\le &\left(\int_{I_k}\basX_k'(x)^2dx\right)^{1/2}\left(\int_{I_k}p_{\Xv^\T\thetavs}^2(x)dx\right)^{1/2}\\
	&\le & \|\psi'\|_{\infty}\CONST_{d} 2^{j_k/2}.
\end{EQA}
Such that
\begin{EQA}
\E[\left|\fv_{\kappavs}'(\Xv^\T \thetavs)\right|]&\le&\CONST(\dimh)\CONST_{\|\etavs\|} \CONST_{d} \|\psi'\|_{\infty}\sum_{k=\dimh}^{\infty}2^{j_k/2}\etas_k\\
	&\le&\CONST(\dimh)\CONST_{\|\etavs\|} \CONST_{d} \|\psi'\|_{\infty}\left(\sum_{j=j_\dimh+1}^{\infty}2^{-2(\alpha-1)j}\right)^{1/2}\\
	&\le& \CONST(\dimh)\CONST_{\|\etavs\|} \CONST_{d} \|\psi'\|_{\infty}\dimh^{-(\alpha-1)}.
\end{EQA}

Collecting both summands 
\begin{EQA}[c]
\|\DF_{\dimh}^{-1}\AF_{\upsilonv\kappav}\kappavs\|\le \CONST\left(\sqrt n\dimh^{-(\alpha+1/2)}+\CONST_{bias} \dimh^{-(\alpha-1)} \right).
\end{EQA}
with some \(\CONST>0\). 
The same arguments give for the case \(\CONST_{bias}=0\)
\begin{EQA}
\|\DF_{\dimh}^{-1}\AF_{\upsilonv\kappav}\HF_{\dimh} ^{-1}\|&\le& \frac{1}{c_{\DF}^2} \left(\sup_{\|\thetav\|=1,\, \|\kappav\|_{l^2}=1}\frac{1}{n}|\thetav\AF_{\thetav\kappav}\kappav|+ \sup_{\|\etav\|=1,\, \|\kappav\|_{l^2}=1}\frac{1}{n}|\etav\AF_{\etav\kappav}\kappav|\right)\\
	&\le& \frac{\CONST_1}{c_{\DF}^2}2\dimh^{-1/2}.
\end{EQA}

\begin{remark}
In case \(\CONST_{bias}>0\) we do not manage to get a bound for \(\thetav b_{\thetav\kappav}\kappav\) for general \(\kappav\in l^2\). How to get a bound for \(\beta(\dimh)\) in this setting remains unclear.
\end{remark}

We bound using the dominated convergence theorem (applicable due to similar bounds as above)
\begin{EQA}[c]
\label{eq: bound for DF zeta, first inequality}
  \|\HF_{\dimh}  \kappavs \|^2\le n \sum_{k=\dimh+1}^{\infty}{\eta^*_k}^2\|p_{\Xv^\T\thetavs}\|_\infty+2n\left|\sum_{l>k}\eta^*_l\eta^*_k\E[\basX_k\basX_l(\Xv^\T\thetavs)]\right|.
\end{EQA}
As above we find
\begin{EQA}[c]
|\E[\basX_k\basX_l(\Xv^\T\thetavs)]|\le 17 s_{\Xv}^{\dimp+1}L_{p_{\Xv}}\|\psi\|_{\infty} 2^{-3j_l/2-1}2^{j_k/2} 1_{\{I_l\cap I_k\neq \emptyset\}}(k,l).
\end{EQA}
We estimate
\begin{EQA}
&&\nquad\sum_{l>k>\dimh}\eta^*_l\eta^*_k\E[\basX_k\basX_l(\Xv^\T\thetavs)]\\
	&\le&17 s_{\Xv}^{\dimp+1}L_{p_{\Xv}}\|\psi\|_{\infty}  \sum_{l>k}\eta^*_l\eta^*_k2^{-3j_l/2-1}2^{j_k/2} 1_{\{I_l\cap I_k\neq \emptyset\}}(k,l)\\
	&\le&17 s_{\Xv}^{\dimp+1}L_{p_{\Xv}}\|\psi\|_{\infty} \sum_{k=1}^\infty\eta^*_k2^{j_k/2} \sum_{l=k+1}^\infty\eta^*_l2^{-3j_l/2-1}1_{\{I_l\cap I_k\neq \emptyset\}}(k,l)\\
	&\le&17 s_{\Xv}^{\dimp+1}L_{p_{\Xv}}\|\psi\|_{\infty}  \sum_{k=1}^\infty\eta^*_k2^{j_k/2} \left(\sum_{l=k+1}^\infty{\eta^*_l}^2l^{2\alpha}\right)^{1/2}\\
	&&\left(\sum_{l=k+1}^\infty l^{-2\alpha}2^{-3j_l}1_{\{I_l\cap I_k\neq \emptyset\}}(k,l)\right)^{1/2}.
\end{EQA}
We continue using that \(l\ge 2^{j_l}\)
\begin{EQA}
&&\nquad\sum_{l=k+1}^\infty l^{-2\alpha}2^{-3j_l}1_{\{I_l\cap I_k\neq \emptyset\}}(k,l)\\
	&\le&\sum_{l=k+1}^\infty 2^{-(3+2\alpha)j_l}1_{\{I_l\cap I_k\neq \emptyset\}}(k,l)\\
	&\le&\sum_{j=j_{k+1}}^\infty 2^{-(3+2\alpha)j}| \{l=(2^j-1)17,\ldots,(2^{j+1}-1)17-1:\,I_l\cap I_k\neq \emptyset\}|\\
	&=&\sum_{j=j_{k+1}}^\infty 2^{-(3+2\alpha)j}\ulcorner 2^{j-j_k}17 \urcorner\\
	&\le &2^{-j_k}18\sum_{j=j_{k+1}}^\infty 2^{-(2+2\alpha)j}\\
	&=&2^{-(3+2\alpha)j_k}18\sum_{j=0}^\infty 2^{-(2+2\alpha)j}\le   2^{-(3+2\alpha)j_k}36.
\end{EQA}
Plugging this in we find
\begin{EQA}
&&\nquad\sum_{l>k>\dimh}\eta^*_l\eta^*_k\E[\basX_k\basX_l(\Xv^\T\thetavs)]\\
	&\le&17 \sqrt{36}s_{\Xv}^{\dimp+1}L_{p_{\Xv}}\|\psi\|_{\infty} \sum_{k=\dimh+1}^\infty\eta^*_k 2^{-(2+2\alpha)j_k/2}\CONST_{\|\fvs\|}\\
	&\le&17 \sqrt{36}s_{\Xv}^{\dimp+1}L_{p_{\Xv}}\|\psi\|_{\infty}\CONST_{\|\fvs\|} \left(\sum_{k=\dimh+1}^\infty{\eta^*_k}^2 k^{2\alpha}\right)^{1/2}\\
		&&\phantom{17 \sqrt{36}s_{\Xv}^{\dimp+1}L_{p_{\Xv}}\|\psi\|_{\infty}}\left(\sum_{k=\dimh+1}^\infty k^{-2\alpha} 2^{-(2+2\alpha)j_k}\right)^{1/2}\\
	&\le&17 \sqrt{36}s_{\Xv}^{\dimp+1}L_{p_{\Xv}}\|\psi\|_{\infty}\CONST_{\|\fvs\|}^2 
	\left(\sum_{k=\dimh+1}^\infty 2^{-(2+4\alpha)j_k}\right)^{1/2}\\
	&\le&17 \sqrt{36}s_{\Xv}^{\dimp+1}L_{p_{\Xv}}\|\psi\|_{\infty}\CONST_{\|\fvs\|}^2 \left(\sum_{j=j_{\dimh}}^\infty 2^{-(1+4\alpha)j_k}\right)^{1/2}\\
	&\le&17 \sqrt{36}s_{\Xv}^{\dimp+1}L_{p_{\Xv}}\|\psi\|_{\infty}\CONST_{\|\fvs\|}^2 2^{-(1+4\alpha)(j_{\dimh})/2}\left(\sum_{j=0}^\infty 2^{-(1+4\alpha)j_k}\right)^{1/2}.
\end{EQA}
From which we obtain
\begin{EQA}
  \|\HF_{\dimh}  \kappavp \|^2&=& n \sum_{k=\dimh+1}^{\infty}{\eta^*_k}^2\|p_{\Xv^\T\thetavs}\|_\infty+2s_{\Xv}^{\dimp}L_{p_{\Xv}}\|\psi\|_{\infty}\CONST_{\|\fvs\|}^2 n 2^{-(1+4\alpha)(j_{\dimh}+1)/2}\\
  	&\le&\|p_{\Xv^\T\thetavs}\|_\infty n \dimh^{-(1+2\alpha)}\dimh\left(\sum_{k=\dimh+1}^{\infty}{\eta^*_k}^2k^{2\alpha}\right)\\
  		&&+ 17^2 \sqrt{36}s_{\Xv}^{\dimp+1}L_{p_{\Xv}}\|\psi\|_{\infty}\CONST_{\|\fvs\|}^2 n \dimh^{-(1/2+2\alpha)}\\
  	&\le& \left(17\|p_{\Xv^\T\thetavs}\|_\infty\CONST_{\|\fvs\|}+17^2 \sqrt{36}s_{\Xv}^{\dimp+1}L_{p_{\Xv}}\|\psi\|_{\infty}\CONST_{\|\fvs\|}^2\right) n\dimh^{-(1+2\alpha)} \dimh.
\end{EQA}

Next we show
\begin{EQA}
    \|\DF_{\dimh}^{-1} \left(\nabla_{\upsilonv\kappav}\E[\LL\big((\Pi_{\dimtotal}\upsilonvs,\lambda\kappavs)\big)] -\AF_{\upsilonv\kappav}\right) \kappavs\|
    &\le& 
    \tau(\dimh). 
\label{eq: new condition for bias}
\end{EQA} 
For this note that
\begin{EQA}
&& \nquad\left(\nabla_{\upsilonv\kappav}\E[\LL\big((\Pi_{\dimtotal}\upsilonvs,\lambda\kappavs)\big)] -\AF_{\upsilonv\kappav}\right) \kappavs\\
	&=&n \left(\begin{array}{l}  \E[\fv_{(0,\lambda\kappavs)}'\Xv \fv_{(0,\kappavs)}(\Xv^\T\thetavs)] \\  \E[\basX \fv_{(0,\kappavs)}(\Xv^\T\thetavs)]\end{array} \right)\\
		&&+ n\left(\begin{array}{ll} \E[\fv_{(0,\kappavs)}'\Xv \fv_{(0,\lambda\kappavs)}(\Xv^\T\thetavs)] \\ 0 \end{array}  \right) .
\end{EQA}
We infer
\begin{EQA}
&& \nquad\|\DF_{\dimh}^{-1} \left(\nabla_{\upsilonv\kappav}\E[\LL\big((\Pi_{\dimtotal}\upsilonvs,\lambda\kappavs)\big)] -\AF_{\upsilonv\kappav}\right) \kappavs\|\\
	&\le& n\E[\fv^2_{(0,\kappavs)}(\Xv^\T\thetavs)]^{1/2}\Bigg(\E\left[\left\|\DF_{\dimh}^{-1}\left(\begin{array}{l}  \fv_{(0,\lambda\kappavs)}'\Xv \\  \basX \end{array} \right)\right\|^2\right]^{1/2}\\
	&&+ \E\left[\left\|\DF_{\dimh}^{-1}\left(\begin{array}{ll} \fv_{(0,\kappavs)}'\Xv \\ 0 \end{array}  \right)\right\|^2\right]^{1/2}\Bigg)\\
	&\le& \frac{\sqrt n}{ c_{\DF}}\left(s_{\Xv}\left\{\E[{\fv'_{(0,\lambda\kappavs)}}^2]^{1/2}+\E[{\fv'_{(0,\kappavs)}}^2]^{1/2}\right\}+\|p_{\Xv^\T\thetav}\|^{1/2}17^{1/4}\sqrt{m}\right)\\
		&&\E[\fv^2_{(0,\kappavs)}(\Xv^\T\thetavs)]^{1/2}.
\end{EQA}
We estimate separately using the same bounds as before to apply the dominated convergence theorem to exchange summation and expectation. We bound as above using \eqref{eq: bound for scaler product basis derivatives l k}
\begin{EQA}
&&\nquad\E[{\fv'_{(0,\lambda\kappavs)}}^2]\\
	&=&\lambda \sum_{k,l=\dimh+1}^{\infty}\eta^*_k\eta^*_l \E[\basX_l'\basX'_k(\Xv^\T\thetavs)]\\
	&\le& 17 s_{\Xv}\|\psi'\|_{\infty}\|p_{\Xv}\|_{\infty}\sum_{k,l=\dimh+1}^{\infty}\eta^*_k\eta^*_l 2^{3(j_l+j_k)/2-(j_l\vee j_k)}1_{\{I_k\cap I_l\neq \emptyset\}}(k,l)\\
	&\le&17 s_{\Xv}\|\psi'\|_{\infty}\|p_{\Xv}\|_{\infty}\sum_{k=\dimh+1}^{\infty}\eta^*_k 2^{3j_k/2} \sum_{l=\dimh+1}^{\infty}\eta^*_l 2^{3j_l/2-(j_l\vee j_k)}1_{\{I_k\cap I_l\neq \emptyset\}}(k,l)\\
	&\le& 17 s_{\Xv}\|\psi'\|_{\infty}\|p_{\Xv}\|_{\infty}\sum_{k=\dimh+1}^{\infty}\eta^*_k 2^{3j_k/2} \left(\sum_{l=\dimh+1}^{\infty}l^{2\alpha}{\fs^*}_l^2\right)^{1/2}\\
		&&\left(\sum_{l=\dimh+1}^{\infty} 2^{(3-2\alpha)2j_l-2(j_l\vee j_k)} 1_{\{I_k\cap I_l\neq \emptyset\}}(k,l)\right)^{1/2}.
\end{EQA}
Observe
\begin{EQA}
&& \nquad\sum_{l=\dimh+1}^{\infty} 2^{(3-2\alpha)j_l-2(j_l\vee j_k)} 1_{\{I_k\cap I_l\neq \emptyset\}}(k,l)\\
	&=&\sum_{j=j_{\dimh}+1}^{\infty} 2^{(3-2\alpha)j-2(j\vee j_k)}\\
		&& \Big|\Big\{l=j 12+2^{j}+r_l\Big|\,r_l\in\{0,\ldots, 2^{j}+11 \},\, I_l\cap I_k\neq \emptyset\Big\}\Big|\\
	&=&\sum_{j=j_{\dimh}+1}^{\infty} 2^{(3-2\alpha)j-2(j\vee j_k)}\left\ulcorner2^{(j-j_k)}17\right\urcorner\\
	& \le& 18\sum_{j=j_{\dimh}+1}^{\infty} 2^{(2-2\alpha)j} =17^3 18 \dimh^{-2\alpha+2}.
\end{EQA}
Such that again using the Cauchy-Schwarz inequality for any \(\lambda\in [0,1]\)
\begin{EQA}
\E[{\fv'_{(0,\lambda\kappavs)}}^2]&\le&  17^{5/2}\sqrt{18}s_{\Xv}\|\psi'\|_{\infty}\|p_{\Xv}\|_{\infty}\CONST_{\|\fvs\|}\dimh^{-\alpha+1}\sum_{k=\dimh+1}^{\infty}\eta^*_k 2^{3j_k/2}\\
	&\le& 17^3\sqrt{18}s_{\Xv}\|\psi'\|_{\infty}\|p_{\Xv}\|_{\infty}\CONST^2_{\|\fvs\|}\dimh^{-2\alpha+3}.
\end{EQA}
Furthermore
\begin{EQA}
&&\nquad\E[\fv^2_{(0,\kappavs)}(\Xv^\T\thetavs)]= \sum_{k,l=\dimh+1}^{\infty}\eta^*_k\eta^*_l\E[\basX_k\basX_l(\Xv^\T\thetavs)]\\
	&\le& 17s_{\Xv}^{\dimp+1}L_{p_{\Xv}}\|\psi\|_{\infty}\sum_{k,l=\dimh+1}^{\infty}\eta^*_k\eta^*_l 2^{-3(j_l\vee j_k)/2+(j_l\wedge j_k)/2} 1_{\{I_l\cap I_k\neq \emptyset\}}(k,l)\\
	&=& 17s_{\Xv}^{\dimp+1}L_{p_{\Xv}}\|\psi\|_{\infty}\sum_{k=\dimh+1}^{\infty}\eta^*_k 2^{-j_k}\sum_{l=\dimh+1}^{\infty}\eta^*_l 1_{\{I_l\cap I_k\neq \emptyset\}}(k,l)\\
	&\le&17s_{\Xv}^{\dimp+1}L_{p_{\Xv}}\|\psi\|_{\infty}\sum_{k=\dimh+1}^{\infty}\eta^*_k 2^{-j_k}\left(\sum_{l=\dimh+1}^{\infty}l^{-2\alpha}{\eta^*_l}^2\right)^{1/2}\\
		&&\left(\sum_{l=\dimh+1}^{\infty}2^{-2\alpha j_l} 1_{\{I_l\cap I_k\neq \emptyset\}}(k,l)\right)^{1/2}\\
	&\le& 17s_{\Xv}^{\dimp+1}L_{p_{\Xv}}\|\psi\|_{\infty}\sum_{k=\dimh+1}^{\infty}\eta^*_k 2^{-j_k}\CONST_{\|\fvs\|}\left(\sum_{j=j_{\dimh}+1}^{\infty}2^{-2\alpha j}18\right)^{1/2}\\
	&\le& 17\sqrt{18}17^{1/2}s_{\Xv}^{\dimp+1}L_{p_{\Xv}}\|\psi\|_{\infty}\CONST_{\|\fvs\|}\dimh^{-\alpha}\sum_{k=\dimh+1}^{\infty}\eta^*_k 2^{-j_k}\\
	&\le& 17\sqrt{36}17^{2}s_{\Xv}^{\dimp+1}L_{p_{\Xv}}\|\psi\|_{\infty}\CONST_{\|\fvs\|}^2\dimh^{-2\alpha}.
\end{EQA}
Together this implies
\begin{EQA}
&& \nquad\|\DF_{\dimh}^{-1} \left(\nabla_{\upsilonv\kappav}\E[\LL\big((\Pi_{\dimtotal}\upsilonvs,\lambda\kappavs)\big)] -\AF_{\upsilonv\kappav}\right) \kappavs\|\\
	&\le& \frac{1}{ c_{\DF}}\left(2s_{\Xv}\left\{17^3\sqrt{18}s_{\Xv}\|\psi'\|_{\infty}\|p_{\Xv}\|_{\infty}\CONST^2_{\|\fvs\|}\right\}^{1/2}+\|p_{\Xv^\T\thetav}\|^{1/2}17^{1/4}\right)\\
		&&\sqrt{\sqrt{36}17^{3}s_{\Xv}^{\dimp+1}L_{p_{\Xv}}\|\psi\|_{\infty}\CONST_{\|\fvs\|}^2}\dimh^{-2\alpha+1/2)}\sqrt{n}\\
	&\le& \CONST_1\dimh^{-2\alpha+1/2}\sqrt{n}.\lceil
\end{EQA}
Clearly
\begin{EQA}[c]
\left|{\kappavs}^\T(\HF_{\dimh} -\nabla_{\kappav\kappav}\E\LL(\Pi_{\dimtotal}\upsilonvs,\lambda\kappavs))\kappavs\right|=0. 
\end{EQA}
To see this simply note that for any \(\fv\in \mathcal S\) and any \(\kappav\in \mathcal S\)
\begin{EQA}[c]
\kappav^\T \nabla_{\kappav\kappav}\E\LL(\thetavs,\fv)\kappav=\E[\fv_{(0,\kappav)}^2(\Xv^\T\thetavs)]=\kappav^\T\HF_{\dimh} \kappav.
\end{EQA}
Furthermore we find that
\begin{EQA}
\thetav^\T d_{\thetav}^2(\upss)\thetav&=&\E[\fs'_{\fs}(\Xv^\T\thetavs)^2(\Xv^\T\nabla\Phi(\thetavs)\thetav)^2]\\
	&\le& \|\fv_{\etavs}'\|^2_{\infty}s_{\Xv}^2\|\nabla\Phi(\thetavs)\|\\
	&\le& \frac{p+2}{4}\CONST_{\|\fs\|}\|\psi'\|^2_{\infty}s_{\Xv}^2 \pi^2,
\end{EQA}
and
\begin{EQA}
\thetav^\T v_{\thetav}^2(\upss)\thetav&=&\E\left[\left(\fv_{\etavs}(\Xv^{\T}\thetavs)- g(\Xv)\right)\left( 
      (\Xv^\T\thetav)^2 \fv_{\etavs}''(\Xv^\T\thetavs)] \right.\right.\\
      &&\left.\left.+|\fv'_{\etavs}(\Xv^{\T}\thetavs)|^2\Xv^{\T}\nabla^2\varphi_{\thetavs}^{\T}[\Xv,\thetav,\thetav]\right)\right]\\
      &\le& \CONST_{bias}\left(s^2_{\Xv}\CONST_{\|\fv_{\etavs}''\|_{\infty}} +34\|\psi'\|^2_{\infty}C^2_{\|\etavs\|}s_{\Xv}^2\|\nabla^2\varphi_{\thetavs}\|_{\infty}\right).
\end{EQA}
This completes the proof.
\end{proof}
\subsection{Proof or Lemma \ref{lem: in example a priori distance of target to oracle}}

Remember the representation the full operator \(\DF\in L(l^2,\{(x_k)_{k\in\N}, x\in\R\})\) in block form
\begin{EQA}[c]
\DF^2(\upsilonvs)=\left( 
      \begin{array}{cc}
        \DP^{2} & \AF \\
       \AF & \HF^2
      \end{array}  
    \right)=\left( 
      \begin{array}{cc}
        \DF_{\dimh}^{2} & \AF_{\upsilonv\kappav} \\
       \AF_{\upsilonv\kappav} & \HF^2_{\kappav\kappav}
      \end{array}  
    \right)= \DF_{\dimh}^{2}=\left( 
      \begin{array}{ccc}
        \DP^{2} & \AF_{\dimh} & \AF_{\thetav\kappav} \\
       \AF_{\dimh} & \HF_{\dimh} &  \AF_{\etav\kappav} \\
       \AF_{\etav\kappav} & \AF_{\thetav\kappav} & \HF^2_{\kappav\kappav}
      \end{array}  
    \right).
\end{EQA}

We proof the claim in two lemmas. The first one concerns condition \( {(\cc{L}{\rr}_{\infty})} \) from Section \ref{sec: synapsis sieve tools}. For this condition we can use the full expectation \(\E\) instead of \(\E_{\varepsilon}\):

\begin{lemma}\label{lem: cond Lr infty}
Assume \((\mathbf{\mathcal A})\). Then there exists a constant \(\gmi>0\) such that
\begin{EQA}[c]
\E\LL(\ups,\upss)\le -\gmi \rr^2.
\end{EQA}
\end{lemma}

\begin{proof}
As in Lemma \ref{lem: D_0 dimh is boundedly invertable} we can make the decomposition
\begin{EQA}
&&\nquad\E\LL(\ups,\upss)\\
	&=&-n\E\left[\left(g(\Xv)-\E[g(\Xv)|\Xv^\T\thetav]\right)^2\right]-n\E\left[\left(g(\Xv)-\E[g(\Xv)|\Xv^\T\thetavs]\right)^2\right]\\
	&&-n\E\left[\left(\E[g(\Xv)|\Xv^\T\thetav]-\sum_{k=1}^{n}\eta_k\basX_k(\Xv^\T\thetav)\right)^2\right].
\end{EQA}
We find with condition (model bias) for all \(\ups=(\thetav,\etav)\in\Ups_{\dimh}\)
\begin{EQA}
&&\nquad-n\E\left[\left(g(\Xv)-\E[g(\Xv)|\Xv^\T\thetav]\right)^2\right]+n\E\left[\left(g(\Xv)-\E[g(\Xv)|\Xv^\T\thetavs]\right)^2\right]\\
	&\le& \begin{cases} -n\gmi_{\theta}, & \|\DP(\thetav-\thetavs)\|\ge \sqrt{n}\rr_{\thetav}/c_{\DF}\\
	 -\gmi_{\theta}\|\DP(\thetav-\thetavs)\|^2,& otherwise.\end{cases}
\end{EQA}
As \(\|\DP(\thetav-\thetavs)\|^2\le n \frac{p+2}{2}\CONST_{\|\fs\|}\|\psi'\|^2_{\infty}s_{\Xv}^2 \pi^2\) we find 
\begin{EQA}
\E\LL(\ups,\upss)\le-\gmi_{\theta}''\|\DP(\thetav-\thetavs)\|^2, &\quad \gmi_{\theta}''=\gmi_{\theta}\min\left\{1 , \frac{1}{\frac{p+2}{2}\CONST_{\|\fs\|}\|\psi'\|^2_{\infty}s_{\Xv}^2 \pi^2}\right\}.
\end{EQA}
We study two cases first assume that \(\|\DP(\thetav-\thetavs)\|^2\ge \tau^2 \rr^2\) for some \(\tau>0\), then we get
\begin{EQA}[c]\label{eq: Lr in bias case for theta component}
-\E\LL(\ups,\upss)\ge \tau^2\gmi_{\theta}'\rr^2.
\end{EQA}
Otherwise - if \(\|\DP(\thetav-\thetavs)\|^2\le \tau^2 \rr^2\) - we have as in the proof of Lemma \ref{lem: D_0 dimh upss is boundedly invertable}
\begin{EQA}[c]
\E\left[\left(\E[g(\Xv)|\Xv^\T\thetav]-\sum_{k=1}^{n}\eta_k\basX_k(\Xv^\T\thetav)\right)^2\right]\ge Q 	(\gmi^*-\CONST\tau)^2\rr^2.
\end{EQA}
Choosing \(\tau>0\) small enough gives the claim.

\end{proof}

The claim of Lemma \ref{lem: in example a priori distance of target to oracle} now is a direct consequence of Lemma A.2 of \cite{AAbias2014}.

%
%
%
%
%
%
%
%
%
%
\subsection{Proof of Lemma \ref{lem: a priori a priori accuracy for rr circ}}

\begin{remark}\label{rem: lower bound of HH}
We assume that the density of the regressors satisfies \(p_{\Xv}\ge c_{p_{\Xv}}>0\) on \(B_{s_{\Xv}+c_{B}}(0)\). This implies that for any \(\thetav\in\R^{\dimp}\) the density of \(\Xv^\T\thetav\) is also bounded away from zero on \([-s_{\Xv},s_{\Xv}]\) by \(\lambda(B^{\dimp-1}_{c_{B}})c_{p_{\Xv}}
\) where \(\lambda(B^{\dimp-1}_{\rr})\) denotes the Lebesgue measure of the \(\dimp-1\) dimensional ball of radius \(\rr>0\) on \(\R^{\dimp-1}\). As we use a orthonormal wavelet basis on \(L^2([-s_{\xv},s_{\xv}])\) this gives
\begin{EQA}
\lambda_{min}(\HF^2(\ups))&=& \inf_{\etav\in l^2} \E[\fv_{\etav}(\Xv^\T\thetav)^2]/\|\etav\|^2\\
	&\ge& \lambda(B^{\dimp-1}_{c_{B}})c_{p_{\Xv}}\int_{[-s_{\Xv},s_{\Xv}]}\fv_{\etav}(x)^2 dx/\|\etav\|^2=\lambda(B^{\dimp-1}_{c_{B}})c_{p_{\Xv}}.
\end{EQA}
\end{remark}

\begin{proof}
Take any \(\thetav\in S^{\dimp,+}_{1}\). Then we have due to the quadratic structure of the problem and using the usual bounds for \(\|\basX\|\le \CONST\sqrt{\dimh}\)
\begin{EQA}
\left\|\tilde\etav^{(\infty)}_{\dimh,\thetav}\right\|&\eqdef& \left\|\argmax_{\etav\in \R^{\dimh}}\LL_{\dimh}(\thetav,\etav)\right\|\\
	&=& \left\|\left(\frac{1}{n}\sum_{i=1}^{n}   \basX\basX^\T (\Xv^\T_i\thetav)\right)^{-1}\frac{1}{n}\sum_{i=1}^{n} (g(\Xv_i)+\varepsilon_i)\basX(\Xv^\T_i\thetav)\right\|\\
	&\le& \left(\|g\|_{\infty}\CONST\sqrt{\dimh}+\left\|\frac{1}{n}\sum_{i=1}^{n}\varepsilon_i \basX(\Xv^\T_i\thetav)\right\| \right)\\
	&&	\left\| \left(\frac{1}{n}\sum_{i=1}^{n}   \basX\basX^\T (\Xv^\T_i\thetav)\right)^{-1}\right\|.\label{eq: to bound in a priori accuracy}
\end{EQA}
We want to bound the above right-hand side. For this we bound
\begin{EQA}
\P\left(\left\|\frac{1}{n}\sum_{i=1}^{n} \varepsilon_i\basX(\Xv^\T_i\thetav)\right\| \ge  t \right)&= &\P\left(\sup_{\substack{\etav\in\R^{\dimh}\\ \|\etav\|=1}} \frac{1}{n}\sum_{i=1}^{n} \varepsilon_i \sum_{k=1}^{\dimh}\eta_k\basX_k(\Xv^\T_i\thetav) \ge  t \right)\\
&\le&\P\left(\sup_{\substack{\etav\in B_1(0)}} \frac{1}{n}\sum_{i=1}^{n} \varepsilon_i\fv_{\etav} (\Xv^\T_i\thetav)\ge  t \right)
\end{EQA}
We want to apply Corollary 2.2 of the supplement of \cite{SP2011} with 
\begin{EQA}[c]
\UP(\etav)=\frac{1}{\sqrt{n}}\sum_{i=1}^{n} \varepsilon_i\fv_{\etav} (\Xv^\T_i\thetav)\), \(\upss=0\in\R^{\dimh}.
\end{EQA}
For this we have to show that 
\begin{EQA}
\log \E \exp \biggl\{ 
        \lambda \frac{\UP(\upsilonv) - \UP(\upsilonvd)}{d(\ups,\upsd)} 
    \biggr\}&\le& \nu^2\lambda^2/2,
\end{EQA}
with \(d(\etav,\etavd)=\|\etav-\etavd\|_{\R^{\dimh}}\). This is indeed the case since by Lemma \ref{lem: bounds for objects} for any pair \(\etav,\etavd \in B_1(0)\)
\begin{EQA}[c]
 \left|\fv_{\etav-\etavd}(\Xv^\T_i\thetav)\right|\le {\CONST}\|\psi\|_{\infty}\sqrt{\dimh}\|\etav-\etavd\|.
\end{EQA}
Using \((\mathbf{Cond}_{\varepsilon})\), the independence of \((\varepsilon_i)\) and \((\Xv_i)\) we find for
\begin{EQA}[c]
\lambda\le \frac{\sqrt{n}}{\CONST\sqrt{\dimh}}\tilde \gm,
\end{EQA}
and any pair \(\etav,\etavd \in B_1(0)\)
\begin{EQA}
&&\nquad \log \E \exp \biggl\{ 
        \lambda \frac{\UP(\upsilonv) - \UP(\upsilonvd)}{d(\ups,\upsd)} 
    \biggr\}\\
	&=& \log\E\exp\left\{\lambda \frac{1}{\sqrt{n}\|\etav-\etavd\|}\sum_{i=1}^{n} \varepsilon_i \fv_{\etav-\etavd}(\Xv^\T_i\thetav)  \right\}\\
	&\le& \sum_{i=1}^n \log\E\exp\left\{ \frac{\lambda}{\sqrt{n}\|\etav-\etavd\|} \varepsilon_i \fv_{\etav-\etavd}(\Xv^\T_i\thetav)  \right\}\\
	&\le& \sum_{i=1}^n \log \E\left[\exp\left\{ \frac{\tilde\nu^2\lambda^2}{n}  \frac{1}{\|\etav-\etavd\|^2}\fv^2_{\etav-\etavd}(\Xv^\T_i\thetav)  \right\}\right]\\
	&\le& \CONST^2 \dimh\tilde\nu^2\lambda^2/2.
\end{EQA}
This implies with Corollary 2.2 of the supplement of \cite{SP2011}
\begin{EQA}[c]
\P\left( \left\|\frac{1}{n}\sum_{i=1}^{n} \varepsilon_i\basX(\Xv^\T_i\thetav)\right\| \ge \CONST \tilde\nu \sqrt{\dimh} \sqrt{\xx+2\dimh}/\sqrt{n} \right)\le \ex^{-\xx}.
\end{EQA}

Two bound the norm of the inverse of the matrix in \eqref{eq: to bound in a priori accuracy} we denote
\begin{EQA}[c]
\bb{M}_n(\thetav)\eqdef \frac{1}{n}\sum_{i=1}^{n}   \basX\basX^\T (\Xv^\T_i\thetav).
\end{EQA}
Note that with Remark \ref{rem: lower bound of HH}
\begin{EQA}[c]
\E\left[\bb{M}_n(\thetav) \right]\ge \lambda(B^{\dimp-1}_{h})c_{p_{\Xv}},
\end{EQA}
while
\begin{EQA}[c]
\sup_{\thetav\in S^{\dimp}_1}\left\|\bb{M}_n(\thetav)-\E\left[\bb{M}_n(\thetav) \right]\right\|=\sup_{(\thetav,\etav)\in S^{\dimp}_1\times S^{\dimh}_1}\left|(P_n-\P)\fv^2_{\etav}(\Xv^\T\thetav)\right|.
\end{EQA}
We bound 
\begin{EQA}
&&\nquad\P\left(\sup_{(\thetav,\etav)\in S^{\dimp}_1\times S^{\dimh}_1}\left|(P_n-\P)\fv^2_{\etav}(\Xv^\T\thetav)\right|\ge t+s\right)\\
	&\le& \P\left(\left|(P_n-\P)\fv^2_{\etavs}(\Xv^\T\thetavs)\right|\ge s\right)\\
	&&+\P\left(\sup_{(\thetav,\etav)\in S^{\dimp}_1\times S^{\dimh}_1}\left|(P_n-\P)\left[\fv^2_{\etav}(\Xv^\T\thetav)-\fv^2_{\etavs}(\Xv^\T\thetavs)\right]\right|\ge t\right).
\end{EQA}
For the first term we can use the bounded differences inequality (Theorem \ref{theo: bounded differences inequality}) to find
\begin{EQA}
\P\left(\left|(P_n-\P)\fv^2_{\etavs}(\Xv^\T_i\thetavs)^2\right|\ge \|f_{\etavs}\|_{\infty}^2\sqrt{\xx}/\sqrt{n}\right)&\le& \ex^{-\xx}.
\end{EQA}
For the second summand we define \(\zetav_{\Xv}(\ups)\eqdef (P_n-\P)\fv_{\etav}(\Xv^\T_i\thetav)^2\). We use the chaining method, i.e. Lemma \ref{lem: basic chaining}. Define \(\Ups_0=\{\upss\}\) and with a sequence \(\rr_k=2^{-k}\rr\) with \(\rr\) to be specified later the sequence of sets \(\Ups_k\) each with minimal cardinality such that
\begin{EQA}
S^{\dimp}_1\times S^{\dimh}_1 \subset \bigcup_{\ups\in\Ups_k} B_{\rr_k}(\ups), &\quad B_{\rr}(\ups)\eqdef \{\upsd\in S^{\dimp}_1\times S^{\dimh}_1,\,\|\upsd-\ups\|\le \rr\}.
\end{EQA}
We can estimate with any \(\ups'\in B_{\rr_k,\DF}(\ups)\)
\begin{EQA}
\inf_{\Ups_{k-1,\dimh}}|\zetav_{\Xv}(\ups)-\zetav_{\Xv}(\upsd)|&=&\left|(P_n-\P)\left\{ \fv_{\etav}(\Xv^\T_i\thetav)^2-\fv_{\etav'}(\Xv^\T_i\thetav')^2\right\} \right|
\end{EQA}
We estimate for an application of the bounded differences inequality
\begin{EQA}
&&\nquad\left|\left\{ \fv_{\etav}(\Xv^\T_i\thetav)^2-\fv_{\etav'}(\Xv^\T_i\thetav')^2\right\} \right|	\\
&\le& \left|\left\{ \fv_{\etav}(\Xv^\T_i\thetav)-\fv_{\etav'}(\Xv^\T_i\thetav')\right\}\left\{ \fv_{\etav}(\Xv^\T_i\thetav)+\fv_{\etav'}(\Xv^\T_i\thetav')\right\}  \right|\\
	&\le& \left(\|\fv_{\etav}\|_{\infty}+\|\fv_{\etav'}\|_{\infty}\right) \left(\| \fv_{\etav-\etav'}\|_{\infty}+\|\fv'_{\etav}\|_{\infty}\|\thetav-\thetav'\|\right).
\end{EQA}
We have as \(\|\eta\|=1\) with Lemma \ref{lem: bounds for objects} 
\begin{EQA}
\|\fv_{\etav}\|_{\infty}&\le& \|\etav\|\sup_{x\in[-s_{\Xv},s_{\Xv}]}\left( \sum_{k=1}^{\dimh} \basX^2_k(x)^2\right)^{1/2}\le \sqrt{17}\|\psi\|\sqrt{\dimh},\\
\|\fv'_{\etav}\|_{\infty}&\le& \|\etav\|\sup_{x\in[-s_{\Xv},s_{\Xv}]}\left( \sum_{k=1}^{\dimh} \basX'^2_k(x)^2\right)^{1/2}\le \sqrt{17}\|\psi'\|\dimh^{3/2}.
\end{EQA}
Consequently
\begin{EQA}[c]
\left|\left\{ \fv_{\etav}(\Xv^\T_i\thetav)^2-\fv_{\etav'}(\Xv^\T_i\thetav')^2\right\} \right|\le \CONST_{\zetav}\dimh^{3/2} \rr_k.
\end{EQA}
This yields with the bounded difference inequality
\begin{EQA}[c]
\P\left(\inf_{\Ups_{k-1,\dimh}} |\zetav_{\Xv}(\ups_k)-\zetav_{\Xv}(\ups_{k-1})|\ge s\CONST_{\zetav}\dimh^{3/2} \rr_k/\sqrt{n} \right)\le \ex^{-s^2}.
\end{EQA}
Now we can define \(\rr\eqdef \frac{(1-1/\sqrt{2})}{\CONST_{\zetav}\dimh^{3/2}}\). Then
\begin{EQA}[c]
\P\left(\inf_{\Ups_{k-1,\dimh}} |\zetav_{\Xv}(\ups_k)-\zetav_{\Xv}(\ups_{k-1})|\ge \frac{2^{-(k-1)}(1-1/\sqrt{2})s}{\sqrt{n}} \right)\le \ex^{-s^2}.\label{eq: chaining k-1 a priori bound matrix part}
\end{EQA}
Set 
\begin{EQA}
s&=&\sqrt{\xx+\log(2)+\dimtotal[1+\log(2)+\log(\CONST_{\zetav}\dimh^{3/2})-\log(1-1/\sqrt{2})]}/{\sqrt{n}}\\
	&\le& \CONST\sqrt{\xx+\dimtotal\log(\dimtotal)}/{\sqrt{n}},
\end{EQA}
and plug it into \eqref{eq: chaining k-1 a priori bound matrix part}, then we find with Lemma \ref{lem: basic chaining}
\begin{EQA}
&&\nquad\P\left(\sup_{\thetav\in S^{\dimp}_1}\left\|\bb{M}_n(\thetav)-\E\left[\bb{M}_n(\thetav) \right]\right\|\ge\CONST\sqrt{\xx+\dimtotal\log(\dimtotal)}/{\sqrt{n}}\right)\\
&\le&\P\left(\sup_{\ups\in\Ups_{\dimh}} \zetav_{\Xv}(\ups)-\zetav_{\Xv}(\upss)\ge \CONST\sqrt{\xx+\dimtotal\log(\dimtotal)}/{\sqrt{n}}\right)\\
&\le&
\sum_{k=1}^{\infty}\exp\bigg\{\dimtotal[1+\log(2)k+\log(\CONST_{\zetav}\dimh^{3/2})-\log(1-1/\sqrt{2})]\\
	 &&-2^{k-1}\left[\xx+\log(2)+\dimtotal[1+\log(2)+\log(\CONST_{\zetav}\dimh^{3/2})-\log(1-1/\sqrt{2})]\right]\bigg\}\\
&\le&\ex^{-\xx}.
\end{EQA}
Together this implies because \(\dimtotal\log(\dimtotal)/\sqrt{n}\to 0\)
\begin{EQA}[c]
\P\left(\sup_{\thetav\in S^{\dimp,+}_1}\left\|\tilde\etav^{(\infty)}_{\dimh,\thetav}\right\|\ge \CONST \sqrt{\dimtotal\log(\dimtotal)+\xx}\right)\le  3\ex^{-\xx}.
\end{EQA}
Adding \(\log(3)\) to \(\xx\) in the above inequality and adapting the constant gives the claim with a probability bound \(\ex^{-\xx}\).
\end{proof}

\subsection{Proof of Lemma~\ref{lem: conditions example}}

Before we prove the claims we need a series of auxiliary lemmas. 

\subsubsection{\(\DF_{\dimh}(\upss_{\dimh})\) is boundedly invertible}	
\begin{lemma}
 \label{lem: D_0 dimh is boundedly invertable}
Under \((\mathbf{\mathcal A})\) we have that
\begin{EQA}[c]\label{eq: bound for cDF}
\DF_{\dimh}(\upss_{\dimh})^2\ge c_{\DF}^2\ge{c_{\DF}^*}^2/ \left(1-\frac{\CONST^*_{\bb{(\cc{L}_{0})}} \left\{\dimh^{3/2}+\CONST_{bias}\dimh^{5/2}\right\}\rr^*}{c_{\DF}^*\sqrt n}\right),
\end{EQA}
where \(c_{\DF}^*>0\) is defined in Lemma \ref{lem: D_0 dimh upss is boundedly invertable} and is independent of \(\dimh,n\) and where \(\rr^*>0\) is defined in \eqref{eq: def of rr star}.
\end{lemma}
\begin{remark}
By the definition of \(\rr^*>0\) in \eqref{eq: def of rr star} it is clear that \(c_{\DF} \approx {c_{\DF}^*}\), once \((\dimh^2+\CONST_{bias}\dimh^3)/\sqrt{n}\to 0\).
\end{remark}

To prove this claim, note that using Lemma \ref{lem: D_0 dimh upss is boundedly invertable} we can prove the following result. It is proved very similarly to Lemma \ref{lem: condition L_0 is satisfied}:
\begin{lemma}
\label{lem: L0 for DF in upss}
We have for any \(\ups\in\{\ups\in\Ups_{\dimh}:\,\|\DF_{\dimh}(\upss)(\ups-\upss)\|\le \rr\}\) and with some constant \(\CONST^*_{\bb{(\cc{L}_{0})}}>0\)
\begin{EQA}[c]
\|I-\DF_{\dimh}^{-1}(\upss_{\dimh})\DF^2_{\dimh}(\upss)\DF^{-1}_{\dimh}(\upss_{\dimh})\|\le \frac{\CONST^*_{\bb{(\cc{L}_{0})}} \left\{\dimh^{3/2}+\CONST_{bias}\dimh^{5/2}\right\}\rr}{c_{\DF}^*\sqrt n}.
\end{EQA}
\end{lemma}

We obtain the claim of Lemma \ref{lem: D_0 dimh is boundedly invertable} because
\begin{EQA}[c]
\DF^2_{\dimh}(\upss_{\dimh})-\DF_{\dimh}(\upss_{\dimh})\left\{I-\DF_{\dimh}^{-1}(\upss_{\dimh})\DF^2_{\dimh}(\upss)\DF^{-1}_{\dimh}(\upss_{\dimh}\right\}=\DF^2_{\dimh}(\upss),
\end{EQA}
such that using Lemma \ref{lem: in example a priori distance of target to oracle} and Lemma \ref{lem: D_0 dimh upss is boundedly invertable}
\begin{EQA}[c]
\left(1+\frac{\CONST^*_{\bb{(\cc{L}_{0})}} \left\{\dimh^{3/2}+\CONST_{bias}\dimh^{5/2}\right\}\rr^*}{c_{\DF}^*\sqrt n}\right)\DF^2_{\dimh}(\upss_{\dimh})\ge \DF^2_{\dimh}(\upss)\ge {c_{\DF}^*}.
\end{EQA}

\subsubsection{Some bounds for the score}
\begin{lemma}\label{lem: bound for supnorm of fv etavs dimh}
We have
\begin{EQA}
|\fv'_{\etavs_{\dimh}}(\xv)|&\le& (C_{\|\fv\|}+ 1)\sqrt{34}s_{\Xv}\|\psi'\|_{\infty},\\
|\fv_{\etav}(\Xv^{\T}\thetav)- \fv_{\etavd}(\Xv^{\T}\thetavd)|&\le&\CONST\frac{\|\DF(\ups-\upsd)\|\sqrt{\dimh}}{\sqrt{n}}\\
	&&+\CONST\left(\frac{\|\DF(\upsd-\upss)\|\dimh^{2}}{\sqrt{n}}+1\right).\label{eq: fv eta theta minus fv etavs thetavs dimh}
\end{EQA}
\end{lemma}

\begin{proof}
Using assumption \((\mathbf{Cond}_{\etavs})\), that \(|M(j)|\le 17\) (in \eqref{def: index set M j}) and \(k=(2^{j_k}-1)17+r_k\) with \(r_k\in \{0,\ldots,2^{j_k}17-1\}\) and \(j_k\in\N_{0}\) we find as \(\alpha >2\)
\begin{EQA}
|\fv'_{\etavs_{\dimh}}(\xv)|&\le& \sum_{j=0}^{j_\dimh-1}\sum_{k\in M(j)}|{\eta_{\dimh}^*}_{k}||\basX_k'(\xv)|\\
	&\le &\sqrt{17}\|\psi'\|_{\infty}\left(\sum_{j=0}^{j_\dimh-1}\sum_{k\in M(j)}|{\eta^*_{\dimh}}_{k}|^2 2^{4j}\right)^{1/2}\left(\sum_{j=0}^{j_\dimh-1}2^{-4 j}2^{3j}\right)^{1/2}\\
	&\le &\sqrt{17}\|\psi'\|_{\infty}\left(\sum_{k=0}^{\dimh-1}|{\eta^*_{\dimh}}_{k}|^2 k^{4}\right)^{1/2}\left(\sum_{j=0}^{j_\dimh-1}2^{-j}\right)^{1/2}\\
	&\le &\sqrt{34}\|\psi'\|_{\infty}C_{\|\etavs_{\dimh}\|},
\end{EQA}
where with Lemma \ref{lem: in example a priori distance of target to oracle} and \(\dimh\in\N\) large enough (\(\dimh^5/n\to 0\) and \(\rr^*\cong \dimh\))
\begin{EQA}
 C_{\|\etavs_{\dimh}\|}&\le& \left(\sum_{k=1}^{\dimh-1}|{\eta^*_{\dimh}}_{k}|^2 k^{4}\right)^{1/2}\\
 &\le& \left(\sum_{k=0}^{\dimh-1}|{\eta^*}_{k}|^2 k^{4}\right)^{1/2} +\left(\sum_{k=0}^{\dimh-1}|{\eta^*_{\dimh}}_{k}-{\eta^*}_{k}|^2 k^{4}\right)^{1/2}\\
  &\le&C_{\|\fv\|}+\dimh^2\|(\etavs_{\dimh}-\Pi_{\dimh}\etavs)\|\\
  &\le& C_{\|\fv\|}+\frac{\dimh^2\rr^*}{\sqrt{n}c_{\DF}}\le C_{\|\etavs\|}+ 1,
\end{EQA}
For the second claim we bound
 \eqref{eq: fv eta theta minus fv etavs thetavs dimh} to bound
\begin{EQA}
|\fv_{\etav}(\Xv^{\T}\thetav)- \fv_{\etavd}(\Xv^{\T}\thetavd)|&\le& |\fv_{\etav-\etavd}(\Xv^{\T}\thetav)|+|\fv_{\etavd}(\Xv^{\T}\thetav)- \fv_{\etavd}(\Xv^{\T}\thetavd)|\\
	&\le& \frac{\rr\sqrt{\dimh}}{c_{\DF}\sqrt{n}}+ s_{\Xv}\|\fv_{\etavd}'\|_{\infty}\rr/\sqrt{n}.
\end{EQA}
It remains to bound using that \(\dimh^5/n\to 0\) and that \(\rr^*\le \CONST \sqrt{\dimh}\)
\begin{EQA}
|\fv_{\etavd}'|&\le& \sqrt{17} \left( \sum_{k=1}^{\dimh}\etavd 2^{4}\right)^{1/2} \left(\sum_{j=1}^{j_{\dimh}}2^{(3-4)j}\right)^{1/2}\\
	&\le& \CONST\left(\frac{\|\DF(\upsd-\upss)\|\dimh^{2}}{\sqrt{n}}+1\right).
\end{EQA}
\end{proof}

\begin{lemma}
 \label{lem: bounds for scores and so on}
We have with \(\varsigmav_{i,\dimh}\) from \eqref{eq: def of varsigmav}
\begin{EQA}[c]
\|\varsigmav_{i,\dimh}(\upsilonvs_{\dimh})\|\le (C_{\|\fv\|}+ 1)\sqrt{34}s_{\Xv}\|\psi'\|_{\infty}+ \sqrt{17} \|\psi\|_{\infty}\sqrt{\dimh},
\end{EQA}
and for any \(\upsilonv,\upsilonv'\in \Upss(\rr)\) with \(\rr\le \CONST\sqrt{\dimh}(1+\CONST_{bias}\log(n))\)
\begin{EQA}
&&\nquad\|\varsigmav_{i,\dimh}(\upsilonv)-\varsigmav_{i,\dimh}(\upsilonv')\|\le \sqrt{34}\Big(s_{\Xv}\|\psi'\|_{\infty}\dimh^{3/2}+2(C_{\|\fv\|}+ 1)\sqrt{\dimh}\|\psi''\|_{\infty}s_{\Xv}\\
  	&&+2\|\psi'\|_{\infty}s_{\Xv}\dimh^{3/2}+\|\psi'\|_{\infty}C_{\|\etavs_{\dimh}\|}\sqrt{2}L_{\nabla \Phi_{\cdot}}\Big)\frac{\|\DF_{\dimh}(\upsilonv-\upsilonv')\|}{\sqrt n c_{\DF}}.
\end{EQA}
\end{lemma}

\begin{proof}
Note
\begin{EQA}
\|\varsigmav_{i,\dimh}(\upsilonvs_{\dimh})\|&=&\| (\fv'_{\etavs_{\dimh}}(\Xv_{i}^{\T}\thetavs_{\dimh}) \nabla\Phi_{\varphi_{\thetavs_{\dimh}}}^{\T} \Xv_{i},  \basX(\Xv_{i}^{\T}\thetavs_{\dimh}))\|\\
  &\le& \| \fv'_{\etavs_{\dimh}}(\Xv_{i}^{\T}\thetavs_{\dimh})\|\|  \Xv_{i}\|+\| \basX(\Xv_{i}^{\T}\thetavs_{\dimh})\|.
\end{EQA}
Such that with \eqref{eq: bound for basis vector} and Lemma \ref{lem: bound for supnorm of fv etavs dimh}
\begin{EQA}\label{eq: bound for f'}
\|\varsigmav_{i,\dimh}(\upsilonvs_{\dimh})\| &\le&(C_{\|\fv\|}+ 1)\sqrt{34}s_{\Xv}\|\psi'\|_{\infty}+ \sqrt{17}\|\psi\|_{\infty}\sqrt{\dimh}.
\end{EQA}
For the second claim we use that for each \(j=1,\ldots,j_{\dimh}-1\)
\begin{EQA}\label{eq: def of set of indizes Nj}
|N(j)|\eqdef\Big|\Big\{&& k\in\{(2^{j}-1)17,\ldots,(2^{j+1}-1)17-1\}:\\
	&& |\basX_{k}(\Xv_{i}^{\T}\thetav')-\basX_{k}(\Xv_{i}^{\T}\thetav)|\vee |\basX'_{k}(\Xv_{i}^{\T}\thetav')-\basX'_{k}(\Xv_{i}^{\T}\thetav)|> 0 \Big\}\Big|\le 34.
\end{EQA}
Furthermore we always have that
\begin{EQA}[c]
|\basX'_{k}(\Xv_{i}^{\T}\thetav')-\basX'_{k}(\Xv_{i}^{\T}\thetav)|\le 2^{j_k5/2}\|\psi''\|_{\infty}s_{\Xv}\|\thetav-\thetav'\|.
\end{EQA}
This implies again using that \(\alpha>2\) that \(\frac{\rr\dimh}{n}\to 0\) for \(\rr^2\le  \CONST\dimh\) and with \(N(j)\subset \N \) from \eqref{eq: def of set of indizes Nj}
\begin{EQA}
\label{eq: bound for f'-difference}
&&\nquad|\fv'_{\etav}(\thetav^{\T} \Xv_{i}) -\fv'_{\etav}(\Xv_{i}^\T\thetav')|\\
  &=&|\sum_{k=1}^{\dimh}\etav_k(e'_k(\Xv_{i}^{\T}\thetav)-e'_k(\Xv_{i}^\T\thetav'))|\\	
  &\le&\left(\sum_{j=0}^{j_\dimh-1}\sum_{k\in N(j)}\eta_{k}2^{5j/2} \right)\|\thetav-\thetav'\|\|\psi''\|_{\infty}s_{\Xv}\\
   &\le&\left\{\left(\sum_{j=0}^{j_\dimh-1}\sum_{k\in N(j)}{\etas_{\dimh}}_{k}2^{5j/2} \right)+ \left(\sum_{j=0}^{j_\dimh-1}\sum_{k\in N(j)}(\eta_{k}-{\etas_{\dimh}}_k)2^{5j/2} \right)\right\}\\
   	&&\|\thetav-\thetav'\|\|\psi''\|_{\infty}s_{\Xv}\\
  &\le&\left\{\sqrt{34}\left(\sum_{k=0}^{\dimh-1}{\etas_{\dimh}}^2_{k} k^{2\alpha}\right)^{\frac{1}{2}}\left(\sum_{j=0}^{j_\dimh-1}2^{(5-2\alpha)j}\right)^{\frac{1}{2}}+\frac{\rr\dimh^{2}}{n} \sqrt{\dimh} \right\}\\
  &&\|\thetav-\thetav'\|\|\psi''\|_{\infty}s_{\Xv}\\
  &\le&\sqrt{34}(C_{\|\fv\|}+ 1)\dimh^{3/2}\|\thetav-\thetav'\|\|\psi''\|_{\infty}s_{\Xv},
\end{EQA}
and with the same arguments
\begin{EQA}
\label{eq: bound for e-difference}
\|\basX(\Xv_{i}^{\T}\thetav)- \basX(\Xv_{i}^{\T}\thetav')\|&\le& \left(\sum_{k=1}^{\dimh}|\basX_k(\Xv_{i}^{\T}\thetav)-\basX_k(\Xv_{i}^{\T}\thetav')|^2\right)^{1/2}\\
  &\le& \sqrt{34}\left(\sum_{j=0}^{j_\dimh-1} 2^{3j}\right)^{1/2}\|\thetav-\thetav'\|\|\psi'\|_{\infty}s_{\Xv}\\
  &\le& \sqrt{34}\dimh^{3/2}\|\thetav-\thetav'\|\|\psi'\|_{\infty}s_{\Xv},
\end{EQA}
and
\begin{EQA}
\label{eq: bound for f'_eta-difference}
   \| f'_{\etav-\etav'}(\thetav^{\T} \Xv_{i})\nabla \varphi_{\thetav}^{\T} \Xv_{i}\|&\le &s_{\Xv}\sum_{k=1}^{\dimh}|\eta_k-\eta'_{\dimh,k}||\basX_k'(\thetav^{\T} \Xv_{i})|\\
  &\le& \sqrt{34}\|\etav-\etav'\|s_{\Xv}\|\psi'\|_{\infty}\left(\sum_{j=0}^{j_\dimh-1} 2^{3j}\right)^{1/2}\\
  &\le& \sqrt{34}\|\etav-\etav'\|s_{\Xv}\|\psi'\|_{\infty}\dimh^{3/2}.
   \end{EQA}
Finally similar to  \eqref{eq: bound for f'} we have with \(M(j)\subset \N \) from \eqref {def: index set M j}
\begin{EQA}
|\fv'_{\etav}(\Xv_{i}^{\T}\thetav)|&\le& \sum_{j=0}^{j_\dimh-1}\sum_{k\in M(j)}|\eta_{k}||\basX_k'(\Xv_{i}^{\T}\thetav)|\\
	&\le &\sqrt{17}\|\psi'\|_{\infty}\left(\sum_{j=0}^{j_\dimh-1}\sum_{k\in M(j)}|\eta_{k}|^2 2^{4j}\right)^{1/2}\left(\sum_{j=0}^{j_\dimh-1}2^{-4j}2^{j}\right)^{1/2}\\
	&\le &\sqrt{17}\|\psi'\|_{\infty}\left(\sum_{k=0}^{\dimh-1}|\eta_{k}|^2 k^{4}\right)^{1/2}\left(\sum_{j=0}^{j_\dimh-1}2^{-j}\right)^{1/2}\\
	&\le &\sqrt{34}\|\psi'\|_{\infty}(C_{\|\etavs\|}+ 1),
\end{EQA}
where since \(\upsilonv'\in\Upss(\rr)\) and \(n\in\N\) large enough (\(\rr^2=O(\dimh)\) and \(\dimh^5(1+\CONST_{bias}\log(n))/n\to 0\)) 
\begin{EQA}
\left(\sum_{k=1}^{\dimh-1}|\eta'_{k}|^2 k^{4}\right)^{1/2}&\le& \left(\sum_{k=0}^{\dimh-1}|{\eta^*}_{k}|^2 k^{4}\right)^{1/2} +\left(\sum_{k=0}^{\dimh-1}|\eta'_{k}-{\eta^*}_{k}|^2 k^{4}\right)^{1/2}\\
  &\le&C_{\|\fv\|}+\dimh^2\left(\|\etav'-\etavs_{\dimh}\|+\|(\etavs_{\dimh}-\Pi_{\dimh}\etavs)\|\right)\\
  &\le& C_{\|\fv\|}+\frac{\dimh^2(\rr+\rr^*)}{\sqrt{n}c_{\DF}}\le C_{\|\etavs\|}+ 1,
\end{EQA}
such that
\begin{EQA}
\label{eq: bound for f' for any point eta}
&&\nquad\|\fv'_{\etav'}(\Xv_{i}^\T\thetav')(\nabla \varphi_{\thetav}^{\T}-\nabla \Phi_{\thetav'}^{\T})\Xv_{i}\|\\
  &\le&\|\psi'\|_{\infty}(C_{\|\etavs\|}+ 1)\sqrt{34}L_{\nabla \Phi_{\cdot}}\|\thetav-\thetav'\|s_{\Xv}.
\end{EQA}

We get combining \eqref{eq: bound for f' for any point eta} , \eqref{eq: bound for f'-difference}, \eqref{eq: bound for f'_eta-difference} and \eqref{eq: bound for e-difference}
\begin{EQA}
 &&\nquad\|\varsigmav_{i,\dimh}(\upsilonv)-\varsigmav_{i,\dimh}(\upsilonv')\|\\
  &=&\|f'_{\etav-\etav'}(\thetav^{\T} \Xv_{i})\nabla \varphi_{\thetav}^{\T} \Xv_{i}\\
  &&+\left[\fv'_{\etav'}(\thetav^{\T} \Xv_{i}) -\fv'_{\etav'}(\Xv_{i}^\T\thetav')\right]\nabla \varphi_{\thetav}^{\T} \Xv_{i}\\
  &&+\fv'_{\etav'}(\Xv_{i}^\T\thetav')(\nabla \varphi_{\thetav}^{\T}-\nabla \Phi_{\thetav'}^{\T})\Xv_{i},\basX(\Xv_{i}^{\T}\thetav)- \basX(\Xv_{i}^{\T}\thetav'))\|\\
  &\le&\sqrt{34}\|\etav-\etav'\|s_{\Xv}\|\psi'\|_{\infty}\dimh^{3/2}\\
  	&&+\sqrt{34}(C_{\|\fv\|}+ 1)\sqrt{\dimh}\|\thetav-\thetav'\|\|\psi''\|_{\infty}s_{\Xv}\\
  	&&+\sqrt{34}\dimh^{3/2}\|\thetav-\thetav'\|\|\psi'\|_{\infty}s_{\Xv}\\
  	&&+\|\psi'\|_{\infty}(C_{\|\etavs\|}+ 1)\sqrt{34}L_{\nabla \Phi_{\cdot}}\|\thetav-\thetav'\|\\
  &\le&\sqrt{34}(s_{\Xv}\|\psi'\|_{\infty}\dimh^{3/2}+2(C_{\|\fv\|}+ 1)\sqrt{\dimh}\|\psi''\|_{\infty}s_{\Xv}\\
  	&&+\|\psi'\|_{\infty}s_{\Xv}\dimh^{3/2}+\|\psi'\|_{\infty}C_{\|\etavs_{\dimh}\|}\sqrt{2}L_{\nabla \Phi_{\cdot}}\Big)\frac{2\|\DF_{\dimh}(\upsilonv-\upsilonv')\|}{\sqrt n c_{\DF}},
\end{EQA}
where we used Lemma \ref{lem: D_0 dimh is boundedly invertable} 
in the last step to find that
\begin{EQA}
\|\thetav-\thetav'\|\vee\|\etav-\etav'\|&\le& \sqrt{\|\thetav-\thetav'\|^2+\|\etav-\etav'\|^2}\le \|\upsilonv-\upsilonv'\|\\
  &\le&\frac{\|\DF_{\dimh}(\upsilonv-\upsilonv')\|}{\sqrt n c_{\DF}}.
\end{EQA}
\end{proof}

\subsubsection{Crude deviation bounds for sums of random matrices}\label{sec: matrix deviation bounds based on tropp}
The next auxiliary Lemma relies on a non-commutative Bernstein inequality; see Theorem 1.4 of \cite{Tropp2012}.
\begin{lemma}
\label{lem: bernstein}
Suppose that \( \vv_{i}\in\R^{\dimp_1} \) are iid random vectors, where \(\dimp\in\N\). Define
\begin{EQA}[c]
\mathbf S^*_n:=\frac{1}{n}\sum_{i=1}^{n}\vv_{i}\vv_{i}^\T-\E[\vv_{1}\vv_{1}^\T],
\end{EQA}
and \( B^{2}:=\E[\|\vv_{1}\|^4]\). Assume that \(\|\vv_{i,\dimh}\vv_{i,\dimh}^\T\|=\|\mathbf M_i\|\le U\in\R\) then it holds
\begin{EQA}[c]
    \P\bigl( \|\mathbf S^*_n \| > n^{-1}t \bigr)
    \le 
    2 \dimp_1 \exp\Bigl\{ - \frac{t^{2}}{4 n B^{2} + 2 U t/3 } \Bigr\}
\label{PXvtsdr}
\end{EQA}
\end{lemma}
\begin{proof}
This lemma is an immediate consequence of the non-commutative Bernstein inequality (Theorem 1.4 in \cite{Tropp2012}). We only have to note that 
\begin{EQA}[c]
\sum_{i=1}^{n}\E[\mathbf M_i^2]\le 2n\E[\|\vv_{1}\|^4]=2n B^{2}.
\end{EQA}
\end{proof}

\begin{lemma}
 \label{lem: bounds for matrix deviations ED_0}
We have with \(\xx\le 9n/2-\log(2\dimh)\) that
\begin{EQA}[c]
\P\left(\|\mathbf S_n\|\ge C_{M}\sqrt{8} \dimh\Big(\xx+\log(2\dimh)\Big)^{1/2}/\sqrt{n}\right)\le \ex^{-\xx}.
\end{EQA}
where with \(\varsigmav_{i,\dimh}\) from \eqref{eq: def of varsigmav}
\begin{EQA}
\mathbf S_n&=&\frac{1}{n}\sum_{i=1}^{n}\varsigmav_{i,\dimh}(\upsilonvs_{\dimh})\varsigmav_{i,\dimh}(\upsilonvs_{\dimh})^\T-\frac{1}{n}\VF^2_{\dimh}(\upsilonvs_{\dimh}).
\end{EQA}
\end{lemma}

\begin{proof}
We want to employ lemma \ref{lem: bernstein}. We estimate using Lemma \ref{lem: bounds for scores and so on}
\begin{EQA}[c]
\|\varsigmav_{i,\dimh}(\upsilonvs_{\dimh})\varsigmav_{i,\dimh}(\upsilonvs_{\dimh})^\T\| \le 34\Big((C_{\|\fv\|}+ 1)\sqrt{2}s_{\Xv}\|\psi'\|_{\infty}+ \|\psi\|_{\infty}\Big)^2\dimh=:C_{M}\dimh,
\end{EQA}
such that \(\|\varsigmav_{i,\dimh}\varsigmav_{i,\dimh}^\T\|=:\|\mathbf M_i\|\le C_{M}\dimh\). Furthermore
\begin{EQA}[c]
\E[\|\varsigmav_{i,\dimh}(\upsilonvs_{\dimh})\|^4]\le C_{M}^2\dimh^2.
\end{EQA}
Plugging these bounds into lemma \ref{lem: bernstein} we get
\begin{EQA}[c]
\P(\|\mathbf S_n\|\ge n^{-1}t)\le  2 \dimh \exp\Bigl\{ - \frac{ t^{2}}{4 nC_{M}^2\dimh^2 + 2 C_{M}\dimh t/3 } \Bigr\}.
\end{EQA}
Setting \(t= C_{M}\sqrt{8 n} \dimh\Big(\xx+\log(2\dimh)\Big)^{1/2}\) and \(\xx\le 9n/2-\log(2\dimh)\) this gives
\begin{EQA}[c]
\P\left(\|\mathbf S_n\|\ge C_{M}\sqrt{8} \dimh\Big(\xx+\log(2\dimh)\Big)^{1/2}/\sqrt{n}\right)\le \ex^{-\xx}.
\end{EQA}
\end{proof}

\begin{lemma}
 \label{lem: bounds for matrix deviations Er}
We have with \(\xx\le 9n/2-\log(2\dimh)\) that
\begin{EQA}[c]
\P\left(\|\mathbf S_n\|\ge \sqrt{8}\CONST(\dimtotal+\xx)^4\Big(\xx+\log(2\dimh)\Big)^{1/2}/\sqrt{n}\right)\le \ex^{-\xx}.
\end{EQA}
where with \(\varsigmav_{i,\dimh}\) from \eqref{eq: def of varsigmav}
\begin{EQA}
\mathbf S_n(\ups)&=&\frac{1}{n}\sum_{i=1}^{n}\varsigmav_{i,\dimh}(\upsilonv)\varsigmav_{i,\dimh}(\upsilonv)^\T-\frac{1}{n}\VF^2_{\dimh}(\upsilonv).
\end{EQA}
\end{lemma}

\begin{proof}
We want to employ lemma \ref{lem: bernstein}. We estimate using Lemma \ref{lem: bounds for scores and so on} and that \(\rr^{\circ}\le \CONST\sqrt{\dimtotal+\xx}\)
\begin{EQA}
\|\varsigmav_{i,\dimh}(\upsilonv)\varsigmav_{i,\dimh}(\upsilonv)^\T\| &\le& 3\|\varsigmav_{i,\dimh}(\upsilonvs_{\dimh})\|^2+ 3\|\varsigmav_{i,\dimh}(\upsilonvs_{\dimh})-\varsigmav_{i,\dimh}(\upsilonv)\|^2\\
	&\le&  C_{M}\dimh+\CONST\frac{\|\DF_{\dimh}(\upsilonv-\upsilonvs_{\dimh})\|^2\dimh^{3}}{ n}\\
	&\le&  C_{M}\dimh+\CONST \dimh^3\rr^2/n.
\end{EQA}
such that \(\|\varsigmav_{i,\dimh}\varsigmav_{i,\dimh}^\T\|=:\|\mathbf M_i\|\le C_{M}\dimh\). Furthermore
\begin{EQA}[c]
\E[\|\varsigmav_{i,\dimh}(\upsilonvs_{\dimh})\|^4]\le \CONST^2 (\dimh^2+\dimh^6\rr^4/n^2).
\end{EQA}
Plugging these bounds into lemma \ref{lem: bernstein} we get
\begin{EQA}[c]
\P(\|\mathbf S_n\|\ge n^{-1}t)\le  2 \dimh \exp\Bigl\{ - \frac{ t^{2}}{4 n\CONST^2 (\dimh^2+\dimh^6\rr^4/n^2) + 2\CONST\left(\dimh+\dimh^3\rr^2/n\right) t/3 } \Bigr\}.
\end{EQA}
Setting \(t=\sqrt{8n}\CONST \left(\dimh+\dimh^3\rr^2/n\right)\Big(\xx+\log(2\dimh)\Big)^{1/2}/n^2  \) and \(\xx\le 9n/2-\log(2\dimh)\) this implies
\begin{EQA}[c]
\P\left(\|\mathbf S_n\|\ge \sqrt{8}\CONST\left(\dimh+\dimh^3\rr^2/n\right)\Big(\xx+\log(2\dimh)\Big)^{1/2}/\sqrt{n}\right)\le \ex^{-\xx}.
\end{EQA}
\end{proof}

\begin{lemma}
 \label{lem: bounds for matrix deviations ED_1}
We have with 
\begin{EQA}[c]
t=C_{M}^2 \|\DF_{\dimh}(\upsilonv-\upsilonv')\|^2 \sqrt{5/n}\dimh^3\Big(\xx+\log(2\dimh)\Big)^{1/2},
\end{EQA}
and \(\xx\le 9n/2-\log(2\dimh)\)
\begin{EQA}[c]
\P(\|\mathbf S_n\|\ge n^{-1}t)\le \ex^{-\xx},
\end{EQA}
where with \(\upsilonv\in\Upss(\rr)\) and with \(\varsigmav_{i,\dimh}\) from \eqref{eq: def of varsigmav}
\begin{EQA}
 \mathbf S_n&=&\frac{1}{n}\sum_{i=1}^{n}(\varsigmav_{i,\dimh}(\upsilonv')-\varsigmav_{i,\dimh}(\upsilonv)) (\varsigmav_{i,\dimh}(\upsilonv')-\varsigmav_{i,\dimh}(\upsilonv))^\T\\
  &&-\E(\varsigmav_{i,\dimh}(\upsilonv')-\varsigmav_{i,\dimh}(\upsilonv))(\varsigmav_{i,\dimh}(\upsilonv')-\varsigmav_{i,\dimh}(\upsilonv))^\T\\
  C_{M}&=&\sqrt{34}\Big(s_{\Xv}\|\psi'\|_{\infty}+3(C_{\|\fv\|}+ 1)\|\psi''\|_{\infty}s_{\Xv}\\
  	&&+3\|\psi'\|_{\infty}s_{\Xv}+\|\psi'\|_{\infty}C_{\|\etav'\|}\sqrt{2}L_{\nabla \Phi_{\cdot}}\Big)\frac{2}{c_{\DF}}.
\end{EQA}
\end{lemma}

\begin{proof}
We estimate using Lemma \ref{lem: bounds for scores and so on}
\begin{EQA}
&&\nquad\|(\varsigmav_{i,\dimh}(\upsilonv')-\varsigmav_{i,\dimh}(\upsilonv))(\varsigmav_{i,\dimh}(\upsilonv')-\varsigmav_{i,\dimh}(\upsilonv))^\T\|\\
  &\le&\|\varsigmav_{i,\dimh}(\upsilonv')-\varsigmav_{i,\dimh}(\upsilonv)\|^2\\	
  &\le&34\Big(s_{\Xv}\|\psi'\|_{\infty}+3(C_{\|\fv\|}+ 1)\|\psi''\|_{\infty}s_{\Xv}\\
  	&&+3\|\psi'\|_{\infty}s_{\Xv}+\|\psi'\|_{\infty}C_{\|\etav'\|}\sqrt{2}L_{\nabla \Phi_{\cdot}}\Big)\frac{4\|\DF_{\dimh}(\upsilonv-\upsilonv')\|^2\dimh^{3}}{ n c_{\DF}^2}\\
  &=:&C_M^2 \frac{\|\DF_{\dimh}(\upsilonv-\upsilonv')\|^2\dimh^{3}}{ n}.
\end{EQA}
With the same estimates we obtain
\begin{EQA}[c]
\E[\|\varsigmav_{i,\dimh}(\upsilonv')-\varsigmav_{i,\dimh}(\upsilonv)\|^4]\le C_{M}^4\dimh^6\frac{\|\DF_{\dimh}(\upsilonv-\upsilonv')\|^4}{ n^2}.
\end{EQA}
Plugging these bounds into Lemma \ref{lem: bernstein} we get with \(d(\ups,\ups')\eqdef \|\DF_{\dimh}(\upsilonv-\upsilonv')\|\)
\begin{EQA}
&&\nquad\P(\|\mathbf S_n\|\ge n^{-1}t)\\
	&\le& 2 \dimh \exp\Bigl\{ - \frac{ t^{2} }{4 d(\ups,\ups')^4C_{M}^4n^{-1}\dimh^6 + 2 d(\ups,\ups')^2C^2_{M}\dimh^3 n^{-1} t/3 } \Bigr\}.
\end{EQA}
Setting \(t=C_{M}^2 \|\DF_{\dimh}(\upsilonv-\upsilonv')\|^2 \sqrt{8/n}\dimh^3\Big(\xx+\log(2\dimh)\Big)^{1/2}\) and \(\xx\le 9n/2-\log(2\dimh)\) this yields
\begin{EQA}[c]
\P(\|\mathbf S_n\|\ge n^{-1}t)\le \ex^{-\xx}.
\end{EQA}
\end{proof}

\subsubsection{Conditions \(\bb{ (\CS \DF_{0})}\), \(\bb{(\CS\rr)}\) and  \( \bb{(\CS \DF_{1,\dimh})}\)}

\begin{lemma}
 \label{lem: condition ED_0 and Er}
With probability greater than \(1-3\ex^{-\xx}\) we have \( \bb{(\CS \DF_{0})} \) with 
\begin{EQA}
 g&=&\sqrt n \sigma^{-1} c_{\DF}\tilde g\Big((C_{\|\etavs\|}+ 1)\sqrt{34}s_{\Xv}\|\psi'\|_{\infty}+\sqrt{17} \|\psi\|_{\infty}\sqrt{\dimh}\Big)^{-1},\\
 \nu_{\dimh}^{2}&=&2\tilde \nu^{2}\sigma^2,
\end{EQA}
and \(\bb{(\CS\rr)}\) with 
\begin{EQA}
 \gm(\rr)&=&\sqrt n c_{\DF}\tilde g\CONST\left(\sqrt{\dimh}+\dimh^{3/2}\rr/\sqrt{n}\right)^{-1},\\
 \nu_{\rr,\dimh}^{2}&=&\tilde \nu^{2}\Big( 1+\CONST\left(\dimh^{3/2}+ \rr\dimh^2/\sqrt{n}\right)\rr/\sqrt{n}\\
 	&&+\CONST\left(\dimh+\dimh^3\rr^2/n\right)\Big(\xx+\log(2\dimh)\Big)^{1/2}/\sqrt{n}\Big).
\end{EQA}
where \(\CONST_{\bb{(\CS\rr)}}>0\) is independent of \(n,\dimh,\xx,\rr\). 
\end{lemma}
             
\begin{proof}
Lemma \ref{lem: D_0 dimh is boundedly invertable} gives with \(\tilde\gammav=\VF^{1/2}_{\dimh}\gamma /\|\VF^{1/2}_{\dimh}\gamma\|\)
\begin{EQA}
 \frac{\langle \nabla \zeta(\upsilonvs_{\dimh}),\gamma\rangle_{\R^{\dimtotal}}}{\|\VF_{\dimh}\gamma\|}&=&\langle \tilde\gamma^{\T} \VF_{\dimh}^{-1}A(\upsilonvs_{\dimh}),\varepsilon \rangle_{\R^{n}}.
\end{EQA}
Consequently - using Lemma \ref{lem: bounds for scores and so on} - we get with \(\mu\le  \sqrt n \sigma^{-1} c_{\DF}\tilde g\Big((C_{\|\etavs\|}+ 1)\sqrt{34}s_{\Xv}\|\psi'\|_{\infty}+\sqrt{17} \|\psi\|_{\infty}\sqrt{\dimh}\Big)^{-1}\), with \(\varsigmav_{i,\dimh}\) from \eqref{eq: def of varsigmav} and assumption \((\mathbf{ Cond}_{\varepsilon})\)
\begin{EQA}
 &&\nquad\sup_{\gamma\in \R^{\dimtotal}}\log\E_{\varepsilon} \exp\left\{ \mu \frac{\langle \nabla \zeta(\upsilonvs_{\dimh}),\gamma\rangle}{\|\VF_{\dimh}(\upsilonvs_{\dimh})\gamma\|}\right\}\\
  &\le& \sum_{i=1}^{n}\sup_{\gamma\in \R^{\dimtotal},\text{ }\| \tilde\gamma\|=1}\log\E \exp\left\{ \mu\langle \tilde \gamma,\VF_{\dimh}^{-1}(\upsilonvs_{\dimh})\varsigmav_{i,\dimh}(\upsilonvs_{\dimh})\rangle\varepsilon_{i}\right\}\\
  &\le&\tilde \nu^{2}\mu^{2}\tilde \gamma^\T\VF_{\dimh}^{-1}(\upsilonvs_{\dimh})\left(\sum_{i=1}^{n}\varsigmav_{i,\dimh}(\upsilonvs_{\dimh})\varsigmav_{i,\dimh}(\upsilonvs_{\dimh})^\T\right)\VF_{\dimh}^{-1}(\upsilonvs)\tilde\gamma\\
  &=&\tilde \nu^{2}\mu^{2}+\tilde \nu^{2}\mu^{2}\tilde\gamma^\T\VF_{\dimh}^{-1}(\upsilonvs_{\dimh}) n\mathbf S_n\VF_{\dimh}^{-1}(\upsilonvs_{\dimh})\tilde\gamma\\
   &\le&\tilde \nu^{2}\mu^{2}+\tilde \nu^{2}\mu^{2}\kappa_n,
\label{expmoments calculation}
\end{EQA}
where 
\begin{EQA}
\kappa_n&=&\tilde\gamma^\T\left(n^{-1}\VF_{\dimh}\right)^{-1/2} \mathbf S_n\left(n^{-1}\VF_{\dimh}\right)^{-1/2}\tilde\gamma,\\
  \mathbf S_n&=&\frac{1}{n}\sum_{i=1}^{n}\varsigmav_{i,\dimh}(\upsilonvs_{\dimh})\varsigmav_{i,\dimh}(\upsilonvs_{\dimh})^\T-\frac{1}{n}\VF_{\dimh}(\upsilonvs_{\dimh}).
\end{EQA}
With Lemma \ref{lem: bounds for matrix deviations ED_0} we infer that if \(\xx\le 9n/2-\log(2\dimh)\)
\begin{EQA}[c]
\P\left(\|\mathbf S_n\|\ge C_{M}\sqrt{8} \dimh\Big(\xx+\log(2\dimh)\Big)^{1/2}/\sqrt{n}\right)\le \ex^{-\xx}.
\end{EQA}
Consequently with probability greater than \(1-\ex^{-\xx}\) we find that for \(n\in\N\) large enough
\begin{EQA}[c]
\kappa_n\le \frac{\CONST_M \sqrt{8} \dimh\Big(\xx+\log(2\dimh)\Big)^{1/2}}{\sqrt n \sigma^2c_{\DF}^2}\le 1.
\end{EQA}
Thus we get \( \bb{(\CS \DF_{0})} \) with probability greater than \(1-\ex^{-\xx}\) and
\begin{EQA}
 g&=&\sqrt n  c_{\DF}\tilde g\Big((C_{\|\etavs\|}+ 1)\sqrt{34}s_{\Xv}\|\psi'\|_{\infty}+\sqrt{17} \|\psi\|_{\infty}\sqrt{\dimh}\Big)^{-1},\\
 \nu_{\dimh}^{2}&=&2\tilde \nu^{2}.
\end{EQA}

Concerning \( \bb{(\CS \rr)} \) we bound using the same arguments as in the proof of Lemma \ref{lem: condition L_0 is satisfied}
\begin{EQA}
\|\VF_{\dimh}(\ups)^{-1}\VF_{\dimh}(\upss_{\dimh})\|^2&\le& 1+\|I-\VF_{\dimh}(\ups)^{-1}\VF_{\dimh}(\upss_{\dimh})^2\VF_{\dimh}(\ups)^{-1}\|\\
	&\le& 1+\CONST\left(\dimh^{3/2}+ \rr\dimh^2/\sqrt{n}\right)\rr/\sqrt{n}.
\end{EQA}
Thus we get with the arguments from above \(\bb{(\CS\rr)}\) using Lemma \ref{lem: bounds for matrix deviations Er} with probability greater than \(1-\ex^{-\xx}\) and
\begin{EQA}
 \gm(\rr)&=&\sqrt n c_{\DF}\tilde g\CONST\left(\sqrt{\dimh}+\dimh^{3/2}\rr/\sqrt{n}\right)^{-1},\\
 \nu_{\rr,\dimh}^{2}&=&\tilde \nu^{2}\Big( 1+\CONST\left(\dimh^{3/2}+ \rr\dimh^2/\sqrt{n}\right)\rr/\sqrt{n}\\
 	&&+\CONST\left(\dimh+\dimh^3\rr^2/n\right)\Big(\xx+\log(2\dimh)\Big)^{1/2}/\sqrt{n}\Big).
\end{EQA}
\end{proof}

\begin{lemma}
 \label{lem: condition ED_1}
With probability greater than \(1-\ex^{-\xx}\) we have \( \bb{(\CS \DF_{1}) }\) with 
\begin{EQA}
\gm&\eqdef& \sqrt n c_{\DF}\rr \dimh^{-3/2}C_{\bb{(\CS \DF_{1}) }}^{-1},	\\
 \omega& \eqdef &\frac{2}{\sqrt n c_{\DF}},\\
 \nu_{1,\dimh}^2&=&\tilde \nu^{2}\CONST_{\bb{(\CS \DF_{1}) }} \dimh^2,
\end{EQA}
where
\begin{EQA}
	C_{\bb{(\CS \DF_{1}) }}&=& \sqrt{34}\Big(s_{\Xv}\|\psi'\|_{\infty}+3(C_{\|\fv\|}+ 1)\|\psi''\|_{\infty}s_{\Xv}
  	+3\|\psi'\|_{\infty}s_{\Xv}\\
  		&&+\|\psi'\|_{\infty}C_{\|\etavs_{\dimh}\|}\sqrt{2}L_{\nabla \Phi_{\cdot}}\Big).
\end{EQA}
\end{lemma}

\begin{proof}
We get with Lemma \ref{lem: bounds for scores and so on}, with Lemma \ref{lem: D_0 dimh is boundedly invertable} and with \(\varsigmav_{i,\dimh}\) from \eqref{eq: def of varsigmav}
\begin{EQA}
 &&\nquad\|\DF_{\dimh}^{-1}(\varsigmav_{i,\dimh}(\upsilonv)-\varsigmav_{i,\dimh}(\upsilonv'))\|\\
 	&\le&\frac{\sqrt{34}}{\sqrt n c_{\DF}}\sum_{i=1}^{n}\varepsilon_{i}\Big(s_{\Xv}\|\psi'\|_{\infty}\dimh^{3/2}+3(C_{\|\fv\|}+ 1)\sqrt{\dimh}\|\psi''\|_{\infty}s_{\Xv}\\
  	&&+3\|\psi'\|_{\infty}s_{\Xv}\dimh^{3/2}+\|\psi'\|_{\infty}C_{\|\etavs_{\dimh}\|}\sqrt{2}L_{\nabla \Phi_{\cdot}}\Big)\frac{2\rr}{\sqrt n c_{\DF}}\\
  	&\eqdef& C_{\bb{(\CS \DF_{1}) }}\frac{2\dimh^{3/2}}{n c_{\DF}^2}\|\DF_{\dimh}(\ups-\ups')\|,
\end{EQA}
We get with,
\begin{EQA}
 \mu&\le& \gm\eqdef \sqrt n c_{\DF} (\rr\dimh)^{-3/2}C_{\bb{(\CS \DF_{1}) }}^{-1}\\
 \omega& \eqdef &\frac{2}{\sqrt n c_{\DF}},
\end{EQA}
and the same calculations as in \eqref{expmoments calculation} with some \(\upsilonv,\upsilonv'\in\Upss(\rr)\), \(\gammav\in\R^{\dimtotal}\) and \(\|\gammav\|=1\)
\begin{EQA}
&&\nquad\log\E_{\varepsilon}[\exp\left\{ \mu \frac{\gammav^\T\DF_{\dimh}^{-1}\left( \nabla\zeta(\upsilonv)- \nabla\zeta(\upsilonv')\right)}{\omega\|\DF_{\dimh}(\upsilonv-\upsilonv')\|}\right\}]\\
&\le&\sum_{i=1}^{n}\log\E_{\varepsilon}[\exp\left\{\mu\varepsilon_{i}\frac{\gamma^\T \DF_{\dimh}^{-1}(\varsigmav_{i,\dimh}(\upsilonv)-\varsigmav_{i,\dimh}(\upsilonv'))}{\omega\|\DF_{\dimh}(\upsilonv-\upsilonv')\|} \right\}]\\
&\le&\frac{\mu^{2}\tilde \nu^{2}}{2}(\omega\|\DF_{\dimh}(\upsilonv-\upsilonv')\|)^{-2}\\
	&&n\gamma^\T \DF_{\dimh}^{-1}\left(\sum_{i=1}^{n}(\varsigmav_{i,\dimh}(\upsilonv')-\varsigmav_{i,\dimh}(\upsilonv))(\varsigmav_{i,\dimh}(\upsilonv')-\varsigmav_{i,\dimh}(\upsilonv))^\T \right)\DF_{\dimh}^{-1}\gamma^\T.
\end{EQA}

We estimate 
\begin{EQA}
&&\nquad\tilde\gamma^\T \DF_{\dimh}^{-1}\left(\sum_{i=1}^{n}(\varsigmav_{i,\dimh}(\upsilonv')-\varsigmav_{i,\dimh}(\upsilonv))(\varsigmav_{i,\dimh}(\upsilonv')-\varsigmav_{i,\dimh}(\upsilonv))^\T \right)\DF_{\dimh}^{-1}\tilde\gamma^\T\\
  &\le& \|\DF_{\dimh}^{-1}n\E\left[(\varsigmav_{i,\dimh}(\upsilonv')-\varsigmav_{i,\dimh}(\upsilonv))(\varsigmav_{i,\dimh}(\upsilonv')-\varsigmav_{i,\dimh}(\upsilonv))^\T\right]\DF_{\dimh}^{-1}\|+\kappa_n\\
  &\le& \E\|\left( n^{-1/2}\DF_{\dimh}\right)^{-1}(\varsigmav_{i,\dimh}(\upsilonv')-\varsigmav_{i,\dimh}(\upsilonv))\|^2+\kappa_n,
\end{EQA}
where 
\begin{EQA}
  \kappa_n&=&\|\left(n^{-1/2}\DF_{\dimh}\right)^{-1} \mathbf S_n\left(n^{-1/2}\DF_{\dimh}\right)^{-1}\|,\\
  \mathbf S_n&=&\frac{1}{n}\sum_{i=1}^{n}(\varsigmav_{i,\dimh}(\upsilonv')-\varsigmav_{i,\dimh}(\upsilonv)) (\varsigmav_{i,\dimh}(\upsilonv')-\varsigmav_{i,\dimh}(\upsilonv))^\T\\
  &&-\E(\varsigmav_{i,\dimh}(\upsilonv')-\varsigmav_{i,\dimh}(\upsilonv))(\varsigmav_{i,\dimh}(\upsilonv')-\varsigmav_{i,\dimh}(\upsilonv))^\T.
\end{EQA}
To controll \(\kappa_n>0\) we apply Lemma \ref{lem: bounds for matrix deviations ED_1} and we infer that with \(t=C_{M}^2 \|\DF_{\dimh}(\upsilonv-\upsilonv')\|^2 \sqrt{5/n}\dimh^3\Big(\xx+\log(2\dimh)\Big)^{1/2}\) and \(\xx\le 9n/2-\log(2\dimh)\) the set \(\{\|\mathbf S_n\|\le n^{-1}t\}\) is of dominating probability and on this set we find
\begin{EQA}
\kappa_n&\le& \frac{C_{M}^2 \|\DF_{\dimh}(\upsilonv-\upsilonv')\|^2 \Big(\xx+\log(2\dimh)\Big)^{1/2}\dimh^3\sqrt{5/n}}{nc^2_{\DF}}\\
	&\le &\omega^2\|\DF_{\dimh}(\upsilonv-\upsilonv')\|^2\frac{C_{M}^2  \sqrt{5}\dimh^3\Big(\xx+\log(2\dimh)\Big)^{1/2}}{\sqrt n}.
\end{EQA}
For \(\rr\le \rups\le \CONST_{\rr}\sqrt{\dimtotal+\xx}\) this gives because \(\dimh^{5/2}/\sqrt{n}\to 0\)
\begin{EQA}
\kappa_n&\le& \CONST_{\kappa}\sqrt{\Big(\xx+\log(2\dimh)\Big)\dimtotal}.
\end{EQA}
We calculate with some \((\thetavd,\etavd)\eqdef(\frac{1}{\sqrt{n}}\DF_{\dimh})^{-1}\gammav\)
\begin{EQA}
&&\nquad n\gammav^\T \DF_{\dimh}^{-1}\E\left[(\varsigmav_{1,\dimh}(\upsilonv')-\varsigmav_{1,\dimh}(\upsilonv))(\varsigmav_{1,\dimh}(\upsilonv')-\varsigmav_{1,\dimh}(\upsilonv))^\T \right]\DF_{\dimh}^{-1}\gammav^\T\\
	&=&\E\left[\left\{\left[\fv'_{\etav}(\Xv^\T\thetav)-\fv'_{\etav'}(\Xv^\T\thetav')\right] (\Xv^\T\thetavd)^2+\fv_{\etavd}(\Xv^\T\thetav)-\fv_{\etavd}(\Xv^\T\thetav')\right\}^2\right]\\
	&\le&3\E\left[\left\{\left[\fv'_{\etav}(\Xv^\T\thetav)-\fv'_{\etav'}(\Xv^\T\thetav')\right] (\Xv^\T\thetavd)^2\right\}^2\right]\\
	&&+3\E\left[\left\{\fv_{\etavd}(\Xv^\T\thetav)-\fv_{\etavd}(\Xv^\T\thetav')\right\}^2\right].
\end{EQA}
We estimate separately
\begin{EQA}
&&\nquad\E\left[\left\{\left[\fv'_{\etav}(\Xv^\T\thetav)-\fv'_{\etav'}(\Xv^\T\thetav')\right] (\Xv^\T\thetavd)^2\right\}^2\right]\\
	&\le& 3 s_{\Xv}^4\left(\E\left[\left\{\fv'_{\etav-\etav'}(\Xv^\T\thetav)\right\}^2\right]+\E\left[\left\{\fv'_{\etav'}(\Xv^\T\thetav)-\fv'_{\etav'}(\Xv^\T\thetav')\right\}^2\right] \right).
\end{EQA}
We again estimate separately denoting \(\gammav=(\etav-\etav')/\|\etav-\etav'\|\)
\begin{EQA}
\E\left[\left\{\fv'_{\etav-\etav'}(\Xv^\T\thetav)\right\}^2\right]&=&\|\etav-\etav'\|^2 2\sum_{k=1}^{\dimh}\sum_{l=k}^{\dimh}(1-1_{k=l}/2)\gamma_k\gamma_l \E[\basX'_l\basX'_l(\Xv^\T\thetav)],
\end{EQA}
We have with \(l=(2^{j_l}-1)17+r_l\in\N\) and \(k=(2^{j_k}-1)17+r_k\in\N\) using \eqref{eq: bound for scaler product basis derivatives l k}
\begin{EQA}
\E[\basX'_k\basX'_l(\Xv^\T\thetav)]	&\le& 17 \CONST_{p_{\Xv}} 2^{j_k}\|\psi'\|^2_{\infty}2^{j_l} 1_{I_k\cap I_l\neq 0}.
	\label{eq: bound for scalar produc basis first derivative}
\end{EQA}
This implies
\begin{EQA}
&&\nquad\frac{1}{\|\etav-\etav'\|^2}\E\left[\left\{\fv'_{\etav-\etav'}(\Xv^\T\thetav)\right\}^2\right]\\
	&=&\sum_{k=1}^{\dimh}\sum_{l=k}^{\dimh}(1-1_{k=l}/2)\gamma_l\gamma_k \E[\basX'_l\basX'_l(\Xv^\T\thetav)]\\
	&\le& 17\CONST_{p_{\Xv}} \|\psi'\|_{\infty}  \sum_{k=0}^{\dimh}\gamma_k 2^{j_k}\left(\sum_{j=j_k}^{j_{\dimh}}\sum_{r=0}^{2^{j}17-1} 2^{2j_l} 1_{I_k\cap I_l\neq 0}((2^{j}-1)17+r,k)\right)^{1/2}\\
	&\le& 17\CONST_{p_{\Xv}} \|\psi'\|_{\infty}  \sum_{k=0}^{\dimh} \gamma_k2^{j_k}\left(\sum_{j=j_k}^{j_{\dimh}}2^{2j_l}\left\ulcorner2^{(j_l-j_k)}17\right\urcorner\right)^{1/2}\\
	&\le& \sqrt{18}17\CONST_{p_{\Xv}} \|\psi'\|_{\infty}  \sum_{k=0}^{\dimh}\gamma_k2^{j_k/2} \left(\sum_{j=j_k}^{j_{\dimh}}2^{3j_l}\right)^{1/2}\\
	&\le& 18^2\CONST_{p_{\Xv}} \dimh^2.
	\label{eq: bound for f etav minus f etavd}
\end{EQA}
Furthermore
\begin{EQA}
&&\nquad\E\left[\left\{\fv'_{\etav'}(\Xv^\T\thetav)-\fv'_{\etav'}(\Xv^\T\thetav')\right\}^2\right]=2\sum_{k=1}^{\dimh}\sum_{l=k}^{\dimh}(1-1_{k=l}/2) \eta'_k\eta'_l\\
	&&\E\left[(\basX'_k(\Xv^\T\thetav)-\basX'_k(\Xv^\T\thetav'))(\basX'_l(\Xv^\T\thetav)-\basX'_l(\Xv^\T\thetav'))\right].
\end{EQA}
With \eqref{eq: bound for expectation of basis difference first derivative} this gives
\begin{EQA}
&&\nquad\E\left[\left\{\fv'_{\etav'}(\Xv^\T\thetav)-\fv'_{\etav'}(\Xv^\T\thetav')\right\}^2\right]\\
	&\le& \|\thetav-\thetav'\|^2\CONST_{p_{\Xv}}\|\psi''\|_{\infty}^2 s_{\Xv}^4 17^2  2\sum_{k=1}^{\dimh} \eta'_k 2^{2j_k}\sum_{l=k}^{\dimh} \eta'_l 2^{2j_l}1_{\{I_k\cap I_l\neq \emptyset\}}\\
	&\le&\|\thetav-\thetav'\|^2\CONST_{p_{\Xv}}\|\psi''\|_{\infty}^2 s_{\Xv}^4 17^2  2\\
		&&\sum_{k=1}^{\dimh} \eta'_k 2^{2j_k}\left( \sum_{l=k}^{\dimh} {\eta'_l}^2 k^{4}\right)^{1/2}\left(\sum_{l=k}^{\dimh} 1_{\{I_k\cap I_l\neq \emptyset\}}\right)^{1/2}.
\end{EQA}
As always 
\begin{EQA}[c]
\rr\le \CONST \sqrt{\dimtotal}(1+\CONST_{bias}\log(n)),
\end{EQA}
implies \(\rr\dimh^2/\sqrt{n}\to 0\) such that 
\begin{EQA}
\left( \sum_{l=k}^{\dimh} {\eta'_l}^2 k^{4}\right)^{1/2}&\le& \left( \sum_{l=k}^{\dimh} {\etas_{\dimh}}_l^2 k^{4}\right)^{1/2}+\left( \sum_{l=k}^{\dimh} |{\eta'_l}-{\etas_{\dimh}}_l|^2 k^{4}\right)^{1/2}\\
	&\le&2(1-\CONST_{\|\etavs\|}),
\end{EQA}
which gives using \eqref{eq: number of intersections}
\begin{EQA}
&&\nquad\E\left[\left\{\fv'_{\etav'}(\Xv^\T\thetav)-\fv'_{\etav'}(\Xv^\T\thetav')\right\}^2\right]\\
	&\le&\|\thetav-\thetav'\|^2(1-\CONST_{\|\etavs\|})\CONST_{p_{\Xv}}\|\psi''\|_{\infty}^2 s_{\Xv}^4 17^{5/2}  4\dimh\sum_{k=1}^{\dimh} \eta'_k 2^{3j_k/2}.
\end{EQA}
Repeating the same arguments gives
\begin{EQA}[c]
\E\left[\left\{\fv'_{\etav'}(\Xv^\T\thetav)-\fv'_{\etav'}(\Xv^\T\thetav')\right\}^2\right] \le\|\thetav-\thetav'\|^2(1-\CONST_{\|\etavs\|})\CONST_{p_{\Xv}}\|\psi''\|_{\infty}^2 s_{\Xv}^4 17^34 \dimh^{3/2},
\end{EQA}
such that
\begin{EQA}[c]\label{eq: bound derivative f etavthetav minus t etavdthetavd}
\E\left[\left\{\left[\fv'_{\etav}(\Xv^\T\thetav)-\fv'_{\etav'}(\Xv^\T\thetav')\right] (\Xv^\T\thetavd)^2\right\}^2\right]\le \CONST \dimh^2\|\ups-\ups'\|^2.
\end{EQA}
Finally we can estimate 
\begin{EQA}
&&\nquad\E\left[\left\{\fv_{\etavd}(\Xv^\T\thetav)-\fv_{\etavd}(\Xv^\T\thetav')\right\}^2\right]\\
	&=&2\sum_{k=1}^{\dimh}\sum_{l=k}^{\dimh}(1-1_{k=l}/2) \eta^\circ_k\eta^\circ_l\\
		&&\phantom{ 2\sum_{k=1}^{\dimh}\sum_{l=k}^{\dimh}}\E\left[(\basX_k(\Xv^\T\thetav)-\basX_k(\Xv^\T\thetav'))(\basX_l(\Xv^\T\thetav)-\basX_l(\Xv^\T\thetav'))\right].
\end{EQA}
Using \eqref{eq: bound for expectation of basis difference} and very similar arguments as before additionally using that \(\|\etavd\|\le 1/c_{\DF}\)
\begin{EQA}
&&\nquad\E\left[\left\{\fv_{\etavd}(\Xv^\T\thetav)-\fv_{\etavd}(\Xv^\T\thetav')\right\}^2\right]\\
	&\le& \|\thetav-\thetav'\|^2\CONST_{p_{\Xv}}\|\psi'\|_{\infty}^2 s_{\Xv}^4 17^2\sum_{k=1}^{\dimh} \eta^\circ_k 2^{j_k}\sum_{l=k}^{\dimh} \eta^\circ_l 2^{j_l}1_{\{I_k\cap I_l\neq \emptyset\}}\\
	&\le& \|\thetav-\thetav'\|^2\CONST_{p_{\Xv}}\|\psi'\|_{\infty}^2 s_{\Xv}^4 17^{5/2}4\dimh^{3/2}\frac{1}{c_{\DF}}\sum_{k=1}^{\dimh}2^{j_k/2}  \eta^\circ_k \\
	&\le& \|\thetav-\thetav'\|^2\CONST_{p_{\Xv}}\|\psi'\|_{\infty}^2 s_{\Xv}^4 17^{5/2}4\dimh^{2}\frac{1}{c_{\DF}^2}.
	\label{eq: bound f thetav minus f thetavd with some etavd}
\end{EQA}
Putting these bounds together gives
\begin{EQA}
&&\nquad n\gammav^\T \DF_{\dimh}^{-1}\E\left[(\varsigmav_{1,\dimh}(\upsilonv')-\varsigmav_{1,\dimh}(\upsilonv))(\varsigmav_{1,\dimh}(\upsilonv')-\varsigmav_{1,\dimh}(\upsilonv))^\T \right]\DF_{\dimh}^{-1}\gammav^\T\\
 &\le& \CONST_{\bb{(\CS \DF_{1}) }}^2\dimh^2 \|\DF_{\dimh}(\ups-\ups')\|^2	\omega^2.
\end{EQA}
This yields \( \bb{(\CS \DF_{1})} \) with
\begin{EQA}
\nu_{1,\dimh}^2&=&\tilde \nu^{2}\CONST_{\bb{(\CS \DF_{1}) }}^2\dimh^2.
\end{EQA} 
\end{proof}

\subsubsection{Condition \(\bb{(\cc{L}_{0})} \)}
\begin{lemma}
 \label{lem: condition L_0 is satisfied}
The condition \(\bb{(\cc{L}_{0})} \) is satisfied where 
\begin{EQA}[c]
\delta(\rr)=\frac{\CONST_{\bb{(\cc{L}_{0})}} \left\{\dimh^{3/2}+\CONST_{bias}\dimh^{5/2}\right\}\rr}{c_{\DF}\sqrt n},
\end{EQA}
where \(\CONST_{\delta,1},\CONST_{\delta,2}>0\) are polynomials of \(\|\psi\|_{\infty},\|\psi'\|_{\infty},\|\psi''\|_{\infty},\CONST_{\|\fvs\|},L_{\nabla \Phi},s_{\Xv}\).
\end{lemma}

\begin{proof}

We will show that \(\frac{1}{n}\|\DF_{\dimh}^2(\upsilonv)-\DF_{\dimh}^2(\upsilonvs_{\dimh})\|\le c^2_{\DF}\delta(\rr)\), which will give the claim due to
\begin{EQA}[c]
\|I_{\dimtotal}-\DF_{\dimh}^{-1}\nabla_{\dimtotal}^{2}\E[\LL(\upsilonv)]\DF_{\dimh}^{-1}\|\le \frac{1}{nc^2_{\DF}}\|\DF^2_{\dimh}(\upsilonv)-\DF^2_{\dimh}(\upsilonvs_{\dimh})\|.
\end{EQA}
We represent
\begin{EQA}
-\nabla_{\dimtotal}^{2}\E[\LL_{\dimh}(\upsilonv)]&\eqdef&\DF_{\dimh}^{2}(\upsilonv)=n d_{\dimh}^2(\upsilonv)+n r_{\dimh}^2(\upsilonv) ,\\
	n d_{\dimh}^2(\upsilonv)  &=&	n \left( 
      \begin{array}{cc}
        d_{\thetav}^{2}(\upsilonv) &a_{\dimh}(\upsilonv) \\
        a_{\dimh}^{\T}(\upsilonv) & h_{\dimh}^{2}(\upsilonv)\\
      \end{array}  
    \right)\eqdef  \left( 
      \begin{array}{cc}
        \DP(\upsilonv)^2 & A_{\dimh}^\T(\upsilonv) \\
        A_{\dimh} (\upsilonv) & \HH_{\dimh}^2(\upsilonv)\\
      \end{array}  
    \right),\\
  r_{\dimh}^2(\upsilonv)&=&	 \E\Bigg[\left(\fv_{\etav}(\Xv^{\T}\thetav)- g(\Xv)\right)\left( 
      \begin{array}{cc}
        v_{\thetav}^{2}(\upsilonv) & b_{\dimh}(\upsilonv) \\
        b_{\dimh}^{\T}(\upsilonv) & 0\\
      \end{array}  
    \right)\Bigg],\\
      v_{\thetav}^{2}(\upsilonv)&=&     2\fv''_{\etav}(\Xv^{\T}\thetav)\nabla\Phi_{\thetav}^{\T} \Xv(\Xv)^{\T}\nabla\Phi_{\thetav}\\
      	&& +|\fv'_{\etav}(\Xv^{\T}\thetav)|^2\Xv^{\T}\nabla^2\varphi_{\thetav}^{\T}[\Xv,\cdot,\cdot],\\
      b_{\dimh}(\upsilonv)&=&\nabla\Phi_{\thetav}\Xv^{\T}\basX'^\T(\Xv^{\T}\thetav),
\end{EQA}
such that 
\begin{EQA}
\frac{1}{n}\|\DF^2_{\dimh}(\upsilonv)-\DF^2_{\dimh}(\upsilonvs_{\dimh})\|&\le&\frac{1}{n}\Big( \|\DP^2(\upsilonv)-\DP^2(\upsilonvs_{\dimh})\|+2\|A_{\dimh}(\upsilonv)-A_{\dimh}(\upsilonvs_{\dimh})\|\\
  &&+\|\HH_{\dimh}^2(\upsilonv)-\HH_{\dimh}^2(\upsilonvs_{\dimh})\|+\|r^2_{\dimh}(\upsilonv)-r^2_{\dimh}(\upsilonvs_{\dimh})\|\Big),
\end{EQA}
so that we can calculate separately
\begin{EQA}
&&\nquad\frac{1}{n}\|\DP^2(\upsilonv)-\DP^2(\upsilonvs_{\dimh})\|\\
  &\le& \E[\|\Xv\|^2\left\{|((\fv'_{\etav})^2-(\fv'_{\etavs_{\dimh}})^2)(\Xv^\T\thetav)|\right.\\
  &&\left.+|(\fv'_{\etavs_{\dimh}})^2(\Xv^\T\thetav)-(\fv'_{\etavs_{\dimh}})^2(\Xv^\T\thetavs_{\dimh})|\right.\\
  &&\left.+2|(\fv'_{\etavs_{\dimh}})^2(\Xv^\T\thetavs_{\dimh})|\|\nabla\Phi(\thetav)^\T\Xv-\nabla\Phi(\thetavs_{\dimh})^\T\Xv\|\right\}].
\end{EQA}
Using Lemma \ref{lem: bound for supnorm of fv etavs dimh} we find
\begin{EQA}
&&\nquad|(\fv'_{\etavs_{\dimh}})^2(\Xv^\T\thetavs_{\dimh})|\|\nabla\Phi(\Xv^\T\thetav)-\nabla\Phi(\Xv^\T\thetavs_{\dimh})\|\\
	&\le& \|\psi'\|_{\infty}(C_{\|\fv\|}+1)\sqrt{2} L_{\nabla \Phi}\|\thetav-\thetavs_{\dimh}\|.
\end{EQA}
Furthermore we have \(M(j)\subset \{1,\ldots,\dimh\}\) in \eqref{def: index set M j}
\begin{EQA}
\label{eq: bound for e'-basis}
\E|(\fv'_{\etav}-\fv'_{\etavs_{\dimh}})(\Xv^\T\thetav)|&\le&\sum_{k=1}^{\dimh}|\eta_k-{\etas_{\dimh}}_k|\E| \basX_k'(\Xv^\T\thetav)|\\
  &\le&\CONST_{p_{\Xv\thetav}}\|\psi'\|\|\etav-\etavs_{\dimh}\|\left(\sum_{k=1}^{j_\dimh}2^{j_k}|M(j)|\right)^{1/2}\\
  &\le& \CONST\|\etav-\etavs_{\dimh}\|\|\psi'\|_{\infty}\dimh.
\end{EQA}
This implies using \eqref{eq: bound for f'} , \eqref{eq: bound for e'-basis} and \eqref{eq: bound for f etav minus f etavd}
\begin{EQA}
&&\nquad\E\left[|((\fv'_{\etav})^2-(\fv'_{\etavs_{\dimh}})^2)(\Xv^\T\thetav)|\right]\\
  &\le& \E\left[|\fv'_{\etav}(\Xv^\T\thetav)|+|\fv'_{\etavs_{\dimh}}(\Xv^\T\thetav)|)|(\fv'_{\etav}-\fv'_{\etavs_{\dimh}})(\Xv^\T\thetav)|\right]\\
  &\le&\E\left[ (|(\fv'_{\etav}-\fv'_{\etavs_{\dimh}})(\Xv^\T\thetav)|+2|\fv'_{\etavs_{\dimh}}(\Xv^\T\thetav)|)|(\fv'_{\etav}-\fv'_{\etavs_{\dimh}})(\Xv^\T\thetav)|\right]\\
  &\le& \E \left[|(\fv'_{\etav}-\fv'_{\etavs_{\dimh}})(\Xv^\T\thetav)|^2\right]\\
  	&&+2\|\psi'\|_{\infty}(C_{\|\fv\|}+1)\sqrt{2}\E[|(\fv'_{\etav}-\fv'_{\etavs_{\dimh}})(\Xv^\T\thetav)|]\\
  &\le&\|\psi'\|_{\infty}\left(2\|\psi'\|_{\infty}(C_{\|\fv\|}+1)\sqrt{2}+\CONST\frac{\rr\dimh}{\sqrt{n}}\right)\dimh\|\etav-\etavs_{\dimh}\|\eqdef \CONST \dimh \|\etav-\etavs_{\dimh}\|,
\end{EQA}
where we used \(\frac{\rr\dimh}{\sqrt{n}}\to 0\) for \(\rr^2\le \CONST\dimh\). Finally we derive with \eqref{eq: bound for f'-difference} and \eqref{eq: bound for f'}
\begin{EQA}
&&\nquad\|(\fv'_{\etavs_{\dimh}})^2(\Xv^\T\thetav)-(\fv'_{\etavs_{\dimh}})^2(\Xv^\T\thetavs_{\dimh})\|\\
  &\le&(\|\fv'_{\etavs_{\dimh}}(\Xv^\T\thetavs_{\dimh})\|+\|\fv'_{\etavs_{\dimh}}(\Xv^\T\thetav)\|)\|\fv'_{\etavs_{\dimh}}(\Xv^\T\thetav)-\fv'_{\etavs_{\dimh}}(\Xv^\T\thetavs_{\dimh})\|\\
  &\le& 4\sqrt{2}\|\psi'\|_{\infty}(C_{\|\fv\|}+1)^2\sqrt{\dimh}\|\psi''\|_{\infty}s_{\Xv}\|\thetav-\thetavs_{\dimh}\|.
\end{EQA}
Collecting everything yields with some constant \(\CONST>0\)
\begin{EQA}
\frac{1}{n}\|\DP^2(\upsilonv)-\DP^2(\upsilonvs_{\dimh})\| &\le &\CONST \dimh \|\upsilonv-\upsilonvs_{\dimh}\|.
\end{EQA}
Furthermore 
\begin{EQA}
&&\nquad\frac{1}{n}\|\HH_{\dimh}^2(\upsilonv)-\HH_{\dimh}^2(\upsilonvs_{\dimh})\|\\
	&=&\sup_{\substack{\gammav\in\R^{\dimh}\\  \|\gammav\|=1}}\sum_{k,l=1}^{\dimh}\gamma_k\gamma_l\left(\E[\basX_k\basX_l(\Xv^\T\thetav)]-\E[\basX_k\basX_l(\Xv^\T\thetavs_{\dimh})] \right)1_{I_l\cap I_k\neq\emptyset}\\
	&\le& 2\sup_{\substack{\gammav\in\R^{\dimh}\\  \|\gammav\|=1}}\sum_{k=1}^{\dimh}\sum_{l=k}^{\dimh}\gamma_k\gamma_l\E\left[\left(\basX_k(\Xv^\T\thetav)-\basX_k(\Xv^\T\thetavs_{\dimh})\right) \basX_l(\Xv^\T\thetav)\right]1_{I_l\cap I_k\neq\emptyset}\\
	&&+2\sup_{\substack{\gammav\in\R^{\dimh}\\  \|\gammav\|=1}}\sum_{k=1}^{\dimh}\sum_{l=k}^{\dimh}\gamma_k\gamma_l
\E\left[\basX_k(\Xv^\T\thetavs_{\dimh})\left(\basX_l(\Xv^\T\thetav)-\basX_l(\Xv^\T\thetavs_{\dimh})\right) \right]1_{I_l\cap I_k\neq\emptyset} .
\end{EQA}
Using \eqref{eq: bound expectation basis l times difference of basis vectors} and \eqref{eq: number of intersections} this gives
\begin{EQA}
&&\nquad\sum_{k=1}^{\dimh}\sum_{l=k}^{\dimh}\gamma_k\gamma_l\E\left[\left(\basX_k(\Xv^\T\thetav)-\basX_k(\Xv^\T\thetavs_{\dimh})\right) \basX_l(\Xv^\T\thetav)\right]1_{I_l\cap I_k\neq\emptyset}\\
	&\le &\|\thetav-\thetavs_{\dimh}\|\|\psi'\| s_{\Xv}^2 17\CONST_{p_{\Xv}} \sum_{k=1}^{\dimh}\gamma_k2^{j_k}\sum_{l=k}^{\dimh}\gamma_l 1_{I_l\cap I_k\neq\emptyset}\\
	&\le &\|\thetav-\thetavs_{\dimh}\|\|\psi'\| s_{\Xv}^2 17^{3/2}\CONST_{p_{\Xv}}\sum_{k=1}^{\dimh}\gamma_k2^{j_k/2}\left(\sum_{j=j_k}^{j_\dimh}2^j\right)^{1/2}\\
	&\le &\|\thetav-\thetavs_{\dimh}\|\|\psi'\| s_{\Xv}^2 17^{3/2}\CONST_{p_{\Xv}}\sqrt{\dimh}\left(\sum_{j=1}^{j_\dimh}2^{j_k}\right)^{1/2}\\
	&\le &\|\thetav-\thetavs_{\dimh}\|\|\psi'\| s_{\Xv}^2 17^{3/2}\CONST_{p_{\Xv}}\dimh,
\end{EQA}
and
\begin{EQA}
&&\nquad\sum_{k=1}^{\dimh}\sum_{l=k}^{\dimh}\gamma_k\gamma_l\E\left[\left(\basX_l(\Xv^\T\thetav)-\basX_l(\Xv^\T\thetavs_{\dimh})\right) \basX_k(\Xv^\T\thetav)\right]1_{I_l\cap I_k\neq\emptyset}\\
	&\le &\|\thetav-\thetavs_{\dimh}\|\|\psi'\| s_{\Xv}^2 17(\CONST_{p_{\Xv}}+\|\psi\|_\infty)\sum_{k=1}^{\dimh}\gamma_k2^{j_k/2}\sum_{l=k}^{\dimh}\gamma_l 2^{j_l/2} 1_{I_l\cap I_k\neq\emptyset}\\
	&\le &\|\thetav-\thetavs_{\dimh}\|\|\psi'\| s_{\Xv}^2 17^{3/2}(\CONST_{p_{\Xv}}+\|\psi\|_\infty)\sum_{k=1}^{\dimh}\gamma_k \left(\sum_{j=j_k}^{j_\dimh}2^{2j}\right)^{1/2}\\
	&\le &\|\thetav-\thetavs_{\dimh}\|\|\psi'\| s_{\Xv}^2 17^{3/2}(\CONST_{p_{\Xv}}+\|\psi\|_\infty)\dimh.
\end{EQA}
Consequently with some constant \(\CONST_{\HH}\in\R\)
\begin{EQA}[c]\label{eq: bound for L0 etav matrix component}
\frac{1}{n}\|\HH_{\dimh}^2(\upsilonv)-\HH_{\dimh}^2(\upsilonvs_{\dimh})\|\le \frac{\CONST_{\HH}\dimh}{\sqrt{n}c_{\DF}}\|\DF(\ups-\upss_{\dimh})\|.
\end{EQA}
Again with some constant \(\CONST>0\)
\begin{EQA}
&&\nquad\frac{1}{n}\|A_{\dimh}(\upsilonv)-A_{\dimh}(\upsilonvs_{\dimh})\|
\le  \CONST\left(\E\Big[\left\|\fv'_{\etav}(\Xv_{1}^{\T}\thetav) -\fv'_{\etavs_{\dimh}}(\Xv_{1}^{\T}\thetavs_{\dimh}) \right\|^2\Big]^{1/2}\right.\\
	&&\left.+\E\Big[\left\| \nabla\Phi(\thetav) -\nabla\Phi({\thetavs_{\dimh}})\right\|^2\Big]^{1/2}+\E\Big[\left\|\basX(\Xv_{1}^{\T}\thetav)- \basX(\Xv_{1}^{\T}\thetavs_{\dimh})\right\|^2\Big]^{1/2}\right).
\end{EQA}
Note that using \eqref{eq: bound f thetav minus f thetavd with some etavd}
\begin{EQA}
&&\nquad\E\Big[\left\|\basX(\Xv_{1}^{\T}\thetav)- \basX(\Xv_{1}^{\T}\thetavs_{\dimh})\right\|^2\Big]\\
	&\le& \sup_{\substack{\etavd\in\R^{\dimh}\\ \|\etavd\|=1}}\E\left[\left\{\fv_{\etavd}(\Xv^\T\thetav)-\fv_{\etavd}(\Xv^\T\thetavs_{\dimh})\right\}^2\right]\\
	&\le& \|\thetav-\thetavs_{\dimh}\|^2\CONST_{p_{\Xv}}\|\psi'\|_{\infty}^2 s_{\Xv}^4 17^{5/2}4\dimh^{2}.
\end{EQA}
Using \eqref{eq: bound derivative f etavthetav minus t etavdthetavd} this yields 
\begin{EQA}
\frac{1}{n}\|A_{\dimh}(\upsilonv)-A_{\dimh}(\upsilonvs_{\dimh})\|&\le & \CONST \dimh\|\ups-\ups'\|.
\end{EQA}

Finally we estimate the fourth term. 
\begin{EQA}\label{eq: two terms restterm L0}
\|r_{\dimh}^2(\upsilonv)-r_{\dimh}^2(\upsilonvs_{\dimh})\|  &\le&\E[|\fv_{\etav}(\Xv^{\T}\thetav)- \fv_{\etavs_{\dimh}}(\Xv^{\T}\thetavs_{\dimh})|\|\tilde \VF^2_{\dimh}(\upsilonv)\|]\\
  	&&+\E[|\fv_{\etavs_{\dimh}}(\Xv^{\T}\thetavs_{\dimh})- g(\Xv)|\|\tilde \VF^2_{\dimh}(\upsilonvs_{\dimh})-\tilde \VF^2_{\dimh}(\upsilonv)\|].
 \end{EQA} 
We estimate separately
\begin{EQA}
&&\nquad\E[|\fv_{\etav}(\Xv^{\T}\thetav)- \fv_{\etavs_{\dimh}}(\Xv^{\T}\thetavs_{\dimh})|\|(\tilde \VF^2_{\dimh})(\upsilonv)\|]\\
&\le& \E[|\fv_{\etav}(\Xv^{\T}\thetav)- \fv_{\etavs_{\dimh}}(\Xv^{\T}\thetavs_{\dimh})|\|v_{\thetav}^{2}(\upsilonv)\|]\\
	&&+\E[|\fv_{\etav}(\Xv^{\T}\thetav)- \fv_{\etavs_{\dimh}}(\Xv^{\T}\thetavs_{\dimh})|\| b_{\dimh}(\upsilonv)\|]
 \end{EQA} 
To bound the first term, first note that again using the wavelet structure
\begin{EQA}
\label{eq: bound for f''}
|\fv''_{\etav}(\Xv^{\T}\thetav)|&\le&|\fv''_{\etav-\etavs_{\dimh}}(\Xv^{\T}\thetav)|+|\fv''_{\etavs_{\dimh}}(\Xv^{\T}\thetav)|  \\
  &\le&\sqrt{34}\|\psi''\|_{\infty}\left(\sum_{k=0}^{\dimh} (\eta_k-{\etas_{\dimh}}_k)^2 
\right)^{1/2}\left(\sum_{j=0}^{j_\dimh-1} 2^{5j}
\right)^{1/2}+\CONST_{\|\fv_{\etavs_{\dimh}}''\|_{\infty}}\\
  &\le& \sqrt{34}\|\psi''\|_{\infty}\|\DF_{\dimh}(\upsilonv-\upsilonvs_{\dimh})\|\frac{\dimh^{5/2}}{c_{\DF}\sqrt{n}}+\CONST_{\|\fv_{\etavs_{\dimh}}''\|_{\infty}},
\end{EQA}
which can be treated as a constant as \(\dimh^5/n \to 0\). Furthermore using \eqref{eq: bound for f'} we have for any \(\varphi\in\R^{\dimp-1}\) with \(\|\varphi\|=1\)
\begin{EQA}
\||\fv'_{\etav}(\Xv^{\T}\thetav)|^2\nabla^2\Phi_{\thetav}^{\T}[\Xv,\varphi,\cdot]\|_{\R^{\dimp}} &\le & 34\|\psi'\|^2_{\infty}C^2_{\|\etavs_{\dimh}\|}s_{\Xv}^2\|\nabla^2\Phi_{\thetavs_{\dimh}}\|_{\infty}.
\end{EQA}
To control \(\E\| b_{\dimh}(\upsilonv)\|^2\) we use \eqref{eq: bound for scalar produc basis first derivative} to bound
\begin{EQA}
\E\| b_{\dimh}(\upsilonv)\|^2&\le& s_{\Xv}^2\sum_{k=1}^{\dimh} \E\basX_k'(\Xv^\T\thetav)^2 \\
	&\le& s_{\Xv}^217^2 \CONST_{p_{\Xv}}^2 \|\psi'\|^2_{\infty}\sum_{k=1}^{\dimh} 2^{2j_k}  \\
	&\le&s_{\Xv}^2 17^2 \CONST_{p_{\Xv}}^2 \|\psi'\|^2_{\infty}\dimh^3.	\label{eq: bound for expectation basis derivative}
\end{EQA}
This implies 
for the first summand in \eqref{eq: two terms restterm L0} with constants \(\CONST, \CONST'>0\) large enough
 \begin{EQA} 
&&\nquad\E[|\fv_{\etav}(\Xv^{\T}\thetav)- \fv_{\etavs_{\dimh}}(\Xv^{\T}\thetavs_{\dimh})|\|\tilde \VF^2_{\dimh}(\upsilonv)\|]\\  	
  &\le&\left(\E[|\fv_{\etav}(\Xv^{\T}\thetav)- \fv_{\etavs_{\dimh}}(\Xv^{\T}\thetav)|^2]^{1/2}\right.\\
  &&\left.+\E[|\fv_{\etavs_{\dimh}}(\Xv^{\T}\thetav)- \fv_{\etavs_{\dimh}}(\Xv^{\T}\thetavs_{\dimh})|^2]^{1/2}\right)\CONST\dimh^{3/2}\\
  &\le& \CONST\dimh^{3/2} \|\upsilonv-\upsilonvs_{\dimh}\|+\CONST\dimh^{3/2}\E[|\fv_{\etav}(\Xv^{\T}\thetav)- \fv_{\etavs_{\dimh}}(\Xv^{\T}\thetav)|^2]^{1/2}.
\end{EQA}
We estimate using \eqref{eq: bound for L0 etav matrix component}, \(\rr\dimh^{3/2}/\sqrt{n}\to 0\) for \(\rr\le \rups\) and constants \(\CONST,\CONST'>0\) large enough
\begin{EQA}
&&\nquad\E[|\fv_{\etav}(\Xv^{\T}\thetav)- \fv_{\etavs_{\dimh}}(\Xv^{\T}\thetav)|^2]^{1/2}=\frac{1}{\sqrt{n}}\|\HH_{\dimh}(\upsilonv)(\etav-\etavs_{\dimh})\|\\
	&\le&\frac{1}{\sqrt{n}}\|\HH_{\dimh}^2(\upsilonv)-\HH_{\dimh}^2(\upsilonvs_{\dimh})\|^{1/2}\|(\etav-\etavs_{\dimh})\|+\frac{1}{\sqrt{n}}\|\HH_{\dimh}(\upsilonvs_{\dimh})(\etav-\etavs_{\dimh})\|\\
	&\le&\left(\frac{1}{\sqrt{n}}\|\HH_{\dimh}^2(\upsilonv)-\HH_{\dimh}^2(\upsilonvs_{\dimh})\|^{1/2}\frac{1}{\sqrt{n}c_{\DF}}+\frac{1}{\sqrt{n}}\right)\|\HH_{\dimh}(\upsilonvs_{\dimh})(\etav-\etavs_{\dimh})\|\\
	&\le& \left\{\left(\rr\dimh^{3/2}/\sqrt{n}\right)^{1/2}+1\right\}\frac{\CONST}{\sqrt{n}c_{\DF}}\|\DF_{\dimh}(\upsilonv-\upsilonvs_{\dimh})\|\\
	&\le&\frac{\CONST'}{\sqrt{n}c_{\DF}}\|\DF_{\dimh}(\upsilonv-\upsilonvs_{\dimh})\|.
\end{EQA}
We also find
\begin{EQA}
&&\nquad|\fv_{\etavs_{\dimh}}(\Xv^{\T}\thetav)- \fv_{\etavs_{\dimh}}(\Xv^{\T}\thetavs_{\dimh})|\\
	&\le& \left(\sum_{k=1}^{\dimh} (\etavs_{\dimh})_k^2 k^{2\alpha}
\right)^{1/2}\left(\sum_{k=1}^{\dimh}|\basX_k'(\Xv^{\T}\thetavs_{\dimh})|^2 k^{-2\alpha}\right)^{1/2}L_{\nabla \Phi}\|\Xv\|\|\thetav-\thetavs_{\dimh}\|\\
  &\le&2\sqrt{34} C_{\|\etavs_{\dimh}\|}\sqrt 2 L_{\nabla \Phi}s_{\Xv}\|\psi'\|_{\infty}\|\thetav-\thetavs_{\dimh}\|.
\end{EQA}
Consequently 
\begin{EQA}
\E[|\fv_{\etav}(\Xv^{\T}\thetav)- \fv_{\etavs_{\dimh}}(\Xv^{\T}\thetavs_{\dimh})|\|\tilde \VF^2_{\dimh}(\upsilonv)\|]&\le& \frac{\CONST\dimh^{3/2}}{\sqrt{n}c_{\DF}}\|\DF_{\dimh}(\upsilonv-\upsilonvs_{\dimh})\|.
\end{EQA}
%
%
Furthermore using that \(|\fv_{\etavs_{\dimh}}(\Xv^{\T}\thetavs_{\dimh})- g(\Xv)|\le \CONST_{bias}\)
\begin{EQA}
&&\nquad\E[|\fv_{\etavs_{\dimh}}(\Xv^{\T}\thetavs_{\dimh})- g(\Xv)|\|\tilde \VF^2_{\dimh}(\upsilonvs_{\dimh})-\tilde \VF^2_{\dimh}(\upsilonv)\|]\\
	&\le& \CONST_{bias}\left(\E[\|v_{\thetav}^{2}(\upsilonvs_{\dimh})-v_{\thetav}^{2}(\upsilonv)\|]+2\E[\|b_{\dimh}(\upsilonvs_{\dimh})-b_{\dimh}(\upsilonv)\|]\right).
\end{EQA}
For this we estimate with some constants \(\CONST_i\) that only depend on \(\|\nabla^2\Phi_{\thetavs_{\dimh}}\|,\allowbreak  s_{\Xv}, \allowbreak \CONST_{\|\fv_{\etavs_{\dimh}}'\|_{\infty}},\allowbreak \CONST_{\|\fv_{\etavs_{\dimh}}''\|_{\infty}}\), etc.
\begin{EQA}
&&\nquad\|v_{\thetav}^{2}(\upsilonvs_{\dimh})-v_{\thetav}^{2}(\upsilonv)\|\\
	&\le& \|2\fv''_{\etav}(\Xv^{\T}\thetav)\nabla\Phi_{\thetav}^{\T} \Xv(\Xv)^{\T}\nabla\Phi_{\thetav}-2\fv''_{\etavs_{\dimh}}(\Xv^{\T}\thetavs_{\dimh})\nabla\Phi_{\thetavs_{\dimh}}^{\T} \Xv(\Xv)^{\T}\nabla\Phi_{\thetavs_{\dimh}}\| \\
	&&+\||\fv'_{\etavs_{\dimh}}(\Xv^{\T}\thetavs_{\dimh})|^2\Xv^{\T}\nabla^2\Phi_{\thetavs_{\dimh}}^{\T}[\Xv,\cdot,\cdot]-|\fv'_{\etav}(\Xv^{\T}\thetav)|^2\Xv^{\T}\nabla^2\varphi_{\thetav}^{\T}[\Xv,\cdot,\cdot]\|\\
	&\le& \CONST_1\left||\fv'_{\etavs_{\dimh}}(\Xv^{\T}\thetavs_{\dimh})|^2-|\fv'_{\etav}(\Xv^{\T}\thetav)|^2\right|+\CONST_2|\fv''_{\etavs_{\dimh}}(\Xv^{\T}\thetavs_{\dimh})-\fv''_{\etav}(\Xv^{\T}\thetav)|\\
		&&+ \CONST_3 \|\thetav-\thetavs_{\dimh}\|.
\end{EQA}
With the same arguments as those used for the bound of \(\frac{1}{n}\|\DP^2(\upsilonv)-\DP^2(\upsilonvs_{\dimh})\|\)
\begin{EQA}[c]
\E\left||\fv'_{\etavs_{\dimh}}(\Xv^{\T}\thetavs_{\dimh})|^2-|\fv'_{\etav}(\Xv^{\T}\thetav)|^2\right|\le \CONST\dimh \|\upsilonv-\upsilonvs_{\dimh}\|.
\end{EQA}
Furthermore
\begin{EQA}
|\fv''_{\etavs_{\dimh}}(\Xv^{\T}\thetavs_{\dimh})-\fv''_{\etav}(\Xv^{\T}\thetav)|&\le&|\fv''_{\etavs_{\dimh}}(\Xv^{\T}\thetavs_{\dimh})-\fv''_{\etavs_{\dimh}}(\Xv^{\T}\thetav)|\\
	&&+|\fv''_{\etavs_{\dimh}}(\Xv^{\T}\thetav)-\fv''_{\etav}(\Xv^{\T}\thetav)|
\end{EQA}
Using
 \begin{EQA}
|N(j)|\eqdef\Big|\Big\{&& k\in\{(2^{j}-1)17,\ldots,(2^{j+1}-1)17-1\}:\\
	&& |\basX''_{k}(\Xv_{i}^{\T}\thetav')-\basX''_{k}(\Xv_{i}^{\T}\thetav)|> 0 \Big\}\Big|\le 34.
\end{EQA}
we estimate
\begin{EQA}
&&\nquad|\fv''_{\etavs_{\dimh}}(\Xv^{\T}\thetavs_{\dimh})-\fv''_{\etavs_{\dimh}}(\Xv^{\T}\thetav)|\\
	&\le& \sqrt{34}\|\psi'''\|_{\infty}\|\thetav-\thetavs_{\dimh}\|
\left(\sum_{k=1}^{\dimh}{\etas_{\dimh}}^2 k^{-2\alpha}\right)^{1/2}\left(\sum_{j=1}^{j_{\dimh}} 2^{(7-2\alpha)j)}\right)^{1/2}\\
	&\le& \CONST\dimh^{3/2}\|\thetav-\thetavs_{\dimh}\|,
\end{EQA}
and
\begin{EQA}
|\fv''_{\etavs_{\dimh}}(\Xv^{\T}\thetav)-\fv''_{\etav}(\Xv^{\T}\thetav)|\le \sqrt{17}\|\etav-\etavs\|\left(\sum_{j=1}^{j_{\dimh}} 2^{5j)}\right)^{1/2}\le \CONST \dimh^{5/2}\|\thetav-\thetavs_{\dimh}\|
\end{EQA}
Furthermore
\begin{EQA}
\E[\|b_{\dimh}(\upsilonvs_{\dimh})-b_{\dimh}(\upsilonv)\|]&\le& \CONST\E\|\basX'(\Xv^{\T}\thetav)-\basX'(\Xv^{\T}\thetavs_{\dimh})\|\\
	&&+\CONST\E[\|\basX'(\Xv^{\T}\thetavs_{\dimh})\|^2]^{1/2}\|\thetav-\thetavs_{\dimh}\|.
\end{EQA}
By \eqref{eq: bound for expectation basis derivative} we have
\begin{EQA}[c]
\E[\|\basX'(\Xv^{\T}\thetavs_{\dimh})\|^2]^{1/2}\le 17\CONST_{p_{\Xv}} \|\psi'\|_{\infty}\dimh^{3/2}.
\end{EQA}
Furthermore
\begin{EQA}
\E\|\basX'(\Xv^{\T}\thetav)-\basX'(\Xv^{\T}\thetavs_{\dimh})\|&\le&\E\left[\|\basX'(\Xv^{\T}\thetav)-\basX'(\Xv^{\T}\thetavs_{\dimh})\|^{2}\right]^{1/2}\\	&=&\left(\sum_{k=1}^{\dimh}\E\left[(\basX_k'(\Xv^{\T}\thetav)-\basX_k'(\Xv^{\T}\thetavs_{\dimh}))^{2}\right]\right)^{1/2}
\end{EQA}
With \eqref{eq: bound for expectation of basis difference first derivative} we find
\begin{EQA}
\E\|\basX'(\Xv^{\T}\thetav)-\basX'(\Xv^{\T}\thetavs_{\dimh})\|&\le&\|\psi''\|_{\infty} s_{\Xv}^2 17\|\thetav-\thetav'\|\left(\sum_{k=1}^{\dimh}2^{4j_k}\right)^{1/2}\\
	&\le&|\psi''\|_{\infty} s_{\Xv}^2 17\|\thetav-\thetav'\|\dimh^{5/2} ,
\end{EQA}
Together this gives
\begin{EQA}
&&\nquad\E[|\fv_{\etavs_{\dimh}}(\Xv^{\T}\thetavs_{\dimh})- g(\Xv)|\|\tilde \VF^2_{\dimh}(\upsilonvs_{\dimh})-\tilde \VF^2_{\dimh}(\upsilonv)\|]\\
	&\le& \CONST\dimh^{3/2}\|\ups-\Pi_{\dimtotal}\upss\|+ \CONST_{bias}\CONST\dimh^{5/2} .
\end{EQA}
Collecting everything we find 
\begin{EQA}[c]
\frac{1}{n}\|\DF^2_{\dimh}(\upsilonv)-\DF^2_{\dimh}(\upsilonvs_{\dimh})\|\le \frac{\CONST}{\sqrt{n}c_{\DF}} \left\{\dimh^{3/2}+ \CONST_{bias}\dimh^{5/2}\right\}\|\DF_{\dimh}(\upsilonv-\upsilonvs_{\dimh})\|.
\end{EQA}
Such that
\begin{EQA}[c]
\delta(\rr)=\frac{\CONST_{\bb{(\cc{L}_{0})}} \left\{\dimh^{3/2}+\CONST_{bias}\dimh^{5/2}\right\}\rr}{c_{\DF}\sqrt n}.
\end{EQA}
\end{proof}


\subsubsection{Condition \({(\cc{L}{\rr})}\) }

Before we start with the actual proof we cite the following important result that will be used in our arguments. 

The next result is a variant of Theorem 4.3 of \cite{Mendelson} and is the key tool of this subsection. 
\begin{theorem}
\label{theo: mendelson theo generalized}
Let for a sequence of independent \(\Xv_i\in \mathcal X\) for some space \(\mathcal X\)
\begin{EQA}[c]
F(\upsilonv)=\sum_{i=1}^n f_i(\upsilonv,\Xv_i)-e,\, \upsilonv\in\Ups\subset\R^{\dimtotal}
\end{EQA}
and assume that with \(\rr>\rr_{\bb{Q}}>0\), \(\Upss(\rr)\subset \Ups\) and \(\chi_{\gmi}:[0,2\gmi]\to \R\) defined in \eqref{eq: def of auxiliary function}
\begin{EQA}
	 \E\left[\sup_{\upsilonv\in \Upss(\rr)^c} (\Pn-\P)\chi_{\gmi}(\ups)\right]\le C_{\chi},&\quad \P(e>C_e)\le \tau_e,\,\\
	 \bb{Q}(\gmi)\eqdef\inf_{\upsilonv\in \Upss(\rr)^c}\P\left(f_i(\upsilonv,\Xv_i)\ge \gmi\rr^2/n\right)>0.
\end{EQA}
Choose 
\begin{EQA}[c]
0< \lambda \le \left(\bb{Q}(2\gmi)-2/n - 2C_{\chi}\right)/4.
\end{EQA}
Then for \(\rr^2\ge C_e/(\lambda\gmi)\vee \rr_{\bb{Q}}^2 \) 
\begin{EQA}[c]
\P\left(\inf_{\upsilonv\in \Upss(\rr)^c} F(\upsilonv) \le \lambda\gmi \rr^2 \right)\le \exp\left\{- n\bb{Q}(2\gmi)^2/4\right\}+\tau_e
\end{EQA}
The auxiliary function is defined as
\begin{EQA}\label{eq: def of auxiliary function}
\bar\chi_{u}(t)=\begin{cases}
			0 & t\le u;\\
			t/u-1 & t\in [u,2u];\\
			1 & t\ge 2u;
		\end{cases} & \text{ }& \chi_{\gmi}(\ups)_i\eqdef \bar\chi_{\gmi}(f_i(\upsilonv)).
\end{EQA}
\end{theorem}

\begin{remark}
The proof is nearly the same as that of Theorem 4.3 of \cite{Mendelson}. The set \(\Upss(\rr)^c\subset \R^{\dimtotal}\) is neither star shaped, nor convex but one can still use the same arguments.
\end{remark}

Now we can start with the proof. We point out that in this Section we will distinguish \(\thetav\in S_{1}^{\dimp,+}\) and \(\varphi_{\thetav}\in W_{S}\) with \(\Phi(\varphi_{\thetav})=\thetav\) from each other. The result is summarized in the following Lemma: 
\begin{lemma}
 \label{lem: condition L_rr}
Assume the conditions \((\mathbf{\mathcal A})\). Then for \(n\in\N\) large enough there exist \(c_{\bb{(Q)}}, c_{\bb{(\cc{L}{\rr})}},\CONST>0\) such that with probability \(1-\exp\left\{-\dimh^{3}\xx\right\}-\exp\left\{- n c_{\bb{(Q)}}/4\right\}\)
\begin{EQA}[c]
-\inf_{\upsilonv\in \Upss( \rr)^c}\E[\LL(\upsilonv,\upsilonvs_{\dimh})]> c_{\bb{(\cc{L}{\rr})}}\rr^2/2,
\end{EQA}
as soon as \(\rr^2\ge \CONST(\dimh+\xx)\).
\end{lemma}
\begin{proof}
We will proof this claim using Theorem \ref{theo: mendelson theo generalized}. First note that we have with expectation taken conditioned on  \((\Xv)=(\Xv_i)_{i=1,\ldots,n}\subset\R^{\dimp}\) and using \eqref{eq: cond exp representation for etavs}
\begin{EQA}
\label{eq: decomposition of kulback leibler divergence}
&&\nquad-\E_{\varepsilon}[\LL(\upsilonv,\upsilonvs_{\dimh})]=-\E[\LL(\upsilonv,\upsilonvs_{\dimh})|(\Xv)]\\
	&=&\sum_{i=1}^n\Big[|\fv_{\etav}(\Xv_i^\T\thetav)-\fv_{\etavs}(\Xv^\T\thetavs)|^2-|\fv_{\etavs_{\dimh}}( \Xv_i^\T\thetavs_{\dimh})-\fv_{\etavs}(\Xv^\T\thetavs)|^2\Big]\\
		&\ge&\sum_{i=1}^n\Big[|\fv_{\etav}(\Xv_i^\T\thetav)-\fv_{\etavs}(\Xv_i^\T\thetavs)|^2\Big]-n\E[|\fv_{\etavs_{\dimh}}( \Xv_i^\T\thetavs_{\dimh})-\fv_{\etavs}(\Xv_i^\T\thetavs)|^2]\\
		&&-n\left|(\Pn-\P)|\fv_{\etavs_{\dimh}}( \Xv^\T\thetavs_{\dimh})-\fv_{\etavs}(\Xv_i^\T\thetavs)|^2 \right|.
\end{EQA}
We define 
\begin{EQA}
e&\eqdef& n\E[|\fv_{\etavs_{\dimh}}( \Xv_i^\T\thetavs_{\dimh})-\fv_{\etavs}(\Xv^\T\thetavs)|^2]\\
	&&+n\left|(\Pn-\P)|\fv_{\etavs_{\dimh}}( \Xv^\T\thetavs_{\dimh})-\fv_{\etavs}(\Xv^\T\thetavs)|^2 \right|,
\end{EQA}
such that
\begin{EQA}[c]
-\E[\LL(\upsilonv,\upsilonvs_{\dimh})|(\Xv)]\ge \sum_{i=1}^n\left(\fv_{\etav}(\Xv_i^\T\thetav)-\fv_{\etavs}(\Xv_i^\T\thetavs)\right)^2-e,
\end{EQA}
This hints that Theorem \ref{theo: mendelson theo generalized} gives the desired result. Consider the following list of assumptions:
\begin{description}
\item[(1)] With some \(\CONST>0\) 
\begin{EQA}
\label{eq: cond 1 for applying mendelson}
n\E[|\fv_{\etavs_{\dimh}}( \Xv^\T\thetavs_{\dimh})-\fv_{\etavs}(\Xv^\T\thetavs)|^2]&\le& 3(2+\CONST) {\rr^*}^2,\\
\end{EQA}
\item[(2)] With probability \(1-\exp\left\{-\dimh^{3}\xx\right\}\) and a constant \(\CONST_{\sum}>0\) 
\begin{EQA}
n\left|(\Pn-\P)|\fv_{\etavs_{\dimh}}( \Xv^\T\thetavs_{\dimh})-\fv_{\etavs}(\Xv^\T\thetavs)|^2] \right|&\le& \CONST_{\sum},
\label{eq: cond 2 for applying mendelson}
\end{EQA}
\item[(3)]For some \(\gmi>0\) and for \(n\in\N\) large enough and \(\rr>\sqrt{\dimh}\) 
\begin{EQA}
\label{eq: prob bound from mendelson}
&&\nquad\bb{Q}(2\gmi)\\
	&\eqdef&\inf_{(\thetav,\etav)\in \Upss(\rr)^c}\P\left[\left(\fv_{\etav}(\Xv_i^\T\thetav)-\fv_{\etavs}(\Xv_i^\T\thetavs)\right)^2 \ge \gmi\rr^2/n\right]>0,
\end{EQA}
\end{description}
This means that in terms of Theorem \ref{theo: mendelson theo generalized} under assumptions (1), (2) and (3) we have \(C_e\le 3(2+\CONST) {\rr^*}^2+\CONST_{\sum}\) and \(\tau_e\le \exp\left\{-\dimh^{3}\xx\right\}\). We prove assumptions (1), (2) and (3) in Lemmas \ref{lem: size of C dimh}, \ref{lem: size of C sum} and \ref{lem: size of Q}, which will give that \(C_e\le \CONST_{\dimh}+3(2+\CONST) {\rr^*}^2\) with probability greater than \(1-\ex^{-\dimh^{3}\xx}\) and that \(\bb{Q}(\gmi)>0\) for a certain choice of \(\gmi>0\) small enough and for \(\rr\ge \CONST\sqrt{\dimh}\) with some constant \(\CONST\). Lemma \ref{lem: mendelson theo applied} completes the proof.
\end{proof}

\begin{lemma}
\label{lem: mendelson theo applied}
Under the assumptions (1), (2) and (3) we get
\begin{EQA}[c]
\inf_{\upsilonv\in \Upss(\rr)^c}-\E[\LL(\upsilonv,\upsilonvs_{\dimh})|(\Xv)]\ge \lambda\gmi \rr^2
\end{EQA}
with probability greater than \(1-\exp\left\{-\dimh^{3}\xx\right\}-\exp\left\{- n\bb{Q}(2\gmi)^2/4\right\}\) for 
\begin{EQA}[c]
\rr^2\ge (3(2+\CONST) {\rr^*}^2+\CONST_{\sum} )/(\lambda\gmi)\vee \CONST\dimh,
\end{EQA}
if 
\begin{EQA}[c]
0< \lambda \eqdef \left(\bb{Q}(2\gmi)-2/n + \CONST \sqrt{\frac{\log(n)\dimtotal}{n}}\right)/4,
\end{EQA}
for a constant \(\CONST>0\) which is a function of \(\|\psi\|_{\infty},\|\psi\|_{\infty},s_{\Xv}\).
\end{lemma}

\begin{proof}
This is a direct consequence of Theorem \ref{theo: mendelson theo generalized}. It remains to bound using the proof of Theorem 8.15 of \cite{Kosorok}
\begin{EQA}
\label{eq: Lr Glivenko bound}
&&\nquad\E\left[\sup_{\upsilonv\in \Upss(\rr)^c} (\Pn-\P)\chi_{\gmi}(\ups)\right] \le  \E\left[\sup_{\upsilonv\in \Ups}  (\Pn-\P)\chi_{\gmi}(\ups)\right]\\
	&\le& 2\CONST^*\E\left[\sqrt{\frac{6\{1+\log N(\delta, \mathcal F, L_1(\Pn))\} }{n}}\right]+\delta,
\end{EQA}
where \(N(\delta, \mathcal F,  L_1(\Pn))\) denotes the \(\delta\)-ball covering number of \( \mathcal F\eqdef \{\chi_{\gmi}(\ups):\,\ups\in\Upsilon\}\) with respect to the norm 
\begin{EQA}[c]
\|h\|_{ L_1(\Pn)}=\Pn|h(\Xv)|=\frac{1}{n}\sum_{i=1}^{n}|h(\Xv_i)|.
\end{EQA}
The universal constant \(\CONST^*>0\) comes from Lemma 8.2 of \cite{Kosorok} (\(\CONST^*=K(\exp(x^2)-1)\)).
The function \(\chi_{\gmi}:\Upss\to \R\) is defined via
\begin{EQA}
\bar\chi_{u}(t)=\begin{cases}
			0 & t\le u;\\
			t/u-1 & t\in [u,2u];\\
			1 & t\ge 2u;
		\end{cases} & \text{ }& \chi_{\gmi}(\ups)_i\eqdef \bar\chi_{\gmi}(|\fv_{\etav}(\Xv_i^\T\thetav)-\fv_{\etavs}(\Xv_i^\T\thetavs)|^2 ).
\end{EQA}
We want to bound the right-hand side of \eqref{eq: Lr Glivenko bound}. For this note that
\begin{EQA}[c]
\log N(\delta, \mathcal F, L_1(\Pn))\le \log N(\delta/(L(\Pn)\vee 1), \Ups, \|\cdot\|_2),
\end{EQA}
where
\begin{EQA}[c]
L(\Pn)=\sup_{\ups,\upsd\in\Ups}\frac{\|\chi_{\gmi}(\ups)-\chi_{\gmi}(\upsd)\|_{L_1(\Pn)}}{\|\ups-\upsd\|_2}.
\end{EQA}
We estimate using that we have \( \diam(\Ups_{\dimh})<\CONST\sqrt{\dimh}\)
\begin{EQA}	
&&\nquad|\chi_{\gmi}(\ups)_i-\chi_{\gmi}(\upsd)_i|\\
	&\le&|\fv_{\etav}(\Xv_i^\T\thetav)-\fv_{\etavd}(\Xv_i^\T\thetavd)|^2\\
	&&+2|(\fv_{\etav}(\Xv_i^\T\thetav)-\fv_{\etavd}(\Xv_i^\T\thetavd))(\fv_{\etav}(\Xv_i^\T\thetav)-\fv_{\etavs}(\Xv^\T\thetavs))|\\
	&\le& 2|\fv_{\etav-\etavd}(\Xv_i^\T\thetav)|^2+2|\fv_{\etavd}(\Xv_i^\T\thetav)-\fv_{\etavd}(\Xv_i^\T\thetavd)|^2\\
	&&+\sqrt{2|\fv_{\etav-\etavd}(\Xv_i^\T\thetav)|^2+2|\fv_{\etavd}(\Xv_i^\T\thetav)-\fv_{\etavd}(\Xv_i^\T\thetavd)|^2}\\
		&&|\fv_{\etav}(\Xv_i^\T\thetav)-\fv_{\etavs}(\Xv^\T\thetavs)|\\
	&\le& 2\|\etav-\etavd\|^2 \dimh \|\psi\|_{\infty}^2+2\|\thetav-\thetavd\|^2s_{\Xv^2}\dimh^3\|\psi'\|_{\infty}^2 \|\etavd\|^2\\
	&&+\sqrt{2\|\etav-\etavd\|^2 \dimh \|\psi\|_{\infty}^2+2\|\thetav-\thetavd\|^2s_{\Xv^2}\dimh^3\|\psi'\|_{\infty}^2 \|\etavd\|^2}\\
		&&\sqrt{\dimh} \|\psi\|_{\infty}(\|\etav\|+\|\etavs\|)\\
	&\le& \CONST_{1}\dimh^3\|\ups-\upsd\|+ \CONST_{2}\dimh^4\|\ups-\upsd\|^2.
\end{EQA}
But note that by the triangular inequality we also have \(|\chi_{\gmi}(\ups)_i-\chi_{\gmi}(\upsd)_i|\le 2\). This gives
\begin{EQA}
\sup_{\ups,\upsd}\frac{\|\chi_{\gmi}(\ups)-\chi_{\gmi}(\upsd)\|_{L_1(\Pn)}}{\|\ups-\upsd\|_2}&\le& \sup_{\ups,\upsd}\left(\frac{2}{\|\ups-\upsd\|_2}\wedge  \CONST_{1}\dimh^3+ \CONST_{2}\dimh^4\|\ups-\upsd\|_2\right)\\
	&=&\CONST_{3}\dimh^3.
\end{EQA}
We infer setting \(\delta =\sqrt{\dimtotal/n}\) 
\begin{EQA}
&&\nquad\sqrt{\frac{6\{1+\log N(\delta, \mathcal F, L_1(\Pn))\} }{n}}+\delta\\
	&\le&\sqrt{\frac{6\{1+\log N(\delta/(L(\Pn)\vee 1), \Ups, \|\cdot\|_2)\} }{n}}+\delta\\
	&\le& \sqrt{\frac{6\{1+\log(\CONST\dimh^3)+\log(1/\delta)\dimtotal\} }{n}}+\delta\\
	&\le&\CONST_1\sqrt{\frac{\log(\dimtotal)+\log(n/\dimtotal)\dimtotal/2}{n}}+\sqrt{\dimtotal/n}\\
	&\le& \CONST_2 \sqrt{\frac{\log(n)\dimtotal}{n}}.
\end{EQA}
The claim follows with Theorem \ref{theo: mendelson theo generalized}.
\end{proof}

It remains to prove the assumptions (1), (2) and (3) which we do in the following three lemmas.

\begin{lemma}
\label{lem: size of C dimh}
We have for some \(\CONST>0\)
\begin{EQA}
 \nsize\E[\|\fv_{\etavs_{\dimh}}( \Xv^\T\thetavs_{\dimh})-\fv_{\etavs}(\Xv^\T\thetavs)\|^2]
   &\le & 3(2+\CONST) {\rr^*}^2.
\end{EQA}
\end{lemma}

\begin{proof}
We find with the Taylor expansion, Lemma A.2 of \cite{AAbias2014} (which is applicable because it only needs \(\bb{(\LL\rr)}\) for the full model and with center \(\upsilonvs\in\Ups\)) and Lemma \ref{lem: conditions theta eta} with some \(\thetavd\in \mathbf{Conv}(\thetavs_{\dimh},\thetavs)\)
\begin{EQA}
 &&\nquad\nsize\E[\|\fv_{\etavs_{\dimh}}( \Xv^\T\thetavs_{\dimh})-\fv_{\etavs}(\Xv^\T\thetavs)\|^2]\\
  &\le& 3\nsize\left(\E[\|\fv_{\etavs}( \Xv^\T\thetavs_{\dimh})-\fv_{\etavs}(\Xv^\T\thetavs)\|^2]+\E[\|\fv_{\etavs_{\dimh}-\etavs}( \Xv^\T\thetavs_{\dimh})\|^2]\right)\\
  &\le&3\left(\|\DP(\thetavd)(\thetavs_{\dimh}-\thetavs)\|^2 +  \|\HF(\upsilonvs_{\dimh})(\etavs_{\dimh}-\fvs)\|^2\right)\\
  &\le&3\left((1+\|I-\DP^{-1/2}n\DP(\xiv)\DP^{-1/2}\|)\|\DP(\thetavs_{\dimh}-\thetavs)\|^2 \right.\\
  &&\left.+ (1+ \|I-\HF^{-1}n\tilde\HF(\upsilonvs_{\dimh})\HF^{-1}\|) \|\HF(\etavs_{\dimh}-\fvs)\|^2\right)\\
  &\le&3\left[ 2+\|I-\DP^{-1/2}n\DP(\thetavd)\DP^{-1/2}\|+  \|I-\HF^{-1}n\HF(\upsilonvs_{\dimh})\HF^{-1}\|\right]\\
  	&&\|\DF(\upsilonvs_{\dimh}-\upsilonvs)\|^2\\
  &\le & 3(2+\CONST) {\rr^*}^2.
\end{EQA}
\end{proof}

\begin{lemma}
\label{lem: size of C sum} 
We have for a constant \(\CONST>0\) that only depends on \(\|\psi\|_{\infty}\), \(\|\psi'\|_{\infty}\) and \(s_{\Xv^2}\) that
\begin{EQA}[c]
\P\left(n\left|(\Pn-\P)|\fv_{\etavs_{\dimh}}( \Xv^\T\thetavs_{\dimh})-\fv_{\etavs}(\Xv^\T\thetavs)|^2\right|\ge \CONST\sqrt{\xx}\right)\le \exp\left\{- \dimh^{3}\xx\right\}.
\end{EQA}
\end{lemma}

\begin{proof}
We want to use the finite difference inequality. As above define
\begin{EQA}
f: \bigotimes_{i=1}^n\R^{\dimp}\to \R, &&  f(\Xv_1,\ldots, \Xv_{n})\eqdef \Pn|\fv_{\etavs_{\dimh}}( \Xv^\T\thetavs_{\dimh})-\fv_{\etavs}(\Xv^\T\thetavs)|^2,
\end{EQA}
and note that for any \(i=1,\ldots,n\) and any alternative realization \(\Xv_{i}'\in \R\)
\begin{EQA}
&&\nquad n|f(\Xv_1,\ldots, \Xv_{i-1},\Xv_{i},\Xv_{i+1},\ldots,\Xv_{n})-f(\Xv_1,\ldots, \Xv_{i-1},\Xv'_{i},\Xv_{i+1},\ldots,\Xv_{n})|\\
	&\le& |\fv_{\etavs_{\dimh}}( {\Xv}_i^\T\thetavs_{\dimh})-\fv_{\etavs}(\Xv_i^\T\thetavs)|^2+|\fv_{\etavs_{\dimh}}( {\Xv'}_i^\T\thetavs_{\dimh})-g({\Xv'}_i)|^2.
\end{EQA}
We have
\begin{EQA}
|\fv_{\etavs_{\dimh}}( \Xv^\T\thetavs_{\dimh})-\fv_{\etavs}(\Xv^\T\thetavs)|^2	&\le& 3|\fv_{\etavs-\etavs_{\dimh}}(\Xv_i^\T\thetav)|^2\\
	&&+3|\fv_{\etavs}(\Xv_i^\T\thetavs)-\fv_{\etavs}(\Xv_i^\T\thetavs_{\dimh})|^2.
	\end{EQA}
As in Lemma \ref{lem: bounds for scores and so on} there are constants \(\CONST,\CONST'\) such that
\begin{EQA}
|\fv_{\etavs-\etavs_{\dimh}}(\Xv_i^\T\thetavs_{\dimh})|^2&\le&3\left|\sum_{k=1}^{\dimh} (\eta^*_k-\eta^*_{k,\dimh})\basX_k(\Xv_i^\T\thetavs_{\dimh})\right|^2+3\left|\sum_{k=\dimh+1}^{\infty} \eta^*_k\basX_k(\Xv_i^\T\thetavs_{\dimh})\right|^2\\
	&\le&3\left|\sum_{k=1}^{\dimh} \basX_k^2(\Xv_i^\T\thetavs_{\dimh})\right|\|\Pi_{\dimh}\etavs-\etavs_{\dimh}\|^2+\CONST(\kappavs)\\
	&\le&\CONST'\left|\sum_{j=0}^{j_\dimh} 2^{j}\right|\|\Pi_{\dimh}\etavs-\etavs_{\dimh}\|^2+\CONST(\kappavs)\\
	&\le&\CONST\dimh\|\Pi_{\dimh}\etavs-\etavs_{\dimh}\|^2+\CONST(\kappavs),
\end{EQA}
where \(\CONST(\kappavs)\le \CONST\dimh^{-2\alpha+1}\). Furthermore again as in Lemma \ref{lem: bounds for scores and so on} there are constants \(\CONST,\CONST'\) such that
\begin{EQA}
|\fv_{\etavs}(\Xv_i^\T\thetavs)-\fv_{\etavs}(\Xv_i^\T\thetavs_{\dimh})|^2&\le&\left|\sum_{k=1}^{\dimh} \eta^*_k\left(\basX_k(\Xv_i^\T\thetavs)-\basX_k(\Xv_i^\T\thetavs_{\dimh}) \right)\right|^2\\
	&\le&\CONST'\left|\sum_{j=0}^{j_\dimh} 2^{3j-2\alpha}\right|\|\thetavs-\thetavs_{\dimh}\|^2\\
	&\le&\CONST\|\thetavs-\thetavs_{\dimh}\|^2.
\end{EQA}
This implies with Lemma \ref{lem: in example a priori distance of target to oracle} and constants \(\CONST_1,\CONST_2>0\)
\begin{EQA}
|\fv_{\etavs_{\dimh}}( \Xv^\T\thetavs_{\dimh})-\fv_{\etavs}(\Xv^\T\thetavs)|^2	&\le& \CONST_1  \left( \frac{\dimh}{n c_{\DF}^2} {{\rr^*}^2}+\dimh^{-2\alpha+1}\right)\le \CONST \dimh^{-3}.
\end{EQA}
Note that \({\rr^*}^2\dimh/n \to 0\). This gives with the bounded difference inequality (Theorem \ref{theo: bounded differences inequality}) that
\begin{EQA}[c]
\P\left(n\left|(\Pn-\P)|\fv_{\etavs_{\dimh}}( \Xv^\T\thetavs_{\dimh})-\fv_{\etavs}(\Xv^\T\thetavs)|^2\right|\ge t\CONST \dimh^{-3}\right)\le \exp\left\{-t^2\right\}.
\end{EQA}
From this we infer with \(t=\dimh^{3}\sqrt{\xx}\to \infty\)
\begin{EQA}[c]
\P\left(n\left|(\Pn-\P)|\fv_{\etavs_{\dimh}}( \Xv^\T\thetavs_{\dimh})-\fv_{\etavs}(\Xv^\T\thetavs)|^2\right|\ge \CONST_{2}\sqrt{\xx}\right)\le \exp\left\{- \dimh^{3}\xx\right\}.
\end{EQA}
\end{proof}

For a set \(A\subset \R^{\dimp}\) we denote by \(\lambda(A)\in\R_+\) its Lebesgue measure and define
\begin{EQA}
\label{eq: def of lambda basX}
&&\nquad\lambda_{\basX}\\
	&\eqdef& \sup\left\{\lambda>0: \inf_{\substack{\bb{v}\in\R^{\dimh}, \|\bb{v}\|=1\\ \thetav\in S_{1}^{\dimp,+}}}\P\left(|\langle \bb{v}, \basX(\Xv^\T\thetav)\rangle|>\lambda\right)>3/4\right\}.
\end{EQA}

\begin{remark}
\(\lambda_{\basX}\ge\R\) in \eqref{eq: def of lambda basX} is strictly greater \(0\) because the basis functions are linearly independent and we assumed the distribution of the regressors \(\Xv\) to be absolutely continuous with respect to the Lebesgue measure.
\end{remark}

\begin{lemma}
\label{lem: size of Q}
Denote the cylinder 
\begin{EQA}[c]
C_{\rho,x,y}(x_0,y_0)\eqdef \{(x,y,z)\in \R^2\times\R^{\dimp-2};\,(x-x_0)^2+(y-y_0)^2\le \rho^2\}.
\end{EQA}
There is a point \((x_0,y_0)\in\R^{2}\) such that \(\bb{Q}(2\gmi)\) in \eqref{eq: prob bound from mendelson} satisfies
\begin{EQA}
&&\nquad\bb{Q}(2\gmi)+3\ex^{-\xx}\ge \frac{1}{2}\wedge c_{p_{\Xv}} \lambda\left(B_{h}(0) \cap C_{h,x,y}(0)\cap B_{s_{\Xv}}(x_0,y_0,0)\right.\\
	&&\left. \text{\phantom{\(\frac{1}{2}\wedge c_{p_{\Xv}} \lambda(\)}} \cap \left\{(x,y)\in\R^2: \sign(y_0)y\ge \sign(y_0)h/2\right\}\right),
\end{EQA}
for \(\tau= \lambda_{\basX}/(8\bb{L}_{\etavs}s_{\Xv})\) and
\begin{EQA}
 2\gmi&=& (1-\corrDF^2)\Bigg(\frac{\lambda_{\basX}^2c_{\DF}^2}{32  }\wedge \frac{\tau c_{\fs_{\etavs}'}^2h^2}{4 \dimp \pi^2 s_{\Xv}^2\|p_{\Xv}\|_{\infty}^2C_{\|\etavs\|}}\Bigg),
\end{EQA}
and for 
\begin{EQA}[c]
\rr \ge  \sqrt{\dimh}\frac{4 \CONST_{\kappav}}{ \lambda_{\basX}\sqrt{(1-\corrDF)}}.
\end{EQA}
\end{lemma}
\begin{remark}
The constants \(h,c_{\fv_{\etavs}'}>0\) are from assumption\break \((\mathbf{Cond}_{\Xv\thetavs})\).
\end{remark}

\begin{proof}
We have to prove
\begin{EQA}[c]
\inf_{\ups \in \Upss(\rr)^c}\P\left[\left(\fv_{\etav}(\Xv_i^\T\thetav)-\fv_{\etavs}(\Xv_i^\T\thetavs)\right)^2 \ge \frac{\gmi\rr^2}{n}\right]>0. \label{eq: prob bound from mendelson corrected} 
\end{EQA}
We carry out the proof in two steps.

1. Before we determine \(\gmi>0\) that allows to prove \eqref{eq: prob bound from mendelson corrected} note that
\begin{EQA}
&&\nquad\|\DF_{\dimh}(\upsilonv-\upsilonvs_{\dimh})\|-\|\DF_{\dimh}(\Pi_{\dimtotal}\upsilonvs-\upsilonvs_{\dimh})\|\le \|\DF_{\dimh}(\upsilonv-\Pi_{\dimtotal}\upsilonvs)\|\\
	&\le& \|\DF_{\dimh}(\upsilonv-\upsilonvs_{\dimh})\|+\|\DF_{\dimh}(\Pi_{\dimtotal}\upsilonvs-\upsilonvs_{\dimh})\|.
\end{EQA}
Slightly modifying Lemma A.3 of \cite{AAbias2014} with \(\thetav=\upsilonv\) gives
\begin{EQA}[c]
\|\DF_{\dimh}(\Pi_{\dimtotal}\upsilonvs-\upsilonvs_{\dimh})\|\le  \Big(\alpha(\dimh) +\tau(\dimh)+2\delta(2\rr^*)\rr^*\Big)\eqdef \rr^*_{\eps}(\dimh),
\end{EQA}
where due to Lemma \ref{lem: conditions theta eta} and the definition of \(\rr^*>0\) in Lemma \ref{lem: in example a priori distance of target to oracle}
\begin{EQA}
\rr^*\le\CONST\sqrt{ \dimh},& \alpha(\dimh)=\CONST\left(\dimh^{-\alpha-1/2}+\CONST_{bias} \dimh^{-(\alpha-1)}\right)\sqrt n, & \tau(\dimh)\le \CONST\dimh^{-2\alpha+1/2}\sqrt{n}.
\end{EQA}
With arguments as above we find that \(\rr^*_{\eps}(\dimh)>0\) is neglect-ably small for \(n\in\N\) large enough. We have with some small \(\eps>0\)
\begin{EQA}
\label{eq: replacing center Upss}
(1-\eps)\|\DF_{\dimh}(\upsilonv-\upsilonvs_{\dimh})\|^2&\le& \|\DF_{\dimh}(\upsilonv-\upsilonvs)\|^2\\
&\le& (1+\eps)\|\DF_{\dimh}(\upsilonv-\upsilonvs_{\dimh})\|^2.
\end{EQA}
Assume that \(n\in\N\) is large enough to ensure that \(\eps<1/2\). Then we find for \(\upsilonv\in\Upss( \rr)^{c}\) and with Lemma B.5 of \cite{AASP2013} and \eqref{eq: replacing center Upss} that 
\begin{EQA}[c]
\|\DP(\varphi_{\thetav}-\varphi_{\thetavs})\|^2+\|\HH_{\dimh}(\etav-\etavs)\|^2\ge (1-\corrDF)\|\DF_{\dimh}(\upsilonv-\upsilonvs)\|^2\ge(1-\corrDF)\rr^2/2.
\end{EQA}

2. Now we show \eqref{eq: prob bound from mendelson}. We treat two cases for \((\varphi_{\thetav},\etav)\in\R^{\dimp-1}\times\R^{\dimh}\) separately. The first case is that \(\|\DP(\varphi_{\thetav}-\varphi_{\thetavs})\|^2\le \frac{1}{4}(1-\corrDF)\rr^2 \). In this situation we can use the smoothness of \(\fv_{\etavs_{\dimh}}\) and \(\fv_{\etavs}\) to determine \(\gmi>0\). In the second case we use the geometric structure of 
\begin{EQA}[c]
\left(\fv_{\etav}(\Xv^\T\thetav)-\fv_{\etavs}(\Xv^\T\thetavs)\right)^2>0,
\end{EQA}
to obtain a good lower bound.\\

Case 1: \(\|\DP(\varphi_{\thetav}-\varphi_{\thetavs})\|^2\le \frac{1}{2}\tau \rr^2 \).
In this case we simply calculate and find
\begin{EQA}
&&\nquad|\fv_{\etav}(\Xv^\T\thetav)-\fv_{\etavs}(\Xv^\T\thetavs)|^2\\
	&\ge& |\fv_{\etav}(\Xv^\T\thetav)-\fv_{\etavs}(\Xv^\T\thetav)|^2\\
		&&-2|\fv_{\etav}(\Xv^\T\thetav)-\fv_{\etavs}(\Xv^\T\thetav)||\fv_{\etavs}(\Xv^\T\thetav)-\fv_{\etavs}(\Xv^\T\thetavs)|\\
	&\ge&|\fv_{\etav}(\Xv^\T\thetav)-\fv_{\etavs}(\Xv^\T\thetav)|^2-2|\fv_{\etav}(\Xv^\T\thetav)-\fv_{\etavs}(\Xv^\T\thetav)|\bb{L}_{\etavs}s_{\Xv}\|\thetav-\thetavs\|.
\end{EQA}
Now
\begin{EQA}[c]
|\fv_{\etav}(\Xv^\T\thetav)-\fv_{\etavs}(\Xv^\T\thetav)|\ge |\fv_{\etav-\etavs}(\Xv^\T\thetav)|-|\fv_{(0,\kappavs)}(\Xv^\T\thetav)|.
\end{EQA}
We find with probability greater than \(3/4\)
\begin{EQA}
|\fv_{\etav-\etavs}(\Xv^\T\thetav)|&=&|\langle \etav-\etavs, \basX(\Xv^\T\thetav)\rangle|\\
	&\ge& \|\HH_{\dimh}(\etav-\etavs)\|  \lambda_{\basX}\\
	&\ge&  \rr \lambda_{\basX}\frac{1}{2}\sqrt{(1-\corrDF^2)},
\end{EQA}
where
\begin{EQA}[c]
\lambda_{\basX}\eqdef \sup\left\{\lambda>0: \inf_{\substack{\etav\in\R^{\dimh}, \|\etav\|=1\\ \thetav\in S_{1}^{\dimp,+}}}\P\left(|\langle \etav,\HH_{\dimh}^{-1} \basX(\Xv^\T\thetav)\rangle|>\lambda\right)>3/4\right\},
\end{EQA}
which is larger \(0\) because the basis functions are linearly independent and we assumed the distribution of the regressors \(\Xv\) to be absolutely continuous to the Lebesgue measure. Remember that by Lemma \ref{lem: conditions theta eta}
\begin{EQA}
\|\HF_{\dimh} ^{1/2}\kappavs\|^2&<&\left(17\|p_{\Xv^\T\thetavs}\|_\infty\CONST_{\|\fvs\|}+17^2 \sqrt{36}s_{\Xv}^{\dimp+1}L_{p_{\Xv}}\|\psi\|_{\infty}\CONST_{\|\fvs\|}^2\right) n\dimh^{-2\alpha}\\
	&\eqdef& \CONST_{\kappav}^2\dimh.
\end{EQA}
We use the Markov inequality to obtain
\begin{EQA}[c]
\P\left(|\fv_{(0,\kappavs)}(\Xv^\T\thetav)|^2\ge 4\CONST_{\kappav}\frac{\dimh}{n}\right)\le \frac{\|\HF_{\dimh} ^{1/2}\kappavs\|^2}{4\CONST_{\kappav}^2\dimh}\le 1/4.
\end{EQA}
This implies that with probability greater than \(1/2=3/4-1/4\)
\begin{EQA}
|\fv_{\etav}(\Xv^\T\thetav)-\fv_{\etavs}(\Xv^\T\thetav)|&\ge&\rr \lambda_{\basX}\frac{1}{2}\sqrt{(1-\corrDF^2)}-4\CONST_{\kappav}\sqrt{\frac{\dimh}{\nsize}}\\
	&\ge& \frac{\sqrt{(1-\corrDF^2)}\lambda_{\basX}}{4\sqrt{n}}\rr ,
\end{EQA}
for 
\begin{EQA}[c]
\rr \ge  \sqrt{\dimh}\frac{4 \CONST_{\kappav}}{ \lambda_{\basX}\sqrt{(1-\corrDF^2)}}.
\end{EQA}
We still have to account for the summand \(\bb{L}_{\etavs}s_{\Xv}\|\thetav-\thetavs\|\) via
\begin{EQA}[c]
\bb{L}_{\etavs}s_{\Xv}\|\thetav-\thetavs\|\le \frac{\bb{L}_{\etavs}s_{\Xv}\sqrt{\tau(1-\corrDF^2)}}{2c_{\DF}\sqrt{n}}\rr.
\end{EQA}
This gives for the choice of \(\tau= \lambda_{\basX}c_{\DF}/(8\bb{L}_{\etavs}s_{\Xv})\)
\begin{EQA}
&&\nquad|\fv_{\etav}(\Xv^\T\thetav)-\fv_{\etavs}(\Xv^\T\thetav)|-2\bb{L}_{\etavs}s_{\Xv}\|\thetav-\thetavs\|\\
	&\ge& \left(\frac{\lambda_{\basX}}{4}-\frac{\bb{L}_{\etavs}s_{\Xv}\sqrt{\tau}}{c_{\DF}}\right)\frac{\sqrt{(1-\corrDF^2)}}{\sqrt{n}}\rr\\
	&=&\frac{\lambda_{\basX}c_{\DF}\sqrt{(1-\corrDF^2)}}{8\sqrt{n}}\rr
\end{EQA}
We obtain in case 1 that \(\bb{Q}(2\gmi)\ge 1/2\) for 
\begin{EQA}
2\gmi/n&\eqdef& \frac{(1-\corrDF^2)\lambda_{\basX}^2c_{\DF}^2}{32  n}.
\end{EQA}

Case 2: \(\frac{1}{2}\tau(1-\corrDF)\rr^2\le \|\DP(\varphi_{\thetav}-\varphi_{\thetavs})\|^2\le \sqrt 2\lambda_{\max}\DP^2 \).\\
Take some \(f:\R\to \R\) with \(f'>c\) and some \((\alpha,\beta)\in\R^2\) with \(\alpha^2+\beta^2=1\). Furthermore take any \(g:\R\to\R\). We are interested in determining 
\begin{EQA}
V(\tau)&\eqdef& \inf_{\substack{f\in C^1(\R),\, f'>c,\\g:\R\to \R}} \lambda\left( \mathcal A(\tau)\right) \\
\mathcal A(\tau)&\eqdef&\left\{(x,y,z)\in \R^2\times\R^{\dimp-2};\,|f(\alpha x+\beta y)-g(x)|>\tau \right\}\\ 	&&\cap C_{\rho,x,y}(0)\cap B_{s_{\Xv}}(x_0,y_0,0)\subset \R^2\times\R^{\dimp-2},\\
C_{\rho,x,y}(x_0,y_0)&\eqdef& \{(x,y,z)\in \R^2\times\R^{\dimp-2};\,(x-x_0)^2+(y-y_0)^2\le \rho^2\},
\end{EQA}
where for a set \(A\subset \R^{\dimp}\) we denote by \(\lambda(A)\in\R_+\) its Lebesgue measure. For this observe
\begin{EQA}[c]
f(\alpha x+\beta y)-g(x)\begin{cases} \ge c\beta y+ f(\alpha x)-g(x ) & \beta>0,\\
\le  c\beta y+ f(\alpha x)-g(x) & \beta\le 0
											\end{cases}
\end{EQA}
Consequently for fixed \(x\in[-\rho,\rho]\) we have \(|f(\alpha x+\beta y)-g(x)|>\rho\beta c/2\) on the set
\begin{EQA}[c]
\{y\in [-\sqrt{\rho^2-x^2},\sqrt{\rho^2-x^2} ]: |c\beta y+ f(\alpha x)-g(x)|>\rho\beta c/2\},
\end{EQA}
which always is of a length greater \(\lambda( [-\sqrt{\rho^2-x^2},\sqrt{\rho^2-x^2} ]\backslash [-\rho/2, \rho/2])\). Addressing the way a centered cylinder intersects with a shifted ball this gives that
\begin{EQA}
V(\rho\beta c/2)&\ge& \lambda\left(C_{\rho,x,y}(0) \cap B_{s_{\Xv}}(x_0,y_0,0)\right.\\
	&&\left. \phantom{\lambda(} \cap \{(x,y,z)\in \R^2\times\R^{\dimp-2};\right.\\
	&&\left. \phantom{\lambda(}(x,y)\in\R^2: - \sign(y_0)y\ge - \sign(y_0)\rho/2\}\right)\\
	&\ge &\lambda(B_{\rho/4}(0))>0,\label{eq: bound for V of tau}
\end{EQA}
for the ball \(B_{h/4}(0)\subset \R^{\dimp}\). Now we can prove the claim. For any \((\thetav,\etav)=\upsilonv\in\Upsilon\), with \(\|\thetav\|=1\), we can represent \(\thetavs=\alpha \thetav+ \beta \thetavd\) with some \(\thetavd\in\thetav^\perp\) with \(\|\thetavd\|=1\) and \(\alpha^2+\beta^2=1\). By assumption \((\mathbf{Cond}_{\Xv\thetavs})\) for any \((\thetav,\etav)=\upsilonv\in\Upsilon\), there exist constants \(c_{\fs'},\,c_{p_{\Xv}}, h >0\) and a value \((x_0,y_0)\in \{x^2+y^2\le s_{\Xv}\}\subset \R^2 \) such that for \((x,y)\in \{(x-x_0)^2+(y-y_0)^2\le h^2 \}\) we have \(|\fs'_{\etavs}(x)|> c_{\fs'}\) and \(p_{\Xv}\ge c_{p_{\Xv}}\). We can estimate using \eqref{eq: bound for V of tau} 
\begin{EQA}
&&\nquad\P\Big\{\left(\fv_{\etavs}(\Xv^\T\thetavs)-\fv_{\etav}(\Xv^\T\thetav)\right)^2\ge c_{\fs'}^2h^2\beta^2/4 \Big\}\\
	&\ge& \inf_{\substack{f\in C^1(\R),\, f'>0,\\g:\R\to \R}} \P\bigg(\{\Xv\in B_{s_{\Xv}}(0)\}	\cap \{\Xv\in C_{h,x,y}(x_0,y_0)\}\\
		&&\phantom{\text{\( \inf_{\substack{f\in C^1(\R),\, f'>0,\\g:\R\to \R}} \P( \)}}\cap\{|f(\alpha x+\beta y)-g(x)|\ge c_{\fs'}h\beta/2\}\bigg)\\
	&\ge& c_{p_{\Xv}}\inf_{\substack{f\in C^1(\R),\, f'>0,\\g:\R\to \R}} \lambda\bigg(B_{s_{\Xv}}(-x_0,-y_0,0) \cap C_{h,x,y}(0)\\
		&&\phantom{\text{\( c_{p_{\Xv}}\inf_{\substack{f\in C^1(\R),\, f'>0,\\g:\R\to \R}} \P( \)}}\cap \{|f(\alpha x+\beta y)-g(x)|\ge c_{\fs'}h\beta/2\}\bigg)\\
	&=&c_{p_{\Xv}}V(h\beta c_{\fs'}/2)\ge \lambda(B_{h/4}(0))>0.
\end{EQA}
We need to express \(\beta>0\) in terms of \(\rr>0\). We can use elementary geometry to obtain
\begin{EQA}[c]
\beta=\sin\left(2\arcsin\left(\frac{\|\thetav-\thetavs\|}{2} \right) \right).
\end{EQA}
Using that \(\sin(2\alpha)=2\sin(\alpha)\cos(\alpha)\) this yields
\begin{EQA}[c]
\beta=\cos\left(\arcsin\left(\frac{\|\thetav-\thetavs\|}{2} \right) \right)\frac{\|\thetav-\thetavs\|}{2}.
\end{EQA}
Now as \(\|\thetav-\thetavs\|^2\le 2\) we get
\begin{EQA}[c]
\beta\ge \cos\left(\arcsin\left(\frac{1}{\sqrt{2}} \right) \right)\|\thetav-\thetavs\|=\frac{\|\thetav-\thetavs\|}{\sqrt{2}}.
\end{EQA}
Furthermore for any \(\varphi_{\thetav},\varphi_{\thetav}\in W_S\) we have with \eqref{eq: replacing center Upss} that
\begin{EQA}[c]
\|\thetav-\thetavs\|^2\ge \frac{2}{\dimp \pi^2}\|\varphi_{\thetav}-\varphi_{\thetavs}\|^2\ge \frac{2}{\dimp \pi^2\|\DP^2\|}\|\DP(\varphi_{\thetav}-\varphi_{\thetavs})\|^2\ge \frac{ \tau}{\dimp \pi^2\|\DP^2\|}\rr^2.
\end{EQA}
With Lemma \ref{lem: D_0 dimh upss is boundedly invertable} this implies
\begin{EQA}[c]
\beta^2\ge\frac{\tau}{2 \dimp \pi^2 s_{\Xv}^2\|f_{\Xv}\|_{\infty}^2C_{\|\fvs\|}}\rr^2/n.
\end{EQA}
Combined this yields that with
\begin{EQA}[c]
2\gmi/n\eqdef \frac{\tau c_{\fs'}^2h^2}{4 n \dimp \pi^2 s_{\Xv}^2\|p_{\Xv}\|_{\infty}^2C_{\|\etavs\|}},
\end{EQA}
it holds
\begin{EQA}
&&\nquad\P\Big\{\left(\fv_{\etav}(\Xv^\T\thetav)-\fv_{\etavs}(\Xv^\T\thetavs)\right)^2\ge 9\gmi \rr^2/n\Big\}\\
	&\ge& c_{p_{\Xv}}\lambda(B_1^{\dimp-2}) \lambda\left(B_{h}(0)\cap \{(x,y)\in\R^2: |y|\le h/2\} \right).
\end{EQA}
This gives the claim.
\end{proof}

\subsubsection{Proof of Condition \(\bb{(\cc{L}{\rr})}\) with modeling bias}
\label{sec: proof of Lr with model bias}
We show the following Lemma
\begin{lemma}
We have with some \(\CONST>0\) and with \(\rr^{\circ}>0\) from \eqref{eq: def of rr circ} that
\begin{EQA}
&&\nquad\P\bigg( \sup_{\ups\in\Upss(\sqrt{n}\rr^{\circ})}\left|\E_{\eps}\LL(\ups,\upss)- \E\LL(\ups,\upss)\right|\ge  \sqrt{\xx+\dimtotal[\CONST\log(\dimtotal )+\log(\rr)]}\bigg)\\
	& \le& \ex^{-\xx}.
\end{EQA}
\end{lemma}

\begin{proof}
We bound
\begin{EQA}
&&\nquad\sup_{\ups\in\Upss(\sqrt{n}\rr^{\circ})}\left|\E_{\eps}\LL(\ups,\upss)- \E\LL(\ups,\upss)\right|\\
	&\le&n\sup_{\ups\in\Upss(\sqrt{n}\rr^{\circ})}\bigg|(P_n-\P)\bigg\{\left(g(\Xv_i)-\fv_{\etavs}(\Xv^\T_i\thetavs)\right)^2\\
		&&\phantom{n\sup_{\ups\in\Upss(\sqrt{n}\rr^{\circ})}}-\left(g(\Xv_i)-\fv_{\etav}(\Xv^\T_i\thetav)\right)^2\bigg\}\bigg|\\
	&\le& n\sup_{\ups\in\Upss(\sqrt{n}\rr^{\circ})}\left|(P_n-\P)\left\{\fv_{\etav}(\Xv^\T_i\thetav)-\fv_{\etavs}(\Xv^\T_i\thetavs)\right\}^2\right|\\
	&&+ n\CONST_{bias}\sup_{\ups\in\Upss(\sqrt{n}\rr^{\circ})}\left|(P_n-\P)\left|\fv_{\etav}(\Xv^\T_i\thetav)-\fv_{\etavs}(\Xv^\T_i\thetavs)\right|\right|.
\end{EQA}
Furthermore
\begin{EQA}
\left\{\fv_{\etav}(\Xv^\T_i\thetav)-\fv_{\etavs}(\Xv^\T_i\thetavs)\right\}^2&\le& \left|\fv_{\etav}(\Xv^\T_i\thetav)-\fv_{\etavs}(\Xv^\T_i\thetavs)\right|\\
	&&\left(\|\fv_{\etavs}\|_{\infty}+\|\fv_{\etavs_{\dimh}}\|_{\infty}+ \CONST \rr\sqrt{\dimh}/\sqrt{n} \right).
\end{EQA}
Thus we have 
\begin{EQA}
&&\nquad\sup_{\ups\in\Upss(\sqrt{n}\rr^{\circ})}\left|\E_{\eps}\LL(\ups,\upss)- \E\LL(\ups,\upss)\right|\\
&\le& n \left(\CONST_{bias}+\|\fv_{\etavs}\|_{\infty}+\|\fv_{\etavs_{\dimh}}\|_{\infty}+ \CONST \rr^{\circ}\sqrt{\dimh} \right)\\
	&&\sup_{\ups\in\Upss(\sqrt{n}\rr^{\circ})}\left|(P_n-\P)\left|\fv_{\etav}(\Xv^\T_i\thetav)-\fv_{\etavs}(\Xv^\T_i\thetavs)\right|\right|.
\end{EQA}
Define \(\zetav_{\Xv}(\ups)\eqdef (P_n-\P)|\fv_{\etav}(\Xv^\T_i\thetav)-\fv_{\etavs}(\Xv^\T_i\thetavs)|\). Then we find using that \(\rr^{\circ}\le \CONST\sqrt{\dimtotal\log(\dimtotal)+\xx}\)
\begin{EQA}
\sup_{\ups\in\Upss(\sqrt{n}\rr^{\circ})}\left|\E_{\eps}\LL(\ups,\upss)- \E\LL(\ups,\upss)\right|&\le& n \CONST \dimh^{3/2}\sup_{\ups\in\Upss(\sqrt{n}\rr^{\circ})}\left|\zetav_{\Xv}(\ups)-\zetav_{\Xv}(\upss)\right|.
\end{EQA}
We want to use Lemma \ref{lem: basic chaining}. Define \(\Ups_0=\{\upss\}\) and with \(\rr_k=2^{-k}\rr\) with \(\rr>0\) to be specified later the sequence of sets \(\Ups_k\) each with minimal cardinality such that
\begin{EQA}
\Ups_{\dimh}\subset \bigcup_{\ups\in\Ups_k} B_{\rr_k}(\ups), &\quad B_{\rr}(\ups)\eqdef \{\upsd\in\Ups_{\dimh},\,\|\DF(\upsd-\ups)\|\le \rr\}.
\end{EQA}
We estimate for an application of the bounded differences inequality
\begin{EQA}
&&\nquad\left|\left\{ \fv_{\etav}(\Xv^\T_i\thetav)-\fv_{\etav'}(\Xv^\T_i\thetav')\right\} \right|\le  \| \fv_{\etav-\etav'}\|_{\infty}+\|\fv'_{\etav}\|_{\infty}\|\thetav-\thetav'\|.
\end{EQA}
We have 
\begin{EQA}
\|\fv_{\etav}\|_{\infty}&\le& \|\etav\|\sup_{x\in[-s_{\Xv},s_{\Xv}]}\left( \sum_{k=1}^{\dimh} \basX^2_k(x)^2\right)^{1/2}\le \sqrt{17}\|\psi\|\sqrt{\dimh}\rr/\sqrt{n},\\
\|\fv'_{\etav-\etav'}\|_{\infty}&\le& \|\etav-\etav'\|\sup_{x\in[-s_{\Xv},s_{\Xv}]}\left( \sum_{k=1}^{\dimh} \basX'^2_k(x)^2\right)^{1/2}\\
	&\le& \sqrt{17}\|\psi'\|\dimh^{3/2}\|\etav-\etav'\|.
\end{EQA}
Consequently again using that \(\rr^{\circ}\le \CONST\sqrt{\dimtotal\log(\dimtotal)+\xx}\)
\begin{EQA}[c]
\left|\left\{ \fv_{\etav}(\Xv^\T_i\thetav)-\fv_{\etav'}(\Xv^\T_i\thetav')\right\} \right|\le \CONST_{\zetav}\dimh^{3/2} \|\ups-\ups'\|.
\end{EQA}
This implies with the bounded difference inequality for any \(\ups_k\in\Ups_{k} \)
\begin{EQA}[c]
\P\left(n\inf_{\Ups_{k-1}} |\zetav_{\Xv}(\ups_k)-\zetav_{\Xv}(\ups_{k-1})|\ge t\CONST_{\zetav}\dimh^{3/2} \frac{\rr_{k-1}}{c_{\DF}} \right)\le \ex^{-t^2}.
\end{EQA}
Define \(\rr\eqdef \frac{(1-1/\sqrt{2})}{\dimh^{3}}\) then we find
\begin{EQA}[c]
\P\left(n\inf_{\Ups_{k-1}} |\zetav_{\Xv}(\ups_k)-\zetav_{\Xv}(\ups_{k-1})|\ge \CONST \dimh^{-3/2} t 2^{-(k-1)}(1-1/\sqrt{2}) \right)\le \ex^{-t^2},\\
|\Ups_{k}|\le \exp\left\{\left(\log(2)k+\log(\rr^{\circ})+\log(n)/2+3\log(\dimh)+\log(1-1/\sqrt{2})\right)\dimtotal \right\}.
\end{EQA}
Set 
\begin{EQA}
\bb{T}(n,\dimh)&\eqdef& \log(\rr^{\circ})+\log(n)/2+3\log(\dimh)+\log(1-1/\sqrt{2},\\
t&\eqdef&\sqrt{\xx+1+\log(2)+\dimtotal \left(\log(2)+\bb{T}(n,\dimh)\right)},
\end{EQA}
then we infer with Lemma \ref{lem: basic chaining}
\begin{EQA}
&&\nquad\P\left(\sup_{\ups\in\Upss(\sqrt{n}\rr^{\circ})}\left|\E_{\eps}\LL(\ups,\upss)- \E\LL(\ups,\upss)\right|\ge \CONST t\right)\\
&\le&\P\left(n\sup_{\ups\in\Upss(\sqrt{n}\rr^{\circ})}\left| \zetav_{\Xv}(\ups)-\zetav_{\Xv}(\upss)\right|\ge \CONST \dimh\log(\dimh) t\right)\\
&\le& \sum_{k=1}^{\infty}\exp\bigg\{\dimtotal\left[\left(\log(2)k+\bb{T}(n,\dimh)\right) -2^{k-1}\left(\log(2)+\bb{T}(n,\dimh)\right)\right]\\
	&&-2^{k-1}(\xx+1+\log(2))\bigg\}\\	
	&\le&\ex^{-\xx}.
\end{EQA}
\end{proof}

We have as in the proof of Lemma A.2 of \cite{AAbias2014}
\begin{EQA}[c]\label{eq: bound for EL upssdimh upss}
-\E\LL(\upss,\upss_{\dimh})=  \E\LL(\upss_{\dimh},\upss)\ge \E\LL(\Pi_{\dimtotal}\upss,\upss)\ge -{\rr^*}^2.
\end{EQA}
Combining this lemma and Equation \eqref{eq: bound for EL upssdimh upss} with Lemma \ref{lem: cond Lr infty} and Lemma \ref{lem: a priori a priori accuracy for rr circ} we find for \(\|\DF_{\dimh}(\ups-\upss_{\dimh})\|^2= \rr^2\ge 2{\rr^*}^2\)that with probability greater than \(1-2\ex^{-\xx}\) 
\begin{EQA}
-\E_{\eps}\LL(\ups,\upss_{\dimh})&\ge& \gmi\rr^2/2-\sqrt{\xx+\CONST\dimtotal[\log(\dimtotal )+\log(n)]}- {\rr^*}^2.
\end{EQA}
Consequently we get for \(\rr\) that additionally satisfies
\begin{EQA}[c]
\rr^2\ge \sqrt{\xx+\CONST\dimtotal[\log(\dimtotal )+\log(n)]}/\gmi\vee 2{\rr^*}^2,
\end{EQA}
that 
\begin{EQA}
-\E_{\eps}\LL(\ups,\upss_{\dimh})&\ge& \gmi\rr^2/4\eqdef \gmi_{bias}\rr^2.
\end{EQA}
Finally observe that by definition \(\LL(\ups,\upss_{\dimh})=\LL_{\dimh}(\ups,\upss_{\dimh})\).

\subsection{Proof of Lemma \ref{lem: additional error from different expectation operator}}

\begin{proof}
Note that with the definitions and with some \(\ups\in\Ups_{\dimh,0}(\rr)\), \(\gammav_0\in \R^{\dimtotal}\) with \(\|\gammav_0\|=1\)\begin{EQA}
&&\nquad\|\DF_{\dimh}^{-1}\nabla(\E-\E_{\varepsilon})[\LL_{\dimh}(\upss_{\dimh})-\LL_{\dimh}(\ups)]\|\\
&\le &\sup_{\ups\in\Ups_{\dimh,0}(\rr)} \|\DF_{\dimh}^{-1}(\E-\E_{\varepsilon})\left[\nabla^2 \LL_{\dimh}(\ups)\right]\DF_{\dimh}^{-1}\|\rr\\
&\le &\frac{1}{\sqrt{n} c_{\DF}} \|(\E-\E_{\varepsilon})\left[\DF_{\dimh}^{-1}\nabla^2 \LL_{\dimh}(\upss_{\dimh})\right]\|\rr\\
&&+\sup_{\ups\in\Ups_{\dimh,0}(\rr)}\left\|(\E-\E_{\varepsilon})\left[\DF_{\dimh}^{-1}\left(\left[\nabla^2 \LL_{\dimh}(\ups)\right]-\left[\nabla^2 \LL_{\dimh}(\upss_{\dimh})\right]\right)\DF_{\dimh}^{-1}\right]\right\|\rr.
\end{EQA}
For the first term we obtain with Lemma~\ref{lem: bound for norm of hessian} and with some constant \(\CONST>0\)
\begin{EQA}[c]
\P\left( \frac{1}{\sqrt{n} c_{\DF}} \|(\E-\E_{\varepsilon})\left[\DF^{-1}\nabla^2 \LL_{\dimh}(\upss_{\dimh})\right]\|\rr\ge \CONST\sqrt{\log(\dimtotal)+\xx}\rr/\sqrt{n}\right)\le \ex^{-\xx}.
\end{EQA}
For the second term we can use similar arguments to those of Lemma \ref{lem: a priori a priori accuracy for rr circ} to find with some constant \(\CONST>0\) that
\begin{EQA}
&&\nquad\P \Bigg(\sup_{\ups\in\Ups_{\dimh,0}(\rr)}\left\|(\E-\E_{\varepsilon})\left[\DF_{\dimh}^{-1}\left(\left[\nabla^2 \LL_{\dimh}(\ups)\right]-\left[\nabla^2 \LL_{\dimh}(\upss_{\dimh})\right]\right)\DF_{\dimh}^{-1}\right]\right\|\\
	&\ge& \CONST\sqrt{\xx+\dimtotal\log(\dimtotal)}  /\sqrt{n}\Bigg)\le \ex^{-\xx}.
\end{EQA}
Adding \(\log(2)\) to \(\xx\) in the above bounds we get the claim after increasing the constants appropriately.
\end{proof}

\subsection{Condition \((bias'')\) is satisfied}
\begin{lemma}
\label{lem: cond bias prime is satisfied}
Under the conditions of Proposition \ref{prop: semi sieve bias} condition \(\bb{(}\bb{bias}''\bb{)}\) is satisfied.
\end{lemma}
\begin{proof}
It suffices to show that
\begin{EQA}
\Cov(\score_{\thetav} \left(\lkh_{i}(\upsilonvs_{\dimh}) -\lkh_{i}(\upsilonvs)\right)) \to 0, &&\Cov(\score_{(\eta_1,\ldots,\eta_\dimh)} \left(\lkh_{i}(\upsilonvs_{\dimh}) -\lkh_{i}(\upsilonvs)\right))\to 0.
\end{EQA}
We calculate
\begin{EQA}
&&\nquad\|\Cov(\score_{\thetav} \left(\lkh_{i}(\upsilonvs_{\dimh}) -\lkh_{i}(\upsilonvs)\right))\|\\
	&\le& \E\|\left(\fv'_{\etavs_\dimh}(\Xv_{i}^{\T}\thetavs_\dimh)-\fv'_{\etavs}(\Xv_{i}^{\T}\thetavs)\right) \nabla\Phi(\thetav)^{\T} \Xv_{i}\|^2\\
	&\le& s_{\Xv}^2 \E\|\fv'_{\etavs_\dimh}(\Xv_{i}^{\T}\thetavs_\dimh)-\fv'_{\etavs}(\Xv_{i}^{\T}\thetavs) \|^2\\
	&\le& 4 s_{\Xv}^2 \left(\E\|\fv'_{\etavs_\dimh-\etavs}(\Xv_{i}^{\T}\thetavs)\|^2+\E\|\fv'_{\etavs_{\dimh}}(\Xv_{i}^{\T}\thetavs_\dimh)-\fv'_{\etavs_\dimh}(\Xv_{i}^{\T}\thetavs) \|^2\right)\\
	&\le& 4 s_{\Xv}^2 \left(\sum_{k=0}^\infty\|\basX_k'\|_{\infty}({\eta^*_\dimh}_k-\eta^*_k)\right)^2+4 s_{\Xv}^4 \left(\sum_{k=0}^{\dimh-1}\|\basX_k''\|_{\infty}{\eta^*_{\dimh}}_k\right)^2\|\thetavs_\dimh-\thetavs \|^2.
\end{EQA}
We estimate separately 
\begin{EQA}
&&\nquad\sum_{k=0}^\infty\|\basX_k'\|_{\infty}({\eta^*_\dimh}_k-\eta^*_k)\le\CONST \|\psi'\|_{\infty}\left(\sum_{k=0}^{\dimh-1}k^{3/2}({\eta^*_\dimh}_k-\eta^*_k)+\sum_{k=\dimh}^\infty k^{3/2}\eta^*_k \right)\\
	&\le& \CONST \|\psi'\|_{\infty}\left(\dimh^{2}\|\etavs_\dimh-\etavs\|+\left(\sum_{k=\dimh}^\infty k^{-2\alpha-3}\right)^{1/2}\left(\sum_{k=\dimh}^\infty2^{\alpha}{\eta^*_k}^2 \right)^{1/2}\right)\\
	&\le& \CONST \|\psi'\|_{\infty}\Bigg( \dimh^{2}\frac{1}{\sqrt{n}c_{\DF}}\|\DF_{\dimh}(\Pi_{\dimtotal}\upsilonvs-\upsilonvs_{\dimh})\|\\
	&&\phantom{\|\psi'\|_{\infty}\Bigg(} + \sqrt{(2\alpha-3)/(2\alpha-4)}\left(\sum_{k=\dimh}^\infty k^{2\alpha}{\eta^*_k}^2 \right)^{1/2} \Bigg)
\end{EQA}
The last term tends to \(0\) because of Lemma \ref{lem: in example a priori distance of target to oracle}, because \(\dimh^2\rr^*/\sqrt{n}\to 0\) and because \(\sum_k 2^{\alpha}{\eta^*_k}^2<\infty\). Furthermore we get with similar steps
\begin{EQA}
&&\nquad\left(\sum_{k=0}^{\dimh-1}\|\basX_k''\|_{\infty}{\eta^*_{\dimh}}_k\right)\|\thetavs_\dimh-\thetavs \|\le  \|\psi''\|_{\infty}\|\thetavs_\dimh-\thetavs \|\left(\sum_{k=0}^{\dimh-1}k^{5/2}{\eta^*_{\dimh}}_k\right)\\
	&\le & \|\psi''\|_{\infty}\|\thetavs_\dimh-\thetavs \|\left\{\left(\sum_{k=0}^{\dimh-1}k^{2\alpha-5}\right)^{1/2}\left(\sum_{k=0}^{\dimh-1}k^{2\alpha}\eta^*_k\right)^{1/2}
\right.\\
 &&\left.+ \frac{1}{\sqrt{n}c_{\DF}}\left(\sum_{k=0}^{\dimh-1}k^{5}\right)^2\|\DF(\upsilonvs-\upsilonvs_{\dimh})\|\right\}\\
 	&\le & \|\psi''\|_{\infty}\|\thetavs_\dimh-\thetavs \|\left\{\dimh \CONST_{\|\etavs\|}+ \frac{1}{\sqrt{n}c_{\DF}}\dimh^3\|\DF_{\dimh}(\Pi_{\dimtotal}\upsilonvs-\upsilonvs_{\dimh})\|\right\}\\
 	&\le & \|\DF_{\dimh}(\Pi_{\dimtotal}\upsilonvs-\upsilonvs_{\dimh})\|\frac{1}{\sqrt{n}c_{\DF}}\dimh \CONST_{\|\etavs\|}\|\psi''\|_{\infty}\\
 		&&+ \frac{1}{n c_{\DF}^2}\dimh^3\|\DF_{\dimh}(\Pi_{\dimtotal}\upsilonvs-\upsilonvs_{\dimh})\|^2\|\psi''\|_{\infty}.
\end{EQA}
Again the last term tends to \(0\). Similarly we calculate
\begin{EQA}
&&\nquad\Cov(\score_{(\eta_1,\ldots,\eta_\dimh)} \left(\lkh_{i}(\upsilonvs_{\dimh}) -\lkh_{i}(\upsilonvs)\right))\le \E\|\basX(\Xv_{i}^{\T}\thetavs_{\dimh})-\basX(\Xv_{i}^{\T}\thetavs)\|^2\\
	&\le& s_{\Xv}^2\|\psi'\|^2_{\infty} \|\thetavs_\dimh-\thetavs \|^2\left(\sum_{k=0}^{\dimh-1}k^{3/2}\right)^2\\
	&\le&  s_{\Xv}^2\|\psi'\|_{\infty} \frac{1}{ n c_{\DF}} \dimh^{3}\|\DF_{\dimh}(\Pi_{\dimtotal}\upsilonvs-\upsilonvs_{\dimh})\|^2,
\end{EQA}
which again is a zero sequence. This gives the claim.
\end{proof}

\subsection{Proof of Lemma \ref{lem: conditions of initial guess met}}

\begin{proof}
Define
\begin{EQA}[c]
\thetav_{l^*}\eqdef \argmin_{\thetav\in G_N}\|\thetav-\thetavs\|.
\end{EQA}
Then by definition
\begin{EQA}
\max_{\etav}\LL_{\dimh}(\tilde \upsilonv^{(0)},\upsilonvs_{\dimh})&\ge& \LL_{\dimh}((\thetav_{l^*},\tilde \etav^{(0)}_{l^*}),\upsilonvs_{\dimh})\ge\LL_{\dimh}((\thetav_{l^*}, \etavs_{\dimh}),\upsilonvs_{\dimh})\\
	&=& -\sum_{i=1}^n (\fv_{\etavs_{\dimh}}(\Xv_i^\T\thetavs_{\dimh})-\fv_{\etavs_{\dimh}}(\Xv_i^\T\thetav_{l^*}))^2\\
	&&+(g(\Xv_i)-\fv_{\etavs_{\dimh}}(\Xv_i^\T\thetavs_{\dimh}))(\fv_{\etavs_{\dimh}}(\Xv_i^\T\thetavs_{\dimh})-\fv_{\etavs_{\dimh}}(\Xv_i^\T\thetav_{l^*}))\\
	&&-(\fv_{\etavs_{\dimh}}(\Xv_i^\T\thetavs_{\dimh})-\fv_{\etavs_{\dimh}}(\Xv_i^\T\thetav_{l^*}))\varepsilon_i.
\end{EQA}
We estimate using the smoothness of \(\fv_{\etavs}\)
\begin{EQA}[c]
|\fv_{\etavs_{\dimh}}(\Xv_i^\T\thetavs_{\dimh})-\fv_{\etavs}(\Xv_i^\T\thetav_{l^*})|\le \CONST s_{\Xv}\|\thetav_{l^*}-\thetavs_{\dimh}\|\le \CONST s_{\Xv}\tau.
\end{EQA}
Furthermore the first order criteria of maximality give for some \(\thetavd\in{\thetavs_{\dimh}}^\perp\)
\begin{EQA}[c]
\E\left[(g(\Xv_i)-\fv_{\etavs_{\dimh}}(\Xv_i^\T\thetavs_{\dimh}))\fv'_{\etavs_{\dimh}}(\Xv_i^\T\thetavs_{\dimh}) \Xv^\T\thetavd\right]=0,
\end{EQA}
We estimate with Taylor expansion
\begin{EQA}
&&\nquad\left\|(\fv_{\etavs_{\dimh}}(\Xv_i^\T\thetavs_{\dimh})-\fv_{\etavs_{\dimh}}(\Xv_i^\T\thetav_{l^*}))-\fv'_{\etavs_{\dimh}}(\Xv_i^\T\thetavs_{\dimh}) \Xv^\T\nabla\Phi_{\thetavs_{\dimh}}(\varphi_{\thetav_{l^*}}-\varphi_{\thetavs_{\dimh}})\right\|\\
	&\le& \CONST\sqrt{\dimh}\|\thetav_{l^*}-\thetavs_{\dimh}\|.
\end{EQA}
Furthermore with the bounded differences inequality
\begin{EQA}
&&\nquad\P\Big(n \left|(P_n-\P)(g(\Xv_i)-\fv_{\etavs_{\dimh}}(\Xv_i^\T\thetavs_{\dimh}))(\fv_{\etavs_{\dimh}}(\Xv_i^\T\thetavs_{\dimh})-\fv_{\etavs_{\dimh}}(\Xv_i^\T\thetav_{l^*}))\right|\\
	&&\phantom{\P\Big(n |(P_n-\P)(g(\Xv_i)-\fv_{\etavs_{\dimh}}(\Xv_i^\T\thetavs_{\dimh})) }\ge \sqrt{\xx}\CONST_{bias} \CONST s_{\Xv}\tau \Big)\le \ex^{-\xx}.
\end{EQA}
Consequently with probability greater than \(1-\ex^{-\xx}\)
\begin{EQA}
\LL(\tilde \upsilonv^{(0)},\upsilonvs)&\ge& -n\CONST^2 s_{\Xv}^2\tau^2-\CONST_{bias}\CONST\left( s_{\Xv}\tau \sqrt{\xx}+ n\sqrt{\dimh}\tau^2 \right)\\
	&&+\sum_{i=1}^n(\fv_{\etavs_{\dimh}}(\Xv_i^\T\thetavs_{\dimh})-\fv_{\etavs_{\dimh}}(\Xv_i^\T\thetav_{l^*}))\varepsilon_i .
\end{EQA}
Clearly we have due to \((\mathbf{Cond}_{\varepsilon})\) for \(\lambda \le \sqrt{n}\tilde \gm/(\CONST s_{\Xv}\tau)\)
\begin{EQA}
&&\nquad\P^{\varepsilon}\left(\sum_{i=1}^n(\fv_{\etavs_{\dimh}}(\Xv_i^\T\thetavs_{\dimh})-\fv_{\etavs_{\dimh}}(\Xv_i^\T\thetav_{l^*}))\varepsilon_i \ge \sqrt{n}t \right)\\
	&\le &\exp\{-\lambda t\}\E^{\varepsilon}\left[ \exp\{ \lambda \sum_{i=1}^n(\fv_{\etavs_{\dimh}}(\Xv_i^\T\thetavs_{\dimh})-\fv_{\etavs_{\dimh}}(\Xv_i^\T\thetav_{l^*}))\varepsilon_i/\sqrt{n}\}\right]\\
	&\le &\exp\{-\lambda t\}\prod_{i=1}^n \E^{\varepsilon}\left[ \exp\{ \lambda (\fv_{\etavs_{\dimh}}(\Xv_i^\T\thetavs_{\dimh})-\fv_{\etavs_{\dimh}}(\Xv_i^\T\thetav_{l^*}))\varepsilon_i/\sqrt{n}\}\right]\\
	&\le &\exp\{-\lambda t+\tilde \nu^2 \CONST^2 s_{\Xv}^2\tau^2\lambda^2/2\}.
\end{EQA}
Setting \(\lambda=\frac{t}{\tilde\nu^2\CONST^2 s_{\Xv}^2\tau^2}\) we get
\begin{EQA}
\P^{\varepsilon}\left(\sum_{i=1}^n(\fv_{\etavs_{\dimh}}(\Xv_i^\T\thetavs_{\dimh})-\fv_{\etavs_{\dimh}}(\Xv_i^\T\thetav_{l^*}))\varepsilon_i \ge \sqrt{n}t \right)&\le &\exp\left\{-\frac{t^2}{2\tilde\nu^2\CONST^2 s_{\Xv}^2\tau^2}\right\}.
\end{EQA}
With \(t=\tilde\nu\CONST s_{\Xv}\tau\sqrt{2\xx}\) and \(\xx\le 2\tilde\nu^2\tilde \gm^2n/(\CONST^2 s_{\Xv}^2\tau^2)\) this gives
\begin{EQA}
\P^{\varepsilon}\left(\sum_{i=1}^n(\fv_{\etavs_{\dimh}}(\Xv_i^\T\thetavs_{\dimh})-\fv_{\etavs_{\dimh}}(\Xv_i^\T\thetav_{l^*}))\varepsilon_i \ge \tilde\nu\CONST s_{\Xv}\tau\sqrt{2n \xx} \right)&\le &\ex^{-\xx}.
\end{EQA}
Consequently
\begin{EQA}[c]
\P\left( \LL_{\dimh}(\tilde \upsilonv^{(0)},\upsilonvs_{\dimh})\le  -\CONST\left\{ (1+\CONST_{bias}\sqrt{\dimh})n\tau^2 +(1+  \CONST_{bias})\sqrt{\xx}\tau\sqrt{n}\right\} \right)\le 2\ex^{-\xx}.
\end{EQA}
For the second claim note that
by Lemma \ref{lem: conditions example} the conditions \( {(\CS \DF_{1})} \) and \( {(\cc{L}_{0})} \) from Section \ref{sec: conditions semi} hold for all \(\rr\le \sqrt n \rr^\circ\). We define
\begin{EQA}[c]
\KL(\xx)\eqdef  \CONST\left\{ (1+\CONST_{bias}\sqrt{\dimh})n\tau^2 +(1+  \CONST_{bias})\sqrt{\xx}\tau\sqrt{n}\right\}.
\end{EQA}
This implies with Lemma \ref{lem: conditions example} and Theorem \ref{theo: large def with K} that 
\begin{EQA}
\RR(\xx)&\le& \CONST\dimh^{3/2} \sqrt{\dimtotal(1+\CONST_{bias}\log(n))+\xx+ (1+\CONST_{bias}\sqrt{\dimh})n\tau^2+\sqrt{n}\tau\sqrt{\xx}}\\
	&\le&\CONST\dimh^{3/2} \sqrt{\dimtotal(1+\CONST_{bias}\log(n))+\xx}\\
		&& +\CONST\dimh^{3/2}\sqrt{ (1+\CONST_{bias}\sqrt{\dimh})n\tau^2+\sqrt{n}\tau\sqrt{\xx}}.
\end{EQA}
We use that \(\tau=o(\dimh^{-3/2})\) if \(\CONST_{bias}=0\) and \(\tau=o(\dimh^{-11/4})\) if \(\CONST_{bias}>0\) to find
\begin{EQA}
\RR(\xx)&\le&\CONST\dimh^{3/2} \sqrt{\dimtotal(1+\CONST_{bias}\log(n))+\xx} +\CONST(\sqrt{n}+\dimh^{1/2}n^{1/4}).
\end{EQA}
Repeating the same arguments as in Section \ref{sec: single index large dev} we can infer that with probability greater than \(1-2\ex^{-\xx}\) the sequence satisfies \((\ups_{k,k(+1)})\subset \Upss(\RR)\) where
\begin{EQA}\label{eq: bound for RR}
\RR(\xx)&\le& \CONST \sqrt{\dimtotal(1+\CONST_{bias}\log(n))+\xx+ (1+\CONST_{bias}\sqrt{\dimh})n\tau^2+\sqrt{n}\tau\sqrt{\xx}}.
\end{EQA}
Furthermore with Lemma \ref{lem: conditions example}
\begin{EQA}[c]
\eps\eqdef \delta(\rr)/\rr+\omega=\CONST \frac{\dimh^{3/2}+\CONST_{bias}\dimh^{5/2}}{\sqrt{n}}.
\end{EQA}
Consequently for moderate \(\xx\) we find if \(\CONST_{bias}=0\) that 
\begin{EQA}[c]
\eps\RR(\xx)=O\left( \dimh^{3/2}/\sqrt{n}\right)O\left( \tau\sqrt{n}+\sqrt{\tau}n^{1/4}\right)+O(\dimh^{2}/\sqrt{n}),
\end{EQA}
such that \(\eps\RR(\xx)\to 0\) if \(\tau=o(\dimh^{-3/2})\). While \(\eps\sqrt{\zz(\xx)}=O(\dimh^{2}/\sqrt{n})\to 0\). 
If \(\CONST_{bias}>0\) we find
\begin{EQA}[c]
\eps\RR(\xx)=O\left( \dimh^{5/2}/\sqrt{n}\right)O\left( \tau \dimh^{1/4}\sqrt{n}+\sqrt{\tau}n^{1/4}\right)+O(\dimh^{3}\log(n)/\sqrt{n}),
\end{EQA}
such that it suffices to ensure that \(\tau=o(\dimh^{-11/4})\) since then \(\dimh^{5/2}\sqrt{\tau}n^{-1/4}=o(\dimh^{-3/8})\to 0\), due to \(n\ge O(\dimh^{6}\log(n)^2)\). In this case \(\eps\sqrt{\zz(\xx)}=O(\dimh^{6}/\sqrt{n})\to 0\).
This gives \((A_3)\) and completes the proof. 
	

\end{proof}

\subsection{Proof of Lemma \ref{lem: conditions for convergence are satisfied}}
\subsubsection{Auxiliary results}
First we need the following uniform bounds:
\begin{lemma}\label{lem: bounds for hessian}
There is a generic constant \(\CONST>0\) such that for any pair \(\ups,\upsd\in\Upss(\RR)\) with \(\varsigmav_{i,\dimh}\) from \eqref{eq: def of varsigmav}
\begin{EQA}
&&\nquad\|\nabla\varsigmav_{i,\dimh}(\upss)\|	\le \CONST(\|\fv'_{\etavs}\|_{\infty}+\|\fv''_{\etavs}\|_{\infty}),\label{eq: score bound for ED 02}\\
&&\nquad\|\DF^{-1/2}_{\dimh}\nabla\varsigmav_{i,\dimh}(\ups)-\DF^{-1/2}_{\dimh}\nabla\varsigmav_{i,\dimh}(\upsd)\| \label{eq: score bound for ED2}\\
	&\le& \frac{\CONST}{c_{\DF}\sqrt{n}}\dimh\left( \dimh^{3/2}+\left( C_{\|\fv\|}+\frac{\dimh^2(\RR+\rr^*)}{nc_{\DF}^2}\right)\dimh^{1/2}\right)\|\DF_{\dimh}(\ups-\upsd)\|. 
\end{EQA}
\end{lemma}
\begin{proof}
Since \(\nabla_{\etav}^{2}\zeta(\ups)=0\) we can estimate with help of Lemma \ref{lem: D_0 dimh is boundedly invertable}
\begin{EQA}[c]
\|\nabla\varsigmav_{i,\dimh}(\upss)\|\le \|\nabla_{\thetav}\varsigmav_{i,\dimh}(\upss)\|+\|\nabla_{\etav}\varsigmav_{i,\dimh}(\upss)\|.
\end{EQA}
We estimate separately
\begin{EQA}
\|\nabla_{\thetav}\varsigmav_{i,\dimh}(\upss)\|&\le& \|\fv''_{\etavs}(\Xv_{i}^{\T}\thetavs) \nabla\Phi(\thetavs)^\T\Xv_{i}\Xv_{i}^\T\nabla\Phi(\thetavs)\|\\
	&&+\|\fv'_{\etavs}(\Xv_{i}^{\T}\thetavs)\Xv_{i}\nabla^2\Phi(\thetav^\T\Xv_i)[\Xv_{i},\cdot,\cdot] \|\\
	&\le&\CONST_0 s_{\Xv}^2\left(|\fv'_{\etavs}(\Xv_{i}^{\T}\thetav)|+ |\fv''_{\etavs}(\Xv_{i}^{\T}\thetav)|\right)\le \CONST(\|\fv'_{\etavs}\|_{\infty}+\|\fv''_{\etavs}\|_{\infty}),
\end{EQA}
This gives \eqref{eq: score bound for ED 02}. For the proof of \eqref{eq: score bound for ED2} we again use \(\nabla_{\etav}^{2}\zeta(\ups)=0\) and estimate with help of Lemma \ref{lem: D_0 dimh is boundedly invertable} 
\begin{EQA}
&&\nquad\|\DF^{-1/2}_{\dimh}\nabla\varsigmav_{i,\dimh}(\ups)-\DF^{-1/2}_{\dimh}\nabla\varsigmav_{i,\dimh}(\upsd)\|\le \frac{1}{c_{\DF}\sqrt{n}}\|\nabla\varsigmav_{i,\dimh}(\ups)-\nabla\varsigmav_{i,\dimh}(\upsd)\|\\
	&\le& \frac{1}{c_{\DF}\sqrt{n}}\left(\|\nabla_{\thetav}\varsigmav_{i,\dimh}(\ups)-\nabla_{\thetav}\varsigmav_{i,\dimh}(\upsd)\|+2\|\nabla_{\etav}\varsigmav_{i,\dimh}(\ups)-\nabla_{\etav}\varsigmav_{i,\dimh}(\upsd)\| \right).
\end{EQA}
We calculate separately
\begin{EQA}
&&\nquad\|\nabla_{\thetav}\varsigmav_{i,\dimh}(\ups)-\nabla_{\thetav}\varsigmav_{i,\dimh}(\upsd)\|\\
	&\le& s_{\Xv}^2\|\fv''_{\etav}(\Xv_{i}^{\T}\thetav) \nabla\Phi(\thetav)\nabla\Phi(\thetav)^{\T}-\fv''_{\etavd}(\Xv_{i}^{\T}\thetavd) \nabla\Phi({\thetavd}^\T\Xv_i)\nabla\Phi({\thetavd}^\T\Xv_i)^{\T}\|\\
	&&+s_{\Xv}^2\|\fv'_{\etav}(\Xv_{i}^{\T}\thetav)\Xv_{i}^\T\nabla^2\Phi(\thetav^\T\Xv_i)-\fv'_{\etavd}(\Xv_{i}^{\T}\thetavd)\Xv_{i}^\T\nabla^2\Phi({\thetavd}^\T\Xv_i)\|.
\end{EQA}
We again separately estimate
\begin{EQA}
&&\nquad\|\fv''_{\etav}(\Xv_{i}^{\T}\thetav) \nabla\Phi(\thetav)\nabla\Phi(\thetav)^{\T}-\fv''_{\etavd}(\Xv_{i}^{\T}\thetavd) \nabla\Phi({\thetavd}^\T\Xv_i)\nabla\Phi({\thetavd}^\T\Xv_i)^{\T}\|\\
	&\le& \|[\fv''_{\etav}(\Xv_{i}^{\T}\thetav)-\fv''_{\etavd}(\Xv_{i}^{\T}\thetavd)] \nabla\Phi(\thetav)\nabla\Phi(\thetav)^{\T}\|\\
		&&+\|\fv''_{\etavd}(\Xv_{i}^{\T}\thetavd) [\nabla\Phi(\thetav)-\nabla\Phi({\thetavd}^\T\Xv_i)]\nabla\Phi(\thetav)^{\T}\|\\
		&&+ \|\fv''_{\etavd}(\Xv_{i}^{\T}\thetavd) \nabla\Phi({\thetavd}^\T\Xv_i)[\nabla\Phi(\thetav)^{\T}-\nabla\Phi({\thetavd}^\T\Xv_i)^{\T}\|.
\end{EQA}
We estimate using that \(\|\nabla\Phi(\thetav)\nabla\Phi(\thetav)^\T\|\le1\)
\begin{EQA}
&&\nquad\|[\fv''_{\etav}(\Xv_{i}^{\T}\thetav)-\fv''_{\etavd}(\Xv_{i}^{\T}\thetavd)] \nabla\Phi(\thetav)\nabla\Phi(\thetav)^{\T}\|\\
&\le&\|\fv''_{\etav}(\Xv_{i}^{\T}\thetav)-\fv''_{\etavd}(\Xv_{i}^{\T}\thetavd)\|\\
 &\le&\|\fv''_{\etav}(\Xv_{i}^{\T}\thetav)-\fv''_{\etavd}(\Xv_{i}^{\T}\thetav)\|+	\|\fv''_{\etavd}(\Xv_{i}^{\T}\thetav)-\fv''_{\etavd}(\Xv_{i}^{\T}\thetavd)\|.
 \end{EQA}
Remember that due to the structure of the basis
 \begin{EQA}
&&\nquad|N(j)|\eqdef\Big|\Big\{ k\in\{(2^{j}-1)17,\ldots,(2^{j+1}-1)17-1\}:\\
	&& |\basX'_{k}(\Xv_{i}^{\T}\thetav')-\basX'_{k}(\Xv_{i}^{\T}\thetav)|\vee |\basX''_{k}(\Xv_{i}^{\T}\thetav')-\basX''_{k}(\Xv_{i}^{\T}\thetav)|\vee |\basX'_{k}(\Xv_{i}^{\T}\thetav)|> 0 \Big\}\Big|\\
	&&\le 34.
\end{EQA}
We get with the same arguments as in the proof of Lemma \ref{lem: bounds for scores and so on}
 \begin{EQA}
&&\nquad\|[\fv''_{\etav}(\Xv_{i}^{\T}\thetav)-\fv''_{\etavd}(\Xv_{i}^{\T}\thetavd)] \nabla\Phi(\thetav)\nabla\Phi(\thetav)^{\T}\|\\
	&\le& \frac{\sqrt{34}}{\sqrt{n}c_{\DF}}\left( \|\psi''\|\dimh^{5/2}+  \|\psi'''\|\left( C_{\|\fv\|}+\frac{\dimh^2(\RR+\rr^*)}{\sqrt{n}c_{\DF}}\right)\dimh^{3/2}\right)\\
	 	&&\|\DF_{\dimh}(\ups-\upsd)\|.
\end{EQA}
For the other two summands we estimate
\begin{EQA}
&&\nquad\|\fv''_{\etavd}(\Xv_{i}^{\T}\thetavd) [\nabla\Phi(\thetav)-\nabla\Phi({\thetavd}^\T\Xv_i)]\nabla\Phi(\thetav)^{\T}\|\\
	&\le& \|\fv''_{\etavd}(\Xv_{i}^{\T}\thetavd)\|\|\left\{\nabla\Phi(\thetav)-\nabla\Phi({\thetavd}^\T\Xv_i)\right\}\nabla\Phi({\thetavd}^\T\Xv_i)
	\|.
\end{EQA}
We can use the smoothness of \(\phi:\R^{\dimp-1}\to S_1\subset \R^{\dimp}\) to find a constant \(\CONST_1\) such that
\begin{EQA}
&&\nquad\|\fv''_{\etavd}(\Xv_{i}^{\T}\thetavd) [\nabla\Phi(\thetav)-\nabla\Phi({\thetavd}^\T\Xv_i)]\nabla\Phi(\thetav)^{\T}\|\\
	&\le& \|\fv''_{\etavd}(\Xv_{i}^{\T}\thetavd)\|\CONST_2\|\thetav-\thetavd\|\\
	&\le& \CONST_1\|\thetav-\thetavd\|\|\psi''\| \sum_{j=0}^{j_{m}-1} \sum_{k\in N(j)}\eta^{\circ}_k 2^{5j/2}\\
	&\le& 17\CONST_1\|\thetav-\thetavd\|\|\psi''\|\left( C_{\|\fv\|}+\frac{\dimh^2(\RR+\rr^*)}{\sqrt{n}c_{\DF}}\right)\dimh^{1/2}.
\end{EQA}
We continue with
\begin{EQA}
&&\nquad\|\fv'_{\etav}(\Xv_{i}^{\T}\thetav)\Xv_{i}^\T\nabla^2\Phi(\thetav^\T\Xv_i)-\fv'_{\etavd}(\Xv_{i}^{\T}\thetavd)\Xv_{i}^\T\nabla^2\Phi({\thetavd}^\T\Xv_i)\|\\
	&\le& \|\fv'_{\etav}(\Xv_{i}^{\T}\thetav)-\fv'_{\etavd}(\Xv_{i}^{\T}\thetavd)\|\|\Xv_{i}^\T\nabla^2\Phi({\thetavd}^\T\Xv_i)\|\\
	&&+\|\fv'_{\etav}(\Xv_{i}^{\T}\thetav)\|\Xv_{i}^\T\nabla^2\Phi(\thetav^\T\Xv_i)-\Xv_{i}^\T\nabla^2\Phi({\thetavd}^\T\Xv_i)\|.
\end{EQA}
Using the smoothness of \(\phi:\R^{\dimp-1}\to S_1\subset \R^{\dimp}\) we find constants \(\CONST_2,\CONST_3\) such that with the same argument as in the proof of Lemma \ref{lem: bounds for scores and so on}
\begin{EQA}
&&\nquad\|\fv'_{\etav}(\Xv_{i}^{\T}\thetav)\Xv_{i}^\T\nabla^2\Phi(\thetav^\T\Xv_i)-\fv'_{\etavd}(\Xv_{i}^{\T}\thetavd)\Xv_{i}^\T\nabla^2\Phi({\thetavd}^\T\Xv_i)\|\\
	&\le& \frac{\sqrt{34}}{\sqrt{n}c_{\DF}}\dimh^{1/2}\left( \CONST_2 s_{\Xv}\|\psi''\|+\|\psi'\|+s_{\Xv}^2\CONST_3\right)\left( C_{\|\fv\|}+\frac{\dimh^2(\RR+\rr^*)}{\sqrt{n}c_{\DF}}\right)\\
		&&\|\DF_{\dimh}(\ups-\upsd)\|.
\end{EQA}
Finally 
\begin{EQA}
&&\nquad\|\nabla_{\etav}\varsigmav_{i,\dimh}(\ups)-\nabla_{\etav}\varsigmav_{i,\dimh}(\upsd)\|\\
	&\le& \|\left(\nabla\Phi(\thetav)^{\T}-\nabla\Phi({\thetavd}^\T\Xv_i)^{\T}\right) \Xv_{i}\|\|\basX'(\thetav^\T\Xv_i)^\T \|\\
	&&+\|\nabla\Phi({\thetavd}^\T\Xv_i)^{\T} \Xv_{i}\|\|\basX'(\thetav^\T\Xv_i)-\basX'({\thetavd}^\T\Xv_i)\|.
\end{EQA}
We estimate separately
\begin{EQA}
 \|\left(\nabla\Phi(\thetav)^{\T}-\nabla\Phi(\thetavd)^{\T}\right) \Xv_{i}\|&\le& \CONST_4 s_{\Xv}^2\frac{1}{\sqrt{n}c_{\DF}}\|\DF_{\dimh}(\ups-\upsd)\|,\\
 \|\basX'(\thetav^\T\Xv_i)^\T \|&\le& \|\psi'\|_{\infty}\left(\sum_{j=0}^{j_{\dimh}}2^{3j}|N(j)|\right)^{1/2}\le \|\psi'\|_{\infty}\sqrt{34}\dimh^{3/2}.
\end{EQA}
Furthermore
\begin{EQA}
\|\nabla\Phi({\thetavd}^\T\Xv_i)^{\T} \Xv_{i}\|&\le& \CONST_5 s_{\Xv},\\
	\|\basX'(\thetav^\T\Xv_i)-\basX'({\thetavd}^\T\Xv_i)\|&\le& \|\psi''\|_{\infty}\sqrt{34}\dimh^{5/2}\frac{1}{\sqrt{n}c_{\DF}}\|\DF_{\dimh}(\ups-\upsd)\|.
\end{EQA}
Putting all estimates together gives \eqref{eq: score bound for ED2}.
\end{proof}

\subsubsection{Condition \( {(\CS \DF_{2})} \)}
Just as for the conditions \( {(\CS \DF_{1})} \) and \( {(\CS \DF_{0})} \) we can show:

\begin{lemma}
 \label{lem: condition ED_2}
We have \( {(\CS \DF_{2})} \) with 
\begin{EQA}
 \omega_2=\frac{1}{\sqrt{n}c_{\DF}},\, & \gm_2 =\sqrt{n}\tilde\gm c_{\DF}\dimh^{-1}\CONST(\RR,\dimtotal)^{-1},\,& \nu_{2}^2=\frac{\tilde\nu^2\dimh^2\CONST(\RR,\dimtotal)^2}{2c_{\DF}},
\end{EQA}
where with \(\CONST>0\) in \eqref{eq: score bound for ED2}
\begin{EQA}[c]\label{eq: def of CONST RR,dimh}
\CONST(\RR,\dimh)\eqdef \CONST\left( \dimh^{3/2}+\left( C_{\|\fv\|}+\frac{\dimh^2(\RR+\rr^*)}{\sqrt{n}c_{\DF}}\right)\dimh^{1/2}\right).
\end{EQA}
\end{lemma}
             
\begin{proof}
Lemma \ref{lem: bounds for hessian} gives for any \(\ups,\upsd\in\Ups(\rr)\) with \(\varsigmav_{i,\dimh}\) from \eqref{eq: def of varsigmav}
\begin{EQA}
&&\nquad \|\DF_{\dimh}^{-1}\nabla\varsigmav_{i,\dimh}(\upsilonv)- \DF^{-1}\nabla\varsigmav_{i,\dimh}(\upsilonvd)\|\\
&\le& \frac{\CONST}{c_{\DF}\sqrt{n}}\dimh\left( \dimh^{3/2}+\left( C_{\|\fv\|}+\frac{\dimh^2(\RR+\rr^*)}{\sqrt{n}c_{\DF}}\right)\dimh^{1/2}\right)\|\DF_{\dimh}(\ups-\upsd)\|\\
&\eqdef& \frac{1}{\sqrt{n}c_{\DF}}\dimh\CONST(\RR,\dimtotal)\|\DF_{\dimh}(\ups-\upsd)\|.
\label{eq: bound for hessian difference}
\end{EQA}
We get with \(\mu\le \gm_2\) and assumption \((\mathbf{ Cond}_{\varepsilon})\) for any pair \(\gammav_1,\gammav_2\in\{\|\gammav\|=1\}\)
\begin{EQA}
 &&\nquad\E_{\varepsilon} \exp\left\{ \frac{\mu}{\omega_2\|\DF_{\dimh}(\ups-\upsd)\|} \gammav_1^\T\left( \DF_{\dimh}^{-1}\nabla^2\left\{\zeta(\upsilonv)- \zeta(\upsilonvd)\right\}\right)\gammav_2\right\}\\
  &=& \E_{\varepsilon} \exp\Bigg\{ \frac{\mu}{\omega_2\|\DF_{\dimh}(\ups-\upsd)\|}\sum_{i=1}^{n} \varepsilon_i \gammav_1^\T\left( \DF_{\dimh}^{-1}\nabla\left\{\varsigmav_{i,\dimh}(\upsilonv)-\varsigmav_{i,\dimh}(\upsilonvd)\right\}\right)\gammav_2\Bigg\}\\
  &=& \prod_{i=1}^{n}\E_{\varepsilon} \exp\Bigg\{ \frac{\mu}{\omega_2\|\DF_{\dimh}(\ups-\upsd)\|} \varepsilon_i 
  	\gammav_1^\T\left( \DF_{\dimh}^{-1}\nabla\left\{\varsigmav_{i,\dimh}(\upsilonv)-\varsigmav_{i,\dimh}(\upsilonvd)\right\}\right)\gammav_2\Bigg\}\\
  	  &\le&\prod_{i=1}^{n}\exp\left\{\frac{\tilde\nu^2 \mu^2}{2\omega^2_2\|\DF_{\dimh}(\ups-\upsd)\|^2}\left(\gammav_1^\T\left( \DF_{\dimh}^{-1}\nabla\left\{\varsigmav_{i,\dimh}(\upsilonv)-\varsigmav_{i,\dimh}(\upsilonvd)\right\}\right)\gammav_2\right)^2\right\}.
\end{EQA}
With \eqref{eq: bound for hessian difference} this implies
\begin{EQA}
&&\nquad\sup_{\substack{\gammav_1,\gammav_2\in \R^{\dimtotal}\\ \|\gammav_i\|=1}}\log\E_{\varepsilon} \exp\left\{ \frac{\mu}{\omega_2\|\DF_{\dimh}(\ups-\upsd)\|} \gammav_1^\T\left( \DF_{\dimh}^{-1}\nabla^2\zeta(\upsilonv)- \DF^{-1}\nabla^2\zeta(\upsilonvd)\right)\gammav_2\right\}\\
  &\le& \frac{\tilde\nu^2 \mu^2}{2c_{\DF}}\dimh^2\CONST(\RR,\dimtotal)^2.
\end{EQA}
\end{proof}

\subsubsection{Bound for Hessian}\label{sec: bound of hessian}
To control the deviation of \(\DF^{-1}\nabla\zeta(\upss)\) we apply the following Theorem of \cite{Tropp2012}:

\begin{theorem}[Corollary 3.7 of \cite{Tropp2012}]\label{theo: tropp matrix bound}
Consider a finite sequence \((\bb{M}_i)_{i=1}^n\subset\R^{\dimtotal\times\dimtotal}\) of independent, selfadjoint, random matrices. Assume that there is a function \(g:(0,\infty)\to \R_+\) and a sequence of matrices \((\bb{A}_i)\subset\R^{\dimtotal\times\dimtotal}\) that satisfy for all \(\mu>0\)
\begin{EQA}
\E\ex^{\mu\bb{M}_i}\preceq \ex^{g(\mu)\bb{A}_i}, & \text{ where }& \bb{M}\preceq\bb{M'} \Leftrightarrow \gammav^\T\bb{M}\gammav\le  \gammav^\T\bb{M}\gammav, \, \forall \gammav\in\R^{\dimtotal}.
\end{EQA}
Then for all \(t\in \R\)
\begin{EQA}
\P\left( \left\|\sum_{i=1}^{n}\bb{M}_i \right\|\ge t\right)\le \dimtotal\inf_{\mu}\exp\left\{-t\mu+g(\mu)\tau \right\}, &\text{ where }&\tau\eqdef \left\| \sum_{i=1}^{n}\bb{A}_i \right\|.
\end{EQA}
\end{theorem}

\begin{lemma}\label{lem: bound for exponential of hessian}
We have for \(\mu\le \tilde \gm \)
\begin{EQA}[c]
\E\exp\left\{\mu \DF^{-1}\nabla^2\zeta(\upss)\right\}\preceq \exp\left\{g(\mu)\diag(1,\ldots,1)\right\},
\end{EQA}
where 
\begin{EQA}[c]
g(\mu)=\begin{cases} \frac{\tilde \nu^2 \CONST^2(\|\fv'_{\etavs}\|_{\infty}+\|\fv''_{\etavs}\|_{\infty})^2\mu^2}{2}, & \text{ if } \mu\le \sqrt{n}\tilde \gm \CONST^{-1}(\|\fv'_{\etavs}\|_{\infty}+\|\fv''_{\etavs}\|_{\infty})^{-1}\\
\infty, &\text{ otherwise.}
		\end{cases}
\end{EQA}
\end{lemma}
\begin{proof}
Due to Lemma \ref{lem: bounds for hessian}
\begin{EQA}
&&\nquad \DF^{-1}\nabla\varsigmav_{i,\dimh}(\upss)\\
	&\preceq&  \diag\left(\frac{1}{\sqrt{n}}\CONST(\|\fv'_{\etavs}\|_{\infty}+\|\fv''_{\etavs}\|_{\infty}),\ldots, \frac{1}{\sqrt{n}}\CONST(\|\fv'_{\etavs}\|_{\infty}+\|\fv''_{\etavs}\|_{\infty})\right).
\end{EQA}
Thus denoting \(\CONST_1\eqdef \CONST(\|\fv'_{\etavs}\|_{\infty}+\|\fv''_{\etavs}\|_{\infty})\)
\begin{EQA}
\exp\left\{\mu \DF^{-1}\nabla^2\zeta(\upss)\right\}&=& \exp\left\{\mu \sum_{i=1}^{n}\DF^{-1}\nabla\varsigmav_{i,\dimh}(\upss)\varepsilon_i\right\}\\
	&\preceq& \exp\left\{\mu \sum_{i=1}^{n}\varepsilon_i\diag\left(\frac{1}{\sqrt{n}}\CONST_1,\ldots, \frac{1}{\sqrt{n}}\CONST_1\right)\right\}.
\end{EQA}
Consequently we obtain due to the independence of the \(\varsigmav_{i,\dimh}(\upss)\) and assumption \((\mathbf{ Cond}_{\varepsilon})\) for \(\mu\le \sqrt{n}\tilde \gm \CONST_1^{-1}\)
\begin{EQA}
&&\nquad\E\exp\left\{\mu \DF^{-1}\nabla^2\zeta(\upss)\right\}\\
	&\le& \prod_{i=1}^n\diag\left(\E\exp\left\{\frac{\mu}{\sqrt{n}}\varepsilon_i\CONST_1\right\},\ldots, \E\exp\left\{\frac{\mu}{\sqrt{n}}\varepsilon_i\CONST_1\right\}\right)\\
	&\le&\diag\left(\exp\left\{\frac{\tilde \nu^2 \mu^2}{2}\CONST_1^2\right\},\ldots, \exp\left\{\frac{\tilde \nu^2 \mu^2}{2}\CONST_1^2\right\}\right)\\
	&=&\exp\left\{ \frac{\tilde \nu^2 \CONST_1^2\mu^2}{2} \diag(1,\ldots,1)\right\}.
\end{EQA}
\end{proof}

\begin{lemma}\label{lem: bound for norm of hessian}
We have with \(\CONST(\|\fv'_{\etavs}\|_{\infty}+\|\fv''_{\etavs}\|_{\infty})\) and if \(\xx\le \frac{1}{2}(\tilde \nu^2 n\tilde \gm^2  \break - \log(\dimtotal))\)
\begin{EQA}[c]
\P\left( \left\|\DF^{-1}\nabla^2\zeta(\upss) \right\|\ge \tilde \nu \CONST(\|\fv'_{\etavs}\|_{\infty}+\|\fv''_{\etavs}\|_{\infty})\sqrt{2\xx+\log(\dimtotal)}\right)\le \ex^{-2\xx}.
\end{EQA}
\end{lemma}
\begin{proof}
With Lemma \ref{lem: bound for exponential of hessian} and Theorem \ref{theo: tropp matrix bound} we obtain for 
\begin{EQA}[c]
t\le \sqrt{n}\tilde \gm \CONST^{-1}(\|\fv'_{\etavs}\|_{\infty}+\|\fv''_{\etavs}\|_{\infty})^{-1},
\end{EQA}
that
\begin{EQA}
\P\left( \left\|\DF^{-1}\nabla^2\zeta(\upss) \right\|\ge t\right)&\le& \dimtotal\inf_{\mu}\exp\left\{-t\mu+\frac{\tilde \nu^2 \CONST^2(\|\fv'_{\etavs}\|_{\infty}+\|\fv''_{\etavs}\|_{\infty})^2\mu^2}{2} \right\}\\
&=&\inf_{\mu}\exp\left\{-t\mu+\tilde \nu^2 \CONST(\|\fv'_{\etavs}\|_{\infty}+\|\fv''_{\etavs}\|_{\infty})^2 \frac{\mu^2}{2} \right\}\\
	&=&\exp\left\{-\frac{t^2}{2\tilde \nu^2 \CONST(\|\fv'_{\etavs}\|_{\infty}+\|\fv''_{\etavs}\|_{\infty})^2 }\right\}.
\end{EQA}
Defining \(t(\xx)\) via
\begin{EQA}[c]
\P\left( \left\|\DF^{-1}\nabla^2\zeta(\upss) \right\|\ge t(\xx)\right)=\ex^{-\xx},
\end{EQA}
we find
\begin{EQA}
t(\xx)\le \tilde \nu \CONST(\|\fv'_{\etavs}\|_{\infty}+\|\fv''_{\etavs}\|_{\infty})\sqrt{2\xx+\log(\dimtotal)},& \text{ if }& \xx\le \frac{1}{2}\left(\tilde \nu^2 n\tilde \gm^2  - \log(\dimtotal)\right).
\end{EQA}
\end{proof}

\subsubsection{Proof of Lemma}
Lemma \ref{lem: bound for norm of hessian} together with Lemma \ref{lem: condition ED_2} gives that in this setting
\begin{EQA}
\frac{\sqrt{1-\corrDF}}{2\sqrt{2}(1+\sqrt \corrDF)}\kappa(\xx,\RR)&\le&\CONST(\|\fv'_{\etavs}\|_{\infty}+\|\fv''_{\etavs}\|_{\infty})^2\sqrt{2\xx+\log(\dimtotal)}/\sqrt{n}\\
	&&+ \frac{\CONST\zzq(\xx,3\dimtotal)}{n} \left(\dimh^{5/2}+\frac{\dimh^{7/2}(\RR+\rr^*)}{\sqrt{n}c_{\DF}}\right)\RR\\
	&&+ \delta(\RR+\rups),	
 \end{EQA} 
if \(\xx\) is chosen moderately. As above 
\begin{EQA}
\zzq(\xx,3\dimtotal)=O(\sqrt{\xx+\dimtotal})=O(\rups), &\quad\|\DF^{-1}\|\le 1/(\sqrt{n}c_{\DF})\\
\delta(\rr)/\rr=O({\dimtotal}^{3/2}+\CONST_{bias}\dimh^{5/2})/\sqrt{n}.
\end{EQA}
In both cases  \(\CONST_{bias}=0\) and \(\CONST_{bias}>0\) the dominating term is the third summand \(\delta(\RR+\rups)\).

Lemma \ref{lem: conditions of initial guess met} tells us that
\begin{EQA}[c]
\RR=O\left(\sqrt{\dimtotal(1+\CONST_{bias}\log(n))+n\tau^2+\sqrt{\xx n}\tau}\right).
\end{EQA}
In case \(\CONST_{bias}=0\) this means that for moderate \(\xx\)
\begin{EQA}[c]
\kappa(\xx,\RR)\le \CONST\left(\frac{{\dimtotal}^2}{\sqrt{n}}+{\dimtotal}^{3/2}\tau+\frac{\sqrt{\tau} {\dimtotal}^{3/2}}{n^{1/4}}\right)(1+o(1)),
\end{EQA}
which tends to zero if \({\dimtotal}^4/n\to 0\) and \(\tau =o({\dimtotal}^{-3/2})\).

In case \(\CONST_{bias}>0\) we have
\begin{EQA}
\rups=\CONST \sqrt{\dimtotal\log(n)}, &\quad \RR=\CONST \sqrt{{\dimtotal}\log(n)+\sqrt{\dimtotal}n\tau^2+\sqrt{\xx n}\tau}.
\end{EQA}
Consequently
\begin{EQA}
\kappa(\xx,\RR)&\le& \CONST \bigg({\dimtotal}^{3}\log(n)/\sqrt{n}+{\dimtotal}^{11/4}\tau\\
	&&+ n^{-1/4}\dimh^{5/2}\sqrt{\tau})  \bigg)(1+o(1)),
\end{EQA}
which tends to \(0\) if \(\dimh^{3}\log(n)/n\to 0\) and \(\tau=o({\dimtotal}^{-11/4})\) since then 
\begin{EQA}[c]
n^{-1/4}\dimh^{5/2}\sqrt{\tau}=o(\dimh^{-3/8}).
\end{EQA}

\bibliographystyle{plain}
\bibliography{../../sources/semiquellen}
\end{document}